\documentclass[12pt,twoside,a4paper]{article} %scrreport

 \sloppypar \textheight 21cm \textwidth 14.4cm \oddsidemargin 0.95cm
\usepackage{amsmath}
\usepackage{amssymb}
\usepackage{amsthm}
\usepackage{mathrsfs}
\usepackage{dsfont}
\usepackage[latin1]{inputenc}

\usepackage{makeidx}

\usepackage[dvips]{graphicx} % für das Einfügen der EPS-Dateien
\usepackage[dvips]{color}    %                "

\usepackage{scrpage}
\renewcommand{\theequation}{\arabic{section}.\arabic{equation}}
\newcounter{saveeqn}
\newcounter{Qnqn}
\newcounter{altcounter}
\newcommand{\alpheqn}{\stepcounter{equation}
  \setcounter{saveeqn}{\value{equation}}%gleichungsnr (1.3a)
  \setcounter{equation}{0}
  \renewcommand{\theequation}{\arabic{section}.\arabic{saveeqn}\alph{equation}}}
\newcommand{\Aeqn}{%\stepcounter{equation}
  \setcounter{saveeqn}{\value{equation}}%gleichungsnr (A)
  \setcounter{equation}{0}
  \renewcommand{\theequation}{\Alph{equation}}}

\newcommand{\reseteqn}{\setcounter{equation}{\value{saveeqn}}%gleichungsnr (1.3)
  \renewcommand{\theequation}{\arabic{section}.\arabic{equation}}}

\newcommand{\Section}{\setcounter{equation}{0} \section}

\newtheorem{satz}{Proposition}[section]
\newtheorem{prop}[satz]{Proposition} %[section]
\newtheorem{lemma}[satz]{Lemma} %[section]
\newtheorem{korollar}[satz]{Corollary} %[section]
\newtheorem*{theorem}{Theorem}
\newtheorem*{vermutung}{Conjecture}
\newtheorem{vermutung2}[satz]{Conjecture}
\newtheorem{Theorem}{Theorem}

\theoremstyle{definition}
\newtheorem*{definition}{Definition}

\newtheorem*{remark}{\sc Remark}

\newenvironment{bemerkung}
   {%
    \begin{remark}%
            }  %
    {

    \hfill $\diamond$
    \end{remark}}

\newcommand{\fa}{{\: \cdot \:}}       % Punkt als Funktionsargument
\newcommand{\ray}{\mathscr{R}}
\newcommand{\cH}{\mathscr{H}}
\newcommand{\cD}{\mathscr{D}}
\newcommand{\const}{{\mathscr{C}}}

\newcommand{\C}{\mathds{C}}                 % Komplexe Zahlen
\newcommand{\R}{\mathds{R}}                 % Reelle Zahlen
\renewcommand{\H}{\mathds{H}}                 % Heisenberggruppe
\newcommand{\N}{\mathds{N}}                 % Nat"urliche Zahlen
\newcommand{\Z}{\mathds{Z}}                 % ganze Zahlen

\newcommand{\lieg}{\mathfrak{g}}               % Liealgebren g
\newcommand{\lieh}{\mathfrak{h}}               % Liealgebren h

\renewcommand{\S}{{\mathscr{S}}}            % Schwartz Funktionen
\newcommand{\Cnu}{C_{0}^\infty}             % C Null Unendlich
\newcommand{\Cu}{C^\infty}                  % C Unendlich
\newcommand{\ed}[1]{\frac{1}{#1}}          % Eins Durch
\newcommand{\dz}[1]{\tfrac{#1}{2}}                  % Durch Zwei
\newcommand{\conv}[2]{#1 \ast #2}           % Faltung
\newcommand{\gaussk}[1]{\lfloor#1\rfloor} % Gaussklammer
\newcommand{\nach}{\circ}                   % Hintereinanderausführung
\renewcommand{\d}[1]{\partial_{\mspace{-1mu}#1}\mspace{-1mu}}
\newcommand{\dt}{\d{t}}
\newcommand{\dx}{\d{x}}
\newcommand{\du}{\d{u}}
\newcommand{\vt}{\vert_{t=0}}

\newcommand{\grad}{\nabla}                  % Gradient

\newcommand{\s}[1]{#1^\prime}               % Strich
\renewcommand{\j}[1]{\langle #1 \rangle}      % japanische Klammer
\renewcommand{\t}[1]{\widetilde{#1}}         % Tilde
\newcommand{\h}[1]{\widehat{#1}}            % Dach
\newcommand{\ab}[1]{\overline{#1}}

\newcommand{\X}{{\cal{X}}}                  % charakteristische Funktion
\renewcommand{\P}{{\cal{P}}}
\newcommand{\U}{{\cal{U}}}              % {\mathscr{U}}

\renewcommand{\l}{\ell}                          % Schreibschrift l
\newcommand{\kphi}{\varphi}                  % kleines Phi
\DeclareMathOperator{\supp}{supp}           % Tr"ager
\DeclareMathOperator{\re}{Re}

\DeclareMathOperator{\sgn}{sgn}
\DeclareMathOperator{\singsupp}{singsupp} %

\newcommand{\skalar}[2]{\left(#1 \vert #2\right)}
\newcommand{\menge}[2]{\{#1;\; #2\}}   % Menge
\newcommand{\bigmenge}[2]{\big\{#1;\; #2\big\}}   % Menge
\newcommand{\Bigmenge}[2]{\Big\{#1;\; #2\Big\}}   % Menge

\newcommand{\teil}{\subseteq}                % Teilmenge
\newcommand{\intao}[1]{\left[#1\right[}     % Intervall abg.-offen
\newcommand{\intaa}[1]{\left[#1\right]}     % Intervall abg.-abgeschlossen
\newcommand{\pint}{\int_0^{\infty}}

\newcommand{\norm}[1]{|#1|}        % Norm auf Rn
\newcommand{\n}[1]{|#1|}        % Norm auf Rn
\newcommand{\bignorm}[1]{\bigl|#1\bigr|}        % Norm auf Rn
\newcommand{\Bignorm}[1]{\Bigl|#1\Bigr|}        % Norm auf Rn
\newcommand{\dnorm}[1]{\|#1\| }      % Doppelstrich Norm
\newcommand{\Bigdnorm}[1]{\Bigl\|#1\Bigr\| }      % Doppelstrich Norm
\newcommand{\Ri}[1]{\tfrac{#1}{\norm{#1}}}   %   x / |x|

\newcommand{\ohne}{\backslash}
\newcommand{\ueber}[2]{\genfrac{}{}{0pt}{}{#1}{#2}}

\newcommand{\absatz}{\vspace{0.2cm}}
\newcommand{\chap}{{\sc Chapter }}

\makeindex

\begin{document}

\begin{titlepage}
\begin{center}
\vspace{2.5cm}
{\Large $L^p$-estimates for the wave equation\\
\vspace{0.4cm} associated to the Gru\v{s}in operator}\\
\vspace{2.5cm} Dissertation\\
\vspace{0.2cm} zur Erlangung des Doktorgrades\\
\vspace{0.2cm} der Mathematisch-Naturwissenschaftlichen Fakultät\\
\vspace{0.2cm} der Christian-Albrechts-Universität\\
\vspace{0.2cm} zu Kiel\\

\vspace{2.5cm}
\includegraphics{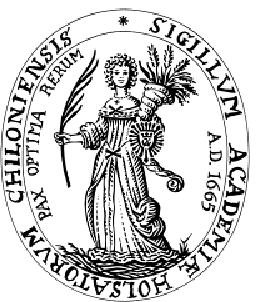}
\vspace{2.5cm}

{\footnotesize vorgelegt von}\\
\vspace{0.2cm}Ralf Meyer\\

\vspace{2.0cm}
Kiel\\
\vspace{0.2cm}Juni 2006\\

\end{center}
\end{titlepage}

\newpage
\thispagestyle{empty}

\vspace*{\fill}

\noindent\begin{tabular}{lr} \hspace{8cm} & \hspace{8cm}\\
Referent/in: & Prof. Dr. D. Müller\\
Korreferent/in: & Prof. Dr. H. König\\
Tag der mündlischen Prüfung: & 14.06.2006\\
Zum Druck genehmigt: & Kiel, den 3.07.2006\\
& \\
& \\
& gezeichnet\\
& Prof. Dr. J. Grotemeyer, Dekan
\end{tabular}

\newpage
\thispagestyle{empty}
\subsection*{Zusammenfassung}

Mit $G:=-(\d{x}²+x²\d{u}²)$ bezeichnen wir den \emph{Gru\v{s}in} Operator auf $\R^2$. Das
Cauchy-Problem der assoziierten Wellengleichung auf $\R \times \R^2$ ist gegeben durch
\begin{equation*}
    \Big(\frac{\partial^2}{\partial t^2}+G\Big) v =0,\quad
    v \rvert_{t=0}=f,\quad \frac{\partial v}{\partial t}\rvert_{t=0}=g,
\end{equation*}
wobei sich $t$ auf die Zeit bezieht und $f,\ g$ geeignete Funktionen sind. Die Lösung
dieses Problems ist formal gegeben durch
\begin{equation*}
  v(t,x,u):=[\cos(t \sqrt{G})f](x,u)+\Big[\frac{\sin(t\sqrt{G})}{\sqrt{G}}g\Big](x,u).
\end{equation*}
Das Thema dieser Dissertation sind Glattheitseigenschaften der Lösung $v$ in Abhängigkeit
von den Anfangsdaten. Wir betrachten dabei einen festen Zeitpunkt $t$. Die Glattheit
einer Lösung messen wir bezüglich Sobolev-Normen
$\dnorm{f}_{L_p^\alpha}:=\dnorm{(1+G)^{\alpha/2}f}_{L_p}$, definiert in Termen des
Differentialoperators $G$. $S_C$ bezeichne den Streifen $S_C:=\menge{(x,u)\in
\R^2}{\n{x}\leq C}$ im $\R^2$. Wir beweisen, dass für $1\leq p \leq \infty$ die Lösung
$v$ in $L_p^{-\alpha}$ liegt, falls unsere Anfangsdaten $f$ und $g$ in einem Streifen
$S_C$, $C>0$, getragene $L_p$-Funktionen sind und zudem $\alpha>\n{1/p-1/2}$ gilt. Hierzu
zeigen wir, dass sich für alle $C>0$ der \mbox{Operator}
$\exp(it\sqrt{G})(1+G)^{-\alpha/2}$, definiert auf dem Schwartzraum $\S$, zu einem
beschränkten Operator von $L_p(S_C)$ nach $L_p(\R^2)$ fortsetzen läßt, sofern
$\alpha>\n{1/p-1/2}$ gilt.

\newpage
\thispagestyle{empty}
\subsection*{Abstract}

Let $G:=-(\d{x}²+x²\d{u}²)$ denote the \emph{Gru\v{s}in} operator on $\R^2$. Consider the
Cauchy problem for the associated wave equation on $\R \times \R^2$, given by
\begin{equation*}
    \Big(\frac{\partial^2}{\partial t^2}+G\Big) v =0,\quad
    v \rvert_{t=0}=f,\quad \frac{\partial v}{\partial t}\rvert_{t=0}=g,
\end{equation*}
where $t$ denotes time and $f,\ g$ are suitable functions. The solution to this problem
is formally given by
\begin{equation*}
  v(t,x,u):=[\cos(t \sqrt{G})f](x,u)+\Big[\frac{\sin(t\sqrt{G})}{\sqrt{G}}g\Big](x,u).
\end{equation*}
The focus of this thesis lies on smoothness properties of the solution $v$ for fixed time
$t$ with respect to the initial data. Smoothness can be measured in terms of Sobolev
norms $\dnorm{f}_{L_p^\alpha}:=\dnorm{(1+G)^{\alpha/2}f}_{L_p}$, defined in terms of the
differential operator $G$. Let $S_C$ denote the strip $S_C:=\menge{(x,u) \in \R^2}{\n{x}
\leq C}$ in $\R^2$. We prove that for $1\leq p \leq \infty$ the solution $v$ is in
$L_p^{-\alpha}$ if our initial data $f$ and $g$ are $L_p$-functions supported in a fixed
strip $S_C$, $C>0$, and if $\alpha>\n{1/p-1/2}$ holds. In fact, we show that for every
$C>0$ the operator $\exp(it\sqrt{G})(1+G)^{-\alpha/2}$, defined for Schwartz functions,
extends to a bounded operator from $L_p(S_C)$ to $L_p(\R^2)$ for all
$\alpha>\n{1/p-1/2}$.

%show that the operator $\cos(t\sqrt{G})(1+G)^{-\alpha/2}$ extends to a bounded operator
%on $L_p(\R^2)$ for $\alpha>\n{1/p-1/2}$ and that the operator
%$\sin(t\sqrt{G})G^{-1/2}(1+G)^{-\alpha/2}$ extends to a bounded operator on $L_p(\R^2)$
%for $\alpha+1>\n{1/p-1/2}$.

%\begin{abstract}
%Let $G:=-(\d{x}²+x²\d{t}²)$. We prove that $e^{i\sqrt{G}}/(1+G)^{\alpha/2}$ extends to a
%bounded operator on $L^p(\R^2)$, for $1\leq p \leq \infty$, when $\alpha>\n{1/p-1/2}$.
%\end{abstract}

\newpage
\tableofcontents \thispagestyle{empty}
\newpage
\thispagestyle{empty} \vspace*{\fill}

\newpage
\setcounter{page}{1}
\Section{Introduction}

\subsection{Context and background}

\noindent Let
\begin{equation}
    \label{context1}
    L(x,\d{x})=-\sum_{\n{\alpha}\leq 2} a_{\alpha}(x)\d{x}^{\alpha}
\end{equation}
be a linear partial differential operator of order $2$ with smooth real coefficients in
an open set $\Omega \teil \R^{d}$ with principal symbol
$L_{pr}(x,\xi):=\sum_{\n{\alpha}=2} a_{\alpha}(x)\xi^{\alpha}$.
%We call such an operator
%\emph{degenerate elliptic}, if $L_{pr}(x,\xi)\geq 0$ for all $(x,\xi) \in
%T^*\Omega=\menge{(x,\xi) \in \Omega \times \R}$.
We say $L$ is \emph{elliptic in $x$}, if
\[L_{pr}(x,\xi)\not=0 \mbox{ for all } \xi \in T_x\Omega\ohne 0:=\menge{\xi\in T_x\Omega}{
\xi \not=0}.\] We call $L$ \emph{elliptic}, if $L$ is elliptic for all $x\in \Omega$. $L$
is called \emph{non-elliptic}, if $L$ is not elliptic.

In addition, we assume that $L$ is positive and essentially selfadjoint. We consider now
the following Cauchy problem for the wave equation associated to $L$ on $\Omega$:
\begin{equation*}
    \frac{\partial^2v}{\partial t^2}+L v =0,\quad
    v \rvert_{t=0}=f,\quad \frac{\partial v}{\partial t}\rvert_{t=0}=g,
\end{equation*}
where $t$ denotes time and $f,\ g$ are suitable functions. The solution to this problem
is formally given by
\begin{equation*}
    \label{sol}
    v(t,x):=
    \cos(t \sqrt{L})f(x)+\frac{\sin(t\sqrt{L})}{\sqrt{L}}g(x),\quad (x,t)
    \in \Omega \times \R.
\end{equation*}
The functions of $L$ are defined by the spectral theorem and the above expression for $v$
makes sense at least for $f,g \in L^2(\R^d)$.

Smoothness properties of the solution $v$, for fixed time $t$, can be measured in terms
of Sobolev norms $\dnorm{f}_{L_p^\alpha}:=\dnorm{(1+L)^{\alpha/2}f}_{L_p}$ adapted to
$L$. We are especially interested in estimates of the following
kind.\absatz\index{lpalpha}

\noindent \textit{For every $t>0$, $1<p<\infty$ and $\alpha>\alpha(d,p)$ there exists a
constant $C_{p,t}^\alpha$ such that
\begin{equation}
  \label{rm1}
  \dnorm{\cos(t \sqrt{L})f}_{L_p^{-\alpha}} \leq C_{p,t}^\alpha \dnorm{f}_p
\end{equation}
and
\begin{equation}
  \label{rm2}
  \Bigdnorm{\frac{\sin(t\sqrt{L})}{\sqrt{L}}g}_{L_p^{-\alpha+1}} \leq C_{p,t}^\alpha
  \dnorm{g}_p
\end{equation}
hold.\\}

\noindent We call these estimates \emph{wave estimates}.
%We say $\alpha_0(d,p)$ is the \emph{critical index} if for all $\alpha<\alpha_0(d,p)$
%there is no constant $C_{p,t}$ such that \ref{rm1} and \ref{rm2} are fulfilled.
\subsubsection*{Wave estimates for the Laplacian on Euclidean space}

For $L=-\Delta$ and $\Omega=\R^d$, we have the usual Cauchy problem on the Euclidean
space. For this case, estimates have been established by Sigrid Sjöstrand
\cite{sjoestrand}, Akihiko Miyachi \cite{miyachi} and Juan Peral \cite{peral}. In 1980,
Peral and Miyachi independently showed the estimates \eqref{rm1}, \eqref{rm2} for
$\alpha(d,p):=(d-1)\norm{1/p-1/2}$. In fact, they showed that these estimates also hold
true for the endpoint $\alpha=\alpha(d,p)=(d-1)\norm{1/p-1/2}$ if $1<p<\infty$. Moreover,
both operators are bounded from the Hardy space $H_1(\R^d)$ to $L_1(\R^d)$ for
$\alpha=(d-1)/2$. These estimates are optimal. In the following we call
$\alpha(d,p):=(d-1)\norm{1/p-1/2}$ the \emph{critical index}.\index{alphadp}

The solution $v$ can also be written as
\begin{equation*}
  \begin{split}
  v(x,t)&=\sum_{\epsilon=\pm 1}(2\pi)^{-d} \int\int e^{i((x-y)\fa\xi +\epsilon \norm{\xi} t)}
  \ed{2}\Big(f(y)+\epsilon \frac{g(y)}{i\norm{\xi}}\Big) \; dy \; d\xi\\
  &=:\sum_{\epsilon=\pm 1} A_{\epsilon,0}^{t}f(x)+A_{\epsilon,1}^{t}g(x).
  \end{split}
\end{equation*}
The operators $A_{\epsilon,k}^{t}$ are Fourier integral operators. Andreas Seeger,
Christopher D. Sogge and Elias M. Stein \cite{seeger} showed that wave estimates
(\ref{rm1}), (\ref{rm2}) hold true for a wide class of Fourier integral operators for the
critical index $\alpha(d,p)$ and $d$ the topological dimension of the underlying space.

In general let
\[P:=\d{t}^2+L(x,\d{x}),\] with $L$ defined as in \eqref{context1}, and put $p(x,\tau,\xi):=-\tau^2+L_{pr}(x,\xi)$.
If $L$ is an elliptic operator, then $P$ is \emph{strictly hyperbolic}, which means that
\begin{equation*}
  \begin{split}
  &p(x,\tau,\xi) \mbox{ has two real distinct roots } \tau_1(x,\xi),\
  \tau_2(x,\xi)\\
  &\mbox{for each } (x,\xi) \mbox{ with } \xi\not=0.
  \end{split}
\end{equation*}
It is well known (see J. J. Duistermaat \cite{duistermaat}) that in this case one can
find elliptic Fourier integral operators $T_{j,k}^t$ such that for small $t$ the solution
$v(t,x)$ of the Cauchy problem with initial conditions $\d{t}^k v(0,\fa)=f_k$ for $k=0,1$
is given by
\begin{equation*}
  v(t,x)=\sum_{j=1,2,\ k=0,1} T_{j,k}^t f_k(x),
\end{equation*}
modulo an infinitely smoothing operator (this technic is often called the
\emph{geometrical optics ansatz}). Therefore, the estimates \eqref{rm1} and (\ref{rm2})
hold true for a wide class of elliptic operators $L$, provided $t$ is small and
$\alpha(d,p)=(d-1)\n{1/p-1/2}$. In fact, Seeger, Sogge and Stein \cite{seeger} showed
these estimates for elliptic differential operators of order $m$ on compact, smooth
manifolds of dimension $d$. Locally such an operator of order $2$ is of the form
(\ref{context1}).\absatz

\subsubsection*{Hörmander type operators}

Now we remove the requirement of ellipticity. To ensure that the question of smoothness
of a solution $v$ of the wave equation is reasonable, we should demand that $L(x,\d{x})$
is \emph{hypoelliptic}, which means that
%. An operator $L(x,\d{x})$ is \emph{hypoelliptic} if and only if%
for all $v \in D'(\Omega)$
\[
  \singsupp v \teil \singsupp L(x,\d{x})v\]
holds. An elliptic operator $L$ is hypoelliptic.

In 1967, Lars Hörmander showed that operators of the form
\[L(x,\d{x}):=-(\sum_{\l=1}^m X_\l^2+X_0),\] where all $X_\l$ are real vector fields on an
open set $\Omega$, are hypoelliptic under the following, rather weak condition:
\begin{equation*}
  \begin{split}
  &\mbox{For all } x \in \Omega, \mbox{ the tangential space } T_x\Omega \mbox{ in } x
  \mbox{ is spanned by } \{X_1,\dots,X_m\}\\
  &\mbox{and finitely many iterated commutators of } X_1,\dots,X_m.
  \end{split}
\end{equation*}
This condition is often called \emph{Hörmander's condition} and an operator which
fulfills it is called an \emph{operator of Hörmander type}. If we write $L$ as
$\sum_{\n{\alpha}\leq 2} a_{\alpha}(x)\d{x}^{\alpha}$, we can show that
$-\sum_{\n{\alpha}=2} a_{\alpha}(x) \xi^{\alpha}\geq 0$. Operators fulfilling this
inequality are often called \emph{degenerate elliptic} operators.

The associated wave equation to a non-elliptic operator is not strictly hyperbolic, and
that is why we cannot use a geometrical optic ansatz to write the solutions by using
Fourier integral operators. Since we have no "straight forward" way of computing the
solutions of the wave equation, we can presently only hope to get results for special
operators. Furthermore, the underlying geometry is sub-Riemannian and, in general,
substantially more complex than the geometry for wave equations associated to elliptic
operators.

Nevertheless, we expect that for many Hörmander type operators (\ref{rm1}) and
(\ref{rm2}) hold true for the critical index $(d-1)\n{1/p-1/2}$, where $d$ is the
topological dimension of the underlying space and that such a result is optimal, except
for the endpoint. As far as we know, there is only one result of this type for a
non-elliptic operator yet known. We define this operator in the following.

%\begin{theorem}{\bf(J.~Peral, A.~Miyachi (1980))}\\
%  Let $d \geq 2,\ 1<p<\infty$. $L:=-\Delta$. The estimates
%  \begin{equation}
%    \label{rm1}
%    \dnorm{\cos(t \sqrt{L})f}_{L^p_{-\alpha}} \leq C_{p,t} \dnorm{f}_p
%  \end{equation}
%  holds if $\alpha\geq (d-1)\norm{1/p-1/2}=:\alpha(d,p)$.
%  \begin{equation}
%    \label{rm2}
%    \Bigdnorm{\frac{\sin(t\sqrt{L})}{\sqrt{L}}g}_{L^p_{-\alpha}} \leq C_{p,t}
%    \dnorm{g}_p
%  \end{equation}
%  holds if $\alpha+1\geq (d-1)\norm{1/p-1/2}$.
%\end{theorem}

%It is well known (J.~J.~Duistermaat \cite{duistermaat}) that if the wave equation for a
%general $L$ is strictly hyperbolic one can express the solutions in terms of elliptic
%Fourier integral operators.
% $T_{j,k}^t$ such that the
%solution $v(t,x)$ of the Cauchy problem
%is given by
%\begin{equation*}
%  v(t,x)=\sum_{j=1,2} T_{j,1}^t f(x)+T_{j,2}^t g(x),
%\end{equation*}
%modulo an infinitely smoothing operator.
%A.~Seeger, C.~D.~Sogge  and E.~M.~Stein \cite{seeger} showed
%that local analogues of the estimates from Miyachi and Peral
%hold true for a wide class of Fourier integral operators.\\

\subsubsection*{Wave estimates for the sub-Laplacian on the Heisenberg group.}

Let $\H_m$ denote the $2m+1$-dimensional Heisenberg group. As a manifold $\H_m$ is the
$\R^{2m+1}$. The vector fields $X_j:=\d{x_j}-\ed{2}y_j \d{u}, \; Y_j:=\d{y_j}+\ed{2} x_j
\d{u}, \; U:=\d{u}$ form a natural basis for the Lie algebra of left-invariant vector
fields.
%The commutator $[X_j,Y_j]$ of the vector fields $X_j,\
%Y_j$ is $U$.
The \emph{sub-Laplacian}
\[L:=-\sum_{j=1}^m (X_j^2+Y_j^2)\]
is non-elliptic. Nevertheless, $L$ is a hypoelliptic operator, since $[X_j,Y_j]=U$ and
hence the Hörmander condition is fulfilled.

In 1999, Detlef Müller and Elias M. Stein \cite{mueller} showed that the estimates
\eqref{rm1}, \eqref{rm2} hold true for the critical index \mbox{$(d-1)\norm{1/p-1/2}$},
where $d:=2m+1$ is the topological dimension of $\H_m$. Except for the endpoint
$\alpha(d,p)$, this result is optimal. One can reduce the proof to showing that the
operator \mbox{$\exp(i\sqrt{L})(1+L)^{-\alpha/2}$} is bounded on $L_p$ for
$\alpha>\alpha(d,p)$. Furthermore, it can be restricted to the case $p=1$. For this case,
Müller and Stein showed that the corresponding convolution kernel of this operator lies
in $L_1(\H_m)$. \absatz

Before we present our result, we want to mention a recent result by Michael Cowling and
Adam Sikora. They studied a sub-Laplacian on the group $SU(2)$. This sub-Laplacian is
also of Hörmander type. Since the $SU(2)$ is connected to the Heisenberg group (see
Fulvio Ricci \cite{ricci}), one can hope to get wave estimates for the wave equation
associated to this sub-Laplacian for the critical index
$\alpha(d,p)=(d-1)\n{1/p-1/2}=2\n{1/p-1/2}$. We know, by oral communication, that Cowling
and Müller are working on this topic and have developed some new technics, which yield,
in principle, the wave estimates on the group $SU(2)$ for the critical index. But they
have not worked out all details yet.

Instead of wave equations, Cowling and Sikora studied multipliers and proved a spectral
multiplier theorem for this sub-Laplacian. By general functional calculus one can deduce
multiplier theorems from wave estimates (see Müller \cite{muellerberlin}).

\subsubsection*{Spectral multiplier theorems}

We say that for an operator $L(x,\d{x})$ a \emph{Mikhlin-Hörmander multiplier theorem}
holds if for all bounded Borel functions $m:\R^+ \to \C$ with
\begin{equation}
    \label{critical}
    \sup_{t\in \intao{1,\infty}} \dnorm{\eta(\fa)m(t\fa)}_{H_s}< \infty
\end{equation}
the operator $m(L)$ is a bounded operator on $L_p$ for $1<p<\infty$ and $m(L)$ is of weak
type $(1,1)$, provided $s$ is bigger than an index $s_0$. $\eta$ is here a non trivial
cut-off function on $\R^+$. Müller and Stein \cite{ms} and independently Waldemar Hebisch
\cite{hebisch} proved a Mikhlin-Hörmander multiplier theorem for the sub-Laplacian $L$ on
$\H_m$ for the index $s_0=d/2=(2m+1)/2$, which is half of the topological dimension of
$\H_m$. This result is optimal, except for the endpoint. It follows also from the
mentioned estimates for the wave equation by Müller and Stein by the method of
subordination. In fact, Hebisch \cite{hebisch}, and also Ricci, Müller and Stein
\cite{muellerricci1} showed multiplier theorems for generalized Heisenberg groups and not
only for $\H_m$.

In 2001, Cowling and Sikora showed that a Mikhlin-Hörmander multiplier theorem for a
sub-Laplacian on the group $SU(2)$ (see \cite{cowling}) holds true for $s_0=3/2$, which
is half of the topological dimension of $SU(2)$ and therefore, this result is the
analogue for $SU(2)$ of the result obtained by Müller and Stein and as well of the result
by Hebisch. This result is optimal, except for the endpoint.

%Though the Cowling/Sikora result does not imply wave estimates for the critical index
%$\alpha(d,p)=2\n{1/p-1/2}$ on $SU(2)$, it is a strong indication that they hold true.
By using the methods of Müller in \cite{muellerberlin}, it should be possible to prove a
Mikhlin-Hörmander multiplier theorem for the Gru\v{s}in operator. We conjecture the
following theorem.
\begin{vermutung}
    Let $m:\R^+ \to \C$ a bounded Borel function that fulfills \eqref{critical}. Than
    $m(G)$ is a bounded operator on $L_p$ for $1<p<\infty$ and $m(G)$ is of weak type
    $(1,1)$, provided $s>1$.
\end{vermutung}

\subsection{The main result}

Let $S_C$ denote the strip $S_C:=\menge{(x',u') \in \R^2}{\n{x'} \leq C}$ in $\R^2$. In
this thesis, we show that for the most basic and historically first studied non-elliptic
Hörmander type operator the estimates (\ref{rm1}) and (\ref{rm2}) hold true with critical
index $\alpha(d,p)$ and $d$ the topological dimension of the underlying space, provided
the initial data $f$ and $g$ are supported in a fixed strip $S_C$. This operator is the
Gru\v{s}in operator.\index{strip}

The \emph{Gru\v{s}in operator} $G$ is defined by
\begin{equation*}
    G:=-(\d{x}^2+x^2\d{u}^2)
%=-(\d{x_1}²+\dots+\d{x_n}²+(x_1\d{t})²+\dots+(x_n\d{t})²)
\end{equation*}
on $\R^{2}$. Though this operator is non-elliptic, it is still a hypoelliptic operator,
since it fulfills the Hörmander condition. Since $G$ is one of the easiest Hörmander type
operators, it is predestinated as a starting point for a systematic study of wave
equations for non-elliptic operators of this type.

$G$ posses less invariance properties than the sub-Laplacian $L$ on $\H_m$. That is why
the study of waves associated to this operator is more difficult than the study of waves
associated to $L$. In contrast to $L$, the Gru\v{s}in operator is not translation
invariant. Waves associated to $G$ that start near the axis $x'=0$ exhibit a behavior
similar to the behavior of waves on $\H_1$. Waves associated to $G$ that start far away
form the axis $x'=0$ behave like waves associated to an elliptic operator. Especially the
transition area, $0<x'\lesssim 1$, is very interesting and gives new insights in the
general theory of wave estimates for operators of Hörmander type.

$G$ is connected to the sub-Laplacian $L$ on $\H_1$, since it can be written as an image
of $L$ under a certain representation of $\H_1$.

Our result reads as follows.
\newpage
\begin{Theorem}
    For every $C>0$, $t>0$, $1\leq p \leq \infty$ and $\alpha>\norm{1/p-1/2}$ there exists a
    constant $C_{p,t,C}^\alpha$ such that for all $f$ and $g$ in $\S$ and supported in $S_C$
    the estimates
    \begin{equation*}
        \Bigdnorm{\frac{\cos(t \sqrt{G})}{(1+G)^{\alpha/2}}f}_{L_p(\R^2)} \leq C^\alpha_{p,t}
        \dnorm{f}_{L_p(\R^2)},
    \end{equation*}
    and
    \begin{equation*}
        \Bigdnorm{\frac{\sin(t\sqrt{G})}{\sqrt{G}(1+G)^{(\alpha-1)/2}}g}_{L_p(\R^2)} \leq
        C^\alpha_{p,t,C}
        \dnorm{g}_{L_p(\R^2)},
    \end{equation*}
    hold.
\end{Theorem}
Since the topological dimension is $d=2$, and hence $d-1=1$, this theorem is a localized
analogue for $G$ of the result by Müller and Stein for the sub-Laplacian on the
Heisenberg group.
%In the following we denote this sub-Laplacian on $\H_m$ by $L$.

We can restrict to $t=1$, since $G$ is homogenous with respect to the dilation
$\delta_r:(x,u)\mapsto (rx,r^2u),\ r>0$. Instead of
\mbox{$\cos(\sqrt{G})(1+G)^{-\alpha/2}$} and
\mbox{$\sin(\sqrt{G})G^{-1/2}(1+G)^{-\alpha/2}$} we study the operator
$\exp(i\sqrt{G})(1+G)^{-\alpha/2}$. The assertion for $\cos(\sqrt{G})(1+G)^{-\alpha/2}$
follows immediately.  For the operator $\sin(\sqrt{G})G^{-1/2}(1+G)^{-\alpha/2}$, we use
that it suffices to prove the assertion for
$\eta(G)\sin(t\sqrt{G})G^{-1/2}G^{-\alpha/2}$, where $\eta$ is a smooth function
supported away from the origin. Thus our theorem can be reduced to the following.

\begin{Theorem}
    \label{theorem}
    For every $C>0$, $1\leq p\leq \infty$, the operator
    \mbox{$\exp(i\sqrt{G})(1+G)^{-\alpha/2}$} extends to a bounded operator from $L_p(S_C)$ to $L_p(\R^2)$,
    provided \mbox{$\alpha>\n{1/p-1/2}$.}
\end{Theorem}
By standard interpolation arguments, it suffices to show the case $p=1$. Since $G$ is not
translation invariant, $\exp(i\sqrt{G})(1+G)^{-\alpha/2}$ has no convolution kernel. To
prove the case $p=1$ we show that this operator has an integral kernel $K$ such that
$\exp(i\sqrt{G})(1+G)^{-\alpha/2}f(x,u)=\int K(x',u',x,u) \; f(x',u') \; d(x',u')$ for
every $f \in \S$ and that
\begin{equation}
    \label{l1abschatzung}
    \dnorm{K(x',u',\fa,\fa)}_{L_1(\R^2)}
\end{equation}
is uniformly bounded for $\n{x'}\leq C$ and $u'\in \R$.

As we have mentioned before, due to the lack of translation invariance, the behavior of a
wave highly depends on its starting point. For waves starting near the axis $x'=0$, we
use ideas of Müller and Stein and adapt them to our situation. For waves starting far
away form the axis $x'=0$, it should be possible to reduce by scaling arguments to the
results by Seeger, Sogge and Stein for elliptic operators. For this case we only present
the general idea and do not go into the details.

It turns out that the most crucial part, but also the most interesting part, of the proof
is the case when waves start near, but not exactly on the axis $x'=0$. For waves starting
at $x'=0$ the methods of Müller and Stein work very well, but as soon as the starting
point is a little bit away from $x'=0$ matters become a lot more difficult. A sketch of
the proof of Theorem \ref{theorem} will be given in Chapter 2.\absatz

Before we start to prove Theorem \ref{theorem}, we have to calculate the integral kernel
of $\exp(i\sqrt{G})(1+G)^{-\alpha/2}$ in a very explicit way.

Peter C. Greiner, David Holcman and Yakar Kannai derived in \cite{greiner} formulas for
the distribution kernel of $\sin(t\sqrt{G})G^{-1/2}$. We do not use their formulas for
two reasons. First, we are interested in the "smoothed" wave propagators
$\cos(t\sqrt{G})(1+G)^{-\alpha/2}$ and $\sin(t\sqrt{G})G^{-1/2}(1+G)^{-\alpha/2}$ resp.
$\exp(i\sqrt{G})(1+G)^{-\alpha/2}$ instead of $\sin(t\sqrt{G})G^{-1/2}$.

The second reason is that Greiner, Holcman and Kannai identify $\R^2$ with $\C$ and their
formulas involve contour integrals. This representation seems not to be explicitly enough
to show wave estimates, or, it is not clearly evident how to use them. Moreover, we do
not know how to get similar formulas for higher dimensional Gru\v{s}in operators by this
approach. Though we prove wave estimates only for the two dimensional case, our method of
calculating the integral kernel of $\exp(i\sqrt{G})(1+G)^{-\alpha/2}$ can also be used,
in principle, for the higher dimensional case.

We use a trick that was used beforehand by Müller and Stein for the sub-Laplacian $L$ on
$\H_m$. $G$ can be written as $(iU)(-iGU^{-1})$ and hence one can derive formulas for the
integral kernel of $\exp(i\sqrt{G})(1+G)^{-\alpha/2}$ by using the functional calculus of
$iU$ and the functional calculus of $-iGU^{-1}$. The calculus for $-iLU^{-1}$, $L$
instead of $G$, has been studied by Robert S. Strichartz in \cite{strichartz}. We adapt
his methods to our situation and calculate $m(-iGU^{-1})$, for a bounded Borel function
$m$. This gives us a representation of the integral kernel of
$\exp(i\sqrt{G})(1+G)^{-\alpha/2}$ that can be handled by oscillatory integral methods.

\subsection{Organization of this thesis}

In \chap 1, Section 1.1, we define the Gru\v{s}in operator and the sub-Laplacian $L$ on
the Heisenberg group $\H_m$.

The Gru\v{s}in operator is the image of $L$ under a certain representation of the
polarized Heisenberg group of dimension 3. Therefore, transference methods are
applicable. We show that for $G$ a weak multiplier theorem holds. This will be done in
Section 2.

The next section, Section 1.3, is devoted to the study of the underlying geometry of our
problem. We give explicit formulas for geodesics belonging to optimal control metric
associated to $G$. This, together with a result by Richard Melrose \cite{melrose}, allows
us to estimate the speed of propagation of our waves. At the end of this section, we
present figures of geodesics and balls belonging to the optimal control metric associated
to G, and a figure of the sphere in the optimal control metric associated to
$L$.\absatz\absatz

In \chap 2, we state our main theorem and a conjecture for higher dimensional Gru\v{s}in
operators. Moreover, we give a short sketch of the proof of Theorem
\ref{theorem}.\absatz\absatz

In \chap 3, we study the joint functional calculus for $iU=i\d{u}$ and $G$. Since $G$ can
be written as $(iU)(-iGU^{-1})$, we are especially interested in $m(-iGU^{-1})$, where
$m$ is a bounded Borel function.

The functional calculus for $-iLU^{-1}$, $L$ instead of $G$, has been studied by
Strichartz \cite{strichartz}. Strichartz derived explicit formulas for joint
eigenfunctions $\phi_{\lambda,n}$ associated to the eigenvalues $\lambda$ of $L$ and the
eigenvalues $\epsilon\lambda/(1+2n)$ of $iU$. He showed that the operators
$m(-iLU^{-1})$, $m$ a bounded Borel function, can be written as a sum over certain
generalized projection operators $\P^\H_{n,\epsilon}$ associated to rays of the
Heisenberg fan. Formally $\P^\H_{n,\epsilon}$ is the convolution operator with kernel
$\int \phi_{\lambda,n} \; d\lambda$. We adapt these ideas to our situation. Since $G$ is
not translation invariant, the corresponding projection operators $\P_{n,\epsilon}$ are
no longer convolution operators. But they are still $L_2$-bounded singular integral
operators. For a bounded Borel function $m$, we get the spectral decomposition
$m(-iGU^{-1})f=\sum_{\epsilon=\pm 1} \sum_{n=0}^\infty
m(\epsilon(2n+1))\;\P_{n,\epsilon}(f)$.\absatz\absatz

In \chap 4, we show that it suffices to prove the theorem for \mbox{$p=1$.}
%We study the operator
%$m^\alpha(G):=(1+G)^{-\alpha/2} \; \exp(i\sqrt{G})$ instead of
%$\cos(\sqrt{g})/(1+G)^{\alpha/2}$ and $\sin{\sqrt{G}}/(\sqrt{G}(1+G)^{\alpha/2})$. By
%$K_{m^\alpha(G)}$ we denote the integral kernel of $m^{\alpha}(G)$.
%To prove the theorem we now have to show that
%\begin{equation*}
%  \dnorm{K_{m^{\alpha}(G)}(x',u',x,u)}_{L^1(x,u)}
%\end{equation*} We study the operator
%$m^\alpha(G):=(1+G)^{-\alpha/2} \; \exp(i\sqrt{G})$ instead of
%$\cos(\sqrt{g})/(1+G)^{\alpha/2}$ and $\sin{\sqrt{G}}/(\sqrt{G}(1+G)^{\alpha/2})$.
Let $K$ denote the integral kernel of $\exp(i\sqrt{G})(1+G)^{-\alpha/2}$. To prove the
theorem, we show that
\begin{equation}
    \label{overview1}
  \dnorm{K(x',u',\fa,\fa)}_{L^1(\R^2)}
\end{equation}
is uniformly bounded for $\n{x'}\leq C$ and $u'\in \R$. Since $G$ is translation
invariant with respect to $u$, we only have to consider the case $u'=0$. We formally show
how one can use scaling arguments and the result by Seeger, Sogge and Stein for elliptic
operators, to proof that \eqref{overview1} is also true for large $x'$.

Furthermore, instead of estimating $\exp(i\sqrt{G})(1+G)^{-\alpha/2}$, we are allowed to
estimate $h^\alpha(G):=\eta(G)\ G^{-\alpha/2} \ \exp(i\sqrt{G})$, where $\eta$ is a
smooth function supported away from the origin. In addition we show, due to the finite
speed of propagation of our waves, that it suffices to show that the integral kernel
$K_{h^\alpha(G)}(x',0,\fa,\fa)$ of $h^\alpha(G)$ is in $L_1(B_{G}((x',0),C)$, where
$B_{G}((x',0),C)$ is a ball with respect to the optimal control metric associated to $G$
centered in $(x',0)$ with radius $C$ and $C$ a constant.\absatz\absatz

In \chap 5, we use a dyadic decomposition of the joint spectrum of $G$ and $iU$ to
decompose the integral kernel $K_{h^{\alpha}(G)}$ in dyadic parts $K_{k,j}$ . The proof
of the theorem is then reduced to showing that $\sup_{\n{x'}\lesssim 1}
\dnorm{K_{k,j}(x',0,\fa,\fa)}_{L_1(B_{G}((x',0),C))}$ is summable in $j$ and $k$. We
derive explicit formulas for these dyadic parts by using the projection operators
$\P_{n,\epsilon}$ we have defined in Chapter 3. Furthermore, we introduce new
coordinates.\absatz\absatz

In \chap 6, we show the desired $L_1$-estimates for the integral kernels $K_{k,j}$.

\subsection{Acknowledgements}

Foremost, I would like to express my gratitude to my advisor Professor Dr. Detlef Müller
for constant support and numberless helpful suggestions.
%Professor Dr. Hermann König and
%Professor Dr. Volker Wrobel are gratefully acknowledged for accepting to be co-referees
%of this work.

I am grateful for the financial support of the Graduiertenkolleg 357 "Effiziente
Algorithmen und Mehrskalenmethoden" of the Deutsche Forschungsgemeinschaft (DFG).

Moreover, I owe thanks to my officemates Martin Hemke, Jörn Peter and Heike Siebert for
encouragement. Last but not least I would like to thank my parents, my brother and
friends for having been there, whenever I needed them.

\newpage
\Section{Preliminaries}

\subsection{Remarks about integral kernels}

We know that if $A:L_p(\R^2)\to L_q(\R^2)$, $1\leq p,q \leq \infty$ is a translation
invariant operator, then there exists a distribution $u\in \S'(\R^2)$ with
$Af=\conv{f}{u}$ for all $f \in \S(\R^2)$. Furthermore, if $u$ is in $L_1(\R^2)$ then the
operator $\S\to C^\infty,\ f\mapsto \conv{f}{u}$ extends to a bounded operator on $L_1$.
(In fact, we know that this operator extends to a bounded operator on $L_1$ if and only
if $u$ is a finite Borel measure.)

Since the operator of interest $G=-(\d{x}^2+x^2\d{u}^2)$ is not translation invariant,
most of the operators we study in this work do not have a convolution kernel. Though by
the Schwartz kernel theorem, we know that for every continuous linear map
$A:C_0^\infty(\R^2)\to D'(\R^2)$ there exists a distribution $K$ such that
\begin{equation}
    <A\phi,\psi>=K(\psi \otimes \phi),
\end{equation}
where $\psi \otimes \phi$ denotes the tensor product of $\psi$ and $\phi$. We designate
$K$ as the \emph{distribution kernel} of $A$.

If $K$ is a
measurable function such that for every $f \in \S$
\[Af(x)=\int K(x',x) \; f(x') \; dx',\]
we say that $K$ is the \emph{integral kernel} of $A$. If we consider more than one
operator, we usually denote the integral kernel of an operator $A$ by $K_A$. We also
write
\[K_A(x',\fa)=A\delta_{x'},\] where $\delta_{x'}$ denotes the Dirac measure at the point
$x'$.

Given such an $A$ and $K$ we assume now that
\begin{equation*}
    \sup_{x'} \int \n{K(x',x)} \; dx \quad \mbox{ and } \quad \sup_{x} \int \n{K(x',x)} \;dx'
\end{equation*}
are bounded. By Schur's test, $A$ is bounded on $L_p$ for $1\leq p \leq \infty$. For the
$L_1$-boundedness, we only need that $\sup_{x'} \int \n{K(x',x)} \; dx$ is bounded.

\begin{definition}
    Let $A$ be an operator with measurable integral kernel $K(x',x)$ such that
    \begin{equation}
        A(f)(x)=\int K(x',x) \; f(x') \; dx',
    \end{equation}
    for all $f\in \S$.
    We define the \emph{Schur norm} of $K$, denoted by $\dnorm{K}_{Schur}$, by
    \begin{equation*}
        \dnorm{K}_{Schur}:=\sup_{x'} \int \n{K(x',x)}\; dx.
    \end{equation*}\index{schurnorm}
\end{definition}
In the following, we only use the phrases "$K$ has bounded Schur norm" or "the kernels
$K_t$ have bounded Schur norms, uniformly for $t\in I$" if $(K_t)_t$ is a family of
kernels and there exists a constant $C$ such that $\dnorm{K_t}_{Schur}\leq C$ for all $t$
in an interval $I$.

Thus an operator $A$ with integral kernel $K$ such that $K$ has bounded Schur norm is
bounded on $L_1$.

\subsection{The Gru\v{s}in operator and the sub-Laplacian on $\H_m$}
\label{grushinsublaplacesection}

\subsubsection*{The Gru\v{s}in operator}
Let $n \in \N$. We define the \emph{Gru\v{s}in operator} $G_n$ on $\R^{n+1}$ by

\begin{equation*}
    G_n:=-(\Delta_x+\norm{x}^2\d{u}^2)
%=-(\d{x_1}²+\dots+\d{x_n}²+(x_1\d{t})²+\dots+(x_n\d{t})²)
\end{equation*}
$G_n$ is positive and essentially selfadjoint on $\Cnu(\R^{n+1})$. Though this operator
is not elliptic for $x=0$, it is still a hypoelliptic operator, since it fulfills the
Hörmander condition. $G$ is one of the easiest Hörmander type operators.

V.~V.~Gru\v{s}in studied in 1970 (see \cite{grushin}) a class of operators that is not
contained in the class of Hörmander type operators. He gave sufficient and necessary
condition for operators in this class to be hypoelliptic. The Gru\v{s}in operator is a
prototype of these operators.

%For example he studied operators of the form
%\[\Delta_y^2+\n{y}^4\Delta{x}^2+\lambda\Delta{x}^2\]
%on $\R^{k+n}$ for $\lambda>0$ and showed that this operator is hypoelliptic if and only
%if $\lambda$ is not an eigenvalue of the operator $\delta_y+\n{y}^4$ in $L_2(\R^n)$.

This work will be restricted to the case $n=1$. Therefore, we define
\[G:=G_1.\]\index{G} For every $r>0$, we define the dilation $\delta_r$ on
$\R^2$ by \begin{equation}
    \delta_r(x,u):=(rx,r^2u). \index{automorphicdilation}
    \end{equation}
Then for every suitable $f$ and $r>0$,
\begin{equation*}
    G(f\nach\delta_r)(x,u)=r^2(Gf)\nach \delta_r(x,u)
\end{equation*}
holds. Hence $G$ is homogenous of degree $2$ with respect to $\delta_r$.

\subsubsection*{The Heisenberg group and the sub-Laplacian}

Let $\H_m$ \index{Hm} denote the Heisenberg group, which is $\R^{2m} \times \R$ endowed
with the group law
\begin{equation*}
    (x,y,u) \fa (x',y',u'):=\big(x+x',y+y',u+u'+\ed{2}\omega((x,y),(x',y'))\big)
\end{equation*}
for $x,y,x',y' \in \R^n,\ u,u' \in \R$, where $\omega$ is the canonical symplectic form
\begin{equation*}
    \omega((x,y),(x',y')):=x\fa y'-x'\fa y,\quad x,y,x',y' \in \R^n
\end{equation*}
on $\R^{2n}$. $\H_m$ is a connected, simply connected and nilpotent Lie group.

The Lebesgue measure on the Euclidean space $\R^{2n+1}$ is a bi-invariant Haar measure on
$\H_m$. The convolution of two functions $f,g \in L_1(\H_m)$ is defined by
\begin{equation*}
    (\conv{f}{g})(x,y,u)=\int f(x',y',u') \; g((x',y',u')^{-1}(x,y,u)) \; d(x',y',u').
\end{equation*}
The  dilation $\delta_r^\H:(x,y,u)\mapsto (rx,ry,r^2u)$ is an automorphism of $\H_m$ for
every $r>0$. The vector fields
\begin{equation*}
    X_j:=\d{x_j}-\ed{2}y_j \d{u}, \quad Y_j:=\d{y_j}+\ed{2} x_j \d{u}, \quad U:=\d{u}
\end{equation*}
form a natural basis for the Lie algebra $\lieh_m$ of left-invariant vector fields with
commutator relations
\begin{equation*}
    \begin{split}
    [X_j,Y_k]&=\delta_{k,j}U_j,\\
    [X_j,X_k]&=[Y_j,Y_k]=[X_j,U]=[Y_j,U]=0.
    \end{split}
\end{equation*}
We define now the \emph{sub-Laplacian}\index{L} on $\H_m$ by
\begin{equation}
    L:=-\sum_{j=1}^m (X_j^2+Y_j^2).
\end{equation}
Explicitly $L$ is given by
\begin{equation*}
    L=-\Delta_{x,y}+x\fa \d{y}-y\fa\d{x}+\frac{1}{4}\n{(x,y)}^2\d{u}^2.
\end{equation*}
$L$ is positive, essentially selfadjoint and homogenous with respect to $\delta^\H_r$.
Moreover, $L$ is non-elliptic, but hypoelliptic, since the Hörmander condition is
fulfilled.\absatz

Another common way of writing $\H_m$ is as a set of matrices
\[\menge{A(p,q,v)}{p,\ q \in \R^n,\ v \in \R}\]
with
\begin{equation*}
    A(p,q,v):=\left(\begin{array}{llcll}
    1 & p_1 & \dots & p_n & v\\
      & 1   &       & 0   & q_1\\
      &     & \ddots &    & \vdots\\
      &  0  &        & 1  & q_n\\
      &     &        &    & 1\\
    \end{array}\right)
\end{equation*}
and endowed with the usual matrix product. By identifying $A(p,q,v)$ with $(p,q,v)$, we
get a group on $\R^{2n}\times \R$ with group law
\begin{equation*}
    (p,q,v)\fa (\s{p},\s{q},\s{v})=(p+p',q+q',v+v'+p\fa q').
\end{equation*}
This group is called the polarized Heisenberg group and we denote it by
$\t{\H}_m$\index{polarizedheisenberg}. $\H_m$ is isomorphic to $\t{\H}_m$ by
\begin{equation*}
  \Psi:\H_1 \to \t{\H_1},\; (x,y,u)\mapsto(p,q,v)=(x,y,u+x\fa y/2)
\end{equation*}
with inverse transformation
\begin{equation*}
  \Psi^{-1}:\t{\H_1} \to \H_1,\; (p,q,v)\mapsto(x,y,u)=(p,q,v-p\fa q/2).
\end{equation*}
A basis of the Lie algebra is given by
\begin{equation*}
    P_j:=\d{p_j},\quad
    Q_j:=\d{q_j}+p_j\d{v}, \quad
    V:=\d{v}
\end{equation*}
and the sub-Laplacian on this group is
\[\t{L}:=-\sum_{j=1}^m(\d{p_j}^2+(\d{q_j}+p_j\d{v})²).\]

For more information about the Heisenberg group we refer to Stein \cite{steinharmonic}
and Taylor \cite{taylor}.

\subsubsection*{$L^p$-estimates for the wave equation on the Heisenberg group}

In analogy to the Cauchy problem for the wave equation on Euclidean space, we consider
the following Cauchy problem on $\H_m \times \R$.
\begin{equation*}
    \frac{\partial^2v}{\partial t^2}+L v =0,\quad
    v \rvert_{t=0}=f,\quad \frac{\partial v}{\partial t}\rvert_{t=0}=g.
\end{equation*}
Since the work of Müller and Stein about $L_p$-estimates for solutions of this Cauchy
problem is the starting point of this work, we want to state its main result once again.

D.~Müller and E.~M.~Stein established in \cite{mueller} the following estimates.
\begin{theorem}{\bf(D.~Müller, E.~M.~Stein (1999))}
    Let $d:=2m+1$ denote the topological dimension of $\H_m$. For every $t>0$, $1\leq p \leq \infty$ and
    \mbox{$\alpha>(d-1)\norm{1/p-1/2}$} there exists a
    constant $C_{p,t}^\alpha$ such that for all $f$ and $g$ in $\S$ the estimates
    \begin{equation*}
        \Bigdnorm{\frac{\cos(t \sqrt{L})}{(1+L)^{\alpha/2}}f}_{L_p(\H_m)} \leq C^\alpha_{p,t}
        \dnorm{f}_{L_p(\H_m)},
    \end{equation*}
    and
    \begin{equation*}
        \Bigdnorm{\frac{\sin(t\sqrt{L})}{\sqrt{L}(1+L)^{(\alpha-1)/2}}g}_{L_p(\H_m)} \leq
        C^\alpha_{p,t}
        \dnorm{g}_{L_p(\H_m)},
    \end{equation*}
    hold.
\end{theorem}
%\begin{theorem}{\bf(D.~Müller, E.~M.~Stein (1999))} Let $d:=2m+1$ be the Euclidean dimension of
%$\H_m$. Let $1\leq p\leq \infty$.
%    \begin{equation*}
%        \Bigdnorm{\frac{\cos(t \sqrt{L})}{(1+L)^{\alpha/2}}f}_{L^p(\H_m)} \leq C_{p,t}
%        \dnorm{f}_{L^p(\H_m)}
%    \end{equation*}
%    holds if $\alpha > (d-1)\norm{1/p-1/2}$.
%    \begin{equation*}
%        \Bigdnorm{\frac{\sin(t\sqrt{L})}{\sqrt{L}(1+L)^{\alpha/2}}g}_{L^p(\H_m)} \leq
%        C_{p,t}
%        \dnorm{g}_{L^p(\H_m)}
%    \end{equation*}
%    holds if $\alpha+1 > (d-1)\norm{1/p-1/2}$.
%\end{theorem}
By a standard interpolation argument and since $L$ is homogeneous with respect to
$\delta_r^\H$, it suffices to prove the theorem for $p=1$ and $t=1$. In fact, Müller and
Stein showed that the operator $\exp(i\sqrt{L})(1+L)^{-\alpha/2}$ extends to a bounded
operator on $L_1(\H_m)$, when $\alpha>m$. For this purpose, they showed that the
corresponding convolution kernel belongs to $L_1(\H_m)$.

The proof of this theorem is strongly involved in the proof of our \mbox{Theorem
\ref{theorem}}. One reason for this is that $G$ is the image of $L$ under a certain
representation of the polarized Heisenberg group $\t{\H}_1$.

\subsection{Transference}
\label{transferencesection}

In this section, we denote the polarized Heisenberg group $\t{\H}_1$ by $\cal{G}$ with
elements $g\in \cal{G}$, $g=(p,q,v)$. The Lie algebra of $\cal{G}$ we denote by $\lieg$.
Define
\begin{equation}
    P:=\d{p},\quad Q:=\d{q}+p\d{v},\quad V:=\d{v}.
\end{equation}
The sub-Laplacian is now given by $L=-(P^2+Q^2)$.

Let $\lieh$ denote the smallest subalgebra of $\lieg$ with $Q \in \lieh$. With
$\exp_G:\lieg \to \cal{G}$ denoting the exponential function, we define now
\[\cal{H}:=\exp(\lieh) \teil \cal{G}.\]
For every $f:\cal{G} \to \C$ we put
\begin{equation*}
    \dnorm{f}_\cH:=\left(\int \norm{f(p,0,v)}^2 \; dp \; dv\right)^{1/2}
\end{equation*}
and we define
\begin{equation*}
    \cH:=\menge{f}{\dnorm{f}_\cH < \infty \mbox{ and } f(hg)=f(g) \; \forall \ (g,h) \in
    \cal{G} \times \cal{H}}.
\end{equation*}
$\cH$ is a Hilbert space with norm $\dnorm{\fa}_{\cH}$. We denote the set of unitary
operators on $\cH$ by $\U(\cH)$.
\begin{equation*}
    \begin{split}
    \pi:\cal{G} &\to \U(\cH),\\
    [\pi(g)f](h)&:=f(hg)
    \end{split}
\end{equation*}
defines an unitary representation $\pi$ of $\cal{G}$ with representation space $\cH$. We
denote the associative algebra of left-invariant differential operators with
$C^\infty$-coefficients by $\cD_\l(\cal{G})$. $d\pi$ denotes the representation of
$\cD_\l(\cal{G})$ derived from $\pi$. Then
\begin{equation*}
    d\pi(P)=\d{p},\quad d\pi(Q)=p\; \d{v},
\end{equation*}
and hence
\begin{equation}
    d\pi(L)=-(\d{p}^2+p^2\d{v}^2).
\end{equation}
Thus the Gru\v{s}in operator $G$ is the image of $L$ under $d\pi$.

For $m \in C_{\infty}(\R^+)$, the operator $m(L)$, defined by the functional calculus for
$L$, is contained in the $C^*$-algebra $C^*(\cal{G})$. We denote the $*$-representation
of $C^*(\cal{G})$ which corresponds to $\pi$ again by $\pi$.
%Then the operator $\pi(m(L))$
%is well defined on $\cH$.
If $m\in \S$, then we know by a result of Hulanicki \cite{hulanicki} that
\[m(L)f=\conv{f}{M},\]
with $M\in \S(G)$. Now, let $m \in C_{\infty}(\R^+)$. If we assume that $m(L)$ is given
by $m(L)f=\conv{f}{M}$ with $M \in L_1(G)$, we have
\[
    \pi(m(L))=\int_{\cal{G}} M(x) \; \pi(x) \; dx.\]
Furthermore, we have that the functional calculus commutes with the representation $\pi$,
i.e.
\begin{equation*}
    \pi(m(L))=m(d\pi(L)).
\end{equation*}
For a proof see e.g. Proposition 1.1 in \cite{muellerrestriktion}.

These facts allow us to compute the integral kernel $M_G$ of $m(G)$ by using the
convolution kernel $M_L$ of $m(L)$, provided $M_L$ is in $L_1(\cal{G})$.
\begin{satz}
    \label{transferproposition}
  Let $m \in C_{\infty}(\R^+)$ such that the operator $m(L)$ has a convolution kernel $M_L\in L_1(\cal{G})$.
  Then for every $f \in \S(\R^2)$
    \begin{equation*}
        [m(G)f](p,v)=\int M_L(p'-p,q',v'-v-pq') \; dq' \;
    f(p',v') \; dp' \; dv'
    \end{equation*}
    holds.
\end{satz}

\begin{proof}
   Let $\t{f} \in \cH\cap \S,\ g:=(p,q,v)\in \cal{G}$. Then
   \begin{equation*}
       \begin{split}
       [\pi(m(L))&\t{f}](g)=\int M_L(x) \; \t{f}(gx) \; dx\\
       &=\int M_L ((p,q,v)^{-1}(p',q',v')) \; dq' \; \t{f}(p',0,v') \; dp' \; dv'\\
       &=\int M_L(p'-p,q'-q,v'-v+p(q-q')) \; dq' \; \t{f}(p',0,v') \; dp' \; dv'\\
       &=\int M_L(p'-p,q',v'-v-pq') \; dq' \; \t{f}(p',0,v') \; dp' \; dv'.
       \end{split}
   \end{equation*}
   Since $m(G)=m(\pi(L))=\pi(m(L))$, we finally get for $f\in \S(\R^2)$
   \begin{equation*}
    [m(G)f](p,v)=\int M_L(p'-p,q',v'-v-pq') \; dq' \; f(p',v') \; dp' \; dv'.
   \end{equation*}
\end{proof}
By this observation and the Fubini theorem, we deduce the following corollary.
\begin{korollar}
    \label{transfer}
    Let $m \in C_{\infty}(\R_{\geq0})$ such that the operator $m(L)$ has a convolution kernel $M_L\in L_1(\cal{G})$.
    Then $m(G)$ is bounded on $L_1(\R^2)$ and has an integral kernel $M_G$ with bounded Schur
    norm. Furthermore, $\dnorm{M_G}_{Schur}=\dnorm{M_L}_{L_1}$.
\end{korollar}

\begin{bemerkung}
    In the coordinates $(x,y,u)$ of the Heisenberg group $\H_1$, and with \mbox{$L=-(X^2+Y^2)$} and
    $G=-(\d{x}^2+x^2\d{u}^2)$,
    we get
    \begin{equation*}
        \label{transferformel}
        [m(G)f](x,u)=\int M_L(x'-x,y',u'-u-(x+x')y'/2) \; dy' \; f(x',u') \; dx' \; du'.
    \end{equation*}
\end{bemerkung}
From the wave estimates for the Heisenberg group by Müller and Stein (see the last part
of Section \ref{grushinsublaplacesection}), we know that the convolution kernel
$M^{\alpha}_L$ of the operator $m^{\alpha}(L):=\exp(i\sqrt{L})(1+L)^{-\alpha/2}$ lies in
$L_1(\H_1)$, for $\alpha>(d-1)/2$ and with $d=3$ the topological dimension of $\H_1$.
Hence we obtain that the operator $\exp(i\sqrt{G})(1+G)^{-\alpha/2}$ is bounded in
$L_1(\R^2)$ for $\alpha>1$.

Comparing this result with our result in Theorem \ref{theorem}, for $p=1$, we see that by
this approach we miss half a derivative. One can show that the $q'$-integral
\[\int M^{\alpha}_L(p'-p,q',v'-v-pq') \; dq'\]
is an oscillatory integral and hence one can hope to get this missing half a derivative
by using the method of stationary phase. Müller and Stein derived an explicit formula for
$M^{\alpha}_L$ as a one dimensional oscillatory integral. So, by using Proposition
\ref{transferproposition} we end up with a two dimensional oscillatory integral.
Unfortunately, this integral turned out to be very complicated. Hence we do not use this
representation of the integral kernel of $\exp(i\sqrt{G})(1+G)^{-\alpha/2}$ for the proof
of Theorem \ref{theorem}.

Nevertheless, Corollary \ref{transfer} is very useful to show that a weak multiplier
theorem for $G$ holds. We use this multiplier theorem in the proof of Theorem
\ref{theorem}.

\subsubsection*{A weak multiplier theorem for $G$}

\begin{satz}{\bf(Multipliers for $G$)}
    \label{reducmultipliertheorem}
    Suppose $\psi \in C^{k}(\R^+)$, with k assumed to be sufficiently large. If $\psi$ satisfies
    the inequalities
    \[
        \left\{ \begin{array}{l@{\quad , \quad}l}
        \n{\xi^{\l} \; \d{\xi}^{\l} \psi(\xi)} \lesssim \xi^{1/2} &\mbox{ for all } 0<\xi \leq 1,\\
        \n{\xi^{\l} \; \d{\xi}^{\l} \psi(\xi)} \lesssim \xi^{-1/2} &\mbox{ for all } 1\leq \xi <\infty,
        \end{array} \right.\\
    \]
    for $0\leq \l \leq k$, then the integral kernel of $\psi(G)$ has bounded Schur norm.
\end{satz}

\begin{proof}
    Müller and Stein showed in \cite{mueller}, Section 1.1 that under these conditions
    for $\psi$ the convolution kernel of $\psi(L)$ is in $L_1(\H_1)$.
%    Is $\psi$ a function then $M_\psi$ refers to the integral kernel of the operator $\psi(L)$. With a
%    standard dyadic partition of unity $1=\sum_{j=-\infty}^{\infty} \chi(2^{-j}\fa)$ for $\R^+$,
%    where $\chi \in \Cnu$ is non-negative and supported in $\intaa{1/2,2}$ and $\psi_j:=2^{\n{j}/2}\chi(2^{-j}\fa) \psi$
%    the function $\psi$ can
%    be written as $\psi(\xi)=\sum 2^{-\n{j}/2} \psi_j(\xi)$.
%    Define $\t{\psi}_j:=\psi_j(2^j \fa)$. Then $\t{\psi}_j$ is supported in $\intaa{1/2,2}$ and
%    $\sup_{j,\xi} \n{\t{\psi}_j^{(\l)}(\xi)}\lesssim 1$ for all $0\leq \l \leq k$.
%    Then the key step in the proof of the Marcinkiewicz-Mikhlin-Hörmander multiplier theorem for $\H_1$
%    (for which see e.g. \cite{FoS}, (Christ, Mauceri/Meda, Müller/Stein, Hebisch [diese Arbeiten habe ich nicht angesehen]),
%    \cite{muellerricci1}, \cite{muellerricci2},) shows that
%    \[ \sup_j \dnorm{M_{\psi_j}}_{L^1(\H_1)}=\sup_j \dnorm{M_{\t{\psi}_j}}_{L^1(\H_1)}<\infty\]
%    \fbox{$\dnorm{\F^{-1}(\psi(2^j\fa))}_{L^1}=2^{-j}\dnorm{\F^{-1}(\psi)(2^{-j}\fa)}_{L^1}=\dnorm{\F^{-1}(\psi)}_{L^1}$}
%    and thus the kernel $M_\psi=\sum_j 2^{-\n{j}/2} M_{\psi_j}$ belongs to $L^1(\H_1)$.
    By Corollary \ref{transfer}, the integrability of the kernel of $\psi(L)$ implies the boundedness of the
    Schur norm of the integral kernel of $\psi(G)$.
\end{proof}

\noindent The fact that $G$ is an image of $L$ under a representation of $\H_1$ implies
much more than we have stated here. In general, let $A$ be an operator on a well-behaved
group and $\pi$ a representation of this group. The process of getting information for
the operator $\pi(A)$ from properties of $A$ is often called "transference method".
Several people have worked on this subject. We refer here to a paper by Ronald R. Coifman
and Guido Weiss \cite{coifman}.
%
%
%
%
%%%%%%%%%%%%%%%%%%%%%%%%%%%%%%%%%%%%%%%%%%%%%%%%%%%%%%%%%%%%%%%%%%%%%%%%%%%%%%%%%%%%%%%%%%%%%%%%%%%%%%%%
%%%%%%%%%%%%%%%%%%%%%%%%%%%%%%%     Carnot-Caratheodory-distance   %%%%%%%%%%%%%%%%%%%%%%%%%%%%%%%%%%%%%
%
%
%
\subsection{The optimal control metric associated to $G$
%\\Gru\v{s}in operator
} \label{carnot} In this section we study the optimal control metric of $G$. The
definition for optimal control metrics resp. Carnot-Carathéodory metrics we have taken
from the paper \cite{strichartzg} by Strichartz.

Let $M$ be a connected $\Cu$ manifold, $X_1,\dots,X_m$ smooth real vector fields on $M$.
Let $x\in M$ and $v\in T_x M$. If $v$ is in the linear span of the vector fields
$X_1(x),\dots,X_m(x)$ we define
\begin{equation*}
    \dnorm{v}^2_x:=\inf \menge{\xi_1^2+ \dots + \xi_m^2}{\xi_1 X_1(x)+\dots + \xi_m
    X_m(x)=v}.
\end{equation*}
If $v$ is not in the linear span of $X_1(x),\dots, X_m(x)$ we define
\begin{equation*}
    \dnorm{v}^2_x:=\infty.
\end{equation*}
Let $I$ be an interval and $\gamma:I \to M$ be a piecewise $C^1$-curve. We call $\gamma$
\emph{admissible}, if $\dnorm{\dot{\gamma}(t)}_{\gamma(t)}<\infty$ for all $t\in I$.

We assume now that $\gamma$ is admissible and $I=\intaa{0,1}$. Set
\begin{equation*}
    L(\gamma):=\int_0^1 \dnorm{\dot{\gamma}(t)}_{\gamma(t)} \; dt.
\end{equation*}
The \emph{Carnot-Carathéodory metric} associated to the vector fields $X_1, \dots,X_m$ on
$M$ we define by
\begin{equation*}
    d_{cc}(x,y):=\inf \menge{L(\gamma)}{\gamma(0)=x,\ \gamma(1)=y,\ \gamma \; \mbox{regular}}.
\end{equation*}
Now let $A:=-(\sum_{\l=1}^m X_\l^2)$. If $A$ is a Hörmander type operator, then we define
the \emph{optimal control metric} $d_A$ associated to $A$ by $d_A:=d_{cc}$, where
$d_{cc}$ is the Carnot-Carathéodory metric associated to the vector fields $X_1,
\dots,X_m$. In \cite{melrose} this metric is called
\emph{$A$-distance}.\absatz\index{dcc} \index{d_a}

Now let $M:=\R^2$ and define vector fields $X_1$ and $X_2$ by
\begin{equation*}
  X_1:=\d{x_1}=(1,0) \quad X_2:=x_1\d{x_2}=(0,x_1).
\end{equation*}
Then $[X_1,X_2]=\d{x_2}=(0,1)$. The Carnot-Carathéodory distance $d_{cc}$ to these vector
fields is the optimal control metric of our operator $G=-(X_1^2+X_2^2)$.

Let $g^{11}(x)=1$, $g^{22}(x)=x_1^2$ and $g^{12}(x)=g^{21}(x)=0$. The Riemannian metric
$g_{jk}$, if there were one, ought to be the inverse of the metric $g^{jk}$, which does
not exists. $g^{jk}$ is often called sub-Riemannian metric.

Here we are interested in balls $B_{G}((x_1,x_2),R)$, belonging to this metric, of radius
$R$ and centered at $(x_1,x_2)$. Especially we like to show that \[B_{G}((x_1,x_2),R)
\teil B(x_1,c R)\times B(x_2,c R (R+\n{x_1})),\] with some constant $c$ and $B(x,R)$ the
ball with respect to the Euclidean metric on $\R$ of radius $R$ and centered at $x$. To
show this we study geodesics.\index{B_A}

Let $x,y \in M$. We want to find an admissible curve $\gamma=(\gamma^1,\gamma^2)$ with
\mbox{$\gamma(0)=x$}, $\gamma(1)=y$ and $d_{G}(x,y)=L(\gamma)$. We can assume that
$\dnorm{\dot{\gamma}(t)}_{\gamma(t)}$ is constant. To find this curve we have to minimize
$\int_0^1 \dnorm{\dot{\gamma}(t)}_{\gamma(t)} \; dt$ with $\gamma(0)=x$ and
$\gamma(1)=y$. Instead of minimizing this integral, we are allowed to minimize
    \begin{equation*}
        \int_0^1 \dnorm{\dot{\gamma}(t)}_{\gamma(t)}^2 \; dt
        =\int_0^1 (\dot{\gamma^1})^2+\frac{(\dot{\gamma^2})^2}{(\gamma^1)^2} \; dt,
    \end{equation*}
    with $\gamma(0)=x,\ \gamma(1)=y$. The minimizer is a solution of the Euler-Lagrange
    equations
    \begin{equation}
        \label{eulerlagrange}
        \begin{split}
        \ddot{\gamma^1}+\frac{(\dot{\gamma^2})^2}{(\gamma^1)^3}&=0\\
        \frac{\dot{\gamma^2}}{(\gamma^1)^2}&=constant.
        \end{split}
    \end{equation}
    For $\gamma(0)=(0,0)$ these equations have the solutions
    \begin{equation*}
        \begin{split}
        \gamma_{b,c}(t)&:=\left(\frac{c}{b} \sin(bt),\
        \frac{c^2}{b}\left(\frac{t}{2}-\frac{\sin(2bt)}{4b}\right)\right),\quad b,c \in \R\\
        \beta_c(t)&:=(ct,0), \quad c \in \R.
    \end{split}
    \end{equation*}
  %The Hamiltonian of the Carnot-Carathéodory distance is given by the Symbol of the Gru\v{s}in
  %operator
  %\begin{equation*}
  %  H(x,\xi)=\ed{2}(\xi_1^2+x_1^2\xi_2^2)
  %\end{equation*}
  For $d_{G}$-Balls centered in the origin we obtain essentially
  \begin{equation*}
    B_{G}(0,R)\sim \menge{(x_1,x_2)}{\norm{x_1}<R,\ \norm{x_2}<R^2}.
  \end{equation*}
  These calculations have also been done by Greiner, Holcman and Kannai in \cite{greiner}.

  With some more calculus, we get also formulas for $\gamma(0)=(x_1,0)$. Define
    \begin{equation*}
    \begin{split}
    \gamma^1&:=c_1\sin(bt)/b+c_2\cos(bt)/b,\\
    \gamma^2&:=\frac{1}{b}\Big(c_1^2\Big(\frac{t}{2}-\frac{\sin(2bt)}{4b}\Big)+\frac{c_1c_2}{b}\sin^2(bt)+c_2^2
    \Big(\frac{t}{2}+\frac{\sin(2bt)}{4b}\Big)\Big)+d,
    \end{split}
    \end{equation*}
with $c_1,c_2,b,d \in \R$. All solutions of \eqref{eulerlagrange} are of one of these
forms, or given by $\beta_c$.

Since $\gamma(0)$ should be $(x_1,0)$ we choose $c_2=x_1b$ and $d=0$. Thus by defining
\begin{equation*}
    \begin{split}
    \gamma^1_{b,c_1}&:=c_1\sin(bt)/b+x_1\cos(bt),\\
    \gamma^2_{b,c_1}&:=\frac{c_1^2}{b}\Big(\frac{t}{2}-\frac{\sin(2bt)}{4b}\Big)+\frac{x_1c_1}{b}\sin^2(bt)+x_1^2b
    \Big(\frac{t}{2}+\frac{\sin(2bt)}{4b}\Big),
    \end{split}
\end{equation*}
the function $\gamma_{b,c_1}(t):\intaa{0,1}\to \R^2,\; t \mapsto
(\gamma^1_{b,c_1}(t),\gamma^2_{b,c_1}(t))$ is part of a geodesic starting in $(x_1,0)$ of
length $\sqrt{c_1^2+x_1^2b^2}$. For every $c_1 \in \R$, there exists a $c\geq \n{x_1b}$
and an $\epsilon \in \{+1,1\}$ such that $c_1=\epsilon\sqrt{c^2-x_1^2b^2}$. Define
\begin{equation*}
    \begin{split}
    \gamma^1_{b,c,\epsilon}&:=\epsilon\sqrt{c^2-x_1^2b^2}\sin(bt)/b+x_1\cos(bt),\\
    \gamma^2_{b,c,\epsilon}&:=\frac{c^2}{b}\Big(\frac{t}{2}-\frac{\sin(2bt)}{4b}\Big)+\frac{\epsilon x_1}{b}\sqrt{c^2-x_1^2b^2}
    \sin^2(bt)+x_1^2\frac{\sin(2bt)}{2}.
    \end{split}
\end{equation*}
Now the function $\gamma_{b,c,\epsilon}(t):\intaa{0,1}\to \R^2,\; t \mapsto
(\gamma^x_{b,c,\epsilon}(t),\gamma^u_{b,c_1,\epsilon}(t))$ is part of a geodesic starting
in $(x_1,0)$ of length $c$.

An easy calculation shows, that we can write the components of $\gamma_{b,c,\epsilon}$ in
the following form
\begin{equation*}
    \begin{split}
    \gamma^1_{b,c,\epsilon}&=\epsilon \n{x_1 b}\sqrt{c^2/(x_1^2b^2)-1}\sin(bt)/b+x_1\cos(bt),\\
    \gamma^2_{b,c,\epsilon}&=\frac{c^2t}{2b}-c^2\frac{\sin(2bt)}{4b^2}+\epsilon x_1\n{x_1}\sqrt{c^2/(x_1^2b^2)-1}
    \sin^2(bt)+x_1^2\frac{\sin(2bt)}{2}.
    \end{split}
\end{equation*}

    With this information it is not difficult to show that there exists a constant
    $c$ with
    \[B_{G}((x_1,0),1) \teil B(x_1,c)\times B(0,c(1+\n{x_1})).\]
    To prove this we have to study the functions $\gamma_{b,1,\epsilon}$. We can restrict
    to \mbox{$\n{x_1}^2\geq C$}, with $C>1$ a sufficiently large constant and hence $\n{b}\leq 1/\n{x_1}\leq 1/C$ is small.
    Therefore, $\sin(2bt)/(2b^2)-t/b\lesssim 1$ and $x_1^2\sin(2bt)\lesssim \n{x_1}$.
    Now, we get for $1/(b^2x_1^2)\geq 2$ by the Taylor expansion of the sine function
     \[\n{\sin(bt)\n{x_1}\sqrt{b^{-2} x_1^{-2}-1}}\lesssim 1, \quad
    \n{\sin^2(bt)x_1^2\sqrt{b^{-2} x_1^{-2}-1}}\lesssim 1.\]
    And for $1/(b^2x_1^2)\leq 2$ we get the same estimates. Hence
    \[\n{\gamma^1_{b,1,\epsilon}(t)-x_1}\lesssim 1 \mbox{ and } \n{\gamma^2_{b,1,\epsilon}(t)}\lesssim
    1+\n{x_1},\]
    for all $t\leq 1$.
    Since $G=-X_1^2-X_2^2$ is homogeneous with respect to the automorphic dilation $\delta_r$ and
    translation invariant with respect to the variable $x_2$ we get
    %\[B_{cc}((x_1,0),R) \teil B(x_1,c R)\times B(0,c R (R+\n{x_1}))\]
    %and
    \[B_{G}((x_1,x_2),R) \teil B(x_1,c R)\times B(x_2,c R (R+\n{x_1})).\]
    We now want to use our previous notation. We denote the first variable by $x$ and the second by $u$.
    Then $G=-(\d{x}^2+x^2\d{u}^2)$ and
    we have proven the following proposition.
    \begin{satz}
        \label{carnotmetric}
        Let $B_{G}((x,u),R)$ denote the ball with respect to the optimal control metric associated to $G$,
        centered in $(x,u)$ and with
        radius $R\in \R^+$. There exists a constant $c$ such that for
        all $x,u \in \R$ and $R \in \R^+$
        \[B_{G}((x,u),R) \teil B(x,c R)\times B(u,c R (R+\n{x}))\]
        holds.
    \end{satz}
This proposition allows us to make a statement about the speed of propagation of our
wave.
\begin{satz}{\bf(Finite wave propagation speed)}
    Let $K_t$ denote the distribution kernel of $\cos(t\sqrt{G})$. There exists a constant $\const_0$ such that
    \begin{equation}
        \label{finitewavespeed}
        \supp K_t \teil \menge{(x',u',x,u)}{(x,u) \in B(x',\const_0 t)\times B(u',\const_0 t(t+\n{x'}))}.
    \end{equation}
\end{satz}
\index{const0}
\begin{proof}
    This assertion follows by Proposition \ref{carnotmetric} and since
    \begin{equation*}
        \supp K_t \teil \menge{(x',u',x,u)}{(x,u) \in B_{G}((x',u'),t)},
    \end{equation*}
    which was shown by Melrose in \cite{melrose}.
    In fact, Melrose showed that this is true for an arbitrary positive selfadjoint differential operator of second order
    on a compact manifold. The compactness is not essential here, since the operator
    $G$ is homogeneous with respect to $\delta_r$ and hence we have to show
    (\ref{finitewavespeed}) only for small $t$.
\end{proof}

\begin{bemerkung}
    A formal proof of \eqref{finitewavespeed} can also be obtained in the following way.

    \noindent The support of the distribution $\cos(t\sqrt{L})\delta_0$ is contained in
    $B^{\H}_{t}$, where $B^{\H}_t$ denotes the ball with respect to the
    optimal control metric associated to $L$ on $\H_1$, centered in $0$ and with radius $t$. There exists a constant $C\geq 1$ such that
    \[B^{H_1}_{\n{t}}\teil \menge{(z,t)\in \H_1}{\n{z}\leq Ct ,\; \n{u}\leq C^2t^2}.\]
    Let $K_{\cos(t\sqrt{L})}$ denote the distribution kernel of $\cos(t\sqrt{L})$. Formally,
    \begin{equation*}
        K_t(x',u',x,u)=\int K_{\cos(t\sqrt{L})}(x'-x,y',u'-u-(x+x')y'/2) \; dy'.
    \end{equation*}
    Observe that $\n{x'-x},\n{y'} \leq
    Ct$ and $\n{u'-u}\geq 2C^2t(t+\n{x'})$ implies
    \[\n{u'-u-(x+x')y'/2}\geq \n{u'-u}-\n{(x+x')y'/2}\geq 3C^2t^2/2>C^2t^2.\]
    Hence, if $K_{\cos(t\sqrt{L})}$ would be integrable, the assertion
    \eqref{finitewavespeed} would follow
    immediately.
    Unfortunately, $K_{\cos(t\sqrt{L})}$ is not integrable and so this second proof is only formally true.
\end{bemerkung}
To get an impression how the geometry looks like, we now want to show some figures.
Define
\[S_{G}(x',u'):=\menge{(x,u)}{d_G((x',u'),(x,u))=1}.\]
%$S_{G}(x',u')$ is the sphere centered in $(x',u')$.
\absatz\vspace{1.5cm}

\begin{center}
\unitlength1cm
\includegraphics{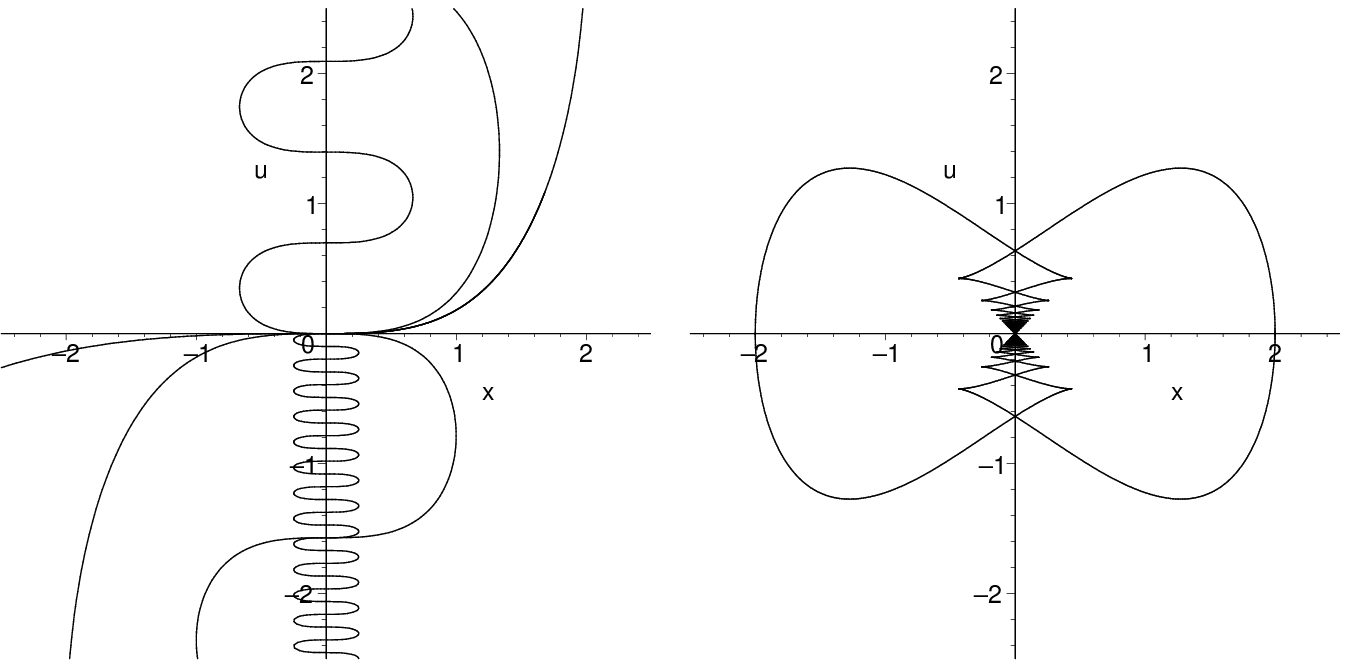}\\
\end{center}
\vspace{0.0cm} \hspace{2.9cm} Figure 1. \hspace{5.05cm} Figure 2.\\
\vspace{1cm}

\noindent Figure 1 shows geodesics starting in the origin. Figure 2 shows the sphere
$S_{G}(0,0)$. It has a highly complicated structure. Remarkable is that it has an inner
structure consisting of infinitely many edges tending to the origin. This a tribute to
the non-ellipticity of $G$ in $0$. The sphere is symmetric with respect to the axis
$x=0$.

\newpage
\begin{center}
\includegraphics{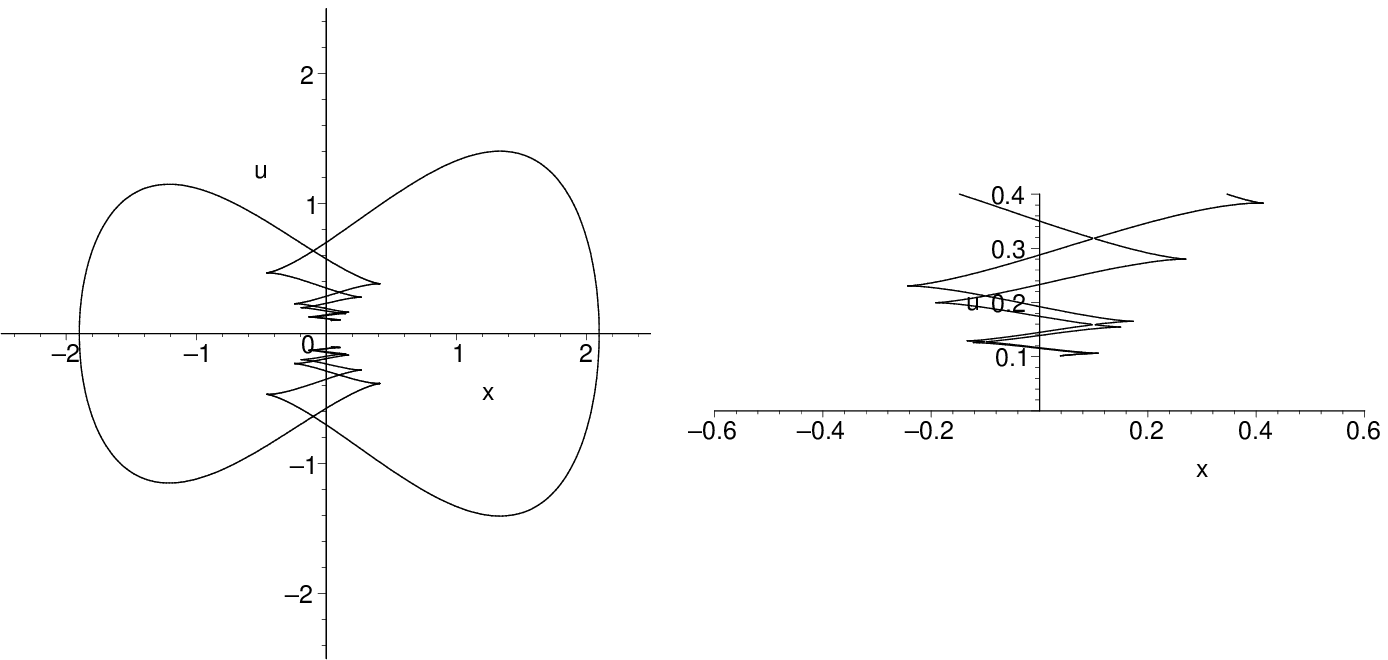}\\
\end{center}
\vspace{0.0cm} \hspace{2.9cm} Figure 3. \hspace{5.05cm} Figure 4.\\
\vspace{1cm}

\noindent Figure 3 shows the sphere $S_G(x',0)$, where $x'=0.1$. Figure 4 is an
enlargement of Figure 3 near the point $(0.1,0)$. The inner structure is given by only
finitely many edges. $S_{G}(0.1,0)$ is, in contrast to $S_{G}(0,0)$, not symmetric with
respect to any axis $x=c$.\vspace{1cm}
%\begin{center}
%\includegraphics{doktorps6.eps}\\
%\end{center}

\begin{center}
\includegraphics{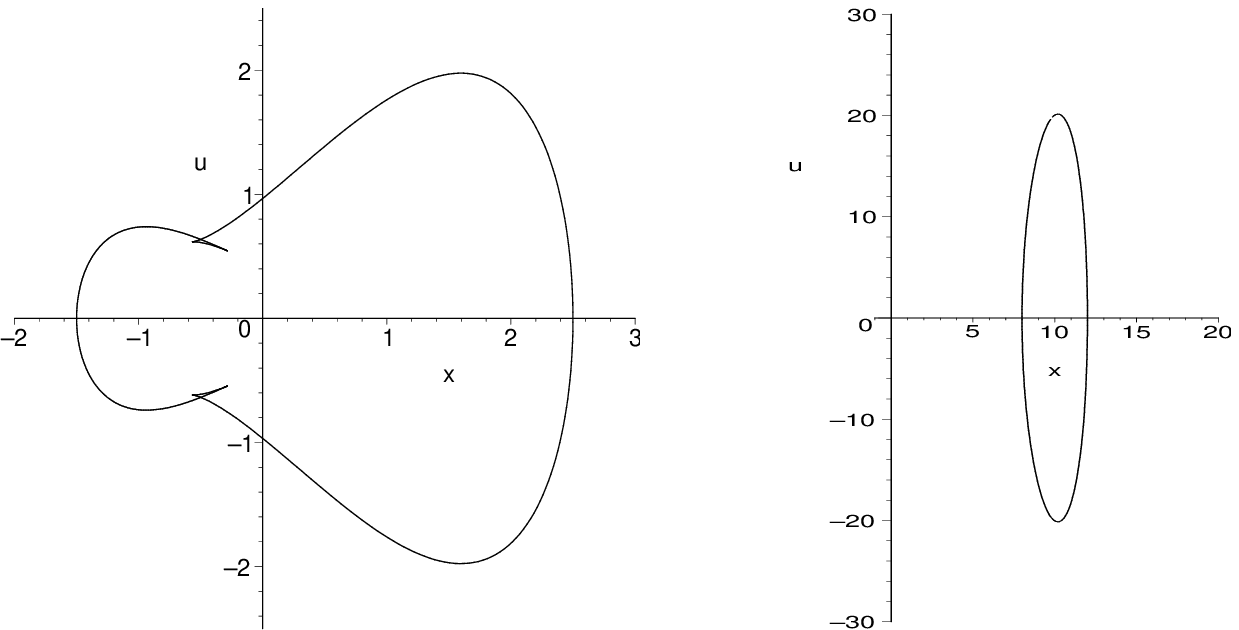}\\ %doktorps3.eps
\end{center}
\vspace{0.0cm} \hspace{2.9cm} Figure 5. \hspace{5.05cm} Figure 6.\\ \vspace{1cm}

\newpage
\noindent Figure 5 shows the sphere $S_{G}(0.5,0)$. The sphere $S_{G}(x,0)$ in figure 6
is given for $x=10$ and is nearly an ellipsoid with elongation comparable to $\n{x}$ in
the $u$-direction and elongation comparable to $1$ in the $x$-direction. Near this sphere
the operator $G$ is elliptic.\vspace{1cm}
% and therefore we can estimate the smoothness of waves starting in
%$(x,0)$ by Fourier integral methods. This will be done in chapter 4, section 1, by using
%the results of Seeger, Sogge and Stein.
\begin{center}
\includegraphics{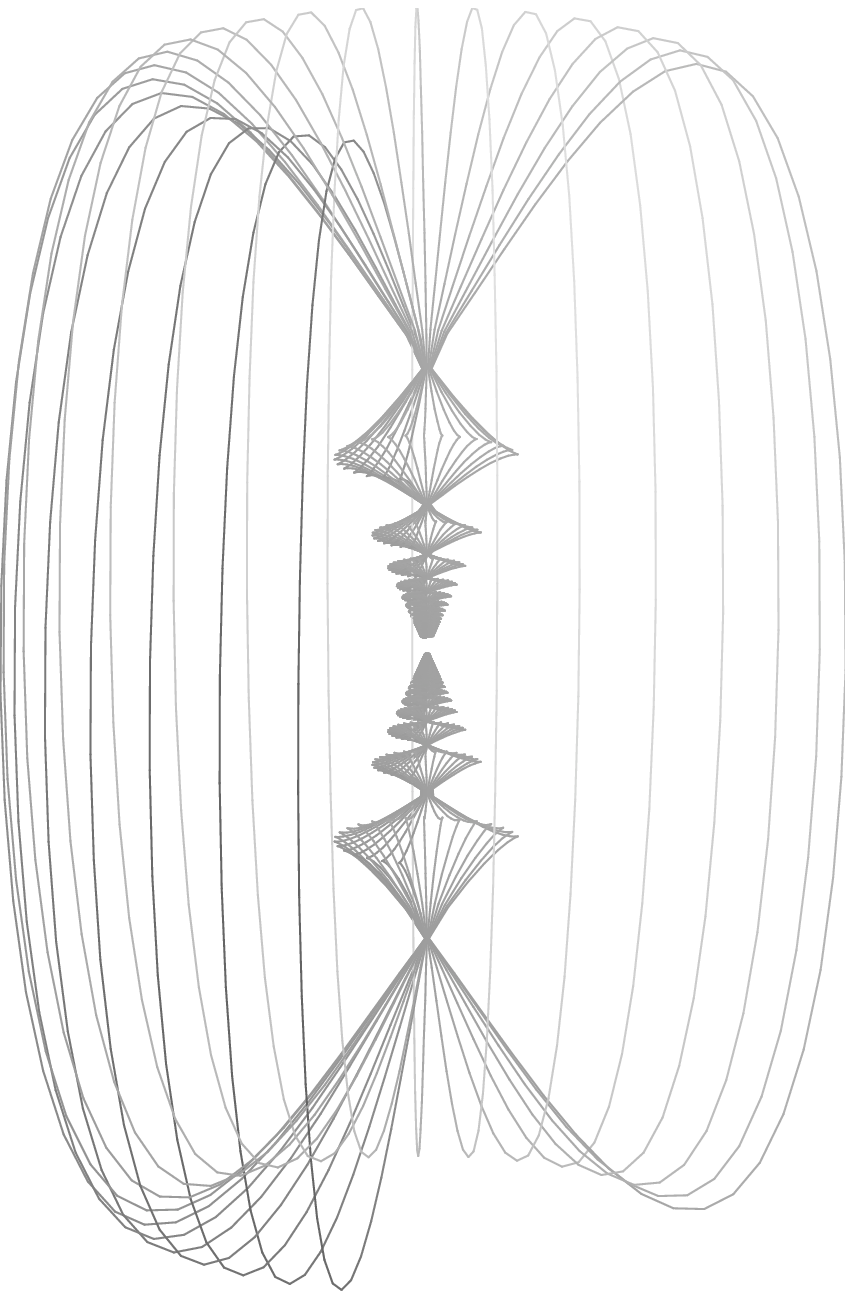}\\
\end{center}
\hspace{6.0cm} Figure 7.\\

\noindent Figure 7 shows the sphere belonging to the optimal control metric of the
sub-Laplacian $L$ on the Heisenberg group $\H_1$. We get this picture by rotating the
sphere in figure 2 around the axis $x=0$. This reflects the fact that $G$ is the image of
$L$ under the representation $\pi$ given in the previous section.\absatz\vspace{1cm}

%Before we end this section we want to make a view remarks about the wave propagators $
%$cos{\sqrt{G}}$
If we consider the Cauchy problem
\begin{equation*}
    \frac{\partial^2v}{\partial t^2}-\Delta v =0,\quad
    v \rvert_{t=0}=f,\quad \frac{\partial v}{\partial t}\rvert_{t=0}=g,
\end{equation*}
for the usual wave equation, then the solution $v$ is given by
\[v(t,x)=\conv{f}{Q_t(x)}+\conv{g}{P_t(x)},\]
where the wave propagators $Q_t$, $P_t$ are distributions with
\[
    \singsupp Q_t=\singsupp P_t=S^t(0)=\menge{x \in \R^d}{\n{x}=t}.\]
For our wave equation we get a similar result. If we denote the distribution kernel of
$\cos(t\sqrt{G})$ and $\sin(t\sqrt{G})/\sqrt{G}$ by $Q_t$ and $P_t$ respectively, we have
\begin{equation*}
    \begin{split}
    \singsupp Q_t&=\singsupp P_t
    \teil \menge{(x',u',x,u) \in \R^2}{d_{G}((x',u'),(x,u))=t}.
    \end{split}
\end{equation*}
This has been proven by Melrose in \cite{melrose}.

Therefore, for given $x'$, the set $S_{x'}:=S_G(x',0)$ in $\R^2$, shown in the figures 2
to 6 has the property that $(\singsupp Q_1) \cap ((x',0)\times \R^2) \teil (x',0)\times
S_{x'}$.
%Or,
%formally, $S_{x'}$ is the set of points so that all singularities of $\cos(\sqrt{G})
%\delta_{(x',0)}$ are contained in $S_{x'}$.\absatz

In the introduction, we mentioned that the most crucial part of the proof of Theorem
\ref{theorem} is the case when waves start near, but not exactly on the axis $x=0$. Now,
since we know the underlying geometry better, we want to pick up this subject once more.

First we thought that estimating waves starting in the origin should be most difficult,
since $G$ is not elliptic in $(0,0)$. Though, by comparing figure 2 and figure 7 one can
see that this situation is very similar to the situation on the Heisenberg group. By a
slightly modification of the methods of Müller and Stein we can prove that the integral
kernel of $\exp(i\sqrt{G})(1+G)^{-\alpha/2}\delta_{(x',0)}$ lies in $L_1(\R^2)$.

Waves starting in $(x',0)$ with $x'$ very far away from $0$ behave like waves for an
elliptic operator, as one can see in figure 6. Therefore, this case should also be not
very difficult. We do not consider this case in detail.

The most difficult case is when $0<\n{x'}<c$, where $c$ is a small constant. A reason for
this is that in this case the set $S_{x'}$ has a highly complex structure, as one can see
in figures 3 and 4. Especially the lack of symmetry, which we have for $S_0$, causes that
explicit formulas for $\exp(i\sqrt{G})(1+G)^{-\alpha/2}\delta_{(x',0)}$ are very
complicated.

%    Define $\delta_r(x,u):=(rx,r^2u)$.
%    \begin{equation}
%        [\d{x}^2+x^2\d{u}^2)f(r\fa,r^2\fa)]|_{(x,u)}=r^2[\d{x}^2+x^2\d{u}^2)f]|_{(rx,r^2u)}
%    \end{equation}
%   also
%   \begin{equation}
%       G(f\nach\delta_r)(x,u)=r^2(Gf)\nach \delta_r(x,u)
%   \end{equation}
%   oder mit $\delta_r(f)(x,u):=f(rx,r^2u)$
%   \begin{equation}
%       G(\delta_r(f))(x,u)=r^2\delta_r(Gf)(x,u)
%   \end{equation}
%   Damit
%   \begin{equation}
 %       \int K_{cos t\sqrt{G}}(x',u',tx,t^2u) \; f(x',u') \; dx' \; du'=
%        \int K_{cos \sqrt{G}}(x',u',x,u) \; f(tx',t^2u') \; dx' \; du'
%    \end{equation}
%   Damit
%   \begin{equation}
%        K_{cos t\sqrt{G}}(x',u',x,u)=K_{cos \sqrt{G}}(x'/t,u'/t^2,x/t,u/t^2)t^{-3}
%    \end{equation}
%    und
%    \begin{equation}
%        \supp K_{cos t\sqrt{G}}(x',u',\fa,\fa)=\supp K_{cos \sqrt{G}}(x'/t,u'/t^2,\fa/t,\fa/t^2)
%    \end{equation}
%   und da $\supp K_{cos \sqrt{G}}(x'/t,u'/t,\fa,\fa)\teil B(x'/t,c )\times B(u'/t^2,c (1+\n{x_1/t}))$
%   gilt
%   \begin{equation}
%       \supp K_{cos \sqrt{G}}(x'/t,u'/t,\fa/t,\fa/t^2)\teil B(x',c t)\times B(u',c t (t+\n{x_1}))
%   \end{equation}
%   Also
%   \begin{equation}
%       \supp K_{cos(t \sqrt{G})}(x',u',\fa,\fa)\teil B(x',c t)\times B(u',c t
%       (t+\n{x_1})).
%   \end{equation}

\newpage
\Section{The main theorem and a conjecture}

Since we now have the fundamentals, we want to present the main theorem one more time.
Furthermore, we state a conjecture for higher dimensional Gru\v{s}in operators, and give
a short sketch of the proof of our theorem.\absatz

The main theorem of this thesis reads as follows. Put $S_{\const_1}:=\menge{(x,u)\in
\R^2}{\n{x}\leq \const_1}$.\index{strip2} \setcounter{Theorem}{0}
\begin{Theorem}
    \label{theorem1}
    For every $\const_1>0$, $t>0$, $1\leq p \leq \infty$ and $\alpha>\norm{1/p-1/2}$ there exists a
    constant $C_{p,t,\const_1}^\alpha$ such that for all $f$ in $\S$ with $\supp f\teil S_{\const_1}$ the estimates
    \begin{equation*}
        \Bigdnorm{\frac{\cos(t \sqrt{G})}{(1+G)^{\alpha/2}}f}_{L_p(\R^2)} \leq C^\alpha_{p,t,\const_1}
        \dnorm{f}_{L_p(\R^2)},
    \end{equation*}
    and
    \begin{equation*}
        \Bigdnorm{\frac{\sin(t\sqrt{G})}{\sqrt{G}(1+G)^{(\alpha-1)/2}}f}_{L_p(\R^2)} \leq
        C^\alpha_{p,t,\const_1}
        \dnorm{f}_{L_p(\R^2)},
    \end{equation*}
    hold.
\end{Theorem}
For higher dimensional Gru\v{s}in operators $G_n$ we conjecture that the following holds.
\begin{vermutung}
    Let $d:=n+1$ denote the topological dimension. For every $\const_1>0$, $t>0$, $1\leq p \leq \infty$ and \mbox{$\alpha>(d-1)\norm{1/p-1/2}$} there exists a
    constant $C_{p,t,\const_1}^\alpha$ such that for all $f$ in $\S$ with $\supp f\teil S_{\const_1}$ the estimates
    \begin{equation*}
        \Bigdnorm{\frac{\cos(t \sqrt{G_n})}{(1+G_n)^{\alpha/2}}f}_{L_p(\R^d)} \leq C^\alpha_{p,t,\const_1}
        \dnorm{f}_{L_p(\R^d)},
    \end{equation*}
    and
    \begin{equation*}
        \Bigdnorm{\frac{\sin(t\sqrt{G_n})}{\sqrt{G_n}(1+G_n)^{(\alpha-1)/2}}f}_{L_p(\R^d)} \leq
        C^\alpha_{p,t,\const_1}
        \dnorm{f}_{L_p(\R^d)},
    \end{equation*}
    hold.
\end{vermutung}
Many our technics we use for the proof of Theorem \ref{theorem1} are also applicable in
the higher dimensional case. Our computations for this case so far give rise to the hope
that the conjecture is really true, but we have not gone into the details yet.\absatz

Instead of Theorem \ref{theorem1} we show the following.
\begin{Theorem}
    \label{theorem2}
    Let $\const_1>0$, $1\leq p \leq \infty$. The operator $\exp(i\sqrt{G})(1+G)^{-\alpha/2}$ extends to
    a bounded operator from $L_p(S_{\const_1})$ to $L_p(\R^2)$ for $\alpha>\n{1/p-1/2}$.
\end{Theorem}
The restriction to time $t=1$ is not essential, since $G$ is homogenous with respect to
the dilation $\delta_r:(x,u)\mapsto (rx,r^2u),\ r>0$, and hence one can deduce the case
$t\not=1$ by the case $t=1$. As we mentioned before, the assertion for
$\cos(\sqrt{G})(1+G)^{-\alpha/2}$ follows immediately. For the operator
$\sin(t\sqrt{G})G^{-1/2}(1+G)^{-(\alpha-1)/2}$, we use that it suffices to show the
assertion for $\eta(G)\sin(t\sqrt{G})G^{-1/2}G^{-(\alpha-1)/2}$, where $\eta$ is a smooth
function supported away from the origin (see Proposition \ref{reduclocalpart} in Section
\ref{reduclocalpartsection}).\absatz

We simplify our notation by defining
\begin{equation*}
    m^{\alpha}(G):=\exp(i\sqrt{G})(1+G)^{-\alpha/2}. \index{malpha}
\end{equation*}
Furthermore, we fix a constant $\const_1>0$. \index{const_1}

Throughout our calculations it turned out that it is very useful to compare our results
and formulas with the results and formulas that have been derived by Müller and Stein.
The topological dimension of the Heisenberg group $\H_m$ is $2m+1$ and the critical index
$\alpha(d,p)=2m\n{1/p-1/2}$ is $m$, for $p=1$. In our work the dimension of $\R^2$ can be
written as $2 \times 1/2+1$ and the critical index $\alpha(d,p)=\n{1/p-1/2}$ is $1/2$,
for $p=1$. Therefore, we set $m$ to be equal to $1/2$,
\[m:=1/2.\index{m}\]
In {\sc Remarks} we consider formulas for $L$ on $\H_m$ and only there we take $m$ to be
in $\N$. If $A$ is some term with respect to $G$, we denote the corresponding term for
$L$ on $\H_m$ by $A^\H$.

\subsubsection*{A short sketch of the proof of Theorem \ref{theorem2}}

By standard interpolation arguments, it suffices to prove the case $p=1$, hence we show
that for all $\alpha>1/2$ the integral kernel of $m^\alpha(G)$ has bounded Schur norm. We
can restrict to high frequencies in the spectrum of $G$, and instead of $m^\alpha(G)$ we
are allowed to study the operator $h^\alpha(G)$, with
\[h^\alpha(\xi):=\eta_N(\xi) \ \xi^{-\alpha/2} \ e^{i\sqrt{\xi}}\]
and $\eta_N$ smooth and supported in $\menge{\xi \in \R^+}{\xi\geq N}$.

Let $\X_B(x',u',x,u):=1$ for all $(x,u)\in B(x',2\const_0)\times
B(u',2\const_0(2+\n{x'})$ and $0$ otherwise. By using the finite speed of propagation
(see Proposition \ref{finitewavespeed}), it suffices to show that the integral kernel
$\X_B K_{h^\alpha(G)}$ has bounded Schur norm. This is Proposition
\ref{reduclocalpart}.\absatz

In the next chapter we present an idea how one can proof our result also in the case that
$f$ is not supported in a strip $S_C$. This idea reads as follows. By scaling in $u$, we
can transform $G$ into the operator $\t{G}:=-\d{x}^2-x^2/x'^2\d{u}^2$. Now let $x'>>1$.
On the set $\n{x-x'}\leq 1,\ u\in \R$ the operator $\t{G}$ is elliptic and just a smooth
perturbation of the Laplacian. Hence it should be possible to use Fourier integral
operator methods to obtain estimates for solutions to the wave equation $\d{t}^2+\t{G}$
(see Seeger, Sogge, Stein \cite{seeger}). From these estimates one would get that
\[\sup_{\n{x'}\geq C,u'\in \R} \dnorm{K_{h^\alpha(G)}(x',u',\fa,\fa)}_{L_1}\]
is bounded, where $C$ is a constant.

We exchange $\X_B$ by a smooth variant $\t{\X}_B$ of $\X_B$. Let $\Omega:=S_{\const_1}$
%=\menge{(x,u)}{\n{x}<\const_1}$
. The proof of the theorem is then reduced to showing that
the operator
\[f\mapsto \int \t{\X}_{B} K_{h^\alpha(G)}(x',u',\fa,\fa) \; f(x',u') \; dx' \; du'\] is
bounded from $L_p(\Omega)$ to $L_p(\R^2)$, for every $1<p<\infty$. This is Proposition
\ref{reduclocalpart2}.

$G$ and $iU$ are strongly commuting operators. Their joint spectrum is the closure of the
union of rays
\begin{equation*}
  \ray_{n, \epsilon}:=
  \Bigmenge{(\epsilon\lambda,\tau)}{\tau=(2n+1)\lambda,\; \lambda>0},\quad \epsilon:=\pm 1,\ n \in
  \N_0.
\end{equation*}
By a dyadic decomposition of the joint spectrum, we write $h^\alpha(G)$ as a sum of
operators $H_{k,j}^\epsilon,\ k\in \Z,\ j\in \N_0,\ \epsilon=\pm 1$, where
$H_{k,j}^\epsilon$ is given by
\begin{equation*}
  \begin{split}
  H_{k,j}^\epsilon f=
  2^{-\alpha k} \sum_n \X_j(2n+1) \; \gamma_n^{\epsilon}(iU) \; (\P_n^\epsilon f),
  \end{split}
\end{equation*}
with $\gamma_n^\epsilon(\lambda):=\t{\X}_{2k-j}(\epsilon\lambda) \;
e^{i\sqrt{(2n+1)\norm{\lambda}}}$, $\X_j:=\X(2^{-j} \fa)$ and
$\t{\X}_{2k-j}:=\t{\X}(2^{-2k+j} \fa)$ cut off functions and $\P_n^\epsilon$ the spectral
projection operator that corresponds to the ray $\ray_{n,\epsilon}$ in the joint
spectrum. These projection operators and especially their integral kernels we study in
Chapter \ref{strichartzchapter}. If we denote the integral kernel of the operator
$H_{k,j}^{\epsilon}$ by $K_{k,j}^{\epsilon}$, then away from the diagonal
$K_{k,j}^{\epsilon}$ is given by
%\begin{equation*}
%    \begin{split}
%    K_{k,j}^\epsilon &(x',0,x,u)=
%   \frac{2^{-\alpha_0 k}}{2\pi}\\
%    &\times\sum_{n=0}^{\infty} \X_j(2n+1) \int_{-\infty}^\infty P_n^\epsilon(x',0,x,u-s)
%    \; e^{i\sqrt{(2n+1)\n{\lambda}}} \; \X_{\l}(\epsilon\lambda) \; e^{-i\lambda s} \; d\lambda \; ds.
%    \end{split}
%\end{equation*}
%We put
%\begin{equation}
%  \Phi_{\l,n}^\epsilon(s):=\frac{1}{2\pi} \int_{-\infty}^{\infty} e^{i\sqrt{(2n+1)\n{\lambda}}} \;
%  \X_{\l}(\epsilon\lambda) \; e^{-i\lambda s} \; d\lambda.
%\end{equation}
%Then $\Phi_{\l,n}^{-1}(s)=\Phi_{\l,n}^{1}(-s)$ and
\begin{equation}
    \label{sketch1}
  K_{k,j}^\epsilon(x',0,x,u)=
  \frac{2^{-\alpha k}}{2\pi} \sum_{n=0}^{\infty} \X_j(2n+1) \int_{-\infty}^{\infty}
  P_n^\epsilon(x',0,x,u-s) \; \Phi_{k,j,n}^\epsilon(s)
  \; ds,
\end{equation}
where $P_n^\epsilon$ is the integral kernel of the projection operator $\P_n^\epsilon$,
and $\Phi_{\l,n}^\epsilon$ is given by an oscillatory integral. We can reduce to the case
$\epsilon=1$. By the method of stationary phase, $\Phi_{k,j,n}^1$ is roughly given by
$2^{3k/2-j} f(\sqrt{(m+n)/(2^{2k-j})}s^{-1}) \; e^{i(m+n)/s}$, with $f \in \Cnu(\R)$,
supported away from the origin. Now we choose $\alpha$ to be the critical index $m=1/2$.
To prove the theorem, it suffices to show that for every $\epsilon>0$ there exists a
constant $C_\epsilon$ with
\begin{equation}
    \label{sketch2}
  \sup\limits_{\n{x'}\leq 2\const_1} \; \sum_{j\in \N_0} \; \dnorm{\t{\X}_{B}K_{k,j}^1(x',0,\fa,\fa)}_{L_1}
  \leq C_{\epsilon} \; 2^{\epsilon k}.
\end{equation}
Because, having this assertion allows us to sum over all $k$, which gives us the desired
result for $h^{\alpha}(G)$. That the assertion (\ref{sketch2}) is true is stated in
Proposition \ref{dyadicprop2}.\absatz

For the proof of Proposition \ref{dyadicprop2}, we use many technical lemmata. Though,
there are two main observations that we want to mention. The first thing we need is an
explicit formula for $P_n^\epsilon$. In Chapter \ref{strichartzchapter} we derive
\begin{equation*}
  P_{n}(x',0,x,u/2)=C \; [Q_n-Q_{n-2}],
\end{equation*}
with
\begin{equation*}
    \begin{split}
    Q_n=\sum_{\l=0}^n \frac{\Gamma(\l+m+1)}{\Gamma(\l+1)}
    \frac{\Gamma(n-\l+m+1)}{\Gamma(n-\l+1)} \; e^{-i2\l\sigma}
    \; e^{in\sigma} \; \frac{(x²+x'²-iu)^{n/2}}{(x²+x'²+iu)^{n/2+m+1}}
    \end{split}
\end{equation*}
and $\sigma$ depending on $x,\ x'$ and $u$. Careful examinations rends that $Q_n$ behaves
like the sum of two terms of the form
\begin{equation*}
    \begin{split}
    Q_n^\pm=\frac{n^m}{\n{1-e^{-i2\sigma}}^{3/2}} \;
    e^{\pm in\sigma} \; \frac{(x²+x'²-iu)^{n/2}}{(x²+x'²+iu)^{n/2+m+1}}.
    \end{split}
\end{equation*}
\begin{bemerkung}
    On the Heisenberg group $\H_m$, Müller, Stein and Strichartz derived similar formulas.
    The corresponding convolution kernel $P^{\H}_{n}$ is given by
    \mbox{$P^{\H}_{n}=C_m[Q^{\H_m}_n-Q^{\H_m}_{n-1}]$}
    with
    \begin{equation*}
        Q^{\H}_n=i^{m+1}(-1)^n \frac{(n+m)!}{n!} \frac{(x²+y²-4iu)^n}{(x²+y²+4iu)^{n+m+1}}.
    \end{equation*}
    Observe that $\frac{(n+m)!}{n!}\sim n^{m}$.
\end{bemerkung}
For $x'=0$, the factor $e^{i\sigma}$ is equal to $i$ and hence $\n{1-e^{-i2\sigma}}$ is
just $2$. Roughly we get
\begin{equation*}
    \begin{split}
    Q_n^\pm(0,0,\fa,\fa)=C_m \; n^m \;
    i^{\pm n} \; \frac{(x²-iu)^{n/2}}{(x²+iu)^{n/2+m+1}}.
    \end{split}
\end{equation*}
This coincidence allows us to use many methods that were used by Müller and Stein in
their proof. Unfortunately, matters become very difficult when we choose $x'$ away from
$0$.\absatz

The second important observation we use is the following. By using our formulas we
derived for $Q_n$, and after some scaling, we see that the integral in (\ref{sketch1}) is
an oscillatory integral with phase function $s\mapsto \kphi(s) +2^{k-j}/s$, where $\kphi$
is roughly given by
\begin{equation*}
  \kphi(s)=\arctan\left(\frac{Rw-2xx_1(u-s)}{2Rxx_1+(u-s)w}\right),
\end{equation*}
with $R:=x^2+x_1^2$ and $w:=((x^2-x_1^2)^2+(u-s)^2)^{1/2}$. By
\begin{equation*}
  \begin{split}
  &X:=\frac{Rw-2xx_1(u-s)}{2Rxx_1+(u-s)w},\quad Y:=\frac{(R²+(u-s)²)w}{2Rxx_1+(u-s)w},\\
  &s:=s,\quad \psi(x,u,s):=(X,Y,s).
  \end{split}
\end{equation*}
we define new coordinates such that $\kphi(s)=\arctan(X)$ and $\d{s}\kphi(s)=X/Y$. For
$x_1=0$, these coordinates coincides with coordinates which were used by Müller and
Stein. On first sight, it is not clear that using them is really possible. The problem is
that the functional determinant of $\psi^{-1}$ can not be computed directly, or, to be
more exact, in a straight forward way.

Nevertheless, it is possible to find an expression for $\det(\psi^{-1})$. We can estimate
the integral in (\ref{sketch1}) by using these new coordinates and a refined partial
integration. We prove Proposition \ref{dyadicprop2}, which completes the proof of Theorem
\ref{theorem2}.
%
%
%
%
%
%
%%%%%%%%%%%%%%%%%%%%%%%%%%%%%%%%%%%%%%%%%%%%%%%%%%%%%%%%%%%%%%%%%%%%%%%%%%%%%%%%%%%%%%%%%%%%
%%%%%%%%%%%%%%%%%%%%%%   harmonische Analysis für den Grushin operator   %%%%%%%%%%%%%%%%%%%
%
%
%
\newpage
\Section{The joint functional calculus for $G$ and $iU$} \label{strichartzchapter}

In \cite{strichartz} Strichartz regarded harmonic analysis on the Heisenberg group $\H_m$
as the \emph{joint spectral theory of the operator $L$ and the operator $iU$}, where
$U=\d{u}$ is the partial derivative with respect to the central variable $u$ of the
group. \index{U}

The joint spectrum of these operators is the \emph{Heisenberg fan}, which is the closure
of the union of rays $\ray_{n, \epsilon}:=
  \bigmenge{(\lambda,\tau)}{\tau=\epsilon\lambda/(m+2n),\; \lambda>0}.$
The spectral decomposition of a function $f$ in $L_2$ can be given as
\[f=\sum_{k,\epsilon} \pint \conv{f}{\phi_{\lambda,k,\epsilon}} \; d\lambda,\]
where the functions $\conv{f}{\phi_{\lambda,k,\epsilon}}$ are joint eigenfunctions of
$iU$ and $L$. The functions $\phi_{\lambda,k,\epsilon}$ can be explicitly calculated in
terms of Laguerre polynomials.

In this chapter we adapt the methods of Strichartz to our situation. Instead of the joint
spectrum of the sub-Laplacian $iU$ and $L$ we study the joint spectrum of $iU$ and $G$,
which also consists of rays in $\R \times \R^+$. Our main concern here is to derive
simple formulas for certain spectral projection operators belonging to these rays and, in
the end, to derive a simple formula for the integral kernel $m^\alpha(G)$.

\subsection{Spectral projection operators to rays}
\label{harmonicgrushin1}

Let $U:=\d{u}$. Since $G$ and $iU$ are essentially self-adjoint and strongly commuting
operators, they have a well defined joint spectrum. This spectrum consists of the union
of rays
\begin{equation*}
  \ray_{n, \epsilon}:=
  \Bigmenge{(\epsilon\lambda,\tau)}{\tau=(2n+1)\lambda,\; \lambda>0}
\end{equation*}
for $\epsilon:=\pm 1$, $n \in \N_0$ together with the limit ray $\ray_\infty
=\menge{(0,\tau)}{\tau\geq 0}$. Here $\epsilon\lambda$ refers to the spectrum $iU$ and
$\tau$ to the spectrum of $G$.\index{ray} In analogy to the Heisenberg group we will call
the closure of the union of rays $\ray_{n, \epsilon}$ the \emph{Gru\v{s}in fan}.

We write
\begin{equation*}
  G=(iU)(-iGU^{-1}).
\end{equation*}
The operator $iU$ is easy. The operator $-iGU^{-1}$ can be written as a sum over spectral
projection operators to rays.
\newpage
\begin{satz}
    \label{spektralzerlegung1}
  Let $m$ be a bounded measurable function on the Gru\v{s}in fan. Then for every $f\in
  \S$
  \begin{equation*}
    \begin{split}
    [m&(G,iU)f](x,u)=\sum_{\epsilon=\pm 1} \sum_{n=0}^\infty
    \pint m(\epsilon\lambda,(2n+1)\lambda) \; [\P_{\lambda,n,\epsilon}f](x,u) \;
    d\lambda\\
    %&=\sum_{\epsilon=\pm 1} \sum_{n=0}^\infty
    %\pint m((2n+1)\lambda,\epsilon \lambda)
    %\int \phi_{\lambda,n,\epsilon}(x',u',x,u) \; f(x',u') \; dx' \; du' \; d\lambda
    \end{split}
  \end{equation*}
  with
  \begin{equation*}
    \begin{split}
    [\P_{\lambda,n,\epsilon}f](x,u)&:=\int \phi_{\lambda,n,\epsilon}(x',u',x,u) \; f(x',u') \; dx' \;
    du',\\
    \phi_{\lambda,n,\epsilon}(x',u',x,u)&:=\frac{\lambda^{1/2}}{2\pi} \;
    e^{-i\epsilon \lambda(u-u')}
    \; h_n(\lambda^{1/2} x') \; h_n(\lambda^{1/2} x).
    \end{split}
  \end{equation*}\index{Pn}
  and
  \begin{equation*}
    h_n(x):=(\sqrt{\pi}\;n!\;2^n)^{-1/2} \; (\d{x}-x)^n e^{-x²/2}
  \end{equation*}
  the $n$-th Hermite function.
\end{satz}

\begin{proof}
%How can we get joint eigenfunctions?

%$H:=-(\d{x}^2-x^2)$ is the \emph{harmonic oscillator Hamiltonian}.
%Let $n \in \N$.
%\begin{equation*}
%  H_n(x):=(-1)^n \; e^{x²} \; \d{x}^n e^{-x²}
%\end{equation*}
%is the $n$th Hermite polynomial. We define the Hermite functions

% are orthogonal with $\dnorm{h_n}=1$
%, i.e.,
%\begin{equation*}
%  \int h_n(x) \; h_m(x) \; dx=\delta_{nm}.
%\end{equation*}
%and they are eigenfunctions of $H$ with eigenvalues $(2n+1)$.\\
%For Schwartz functions $f$ we get the expression
%\begin{equation*}
%  f(x_1)=\sum_{n=0}^\infty [\int h_n(x) \; f(x) \; dx] \; h_n(x_1)
%\end{equation*}
%and
%\begin{equation*}
%  f(x_1)=f\big(\frac{\lambda x_1}{\lambda}\big)=\sum_{n=0}^\infty
%  [\int \lambda \; h_n(\lambda x) \; f(x) \; dx \; h_n(\lambda x_1).
%\end{equation*}
%The Fourier inversion formula implies
%\begin{equation*}
%  f(u_1)=\frac{1}{2\pi} \; \sum_{\epsilon=\pm 1}
%  \pint \int e^{i\epsilon \lambda(u_1-u)} \; du\; d\lambda.
%\end{equation*}
Observe that under the Fourier transform with respect to the variable $u$ the operator
$G$ is the Hermite operator $H_\lambda:=-(\d{x}^2-\lambda x^2)$. By using the spectral
decomposition of the Hermite operator and the Fourier inversion formula with respect to
the variable $u$, we get for any Schwartz function $f$ on $\R^2$
\begin{equation*}
  \begin{split}
  &f(x,u)
 % &=\frac{1}{2\pi} \; \sum_{\epsilon=\pm 1} \sum_{n=0}^\infty
 % \pint \lambda^{1/2} \; \int e^{i\epsilon \lambda(u'-u)} \; h_n(\lambda^{1/2} x) \;
 % f(x,u) \; dx \; du \; h_n(\lambda^{1/2} x') \; d\lambda\\
  =\sum_{\epsilon=\pm 1} \sum_{n=0}^\infty
  \pint \int \phi_{\lambda,n,\epsilon}(x',u',x,u) \; f(x',u') \; dx' \; du' \; d\lambda,
  \end{split}
\end{equation*}
with
\begin{equation*}
  \phi_{\lambda,n,\epsilon}(x',u',x,u):=\frac{\lambda^{1/2}}{2\pi} \;
  e^{-i\epsilon \lambda(u-u')}
  \; h_n(\lambda^{1/2} x') \; h_n(\lambda^{1/2} x).
\end{equation*}
The function $\Phi_{\lambda,n,\epsilon}(x,u):=e^{-i\epsilon \lambda u}
  \; h_n(\lambda^{1/2} x)$
is a joint eigenfunction of $G$ and $iU$ with
\begin{equation*}
  \begin{split}
  G\Phi_{\lambda,n,\epsilon}&=(2n+1)\lambda \; \Phi_{\lambda,n,\epsilon}\\
  iU\Phi_{\lambda,n,\epsilon}&=\epsilon\lambda \; \Phi_{\lambda,n,\epsilon}.
  \end{split}
\end{equation*}
Hence
\begin{equation*}
  \frac{G}{iU}\Phi_{\lambda,n,\epsilon}=\epsilon(2n+1) \; \Phi_{\lambda,n,\epsilon}
\end{equation*}
the proposition follows by the spectral theorem.
\end{proof}
This proposition implies the next corollary.

\begin{korollar}
    \label{spektralzerlegung2}
  Let $m$ be a bounded measurable function then for every $f\in\S$
  \begin{equation*}
    \begin{split}
    [m(-iGU^{-1})f](x,u)
    =\sum_{\epsilon=\pm 1} \sum_{n=0}^\infty
    m(\epsilon(2n+1)) \; [\P_{n,\epsilon}f](x,u)
    \end{split}
  \end{equation*}
  with
  \begin{equation*}
    \begin{split}
    &[\P_{n,\epsilon}f](x,u)=\\
    &\pint
    \int
    \frac{\lambda^{1/2}}{2\pi} \;
    e^{-i\epsilon \lambda(u-u')}
    \; h_n(\lambda^{1/2} x') \; h_n(\lambda^{1/2} x) \; f(x',u') \; dx' \; du'\;
    d\lambda.
    \end{split}
  \end{equation*}
\end{korollar}

%\begin{prop}
%  Let $m$ be a bounded measurable function then
%  \begin{equation*}
%    \begin{split}
%    &[m(-iGT^{-1})f](x_1,u_1)\\
%    &=
%    \sum_{\epsilon=\pm 1} \sum_{n=0}^\infty
%    m(\epsilon(2n+1)) \; \pint
%    \int \phi_{\lambda,n,\epsilon}(x_1,u_1,x,u) \; f(x,u) \; dx \; du\; d\lambda\\
%    &=:\sum_{\epsilon=\pm 1} \sum_{n=0}^\infty
%    m(\epsilon(2n+1)) \; [P_{n,\epsilon}f](x_1,u_1).
%    \end{split}
%  \end{equation*}
%\end{prop}

The operator $\P_{n,\epsilon}$ will be called the \emph{spectral projection operator to
the ray $\ray_{n, \epsilon}$}.

Observe that with $S(x,u):=(x,-u)$
\[[\P_{n,-1}f](x,u)=-[\P_{n,1}(f\nach S)](x,-u).\]
Hence it suffices to study $\P_{n,1}$. We assume now $\epsilon=1$ and define
$\P_n:=\P_{n,1}$ to make the notation simpler.

For $x,\ u \in \R$ and $f \in \S$ supported away form $u$, we see by using partial
integration that $[\P_{n}f](x,u)$ can be given by an absolutely convergent integral. By
Fubini's theorem we get
\begin{equation*}
    \begin{split}
    &[\P_{n}f](x,u)=\\
    &\int
    \pint
    \frac{\lambda^{1/2}}{2\pi} \;
    e^{-i \lambda(u-u')}
    \; h_n(\lambda^{1/2} x') \; h_n(\lambda^{1/2} x) \; d\lambda \; f(x',u') \; dx' \; du'.\\
    \end{split}
  \end{equation*}
We define now
\begin{equation}
    P_{n}(x',u',x,u):=\pint
    \frac{\lambda^{1/2}}{2\pi} \;
    e^{-i\lambda(u-u')}
    \; h_n(\lambda^{1/2} x') \; h_n(\lambda^{1/2} x) \;  d\lambda
\end{equation}
and
\begin{equation}
    \phi_{\lambda,n}(x',u',x,u):=\frac{\lambda^{1/2}}{2\pi} \;
    e^{-i\lambda(u-u')}
    \; h_n(\lambda^{1/2} x') \; h_n(\lambda^{1/2} x).
\end{equation}
Then
\[ [\P_{n}f](x,u)=\int
    P_{n}(x',u',x,u) \; f(x',u') \; dx' \; du',\]
for $f\in\S$, supported away form $u$.
\begin{bemerkung}
    Strichartz showed that the projection operators $\P^{\H}_{n}$ for $iU$ and $L$
    are Calderon-Zygmund operators and that
    \[\P^{\H}_{n}=c_n \; \delta_0 + p.v.P^{\H}_{n}.\]
    $P^{\H}_{n}$ is a kernel of similar type than $P_{n}$ and $p.v.P^{\H}_{n}$
    is the operator with kernel $P^{\H}_{n}$ in the principal value sense.

    The same should also be true here, but we do not need this additional information.
\end{bemerkung}
We compute $P_{n}$ explicitly. Recall the Mehler formula, according to which for all $x,y
\in \R$ and $z \in \C,\; \norm{z}<1$
\begin{equation*}
  \begin{split}
  \sum_{n=0}^\infty &h_n(x) \; h_n(y) \; z^n=
  \frac{1}{\sqrt{\pi}} \sum_{n=0}^\infty \frac{H_n(x)H_n(y) z^n}{2^n n!} \; e^{-\dz{x²+y²}}\\
  &=\frac{1}{\sqrt{\pi}\sqrt{1-z²}}e^{\tfrac{2xyz-z²(x²+y²)}{1-z²}} \; e^{-\dz{x²+y²}}
  \end{split}
\end{equation*}
holds. Thus
\begin{equation*}
  \begin{split}
  &\sum_{n=0}^\infty r^n\pint \phi_{\lambda,n}(x',u',x,u) \; d\lambda\\
  &=\sum_{n=0}^\infty r^n\pint \frac{\lambda^{1/2}}{2\pi} \; e^{-i \lambda(u-u')}
  \; h_n(\lambda^{1/2} x) \; h_n(\lambda^{1/2} x') \; d\lambda\\
  &=\frac{1}{2\pi^{3/2}} \pint \frac{\sqrt{\lambda}}{\sqrt{1-r²}}
  e^{\tfrac{2r\lambda xx'-r²(\lambda x²+\lambda x'²)}{1-r²}} \;
  e^{-\lambda\tfrac{x²+x'²}{2}}
  e^{-i\lambda(u-u')}\; d\lambda\\
  &=\frac{1}{2\pi^{3/2}} \pint \frac{\sqrt{\lambda}}{\sqrt{1-r²}}
  e^{-\lambda\big[(x²+x'²)\tfrac{1+r²}{2(1-r²)}-\tfrac{2rxx'}{1-r²}+i(u-u')\big]}
  \; d\lambda.
    \end{split}
\end{equation*}
For small $r$ this integral converges since, $x²+x'²\geq 2xx'$.
\begin{bemerkung}
  For $x'=0=u'$ we obtain
  \[\frac{1}{2\pi^{3/2}} \pint \frac{\sqrt{\lambda}}{\sqrt{1-r²}}
  e^{-\lambda[x²\tfrac{1+r²}{2(1-r²)}+iu]}\; d\lambda.\]
  On the Heisenberg group $\H_1$, Strichartz established the following formula
  \[\sum_{n=0}^\infty r^n\pint \phi_{\H,\lambda,n}(z,u) \; d\lambda
  =\frac{1}{4\pi²} \pint \frac{\lambda}{1-r} \;
  e^{-\lambda[\norm{z}²\tfrac{1+r}{4(1-r)}+i u]}\; d\lambda,\]
  with $\norm{z}²=x²+y²$.
\end{bemerkung}
The $\lambda$-integration can be computed easily, since the Gamma function has the
integral representation
\begin{equation*}
  \Gamma(z)=\pint t^{z-1} \;e^{-t}\; dt,
\end{equation*}
for $\re(z)>0$. We get
\begin{equation*}
    \begin{split}
  &\sum_{n=0}^\infty r^n\pint \phi_{\lambda,n}(x',u',x,u) \; d\lambda
  =\frac{\Gamma(3/2)2^{1/2}}{\pi^{3/2}} (1-r²) \\
  &\big( (x²+x'²)+2i(u-u')+r²((x²+x'²)-2i(u-u'))-4rxx'\big)^{-3/2}.
  \end{split}
\end{equation*}
Put
\begin{equation}
  \label{phiepsilondef}
  \begin{split}
  \phi&(x',u',x,u):=\frac{\Gamma(3/2)2^{1/2}}{\pi^{3/2}} \; (1-r²)\\
  &\times \big( (x²+x'²)+2i(u-u')+r²((x²+x'²)-2i(u-u'))-4rxx'\big)^{-3/2}
  \end{split}
\end{equation}
%\fbox{Ab hier setze $u_1=0$ und $\t{u}=2u$, $u:=\t{u}$.}
Since our operator is translation invariant with respect to the $u$-direction, we only
have to consider the case $u'=0$. In the following, we are interested in
$P_{n}(x',0,x,u/2)$, thus we replace $2u$ by $u$.

We use the abbreviations
\begin{equation}
    \begin{split}
    R&:=x²+x'²,\quad z:=R+iu,\quad p:=4xx', \index{R}
    \end{split}
\end{equation}
and
\begin{equation}
    \begin{split}
    w&:=\sqrt{x^4+x'^4-2(xx')^2+u^2}=\sqrt{R²-4(xx')²+u²},\\ \index{w}
    a&:=\sqrt{(x^2+x'^2)^2+u^2}=\sqrt{R^2+u^2}. \index{a}
    \end{split}
\end{equation}
With these definitions we get $\phi=C_m \; (1-r²)( z+r²\ab{z}-pr)^{-m-1},$ with $C_m$ a
constant. Recall that $m$ is always $1/2$, except when we speak about the Heisenberg
group $\H_m$.
\begin{lemma}
    \label{bnlemma}
  Let $\alpha \in \N/2$ then
  \begin{equation*}
    \frac{\d{r}^n}{n!}|_{r=0} (z+r²\ab{z}-rp)^{-\alpha}=C_{\alpha} \;b_{n,\alpha} \;
    \frac{\ab{z}^{n/2}}{z^{\alpha+n/2}} \; e^{in\sigma}
  \end{equation*}
  holds with
  \begin{equation*}
    \begin{split}
%    b_{n,\alpha}=&\frac{\Gamma(n+\alpha)}{\Gamma(n+1)} \; [1+e^{-i2n\sigma}]
%    +\sum_{\l=1}^{n-1} \frac{\Gamma(\l+\alpha)}{\Gamma(\l+1)}
%    \frac{\Gamma(n-\l+\alpha)}{\Gamma(n-\l+1)} \; e^{-i2\l\sigma}\\
    b_{n,\alpha}&=\sum_{\l=0}^{n} \frac{\Gamma(\l+\alpha)}{\Gamma(\l+1)}
    \frac{\Gamma(n-\l+\alpha)}{\Gamma(n-\l+1)} \; e^{-i2\l\sigma},\\
    e^{i\sigma}&=\left(\frac{p}{2\norm{z}}+i\sqrt{1-\frac{p²}{4\norm{z}²}}\right)\\
    &=\frac{2xx'+i\sqrt{x^4+x'^4-2(xx')²+u^2}}{\sqrt{(x²+x'²)²+u²}}
    =\frac{2xx'+iw}{\sqrt{R²+u²}}
    \end{split}
    \index{sigma}
  \end{equation*}
  and $C_\alpha$ a constant only depending on $\alpha$.
\end{lemma}

\begin{proof}
  Put $e^{i\vartheta}:=\Ri{z}$ and $\t{r}:=r e^{-i\vartheta}$. Then
  \begin{equation*}
    \begin{split}
    (z+r²\ab{z}-rp)^{-\alpha}
    =&z^{-\alpha}\left(1+\t{r}²-\t{r}\tfrac{p}{\norm{z}}\right)^{-\alpha}\\
    =&z^{-\alpha}[(\t{r}-\beta_1)(\t{r}-\beta_2)]^{-\alpha}.
    \end{split}
  \end{equation*}
  The roots $\beta_1$ and $\beta_2$ are given by
  $\beta_1=\tfrac{p}{2\norm{z}}+i\sqrt{1-\tfrac{p²}{4\norm{z}²}}=e^{i\sigma}$
  and $\beta_2=\ab{\beta_1}$. We have $\norm{\beta_1}=1=\norm{\beta_2}$ and
  \begin{equation*}
    (z+r²\ab{z}-rp)^{-\alpha}=z^{-\alpha}\; e^{i2\vartheta\alpha} \;
    [(r-e^{i(\vartheta+\sigma)})(r-e^{i(\vartheta-\sigma)})]^{-\alpha}.
  \end{equation*}
  Thus
  \begin{equation*}
    \begin{split}
    \frac{\d{r}^n}{n!} |_{r=0} & (z+r²\ab{z}-rp)^{-\alpha}\\
    =&\frac{z^{-\alpha} e^{i2\vartheta \alpha}}{n!} \Big[
    %(-1)^{-2\alpha} \;
    %\alpha \fa \dots \fa (\alpha+n-1) \;
    %e^{-i2\vartheta\alpha} \; e^{-i\vartheta n} \; e^{i\sigma n}\\
    %&+(-1)^{-2\alpha} \;
    %\alpha \fa \dots \fa (\alpha+n-1) \;
    %e^{-i2\vartheta\alpha} \; e^{-i\vartheta n} \; e^{-i\sigma n}\\
    \sum_{\l=0}^{n} \frac{n!}{\l!(n-\l)!}
    (-1)^{\l} \; \alpha \fa \dots \fa (\alpha+\l-1) \;
    (-e^{i(\vartheta+\sigma)})^{-\alpha-\l}\\
    &\times (-1)^{n-\l}\;\alpha\fa\dots\fa(\alpha+(n-\l)-1)\;
    (-e^{i(\vartheta-\sigma)})^{-\alpha-(n-\l)}\Big]\\
    %=&(-1)^{2\alpha} \; \Gamma(\alpha)^{-1} \; z^{-\alpha} \; \frac{\Gamma(n+\alpha)}{\Gamma(n+1)} \; e^{-i\vartheta n}[e^{i\sigma n} +e^{-i\sigma
    %n}]\\
    =&(-1)^{2\alpha} \; \Gamma(\alpha)^{-2} \; z^{-\alpha} \; e^{-i(\vartheta-\sigma)n} \;
    \sum_{\l=0}^{n} \frac{\Gamma(\l+\alpha)}{\Gamma(\l+1)}
    \frac{\Gamma(n-\l+\alpha)}{\Gamma(n-\l+1)} \; e^{-i2\l\sigma}\\
    =&C_\alpha \sum_{\l=0}^{n} \frac{\Gamma(\l+\alpha)}{\Gamma(\l+1)}
    \frac{\Gamma(n-\l+\alpha)}{\Gamma(n-\l+1)} \; e^{-i2\l\sigma}\; z^{-\alpha} \; e^{-i(\vartheta-\sigma)n} \;\\
    =&C_\alpha \; b_{n,\alpha} \; \frac{\ab{z}^{n/2}}{z^{\alpha+n/2}} \; e^{i\sigma n}.
    \end{split}
  \end{equation*}
\end{proof}
\noindent We are only interested in $\alpha=3/2$ and $\alpha=1/2$. Thus we put
\begin{equation}
    q_{n,m}:=b_{n,3/2} \mbox{ and } q_{n,m-1}:=b_{n,1/2}.
\end{equation}
Furthermore, we observe that
\begin{equation*}
  1-e^{i2\sigma}=\frac{R²+u²-(4(xx')²+i4xx'w-w²)}{R²+u²}=2\frac{w²-i2xx'w}{R²+u²},
\end{equation*}
and thus
\begin{equation*}
  \begin{split}
  \norm{1-e^{i2\sigma}}=2\frac{\sqrt{w^4+4(xx')²w²}}{R²+u²}
  =2w\frac{\sqrt{w²+4(xx')²}}{R²+u²}=2\frac{w}{a}.
  \end{split}
\end{equation*}
For $z=R+iu$, we get
\begin{equation}
    \label{arctan1}
    \begin{split}
    \Big(&\frac{R-iu}{R+iu}\Big)^{1/2} \; e^{i\sigma}=\Big(\frac{\ab{z}}{z}\Big)^{1/2} \; e^{i\sigma}=
    \Big(\frac{\ab{z}^2}{\n{z}^2}\Big)^{1/2} \; e^{i\sigma}
    =\frac{\ab{z}}{\n{z}} \; e^{i\sigma}
    =\frac{(R-iu)(2xx'+iw)}{\n{z}^2}\\
    &=\frac{R2xx'+uw+i(Rw-2xx'u)}{\n{z}^2}
    =\frac{\n{z}^2}{\n{z}^2} \; \exp\Big(i\arctan\Big(\frac{Rw-2xx'u}{R2xx'+uw}\Big)\Big)\\
    &=\exp\Big(i\arctan\Big(\frac{Rw-2xx'u}{R2xx'+uw}\Big)\Big).
    \end{split}
\end{equation}
$\arctan$ denotes the branch of $\tan^{-1}$ taking values in $\intaa{0,\pi}$, since
\mbox{$\n{Rw} \geq \n{2xx'u}$} and thus the imaginary part is always positive. We also
get
%\begin{equation*}
%  \begin{split}
%    \Big(&\frac{R-iu}{R+iu}\Big)^{1/2} \; e^{-i\sigma}=
%    \Big(\frac{\ab{z}}{z}\Big)^{1/2} \; e^{-i\sigma}=
%    \frac{R2xx_1-uw+i(-Rw-2xx'u)}{\n{z}^2}\\
%   &=\exp\Big(i\arctan\Big(\frac{-Rw-2xx'u}{R2xx'-uw}\Big)\Big)=
%   \exp\Big(i\arctan\Big(\frac{Rw+2xx'u}{uw-R2xx'}\Big)\Big)
%  \end{split}
%\end{equation*}
%with the $\arctan$ taking values in $\intaa{\pi,2\pi}$ or
\begin{equation}
    \label{arctan2}
    \Big(\frac{R-iu}{R+iu}\Big)^{1/2} \; e^{-i\sigma}=
    \exp\Big(-i\arctan\Big(\frac{Rw+2xx'u}{R2xx'-uw}\Big)\Big),
\end{equation}
where $\arctan$ denotes the same branch of $\tan^{-1}$ taking values in $\intaa{0,\pi}$.
For the integral kernel $P_n$ of $\P_{n}$ we get
\begin{equation}
  P_{n}(x',0,x,u/2)=\pint \phi_{\lambda,n}(x',0,x,u/2) \; d\lambda=C \; [Q_n-Q_{n-2}].
\end{equation}
For $Q_n$ and with $m=1/2$ we have the following formula
\setcounter{Qnqn}{\value{equation}}
\begin{equation}
    \label{qndefinition}
    \begin{split}
    Q_n=Q_n(x',0,x,u)=&\sum_{\l=0}^n \frac{\Gamma(\l+m+1)}{\Gamma(\l+1)}
    \frac{\Gamma(n-\l+m+1)}{\Gamma(n-\l+1)} \; e^{-i2\l\sigma}\\
    &\times\; e^{in\sigma} \; \frac{(x²+x'²-iu)^{n/2}}{(x²+x'²+iu)^{n/2+m+1}},
  \end{split}
\end{equation}
with
\begin{equation}
  \begin{split}
  e^{i\sigma}
  %\&=\left(\frac{p}{2\norm{z}}+i\sqrt{1-\frac{p²}{4\norm{z}²}}\right)\\
  &=\frac{2xx'+i\sqrt{x^4+x'^4-2(xx')²+u^2}}{\sqrt{(x²+x'²)²+u²}}
  =\frac{2xx_1+iw}{a}.
  \end{split}
\end{equation}

%
%%%%%%%%%%   alternative Formel für die Spectralprojectoren   %%%%%%%%%%%
%
For completeness we also want to state another formula for the projection operators. It
is based on the following observation
\begin{lemma}
  Let $f,g \in \Cu$ with $g^{(n)}(r)|_{r=0} =0$ for all $n>2$. Then
  \begin{equation*}
    \frac{1}{n!} (f\nach g)^{(n)}|_{r=0}=
    \sum_{\l=0}^{\gaussk{n/2}} \frac{1}{2^\l \l! (n-2\l)!} \; (g')^{n-2\l} \; (g'')^\l
    \; f^{(n-\l)}\nach g|_{r=0}.
  \end{equation*}
\end{lemma}
By this lemma we get easily
\begin{equation*}
  \begin{split}
  \frac{\d{r}^n}{n!}|_{r=0} &(z+r²\ab{z}-rp)^{-\alpha}\\
  &=\sum_{\l=0}^{\gaussk{n/2}} \frac{1}{2^\l \l! (n-2\l)!} \; (-p)^{n-2\l} \; (2\ab{z})^\l
  \;  (-1)^{(n-\l)} \frac{\Gamma(n-\l+\alpha)}{\Gamma(\alpha)} \; z^{-\alpha-(n-\l)}\\
  &=\frac{1}{\Gamma(\alpha)}
  \sum_{\l=0}^{\gaussk{n/2}} (-1)^\l \; \frac{\Gamma(n-\l+\alpha)}{\l! \; (n-2\l)!} \;
  p^{n-2\l} \; \frac{\ab{z}^\l}{z^{\alpha+(n-\l)}}\\
  &=\frac{1}{\Gamma(\alpha)}
  \sum_{\l=0}^{\gaussk{n/2}} (-1)^\l \; \frac{\Gamma(n-\l+\alpha)}{\l! \; (n-2\l)!} \;
  (4xx_1)^{n-2\l} \; \frac{(x²+x'²-iu)^\l}{(x²+x'²+iu)^{\alpha+(n-\l)}}.
  \end{split}
\end{equation*}
\noindent Hence
\begin{equation*}
  \pint \phi_{\lambda,n}(x',t',x,0) \; d\lambda=C_{m} \; [Q_n-Q_{n-2}].
\end{equation*}
with $Q_n$ given by
%\begin{equation}
%  Q_n=\sum_{\l=0}^n \frac{\Gamma(\l+m+1)}{\Gamma(\l+1)}
%    \frac{\Gamma(n-\l+m+1)}{\Gamma(n-\l+1)} \; e^{-i2\l\sigma}
%   \frac{(x²+x_1²-i\epsilon t)^{n/2}}{(x²+x_1²+i\epsilon t)^{m+1+n/2}} \; e^{in\sigma},
%\end{equation}
\setcounter{altcounter}{\value{equation}} \setcounter{equation}{\value{Qnqn}} \alpheqn
\begin{equation}
    \begin{split}
    Q_n=&\sum_{\l=0}^{\gaussk{n/2}} (-1)^\l \; \frac{\Gamma(n-\l+m+1)}{\l! \; (n-2\l)!} \;
    (4xx')^{n-2\l}\\ &\times \frac{(x²+x'²-i u)^\l}{(x²+x'²+i u)^{m+1+(n-\l)}}.
    \end{split}
\end{equation}
\reseteqn \setcounter{equation}{\value{altcounter}} It turns out, that this formula is
not very useful. In the following we only use formula \eqref{qndefinition}.
\subsubsection*{The region where $w$ is small}

For the region where $w$ is small we establish a second formula for the $P_{n}$. Let
\begin{equation*}
  \begin{split}
  \Phi&:=-\frac{C_m}{im}(z+r²\ab{z}-rp)^{-m}\\
  &=-\frac{C_m}{im}(x²+x_1²+iu+r²(x²+x_1²-iu)-4rxx_1)^{-m}.
  \end{split}
\end{equation*}
Then $\phi=\d{u}\Phi$ and by Lemma \ref{bnlemma} we obtain
\begin{equation*}
    \begin{split}
    \frac{\d{r}^n}{n!}|_{r=0} \; \Phi_{\epsilon}&=\frac{\d{r}^n}{n!}|_{r=0}\;
    \frac{C_m}{i(\alpha+1)} (z+r²\ab{z}-rp)^{-\alpha+1}\\
    &=C_m \; q_{n,m-1} \frac{\ab{z}^{n/2}}{z^{m+n/2}} \; e^{in\sigma}
    =:C_m \; q_{n,m-1} \frac{(R-iu)^{n/2}}{(R+iu)^{m+n/2}} \; e^{in\sigma}.
    \end{split}
\end{equation*}
Define
\begin{equation}
    R_n:=q_{n,m-1} \frac{(R-iu)^{n/2}}{(R+iu)^{m+n/2}} \; e^{in\sigma}.
\end{equation}
For $P_n$ we get
\begin{equation}
  \begin{split}
  \label{Pn=dtRn}
  P_{n}&=C_m \; [Q_n-Q_{n-2}]=\tfrac{1}{n!}\d{r}^n|_{r=0} \; \phi
  =\tfrac{1}{n!}\d{r}^n|_{r=0} \; (\d{t}\Phi)
  =\tfrac{1}{n!} \d{t} \d{r}^n|_{r=0} \; \Phi\\
  &=C_m \; \d{t} R_n.
  \end{split}
\end{equation}
\begin{bemerkung}
In the Heisenberg situation, the corresponding function
  $\phi^{\H}$ is given by
  \begin{equation*}
    \begin{split}
    \phi^{\H}&(x,y,u)=2^{m+1}\pi^{-m-1}m!\;(1-r)\\
    &\times \big( (x²+y²)+4i u+r((x²+y²)-4i u))^{-m-1}.
    \end{split}
  \end{equation*}
For the convolution kernel $P^{\H}_{n}$ we have a very similar expressions as for
$P_{n}$.
\[P^{\H}_{n}=C_m[Q^{\H_m}_n-Q^{\H_m}_{n-1}]\] and
\[P^{\H}_{n}=C_m \; \d{t} R^{\H_m}_n,\]
with
    \begin{equation*}
        Q^{\H}_n=i^{m+1}(-1)^n \frac{(n+m)!}{n!} \frac{(x²+y²-4iu)^n}{(x²+y²+4iu)^{n+m+1}}
    \end{equation*}
    and
    \begin{equation*}
        R^{\H}_n=
        %\frac{(m+n-1)!}{n!}\frac{(4u+i(x^2+y^2)^n}{(4u-i(x^2+y^2))^{m+n}}
        =i^m(-1)^{n}\frac{(m+n-1)!}{n!}\frac{((x^2+y^2)-4iu)^n}{((x^2+y^2)+4iu)^{m+n}}.
    \end{equation*}
\end{bemerkung}
%\ffbox{
%  Beachte die partielle Integration liefert folgende Faktoren\\
%  $2^{-j}$ wegen $n^{m-1}$ statt $n^m$\\
%  $\norm{z}^{1}=a$ wegen $z^{-m-n}$ statt $z^{-m-n-1}$\\
%  $2^{2k-j}$ aus der Ableitung von $\Phi_{k,j,n}$\\
% Insgesamt bekommt man $2^{2(k-j)}a$, das heißt man gewinnt für $2^{2(k-j)}a\leq 1$, 5
%skaliert gilt hie
%  \[2^{k-j}a\leq 1.\]

%  Für den Grushin Operator sollte man $2^{-j}w/a$ statt $2^{-j}$
%  gewinnen (man braucht weniger partielle Summationen in $\l$. Das heißt man gewinnt für
%  $2^{2(k-j)}aw/a\leq 1$, skaliert bedeutet das
%  \[2^{k-j}w\leq 1.\])}

Careful examination rends that $Q_{n}$, given by (\ref{qndefinition}), can be roughly
written as a sum of two functions $Q_n^+$ and $Q_n^-$, where each behaves like
\begin{equation*}
  \frac{n^{m}}{\n{\sigma}^{m+1}} \; e^{\pm in\sigma} \; \frac{(x²+x'²-iu)^{n/2}}{(x²+x'²+iu)^{n/2+m+1}}.
\end{equation*}
For $R_n$ we get a similar expression.
%The product of the oscillatory factors  gives a factor of the form
%\begin{equation*}
%  \exp\left(in\arctan\left(\frac{Rw\pm2xx'u}{uw-2Rxx'}\right)\right).
%\end{equation*}
These calculations will be done in the next section.
%
%
%
%%%%%%%%%%%%%%%%%%%%%%%%%%%%%%%%%%%%%%%%%%%%%%%%%%%%%%%%%%%%%%%%%%%%%%%%%%%%%%%%%%%%%%%%%%%%%
%%%%%%%%%%%%%%%%%%%%%%%%%%%%%%    Eigenschaften von Q_n    %%%%%%%%%%%%%%%%%%%%%%%%%%%%%%%%%%
%
%
%
\subsection{Properties of the function $Q_n$}
\label{harmonicgrushin2}

In the last chapter we obtained the following formulas
\begin{equation*}
  P_{n}(x',0,x,u/2)=\pint \phi_{\lambda,n}(x',0,x,u/2) \; d\lambda=C \; [Q_n-Q_{n-2}]
\end{equation*}
with
\begin{equation*}
  \begin{split}
    Q_n=Q_n(x',0,x,u)&=\sum_{\l=0}^n \frac{\Gamma(\l+m+1)}{\Gamma(\l+1)}
    \frac{\Gamma(n-\l+m+1)}{\Gamma(n-\l+1)} \; e^{-i2\l\sigma}\\
    &\times\; e^{in\sigma} \; \frac{(x²+x'²-iu)^{n/2}}{(x²+x'²+iu)^{n/2+m+1}}.
  \end{split}
\end{equation*}
where $m=1/2$.\absatz

Let $\X^+,\ \X^- \in \Cnu$ with $\X^+(x)+\X^-(x)=1$ for all $x\in \int{0,1}$. Furthermore
$\X^+(x)=0$ for $x\geq 3/4$ and $\X^-(x)=0$ for $x\leq 1/4$.
%For every $n \in \N$ let $\X_n^+, \X_n^- \in C^{\infty}$ with
%\[\X_n^+(x)=1 \mbox{ for all }0 \leq x \leq \gaussk{n/2}+1/4 \mbox{ and } \X_n^+(x)=0 \mbox{ for } x\geq
%\gaussk{n/2}-1/4,\]
%\[\X_n^-(x)=1 \mbox{ for all }0 \geq \gaussk{n/2}+3/4 \mbox{ and } \X_n^+(x)=0 \mbox{ for } x\leq
%\gaussk{n/2}+1/4.\]

%In addition we demand that $X^+(x) + \X^-(x)=1$ for all $0 \leq x \leq n$ and $\d{x}^N
%\X_n^{\pm}(x) \leq C$ for all $N \in \N_0$. Define
\begin{equation*}
    q_{n,m}^{\pm}:=\sum_{\l=0}^n \X^\pm (\l/n)\frac{\Gamma(\l+m+1)}{\Gamma(\l+1)}
    \frac{\Gamma(n-\l+m+1)}{\Gamma(n-\l+1)} \; e^{-i2\l\sigma}.
\end{equation*}
Then
\begin{equation}
    Q_n=(q_{n,m}^{+}+q_{n,m}^{-})e^{in\sigma} \; \frac{(x²+x'²-iu)^{n/2}}{(x²+x'²+iu)^{n/2+m+1}}.
\end{equation}
By the transformation $\l \to n-\l$, we get easily
\begin{equation*}
    \begin{split}
    q_{n,m}^{-}&=\sum_{\l=0}^n \X^- ((n-\l)/n)\frac{\Gamma(\l+m+1)}{\Gamma(\l+1)}
    \frac{\Gamma(n-\l+m+1)}{\Gamma(n-\l+1)} \; e^{i2\l\sigma} \; e^{-i2n\sigma}\\
    &=:\t{q}_n^{-}\;e^{-i2n\sigma}
    \end{split}
\end{equation*}
Thus
\begin{equation}
    \label{Qnqn+qn-}
    \begin{split}
    Q_n&=(q_{n,m}^{+} \; e^{in\sigma}+q_{n,m}^{-}e^{in\sigma}) \frac{(x²+x_1²-iu)^{n/2}}{(x²+x_1²+iu)^{n/2+m+1}}\\
    &=(q_{n,m}^{+} \; e^{in\sigma} + \t{q}_{n,m}^{-} \; e^{-in\sigma}) \frac{(x²+x_1²-iu)^{n/2}}{(x²+x_1²+iu)^{n/2+m+1}}
    \end{split}
\end{equation}
$\X^{-}((n-\fa)/n)$ is of the same type as $\X^{+}$ and $q_{n,m}^{+}$ and
$\t{q}_{n,m}^{-}$ behave in the same way. So, we can exchange $\t{q}_{n,m}^-$ in
\eqref{Qnqn+qn-} by $q_{n,m}^+$.
%We define $q_{n,m}:=q_{n,m}^{+}$.
%
%
%%%%%%%%%%%%%%%%%%%%%%    partielle Summation und die Betafunktion    %%%%%%%%%%%%%%%%%%%%%%%%%%%
%
%
%
\subsubsection*{Partial summation and the beta function}

%\ffbox{$\Gamma(n)=(n-1)!$}
We define for a sequence $a:=a_n$ the difference Operator
$\Delta$ by
\begin{equation*}
  \Delta(a)_n=\Delta a_n:=a_n-a_{n-1}.
\end{equation*}
The product rule reads as follows
\begin{equation}
    \label{productrule}
    \begin{split}
    \Delta(a_n b_n)&=a_n b_n-a_{n-1}
    b_{n-1}=(a_n-a_{n-1})b_n+a_{n-1}(b_{n}-b_{n-1})\\
    &=\Delta(a)_nb_n+a_{n-1}\Delta(b)_n.
    \end{split}
\end{equation}
\noindent Define
\[\Gamma_\l:=\frac{\Gamma(\l+m+1)}{\Gamma(\l+1)}.\]
Recall the definition of the beta function. For $\re z>0$ and $\re w>0$ the function
$B(z,w)$ is defined by
\begin{equation*}
  B(z,w):=\int_0^1 t^{z-1} (1-t)^{w-1} dt.
\end{equation*}
and $B(z,w)=\frac{\Gamma(z)\Gamma(w)}{\Gamma(z+w)}$. So we have
\begin{equation}
    \label{gammalbeta}
    \Gamma_\l=\frac{(\l+1)B(\l+3/2,1/2)}{\Gamma(1/2)}.
\end{equation}
%\[a_\l^n=\frac{\Gamma(\l+\alpha)}{\Gamma(\l+1)}
%    \frac{\Gamma(n-\l+\alpha)}{\Gamma(n-\l+1)}=
%    \frac{(\l+1)B(\l+\alpha,1/2)}{\Gamma(1/2)}\frac{(n-\l+1)B(n-\l+\alpha,1/2)}{\Gamma(1/2)}
%\]
\begin{lemma}
    Let $k,N \in \N$, $z,w \in \C$ and $\beta \in \N/2$.
    \begin{list}{}
        \item{(a)
        \begin{equation*}
        \Delta_k^N B(z+k,w)=(-1)^N B(z+k-N,w+N)
        \end{equation*}
        holds if $\; \re z+k-N>0$ and $\; \re w>0$.}
        \item{(b)
        \begin{equation*}
        B(k-\beta,\beta)\leq c \; \Gamma(\beta) \; k^{-\beta}
        \end{equation*}
        holds if $k-\beta>0$.}
        \item{(c)
        \begin{equation*}
        \Delta_\l^N \Gamma_\l \leq c_N \; \l^{1/2-N}
        \end{equation*}
        holds if $\l-N>0$.}
    \end{list}
\end{lemma}

\begin{proof}
  The first assertion follows by induction. For the main step we use the equality
  \begin{equation*}
    \begin{split}
    \Delta_k B(z+k,w)&=B(z+k,w)-B(z+k-1,w)\\
    &=\int_0^1 t^{z+k-1} (1-t)^{w-1} dt-\int_0^1 t^{z+k-2} (1-t)^{w-1} dt\\
    &=-\int_0^1 t^{z+k-2} (1-t) (1-t)^{w-1} dt=-\int_0^1 t^{z+k-2} (1-t)^{w} \; dt\\
    &=-B(z+k-1,w+1).
    \end{split}
  \end{equation*}
  For the second assertion we study the case $\beta \in \N$ first. In this case we have
  \begin{equation*}
    B(k-\beta,\beta)=\frac{\Gamma(k-\beta)\Gamma(\beta)}{\Gamma(k)}
    \leq \Gamma(\beta) \; k^{-\beta}.
  \end{equation*}
  For $\beta \in \N-1/2$ we use the Wallis product formula. This formula can be derived by
  using the product development of the sine function and reads as follows
  \begin{equation*}
    \lim_{n \to \infty} \left(\frac{2 \fa 4 \fa 6 \fa \dots \fa (2n)}
      {1 \fa 3 \fa 5 \fa \dots \fa (2n-1)} \right)^2 \frac{1}{2n}=\frac{\pi}{2}.
  \end{equation*}
  Thus $\Gamma(k+1/2)/\Gamma(k)\leq c \; k^{1/2}$ and
  \begin{equation*}
    \begin{split}
    B(k-\beta,\beta)&=\frac{\Gamma(k-\beta)\Gamma(\beta)}{\Gamma(k)}\\
    &=\Gamma(\beta)\frac{\Gamma(k-\beta)}{((k-1)\fa(k-2)\dots\fa(k-\beta-1/2))
    \Gamma(k-\beta-1/2)}\\
    &\leq  c \; \Gamma(\beta) \; \frac{k^{1/2}}{k^{\beta+1/2}}
    =c \; \Gamma(\beta) \; k^{-\beta}.
    \end{split}
  \end{equation*}
  By using (\ref{gammalbeta}) the third assertion is now an easy consequence.
\end{proof}

\noindent We also need some estimates for derivatives of the beta function. Put
\[\psi:=(\log \Gamma)'=\Gamma'/ \Gamma.\] This function is sometimes called the
\emph{digamma function}. The first derivative is connected with the half series
\begin{equation*}
  \zeta(x,s):=\sum\limits_{n=0}^{\infty} \frac{1}{(n+x)^s} \mbox{ for } x>0,
\end{equation*}
called the \emph{Hurwitz zeta function}. In fact we have
\begin{equation*}
  \psi'(x)=\sum\limits_{n=0}^{\infty} \frac{1}{(n+x)^2}=\zeta(x,2)
\end{equation*}
and the derivatives $\psi^{(N)}(x)$ of $n$-th order in the point $x$ is bounded by $c_N
\n{x}^{-N}$.
\begin{bemerkung}
  For $x=1$, the Hurwitz zeta function is the well known \emph{Riemannian zeta function}
  $\zeta(s)$.
\end{bemerkung}

\noindent By these facts we obtain the following lemma.
\begin{lemma}
  For all $x \gtrsim 1,\ w>0$
  \begin{equation*}
    \d{x}^N B(x,w)\leq C_w \; \n{x}^{-w-N}
  \end{equation*}
  holds.
\end{lemma}

\begin{proof}
  First we want to study the case $N=1$. Observe that
  \begin{equation*}
    \begin{split}
    \d{x} B(x,w)&=\d{x}
    \frac{\Gamma(x)\Gamma(w)}{\Gamma(x+w)}=\frac{\Gamma'(x)}{\Gamma(x+w)}\Gamma(w)
    -\frac{\Gamma'(x+w)}{\Gamma(x+w)}\frac{\Gamma(x)\Gamma(w)}{\Gamma(x+w)}\\
    &=[\psi(x)-\psi(x+w)]B(x,w).
    \end{split}
  \end{equation*}
  The estimate follows by the mean value theorem.
  The higher cases $N >1$ follow in the same way by using the product formula \eqref{productrule}.
\end{proof}
By these observations it is obvious to regard the functions $\Gamma_\l$ and
$\Gamma_{n-\l}$ as ``symbols'' of order $1/2$ with respect to the operators $\Delta$ and
$\partial$.
%
%
%
%\subsubsection*{Symbols}
%
%
%
By the previous lemmata, one gets
\begin{equation}
    \label{symbols1}
    \n{\d{\l}^{N_1} \Delta_\l^{N_2} \Gamma_\l} \leq c_{N_1,N_2} \; \l^{1/2-N_1-N_2}
\end{equation}
for all $N_2 \leq \l$. In fact
\begin{equation*}
    \begin{split}
    \d{\l}^{N_1} \Delta_\l^{N_2} \Gamma_\l=&\frac{1}{\Gamma(1/2)} \;
    \d{\l}^{N_1} \Delta_\l^{N_2} (\l+1) B(\l+3/2,1/2)\\
    =&\frac{1}{\Gamma(1/2)} \; \d{\l}^{N_1} \big( \Delta_\l^{N_2-1} B(\l+3/2,1/2)+\l\;
    \Delta_\l^{N_2}B(\l+3/2,1/2)\big)\\
    =&\frac{1}{\Gamma(1/2)} \; \d{\l}^{N_1} \big( (-1)^{N_2-1} B(\l+5/2-N_2,-1/2+N_2)\\&+(-1)^{N_2} \;
    \l \; B(\l+3/2-N_2,1/2+N_2)\big)
    \end{split}
\end{equation*}
and the estimate (\ref{symbols1}) follows. Similarly
\begin{equation}
    \label{symbols2}
    \n{\d{\l}^{N_1} \Delta_\l^{N_2} \Delta_n^{N_3} \Gamma_{n-\l}} \leq c_{N_1,N_2,N_3} \;
    n^{1/2-N_1-N_2-N_3}
\end{equation}
for all $N_2+N_3 \leq n/4$ with $\l \leq 3n/4$.
%\ffbox{ Beachte
%\[\Delta_k B(z-k,w)=B(z-k,w+1)\]
%\begin{equation*}
%    \begin{split}
%    &\Delta_\l^{N_2} \Delta_n^{N_3}
%    \Gamma_{n-\l}=\frac{1}{\Gamma(1/2)} \; \Delta_\l^{N_2}
%   \big( \Delta_n^{N_3-1} B(n-\l+3/2,1/2)\\&+(n-\l)\;
%    \Delta_n^{N_3}B(n-\l+3/2,1/2)\big)\\
%   =&\frac{1}{\Gamma(1/2)} \big(\Delta_\l^{N_2} \Delta_n^{N_3-1} B(n-\l+3/2,1/2)\\
%   &+(-1) \; \Delta_\l^{N_2-1} \Delta_n^{N_3} B(n-\l+3/2,1/2)\\&+(n-\l+1)
%   \Delta_\l^{N_2} \Delta_n^{N_3} B(n-\l+3/2,1/2)\big)\\
%    =&\frac{1}{\Gamma(1/2)} \big((-1)^{N_3-1} B(n-\l+5/2-N_3,-1/2+N_3+N_2)\\
%    &+(-1)^{N_3+1} \; B(n-\l+3/2-N_3,-1/2+N_3+N_2)\\&+(n-\l+1)
%    (-1)^{N_3} \; B(n-\l+3/2-N_3,1/2+N_3+N_2)\big)
%    \end{split}
%\end{equation*}}
\noindent The function $(\l,n)\mapsto \X(\l/n)$ behaves even better. An easy calculation
shows that
\begin{equation}
    \label{symbols3}
    \n{\d{\l}^{N_1}\Delta_n^{N_2} \Delta_\l^{N_3} \X(\l/n)}\leq c_{N_1,N_2,N_3} \; n^{-(N_1+N_2+N_3)}
\end{equation}
holds, for all $\l \leq 3n/4$. In fact
\[\Delta_\l \X(\l/n)=\X(\l/n)-\X((\l-1)/n)=n^{-1} \; \t{\X}_n(\l)\]
with $\t{\X}_n:=n(\X(x/n)-\X((x-1)/n))$ and $\t{\X}_n$ is of the same type than $\X_n$.
Furthermore,
\[\Delta_n \X(\l/n)=\X(\l(n-1)/(n^2-n))-\X(\l n/(n^2-n))=n^{-1} \; \t{\t{\X}}_n(\l)\]
with $\t{\t{\X}}_n:=n(\X(x/n)-\X(x/(n-1))$ and $\t{\t{\X}}_n$ is of the same type than
$\X_n$.

%This Lemma together with the proof of Lemma \ref{poisson} shows
%\begin{korollar}
%  \label{poissonkorollar}
%  Let $\X_n \in \Cnu$ with $\X_n(x)=0$ for $x \leq 1/2$ or $x\geq n+1$ and $\X_n(x)=1$ for
%  $1 \leq x \leq n$. For all $\sigma \leq  1/2$
% \begin{equation}
%   \sum\limits_{\l=1}^n  B(\l+3/2,1/2) \;
%    e^{i\sigma\l}=\sum \X_n(\l) \; B(\l+3/2,1/2) \; e^{i\sigma\l}
%    \lesssim \n{\sigma}^{-1/2}
%  \end{equation}
%  holds.
%\end{korollar}

%\[\Delta a_\l^n=\Delta[\Gamma_\l
%\Gamma_{n-\l}]=\Gamma_\l\Gamma_{n-\l}-\Gamma_{\l-1}\Gamma_{n-\l+1}=
%(\Delta\Gamma_{\l})\Gamma_{n-\l}+\Gamma_{\l-1}(\Delta\Gamma_{n-\l})
%\]

%\[\n{\Delta a_\l^n}=\Delta[
%    \frac{(\l+1)B(\l+2-1/2,1/2)}{\Gamma(1/2)}\frac{(n-\l+1)B(n-\l+2-1/2,1/2)}{\Gamma(1/2)}]
%    \lesssim \l^{-1/2}(n-\l)^{1/2}+\l^{1/2}(n-\l)^{-1/2}\]
\absatz

For the estimates to come we also need "half derivatives". In some sense, we get the next
lemma by applying the operator "$\Delta^{1/2}$". We do our calculation in the Fourier
space, where we estimate integrals. Finally, we use the Poisson summation formula.
\begin{lemma}
  \label{poisson}
  Let $\beta\in \R$ and $c>1$. Let $\X_n \in \Cnu$ with $\X_n(x)=0$ for $x \leq 1/2$ or $x\geq \gaussk{n/c}+1$ and
  $\X_n(x)=1$ for
  $1 \leq x \leq \gaussk{n/c}$. For all $\sigma$ with $\n{1-e^{i\sigma}} \leq  1/2$
  \begin{equation*}
    \bignorm{\sum\limits_{\l=1}^{\gaussk{n/c}}  \l^{-1/2} (n-\l)^{\beta} \; e^{i\sigma\l}}=
    \bignorm{\sum \X_n(\l) \; \l^{-1/2} \;
    (n-\l)^{\beta} \; e^{i\sigma\l}} \lesssim \frac{n^{\beta}}{\n{1-e^{i\sigma}}^{1/2}}
  \end{equation*}
  holds.
\end{lemma}

\begin{proof}
  Since our function is periodic we can assume that $0\leq \sigma\leq 1/2$ holds. Put
  \begin{equation*}
    \begin{split}
    I_n(\sigma-x)=I&:=\int \X_n(t) \; t^{-1/2} (n-t)^{\beta} \; e^{i(\sigma-x)t} \; dt\\
    &=n^{1/2+\beta} \; \int \X_n(nt) \; t^{-1/2} (1-t)^{\beta} \; e^{in(\sigma-x)t} \; dt.
    \end{split}
  \end{equation*}
  Then we have
  \begin{equation*}
    \sum \X_n(\l) \; \l^{-1/2}(n-\l)^{\beta} \; e^{i\sigma\l}=
    \sum_{x\in \Z} I_n(\sigma-x).
  \end{equation*}
  Put $y:=n(\sigma-x)$. For $\n{y}\leq C$ we have $\n{I_n}\lesssim n^{1/2+\beta}$.

  For $\n{y}\geq C$ we
  get easily by partial integration and since $\X_n(n\fa)'$ is zero outside a set of measure
  $\sim 1/n$.
  \begin{equation*}
    \begin{split}
    \n{I}&\lesssim \frac{n}{\n{y}^N} \; \int_0^1 \n{\d{t}^N [\X_n(nt) t^{-1/2}(1-t)^{\beta}]} \; dt\\
    &=\frac{n}{\n{y}^N} \; \sum_{M=0}^N c_M \; \int_0^1 \n{\d{t}^M [\X_n(nt)] \; [\d{t}^{N-M} t^{-1/2}(1-t)^{\beta}]} \; dt\\
    &\lesssim \frac{n}{\n{y}^N} \; \sum_{M=1}^N c_M \; \int_0^1 \n{\d{t}^M [\X_n(nt)] \; [\d{t}^{N-M} t^{-1/2}(1-t)^{\beta}]} \; dt
    +c_0 \;  \frac{n}{\n{y}^N} (1/n)^{1/2-N}\\
    &\lesssim  \frac{n}{\n{y}^N} \; n^M \; 1/n \; (1/n)^{-1/2-N+M}+n^{1/2}\big(\frac{n}{\n{y}}\big)^N
    =n^{1/2}\big(\frac{n}{\n{y}}\big)^N\\
    &=n^{1/2} \; \n{\sigma-x}^{-N}.
    \end{split}
  \end{equation*}
  For $\n{y}\geq C$ we also have a second estimate. Observe that
  \begin{equation*}
    \begin{split}
    I&=\frac{n}{y^{1/2}} \int_0^{y} \Psi_n(t/y) \; t^{-1/2} (1-t/y)^{\beta} \; e^{it} \; dt
    =\frac{n}{y^{1/2}} \int_0^{1} \dots \; dt +\frac{n}{y^{1/2}} \int_1^{y} \dots \; dt\\
    &=:\frac{n}{y^{1/2}}(I_1 +I_2),
    \end{split}
  \end{equation*}
  with $\Psi_n(\fa)=\X_n(n\fa)$, holds.
  The first Integral can be estimated by a constant. For the second integral we get by
  integration by parts
  \begin{equation*}
    \begin{split}
    \n{I_2}&\lesssim \n{\Psi_n(1/y)}+\n{\int_{1}^{y}
    (\d{t}\Psi_n(t/y)) \; t^{-1/2}(1-t/y)^{\beta} \; e^{it} \; dt}\\
    &+\n{\int_1^y \Psi_n(t/y) \; (\d{t}t^{1/2}(1-t/y)^{\beta}) \; e^{it} \;
    dt}.
    \end{split}
  \end{equation*}
  But this is bounded by a constant, since $\Psi_n$ is a bounded
  function, the support of $\d{t}\Psi_n(\fa/y)$ has measure $\sim y/n$,
  $\d{t}\Psi_n(t/y)\lesssim n/y$, $\d{t}t^{1/2}$ is
  integrable and $\d{t}(1-t/y)^{\beta}\lesssim 1/y$.
  %\begin{equation}
  %  \n{I_2} \lesssim \frac{n^{3/2}}{y^{1/2}} \; n^{-1} \; 1_+ (2/n-1/y)+
  %  \frac{n}{y^{1/2}} \; n^{-1} \; 1_+ (1+1/n-1/y)\lesssim 1.
  %\end{equation}
  For $x=0$, we get the following estimates
  \begin{equation*}
    \begin{split}
    \n{I_n(\sigma-x)}&\lesssim n^{1/2+\beta}\lesssim \frac{n^{\beta}}{\sigma^{1/2}},\quad \mbox{ if } n\sigma\leq C \mbox{ and}\\
    \n{I_n(\sigma-x)}&\lesssim \frac{n^{1/2+\beta}}{y^{1/2}}=\frac{n^{\beta}}{\sigma^{1/2}},\quad \mbox{ if } n\sigma\geq C.
    \end{split}
  \end{equation*}
  Hence $\n{I_n(\sigma)}\lesssim n^{\beta}\sigma^{-1/2}$.
  For $\n{x}\geq 1$ we have $y \geq C$ and
  \begin{equation*}
    \sum_{\n{x} \geq 1} \n{I_n(\sigma-x)}\lesssim n^{\beta} \; \sum_{\n{x} \geq 1} \n{x}^{-2}\lesssim n^{\beta}
    \lesssim n^{\beta} \; \sigma^{-1/2}.
  \end{equation*}
  The lemma follows by the Poisson summation formula.
\end{proof}

\begin{korollar}
    \label{poissonkorollar}
    \begin{equation}
        \sum_{\l=1}^{n-1} \Delta_n[\X((\l-1)/n)\Gamma_{n-\l+1}] \Delta_\l\Gamma_\l
        e^{-i2\l\sigma} e^{i\omega n}\lesssim \frac{n^{-1/2}}{\n{1-e^{-i2\sigma}}^{1/2}}
    \end{equation}
    holds for all $\sigma$ with $\n{1-e^{i\sigma}} \leq  1/2$.
\end{korollar}

\begin{proof}
    \eqref{symbols2} and \eqref{symbols3} together with the product formula \eqref{productrule} imply that the function $(n,\l)\mapsto \Delta_n[\X((\l-1)/n)\Gamma_{n-\l+1}]$
    behaves like $(n-\l)^{-1/2}$. \eqref{symbols1} implies that the function $\l \mapsto
    \Delta_\l\Gamma_\l$ behaves like $\l^{-1/2}$.
    Furthermore, $\X((\l-1)/n)$ is supported in $\intaa{0,\gaussk{n/c}}$ with a constant $c>1$.
    Hence this corollary follows by the proof of Lemma \ref{poisson} with slightly
    modifications.
\end{proof}

\begin{satz}
    \label{propqn1}
    \begin{itemize}
    \item[(a)] For every $\epsilon$ there exists a constant $C_\epsilon$ such that
    \begin{equation}
    \label{estimaterhoj}
    \begin{split}
        \Bignorm{\sum_{n\sim 2^j} &\sum_{\l=0}^n \X(\l/n) \; \frac{\Gamma(\l+m+1)}{\Gamma(\l+1)}
        \frac{\Gamma(n-\l+m+1)}{\Gamma(n-\l+1)}\; e^{-i2\l\sigma} \;
        e^{i\omega n}}\\
        &\leq C_\epsilon
        \frac{2^{j/2+j\epsilon}}{\n{1-e^{-i2\sigma}}^{3/2}} \;
        \frac{2^j}{(1+2^j\n{1-e^{i\omega}})^{1+\epsilon}},\mbox{ for all } \sigma \mbox{ and } \omega.
    \end{split}
    \end{equation}
    \item[(b)] For every $\epsilon$ there exists a constant $C_\epsilon$ such that
    \begin{equation}
    \label{estimaterhoj2}
    \begin{split}
        \Bignorm{\sum_{n\sim 2^j} &\sum_{\l=0}^n \X(\l/n) \; \frac{\l \; \Gamma(\l+m+1)}{\Gamma(\l+1)}
        \frac{\Gamma(n-\l+m+1)}{\Gamma(n-\l+1)}\; e^{-i2\l\sigma} \;
        e^{i\omega n}}\\
        &\leq C_\epsilon
        \frac{2^{j/2+j\epsilon}}{\n{1-e^{-i2\sigma}}^{5/2}} \;
        \frac{2^j}{(1+2^j\n{1-e^{i\omega}})^{1+\epsilon}},\mbox{ for all } \sigma \mbox{ and } \omega.
    \end{split}
    \end{equation}
    \end{itemize}
\end{satz}

%\begin{satz}
%    For every $\epsilon$ there exists a constant $C_\epsilon$ such that
%    \begin{equation}
%    \label{estimaterhoj}
%    \begin{split}
%        \Bignorm{\sum_{n\sim 2^j} &q_{n} \; e^{i\omega n}}\leq C_\epsilon
%       \frac{2^{j/2+j\epsilon}}{{\sigma}^{3/2}} \;
%        \frac{2^j}{(1+2^j\n{\omega})^{1+\epsilon}},\mbox{ for all } \omega \in \intaa{0,2\pi}.
%   \end{split}
%    \end{equation}
%    For every $\epsilon$ there exists a constant $C_\epsilon$ such that
%    \begin{equation}
%    \label{estimaterhoj}
%    \begin{split}
%       \Bignorm{\sum_{n\sim 2^j} &(\d{s}q_{n}) \; e^{i\omega n}}\leq C_\epsilon
%       \frac{2^{j/2+j\epsilon}\n{\d{s}\sigma}}{\n{\sigma}^{5/2}} \;
%       \frac{2^j}{(1+2^j\n{\omega})^{1+\epsilon}},\mbox{ for all } \omega \in \intaa{0,2\pi}.
%    \end{split}
%    \end{equation}
%\end{satz}

    The proof is mainly based on the fact that $\Gamma_\l$ behaves like $\l^{1/2}$ and
    that $\X(\l/n) \ \Gamma_{n-\l}$ behaves like $n^{1/2}$, since $\l \leq 3n/4$.
    Observe that the double sum in (\ref{estimaterhoj}) is bounded by $2^{3j}$.
    By
    partial summation with respect to $n$ we win $2^{-j}/\n{1-e^{i\omega}}$.
    By partial summation with respect to $\l$ we get roughly two terms $\l^{-1/2}\; (n-\l)^{1/2}$
    and $\l^{1/2} \; (n-\l)^{-1/2}$. We win $2^{-j}/\n{1-e^{-i2\sigma}}$.
    Hence the double sum is bounded by $2^{2j}/\n{1-e^{-i2\sigma}}$.
    To get an additional factor $2^{-j/2}/\n{1-e^{-i2\sigma}}^{1/2}$ we use Corollary
    \ref{poissonkorollar}. Hence, together with the partial summation in $n$, we get our
    result.

    Just one more partial summation in $\l$ would not give us the desired result, since we
    would have to sum over $\l^{-3/2} \; (n-\l)^{1/2}$ which gives, after summing all up,
    $2^{3/2j}/\n{1-e^{-i2\sigma}}^2$ by interpolating this with our previous result we would get
    $2^{3/2j+1/4j}/\n{1-e^{-i2\sigma}}^{3/2}$, which is not sufficient.\absatz

\begin{proof}
    For $j\leq 10$ the estimate (\ref{estimaterhoj}) is obvious. So we can restrict to
    $j\geq 10$. Now $n\geq c 2^{10}$. We assume that $\omega \in
    \intaa{-\pi,\pi}$, then $\n{1-e^{i\omega}}\sim \n{\omega}$. Furthermore, we assume that $\sigma \in
    \intaa{-\pi/2,\pi/2}$, then $\n{1-e^{-i2\sigma}}\sim \n{\sigma}$.

    Since
    \[\sum_{\l=\l_0}^{\l_1} (a_\l-a_{\l-1}) \; e^{i\omega \l}=
    \sum_{\l=\l_0}^{\l_1-1} a_\l \; e^{i\omega \l}(1-e^{i\omega})+a_{\l_1}\; e^{i\omega \l_1}-a_{\l_0-1}
    \; e^{i\omega \l_0}
    \]
    for an arbitrary sequence $a_\l$ holds, we get
    \begin{equation*}
        \begin{split}
        \sum_{\l=0}^n \X(\l/n) &\; \frac{\Gamma(\l+m+1)}{\Gamma(\l+1)}
        \frac{\Gamma(n-\l+m+1)}{\Gamma(n-\l+1)}\; e^{-i2\l\sigma}\\
        =&\sum_{\l=0}^1 \X(\l/n) \Gamma_\l \Gamma_{n-\l} \; e^{-i2\l\sigma} +
        \sum_{\l=2}^n \X(\l/n) \Gamma_\l \Gamma_{n-\l} \; e^{-i2\l\sigma}\\
        =&\sum_{\l=0}^1 \X(\l/n) \Gamma_\l \Gamma_{n-\l} \; e^{-i2\l\sigma}
        +\Gamma_1\Gamma_{n-1} \; e^{-i2\sigma} \; (1-e^{-i2\sigma})^{-1}\\
        &+\sum_{\l=1}^n \Delta[\X(\l/n) \Gamma_\l \Gamma_{n-\l}] \; e^{-i2\l\sigma} \;
        (1-e^{-i2\sigma})^{-1}=:q_n^1+q_n^2+q_n^3,
        \end{split}
    \end{equation*}
    since $\X(x)=0$ for $x\geq 1$.
    By partial summation with respect to $n$ the function $\sum_{n\sim 2^j} q_n^1 \; e^{i\omega n}$ can
    be written as a sum of two terms of the form
    \[(1-e^{i\omega})^{-1} \sum_{n\sim 2^j} \sum_{\l=0}^1
        \Delta_n[\X(\l/n) \Gamma_\l \Gamma_{n-\l}] \; e^{-i2\l\sigma} e^{i\omega n}
    \]
    and
    \[(1-e^{i\omega})^{-1} \sum_{\l=0}^1 \X(\l/n) \Gamma_\l \Gamma_{n-\l} \;
        e^{-i2\l\sigma} e^{i\omega n}\vert_{n\sim 2^j}.\]
    By the symbol properties (\ref{symbols2})$-$(\ref{symbols3}) of the functions
    $\Gamma_{n-\l}$ and $\X(\l/n)$ both terms are bounded by
    $2^{j/2}\n{\omega}^{-1}$.

    By the same arguments we obtain that
    $\sum_{n\sim 2^j} q_n^2 \; e^{i\omega n}$ is bounded by
    $2^{j/2}\n{\sigma}^{-1}\n{\omega}^{-1}$. We are left with the sum over $q_n^3$.
    Once more we use partial summation with respect to $n$.
    Observe that
    \begin{equation*}
        \sum_{n\sim 2^j} q_n^3 \; e^{i\omega n}=(1-e^{-i2\sigma})^{-1}
        \sum_{n\sim 2^j} \sum_{\l=1}^n \Delta_\l[\X(\l/n) \Gamma_\l \Gamma_{n-\l}] \;
        e^{-i2\l\sigma} \; e^{i\omega n}
    \end{equation*}
    can be written as
    \Aeqn
    \begin{equation}
        \begin{split}
        \big((1-e^{i\omega})(1-e^{-i2\sigma})\big)^{-1} \sum_{n\sim 2^j} \Delta_n \Big[&\sum_{\l=1}^n
        \Delta_\l[\X(\l/n) \Gamma_\l \Gamma_{n-\l}] \;
        e^{-i2\l\sigma} \Big] \; e^{i\omega n}\\
        =\big((1-e^{i\omega})(1-e^{-i2\sigma})\big)^{-1}
        \Big(&\sum_{n\sim 2^j} \sum_{\l=1}^{n-1} \Delta_n\Delta_\l[\X(\l/n) \Gamma_\l \Gamma_{n-\l}] \;
        e^{-i2\l\sigma} e^{i\omega n}\\
        +&\sum_{n\sim 2^j} \Delta_\l[\X(\l/n) \Gamma_\l \Gamma_{n-\l}]\vert_{\l=n} \;
        e^{-i2n\sigma} e^{i\omega n}\Big),
        \end{split}
    \end{equation}
    together with boundary terms of the form
    \begin{equation}
        (1-e^{i\omega})^{-1} \sum_{\l=1}^n \Delta_\l[\X(\l/n) \Gamma_\l \Gamma_{n-\l}] \;
        e^{-i2\l\sigma} \; (1-e^{-i2\sigma})^{-1} \; e^{i\omega n}\vert_{n\sim 2^j}.
    \end{equation}
    \reseteqn
    Since $\X((n-1)/n)=0=\Delta\X(1)$ we just have to study the first summand in $(A)$.
    Using (\ref{symbols1})$-$(\ref{symbols3}) and the product rule it is easy to see
    that $(A)$ is bounded by $c \; 2^j\n{\sigma}^{-1}\n{\omega}^{-1}$, where $c$ is a
    constant independent of $\sigma,\ \omega$ and $j$. Unfortunately we need a slightly better
    estimate. Observe that
    \begin{equation*}
        \begin{split}
        \Delta_n\Delta_\l[\X(\l/n) \Gamma_\l \Gamma_{n-\l}]=&
        \Delta_n\Delta_\l[\X(\l/n)\Gamma_{n-\l}] \; \Gamma_\l\\
        &+\Delta_n[\X((\l-1)/n)\Gamma_{n-\l+1}] \; \Delta_\l\Gamma_\l
        \end{split}
    \end{equation*}
    Applying one more time the summation operator $\Delta_\l$ to the first term and summing over all $\l$ gives us
    \[\sum_{\l=1}^n \Delta_\l[\Delta_n\Delta_\l[\X(\l/n)\Gamma_{n-\l}] \; \Gamma_\l]\lesssim n^{-1}\]
    Thus, we get
    \begin{equation*}
        \big((1-e^{i\omega})(1-e^{-i2\sigma})\big)^{-1}
        \sum_{n\sim 2^j} \sum_{\l=1}^{n-1} \Delta_n\Delta_\l[\X(\l/n)\Gamma_{n-\l}] \; \Gamma_\l \;
        e^{-i2\l\sigma} e^{i\omega n}\lesssim \frac{1}{\n{\sigma}^{2}\n{\omega}}.
    \end{equation*}
    By interpolating this result with our previous one, we get the bound
    $2^{j/2} \n{\sigma}^{-3/2}\n{\omega}^{-1}$.
    For the second term, involving $\Delta_n[\X((\l-1)/n)\Gamma_{n-\l+1}] \;
    \Delta_\l\Gamma_\l$, we need a more refined method then just partial summation.
    For $\n{1-e^{-i2\sigma}}\geq 1/2$ we use partial summation with respect to $\l$ as we
    did for the first term. We get the estimate $2^{j/2}
    \n{\sigma}^{-2}\n{\omega}^{-1}$. This is weaker than it should be, but since $\n{1-e^{-i2\sigma}}\geq
    1/2$ this is bounded by $2^{j/2} \n{\sigma}^{-3/2}\n{\omega}^{-1}$.
    For $\n{1-e^{-i2\sigma}}\leq 1/2$ we use Corollary \ref{poissonkorollar} and get the expected bound directly. In fact
    \begin{equation*}
        \big((1-e^{i\omega})(1-e^{-i2\sigma})\big)^{-1}
        \sum_{n\sim 2^j} \sum_{\l=1}^{n-1} \Delta_n[\X((\l-1)/n)\Gamma_{n-\l+1}] \Delta_\l\Gamma_\l
        e^{-i2\l\sigma} e^{i\omega n}\lesssim \frac{2^{j/2}}{\n{\sigma}^{3/2}\n{\omega}}.
    \end{equation*}
    For $(B)$ all the calculations are very similar. The first estimate we get is that
    $(B)$ is bounded by $c \; 2^j\n{\sigma}^{-1}\n{\omega}^{-1}$.
    We write
    \begin{equation*}
        \begin{split}
        \sum_{\l=1}^n &\Delta_\l[\X(\l/n) \Gamma_\l \Gamma_{n-\l}]\;e^{-i2\l\sigma}\\
        =&\sum_{\l=1}^n \big((\Delta\X(\l/n))\Gamma_\l
        \Gamma_{n-\l}+\X((\l-1)/n)(\Delta\Gamma_{n-\l})\Gamma_\l\big)e^{-i2\l\sigma}\\
        &+\sum_{\l=1}^n \X((\l-1)/n)\Gamma_{n-\l+1}(\Delta\Gamma_\l) \;e^{-i2\l\sigma}.
        \end{split}
    \end{equation*}
    For the first sum we use one more time partial summation which gives us an additional
    factor $(n \n{\sigma})^{-1}$.
    Now by interpolation we get
    \[\big((1-e^{i\omega})(1-e^{-i2\sigma}))\big)^{-1} \sum_{\l=1}^n \Delta_\l[\X(\l/n) \Gamma_{n-\l}] \; \Gamma_\l
        \; e^{-i2\l\sigma} e^{i\omega n}\vert_{n\sim 2^j}
        \lesssim \frac{2^{j/2}}{\n{\sigma}^{3/2} \n{\omega}}.\]
    As before we use Corollary \ref{poissonkorollar} for the second sum and obtain
    \[\big((1-e^{i\omega})(1-e^{-i2\sigma}))\big)^{-1} \sum_{\l=1}^n \X((\l-1)/n) \Gamma_{n-\l+1} \;
        (\Delta\Gamma_\l) \; e^{-i2\l\sigma} e^{i\omega n}\vert_{n\sim 2^j}
        \lesssim \frac{2^{j/2}}{\n{\sigma}^{3/2} \n{\omega}}.\]
    Of course, if we omit the partial summation in $n$ and do similar calculations, we also get an
    estimate of the form
    \begin{equation*}
        \Bignorm{\sum_{n\sim 2^j} \sum_{\l=0}^n \X(\l/n) \; \frac{\Gamma(\l+m+1)}{\Gamma(\l+1)}
        \frac{\Gamma(n-\l+m+1)}{\Gamma(n-\l+1)}\; e^{-i2\l\sigma} \;
        e^{i\omega n}}
        \leq C
        \frac{2^{j/2}}{\n{\sigma}^{3/2}} \; 2^j.
    \end{equation*}
    Observe that
    $\frac{2^{j/2}}{\n{\sigma}^{3/2} \n{\omega}}\leq \frac{2^{j/2}}{\n{\sigma}^{3/2}
    \n{\omega}^{1+\epsilon}}$
    and
    \begin{equation*}
        \begin{split}
        \min\{\frac{2^{j/2}}{\n{\sigma}^{3/2}
        \n{\omega}^{1+\epsilon}}&,\ \frac{2^{j/2}}{\n{\sigma}^{3/2}} \; 2^j\}
        \lesssim
        \frac{2^{j/2+j\epsilon}}{\n{1-e^{-i2\sigma}}^{3/2}} \;
        \frac{2^j}{(1+2^j\n{1-e^{i\omega}})^{1+\epsilon}}
        \end{split}
    \end{equation*}
    which completes the proof (a).\absatz

    (b) we get in a similar way.
    We write
    \begin{equation*}
        \begin{split}
        \sum_{\l=0}^n \X(\l/n) &\; \frac{\l \ \Gamma(\l+m+1)}{\Gamma(\l+1)}
        \frac{\Gamma(n-\l+m+1)}{\Gamma(n-\l+1)}\; e^{-i2\l\sigma}\\
        =&\sum_{\l=0}^2 \X(\l/n) \l \ \Gamma_\l \Gamma_{n-\l} \; e^{-i2\l\sigma}
        +2 \ \Gamma_2\Gamma_{n-2} \; e^{-i2\sigma} \; (1-e^{-i2\sigma})^{-1}\\
        &+\sum_{\l=2}^n \Delta[\X(\l/n) \l \ \Gamma_\l \Gamma_{n-\l}] \; e^{-i2\l\sigma} \;
        (1-e^{-i2\sigma})^{-1}=:\t{q}_n^1+\t{q}_n^2+\t{q}_n^3.
        \end{split}
    \end{equation*}
    For $\t{q}_n^1$ and $\t{q}_n^2$ we get with the same methods as in (a)
    \begin{equation*}
        \sum_{n\sim 2^j} \t{q}_n^1 \; e^{i\omega n},\ \sum_{n\sim 2^j} \t{q}_n^2 \; e^{i\omega
        n}\lesssim \frac{2^{j/2}}{\n{\sigma}\n{\omega}}.
    \end{equation*}
    For $\t{q}_n^3$ we have
    \begin{equation*}
        \begin{split}
        \sum_{\l=2}^n \Delta[\X(&\l/n) \l \ \Gamma_\l \Gamma_{n-\l}] \; e^{-i2\l\sigma} \;
        (1-e^{-i2\sigma})^{-1}\\
        =&\sum_{\l=2}^n \Delta[\l \ \Gamma_\l] \; \X(\l/n) \Gamma_{n-\l}] \; e^{-i2\l\sigma} \;
        (1-e^{-i2\sigma})^{-1}\\
        &+\sum_{\l=2}^n (\l-1) \ \Gamma_{\l-1} \Delta[\X(\l/n) \Gamma_{n-\l}] \; e^{-i2\l\sigma} \;
        (1-e^{-i2\sigma})^{-1}.
        \end{split}
    \end{equation*}
    Since $\Delta[ \l \ \Gamma_\l]$ behaves like $\Gamma_\l$ we get by the proof of (a)
    \[(1-e^{-i2\sigma})^{-1} \; \sum_{n \sim 2^j} \sum_{\l=2}^n \Delta[\l \ \Gamma_\l] \; \X(\l/n) \Gamma_{n-\l} \; e^{-i2\l\sigma} \;
        (1-e^{-i2\sigma})^{-1} \; e^{i\omega n}\lesssim
        \frac{2^{j/2}}{\n{\sigma}^{5/2}\n{\omega}}.\]
    For the second sum we use one time more partial summation in $\l$ and get two terms,
    that behave like
    \Aeqn
    \begin{equation}
        (1-e^{-i2\sigma})^{-2} \; \sum_{\l} \Gamma_\l \; \Delta[\X(\l/n) \Gamma_{n-\l}] \;
    e^{-i2\l\sigma}
    \end{equation}
    and
    \begin{equation}
        (1-e^{-i2\sigma})^{-2} \; \sum_{\l} \l \ \Gamma_\l \; \Delta^2[\X(\l/n) \Gamma_{n-\l}] \;
    e^{-i2\l\sigma}
    \end{equation}
    \reseteqn
    By partial summation with respect to $n$ we get
    \[\norm{\sum_{n\sim 2^j} \sum_{\l} \Gamma_\l \; \Delta[\X(\l/n) \Gamma_{n-\l}] \;
    e^{-i2\l\sigma} \; e^{i\omega n}}\lesssim 2^j \norm{\omega}^{-1}.\]
    Now, one more partial summation in $\l$ yields that this sum is bounded by
    $\norm{\sigma}^{-1} \norm{\omega}^{-1}$. Interpolating these two results
    gives us our expected result. For (B) we get the same results.
    We omit the details.
\end{proof}

For $q_{n,m-1}^+$ we get similar results.

\begin{satz}
    \label{propqn2}
    \begin{itemize}
    \item[(a)] For every $\epsilon$ there exists a constant $C_\epsilon$ such that
    \begin{equation}
    \begin{split}
        \Bignorm{\sum_{n\sim 2^j} &\sum_{\l=0}^n \X(\l/n) \; \frac{\Gamma(\l+m)}{\Gamma(\l)}
        \frac{\Gamma(n-\l+m)}{\Gamma(n-\l)}\; e^{-i2\l\sigma} \;
        e^{i\omega n}}\\
        &\leq C_\epsilon
        \frac{2^{-j/2+j\epsilon}}{\n{1-e^{-i2\sigma}}^{1/2}} \;
        \frac{2^j}{(1+2^j\n{1-e^{i\omega}})^{1+\epsilon}},\mbox{ for all } \sigma \mbox{ and } \omega.
    \end{split}
    \end{equation}
    \item[(b)] For every $\epsilon$ there exists a constant $C_\epsilon$ such that
    \begin{equation}
    \begin{split}
        \Bignorm{\sum_{n\sim 2^j} &\sum_{\l=0}^n \X(\l/n) \; \frac{\l \; \Gamma(\l+m)}{\Gamma(\l)}
        \frac{\Gamma(n-\l+m)}{\Gamma(n-\l)}\; e^{-i2\l\sigma} \;
        e^{i\omega n}}\\
        &\leq C_\epsilon
        \frac{2^{-j/2+j\epsilon}}{\n{1-e^{-i2\sigma}}^{3/2}} \;
        \frac{2^j}{(1+2^j\n{1-e^{i\omega}})^{1+\epsilon}},\mbox{ for all } \sigma \mbox{ and } \omega.
    \end{split}
    \end{equation}
    \end{itemize}
\end{satz}

\begin{proof}
    The proof is similar to the proof of the previous proposition. Hence we omit it.
\end{proof}

\newpage
\Section{Reductions}
%
%
%
%%%%%%%%%%%%%%%%%%%%%%%%%%%%%%%%%%%%%%%%%%%%%%%%%%%%%%%%%%%%%%%%%%%%%%%%%%%%%%%%%%%%%%%%%%%%%
%%%%%%%%%%%%%%%%%%%%%%%%%%%%%%%%    Interpolation    %%%%%%%%%%%%%%%%%%%%%%%%%%%%%%%%%%%%%%%%
%
%
%
\subsection{Interpolation, reduction to $p=1$}

Define the analytic family of operators $T_\alpha:=e^{i\sqrt{G}}/(1+G)^{\alpha/2}$. If we
assume now that the case $p=1$ is true, we have
\begin{equation*}
    \begin{split}
    \dnorm{T_\alpha f}_2&\leq C_\alpha \dnorm{f}_2,\quad \mbox{ if $\re \alpha=0$,}\\
    \dnorm{T_\alpha f}_1&\leq C_\alpha \dnorm{f}_1,\quad \mbox{ if $\re \alpha>1/2$.}
    \end{split}
\end{equation*}
Müller and Stein showed that the convolution kernel of the operator
$(1+L)^{-\epsilon+i\gamma}$ has bounded $L_1$-norm, for $\gamma \in \R$ and $\epsilon>0$.
The $L_1$-norms grow polynomially in $\gamma$. By transference, Corollary \ref{transfer},
the operators $(1+G)^{-\epsilon+i\gamma}$ have integral kernels with bounded Schur norms
and the Schur norms grow polynomially in $\gamma$. Hence we can use the analytic
interpolation theorem in \cite{steininterpolation} and a standard duality argument to
deduce the theorem, for arbitrary $1\leq p\leq \infty$.
%
%
%
%%%%%%%%%%%%%%%%%%%%%%%%%%%%%%%%%%%%%%%%%%%%%%%%%%%%%%%%%%%%%%%%%%%%%%%%%%%%%%%%%%%%%%%%%%%%%
%%%%%%%%%%%%%%%%%%%%%%   Reduktion auf den lokalen Anteil des Kerns  %%%%%%%%%%%%%%%%%%%%%%%%
%
%
%
\subsection{Reduction to an estimate for the local part of the kernel, part 1}
\label{reduclocalpartsection}

From now on, we take $\alpha$ to be always bigger than $1/2$, let
\[\alpha>\frac{1}{2}.\]\index{alpha}
%\fbox{Für $\alpha_0$ gilt: $\alpha_0>m=(d-1)/2=1/2$} \fbox{Für $M$ gilt: $M=2m=1$}
Let $\eta$ be an even $\Cnu(\R)$ function, such that $\eta(\xi)=1$ for small $\xi$, and
$\eta(\xi)=0$, if $\norm{\xi}\geq 1$. For some large constant $N>1$, to be chosen later,
put $\eta_N(\xi):=\eta(\xi/N)$. Define
\begin{equation}
    h^\alpha_t(\xi):=(1-\eta_N)(\xi) \; \xi^{-\alpha/2} \; e^{it\sqrt{\xi}}, \mbox{ for
    all } \xi \in \R^+.\index{halpha}
\end{equation}
and $h^\alpha:=h^{\alpha}_1$. Then $h_t^\alpha(\xi)=0$ for all $\xi< N$. Furthermore, we
define \[B_r(x',u'):=B(x',\const_0 r)\times B(u',\const_0 r(r+\n{x'})).\]

%%%%%%%%%%%%%%%%%%%%%%%%%%%%%%%%%%%      Reduktion auf lokalen Anteil     %%%%%%%%%%%%%%%%%%%%%%%%%%%%%%%%%%%%%%%%%%%
\begin{prop}
    \label{reduclocalpart}
    Put
    %\begin{equation*}
    %    \X_B(x',u',x,u):=
    %    \left\{ \begin{array}{c@{\;, \;}l}
    %    1 & \mbox{ if } \n{x'-x}\leq C \mbox{ and } \n{u'-u}\leq C(1+\n{x'})\\
    %    0 & \mbox{ otherwise }.\end{array} \right.
    %\end{equation*}
    \begin{equation*}
        \X_B(x',u',x,u):=
        \left\{ \begin{array}{c@{\;, \;}l}
        1 & \mbox{ if } (x,u) \in B_2(x',u')\\
        0 & \mbox{ otherwise }.\end{array} \right.
    \end{equation*}
    Then the integral kernel $(1-\X_B) K_{h^\alpha(G)}$ has bounded Schur norm.
    Furthermore, to prove the theorem, it suffices to show that
    $\X_{B} K_{h^\alpha(G)}$
    has bounded Schur norm.
\end{prop}

%\ffbox{Observe now that $\norm{x-\t{x}_1}\leq C_0$ implies
%\begin{equation*}
%  \begin{split}
%  \norm{\t{x}^2-\t{x}_1^2}&=\norm{(\t{x}-\t{x}_1)(\t{x}+\t{x}_1)}
%  \leq C_0(C_0+2\norm{\t{x}_1})\\
%  \norm{\t{x}^2+\t{x}_1^2}&\leq (C_0+2\norm{\t{x}_1})^2
%  \end{split}
%\end{equation*}
%Hence for $t_1=0$
%\begin{equation}
%  \label{reduction0}
% \norm{\t{x}^2-\t{x}_1^2}=\leq C_0(C_0+2\norm{\t{x}_1}),
%  \quad \norm{\t{t}}\leq C_0(C_0+2\norm{\t{x}_1})
%\end{equation}
%Hence for $w:=((x²-\t{x}_1²)²+(t-s)²)^{1/2}$
%\begin{equation}
%  \label{reduction1}
%  w\leq 2C_0(C_0+2\norm{\t{x}_1})\lesssim 1+\norm{\t{x}_1}=1+\norm{\t{x}_1}
%\end{equation}
%holds.}

%\noindent Define for $r>0$ \[B_r:=B(x',\sqrt{Cr})\times
%B(u',\sqrt{Cr}(2\sqrt{Cr}+\n{x'}))\] \noindent Now we turn to the proof of Proposition
%\ref{reduclocalpart}.
\begin{proof}
    Define
    \begin{equation*}
        \begin{split}
        f_{\alpha,t}(\xi)&:=(1-\eta_N)(\xi) \; \n{\xi}^{-\alpha/2} \; \cos(t\sqrt{\xi}),\\
        g_{\alpha,t}(\xi)&:=(1-\eta_N)(\xi) \; \n{\xi}^{-\alpha/2} \; \sin(t\sqrt{\xi})
        \end{split}
    \end{equation*}
    Then $h^\alpha(\xi)=f_{\alpha,1}(\xi)+i g_{\alpha,1}(\xi).$
    Let
    \[\t{g}_{\alpha,t}(\xi):=(1-\eta_N)(\xi) \; \n{\xi}^{-\alpha/2} \; \sin(t\sqrt{\n{\xi}})\; \sgn(\xi).\]
    Then $\t{g}_{\alpha,t}$ is an odd function and
    \begin{equation*}
        \begin{split}
        h^\alpha(\xi)&=f_{\alpha,1}(\xi)+i \t{g}_{\alpha,1}(\xi),\quad \mbox{ for all }\xi\geq 0,\\
        h^\alpha(G)&=f_{\alpha,1}(G)+i \t{g}_{\alpha,1}(G).
        \end{split}
    \end{equation*}
    It is easily seen that for $\xi>0$
    \[(1-\eta_N)(\xi)\n{\xi}^{-\alpha/2}=\int \kphi(\tau)\; \cos(\tau\sqrt{\xi}) d\tau,\]
    where $\kphi$ is such that $\eta \kphi \in L_1$ and $(1-\eta)\kphi \in \S$.
    Hence
    \[f_{\alpha,t}(\xi)=\int (\eta \kphi)(\tau) \; \cos(\tau\sqrt{\xi}) \; \cos(t
    \sqrt{\xi})\;
    d\tau+\Phi(\sqrt{\xi}) \; \cos(t\sqrt{\xi}),\]
    with $\Phi\in \S$.
    By Proposition \ref{finitewavespeed} the support of the distribution
    \begin{equation*}\label{}
        \int (\eta\kphi)(\tau) \; \cos(\tau\sqrt{G}) \; \cos(t \sqrt{G}) \;
        d\tau
    \end{equation*}
    lies in the support of $\X_B$ for all $t\leq 1$.
    Define \[\psi(\xi):=\Phi(\sqrt{\xi})\cos(t\sqrt{\xi})-\Phi(0)e^{-\xi}.\]
    Then
    \begin{equation*}
        \begin{split}
        \n{\xi^{\l} \; \d{\xi}^{\l} \psi(\xi)} \lesssim \xi^{1/2} &\mbox{ for all } 0<\xi \leq 1,\\
        \n{\xi^{\l} \; \d{\xi}^{\l} \psi(\xi)} \lesssim \xi^{-1/2} &\mbox{ for all } 1\leq \xi < \infty.
        \end{split}
    \end{equation*}
    And thus the integral kernel of $\Phi(\sqrt{G})\cos(t\sqrt{G})$ has bounded Schur norm by
    Proposition \ref{reducmultipliertheorem} and
    the fact that $e^{-G}$ is the heat-kernel, which has bounded Schur norm.
    Thus, if we denote the integral kernel of $f_{\alpha,t}(G)$ by $K_f$, then $(1-\X_B) \ K_f$
    has bounded Schur norm.

    Since $\t{g}_{\alpha,t}$ is an odd function we get by the Fourier transform and the summation
    formula for the $\cos$ function
    %\fbox{$\cos(a\pm b)=\cos(a)\cos(b)\mp\sin(a)\sin(b)$}
    \begin{equation*}
        \begin{split}
        \t{g}_{\alpha,t}&=\int \kphi(\tau)\; \sin(\tau\sqrt{\xi}) \; \sin(t \sqrt{\xi}) \; d\tau\\
        &=\int \kphi(\tau)\; \cos(\tau\sqrt{\xi}) \; \cos(t \sqrt{\xi}) \; d\tau
        -\int \kphi(\tau)\; \cos((\tau+t)\sqrt{\xi}) \; d\tau\\
        \end{split}
    \end{equation*}
    with some appropriate function $\kphi$. This can be written as
    \begin{equation*}
        \begin{split}
        &\int (\eta\kphi)(\tau) \; \cos(\tau\sqrt{\xi}) \; \cos(t \sqrt{\xi}) \; d\tau +
        \Phi(\sqrt{\xi})\;\cos(t\sqrt{\xi})\\
        &-\int (\eta\kphi)(\tau)\; \cos((\tau+t)\sqrt{\xi}) \; d\tau-
        \int ((1-\eta)\kphi)(\tau)\; \cos((\tau+t)\sqrt{\xi}) \; d\tau,\\
        \end{split}
    \end{equation*}
    with a smooth function $\Phi \in \S$.
    As for $f_{\alpha,t}$, we know that for all $\tau \leq 1$ the supports of the distributions
    \begin{equation*}\label{}
        \begin{split}
        &\int (\eta\kphi)(\tau) \; \cos(\tau\sqrt{G}) \; \cos(t \sqrt{G}) \; d\tau  \mbox{ and }\\
        &\int (\eta\kphi)(\tau)\; \cos((\tau+t)\sqrt{G}) \;
        d\tau
        \end{split}
    \end{equation*}
    lie in the support of $\X_B$.
    Furthermore, the kernel of $\Phi(\sqrt{G})\cos(t\sqrt{G})$ has bounded Schur norm for all
    $\n{t}\leq 1$, as we have seen before.

    We are left with the operator $\int ((1-\eta)\kphi)(\tau)\; \cos((\tau+t)\sqrt{G}) \; d\tau$
    which can be written as
    \begin{equation*}\label{}
        \begin{split}
        \int ((1-\eta)\kphi)(\tau)\; &[\cos(\tau\sqrt{G})\cos(t\sqrt{G})-\sin(\tau\sqrt{G})\sin(t\sqrt{G})] \;
        d\tau=\\
        &\Phi_1(\sqrt{G})\cos(t\sqrt{G})+\Phi_2(\sqrt{G})\sin(t\sqrt{G})
        \end{split}
    \end{equation*}
    with appropriate functions $\Phi_1$ and $\Phi_2$ supported away from the origin.
    Once more we can show that the functions
    \[\Phi_1(\sqrt{\xi})\cos(t\sqrt{\xi})-\Phi_1(0)e^{-\xi},\;
    \Phi_2(\sqrt{\xi})\sin(t\sqrt{\xi})-\Phi_2(0)e^{-\xi}\]
    fulfill the conditions of Proposition
    \ref{reducmultipliertheorem} and thus the integral kernels of the operators
    $\Phi_1(\sqrt{G})\cos(t\sqrt{G})$ and $\Phi_2(\sqrt{G})\sin(t\sqrt{G})$
    have bounded Schur norms, since the heat-kernel has bounded Schur norm.
    Thus, if we denote the integral kernel of $g_{\alpha,t}(G)$ by $K_g$, then $(1-\X_B) \ K_g$
    has bounded Schur norm.
    Hence, if we know that $\X_B K_{h^\alpha(G)}$ has bounded Schur norm, we could
    conclude that $K_{h^\alpha(G)}$ has bounded Schur norm.\absatz

    The last thing we have to prove is that if the kernel of $h^\alpha(G)$ has bounded Schur norm,
    this is also true for the kernel of $(1+G)^{-\alpha/2}e^{i\sqrt{G}}$.
    We write
    \[(1+\xi)^{-\alpha/2}\;
    e^{i\sqrt{\xi}}=\eta_N(\xi)(1+\xi)^{-\alpha/2}e^{i\sqrt{\xi}}+\xi^{\alpha/2}(1+\xi)^{-\alpha/2}
    (1-\eta_N)(\xi)\xi^{-\alpha/2}e^{i\sqrt{\xi}}.\]
    The functions $\eta_N(\xi)(1+\xi)^{-\alpha/2}e^{i\sqrt{\xi}}-e^{i\xi}$,
    $\xi^{\alpha/2}(1+\xi)^{-\alpha/2}-1+e^{-\xi}$
    satisfy the
    hypothesis of Proposition \ref{reducmultipliertheorem} and thus the operators
    $\eta_N(G)(1+G)^{-\alpha/2}e^{i\sqrt{G}}$ and $G^{\alpha/2}(1+G)^{\alpha/2}$ have
    bounded Schur norms. Since $(1-\eta_N)(G)G^{-\alpha/2}e^{i\sqrt{G}}=h^\alpha(G)$ the proposition
    has been proven.
\end{proof}

A further reduction allows us to exchange the indicator function $\X_{B}$ by a smooth
variant, this will be shown in part 2 of this section.

\subsection{The case of large $x'$}\label{reducx1chapter}

Before we go on with the proof of the theorem, we want to mention how one can get uniform
estimates for $\dnorm{\X_{B} K_{h^\alpha(G)}(x',0,\fa,\fa)}_{L_1}$ for $x'>>1$. Since we
have not done all calculations in this section in detail, we formulate the key estimate
as a conjecture and show how this conjecture implies the result for $x'>>1$.

Let us define operators $G_{\epsilon}$ for $\epsilon \leq 1/4$ in the following way
\[G_\epsilon:=-\d{x}^2-(1+2\epsilon x+\epsilon^2 x^2)\d{u}^2.\]
These operators are uniformly elliptic operators on the set
\mbox{$\menge{(x,u)}{\norm{(x,u)}\leq c}$}, for every $c>0$. By Fourier integral methods,
one can show that solutions to corresponding wave equations fulfill estimates in the
sense of (\ref{rm1}) and (\ref{rm2}), uniformly in $\epsilon \leq 1/4$.
\begin{satz}
    \label{FIO}
    Let $\alpha\geq 1/2$, $1<p<\infty$ and $c>0$.
    There exists a constant $C_{\alpha,p,c}$ such that for all $\epsilon\leq 1/4$, $t$ sufficiently small and $f$
    supported in \mbox{$\menge{(x,u)}{\norm{(x,u)}\leq c}$}
    \begin{equation}
        \label{FIO00}
    \Bigdnorm{\exp(it\sqrt{G_\epsilon})f}_{L_p} \leq C_{\alpha,p,c} \dnorm{f}_{W_p^{\alpha}}
    \end{equation}
    holds.
\end{satz}
\noindent Such estimates are well known. At the end of the chapter we will give a sketch
of the proof. For details see \cite{seeger} and \cite{sogge}. In fact, Seeger, Sogge and
Stein showed that Proposition \ref{FIO} is true for an elliptic operator defined on a
compact manifold instead of $G_\epsilon$. The compactness is here not necessary, since we
have finite speed of wave propagation and since the functions $f$ are supported in a
compact set.

Since the operator $G_\epsilon$ is elliptic on the support of $f$ it should be possible
to exchange the usual Sobolev norm $\dnorm{\fa}_{W_p^{\alpha}}$ by norms
$\dnorm{\fa}_{L^{\alpha}_p}:=\dnorm{(1+G_\epsilon)^{\alpha/2}\fa}_{L_p}$. In addition,
the assertion (\ref{FIO00}) should also hold true for $p=1$ and with a slightly worse
exponent $\alpha>1/2$. We conjecture the following.
\begin{vermutung2}
    \label{FIOvermutung}
    Let $\alpha> 1/2$ and $c>0$.
    There exists a constant $C_{\alpha,c}$ such that for all $\epsilon\leq 1/4$, $t$ sufficiently small and $f$
    supported in \mbox{$\menge{(x,u)}{\norm{(x,u)}\leq c}$}
    \begin{equation}
        \label{FIO00}
    \Bigdnorm{\frac{\exp(it\sqrt{G_\epsilon})}{(1+G_\epsilon)^{\alpha/2}}f}_{L_1} \leq C_{\alpha,c}
    \dnorm{f}_{L_1}
    \end{equation}
    holds.
\end{vermutung2}
This conjecture implies now the following proposition.
\begin{satz}
    \label{reduktionx1}
    Conjecture \ref{FIOvermutung} implies that for
    every $\alpha > 1/2$ there exist constants $C,C_\alpha$, such that
    \index{const1}
    \begin{equation}
        \sup\limits_{\n{x'}\geq C} \dnorm{K_{h^\alpha(G)}(x',0,x,u)}_{L_1(x,u)}\leq C_{\alpha}
    \end{equation}
    holds.
\end{satz}
Hence, if the conjecture was true, this proposition together with our theorem would imply
that the operator $\exp(i\sqrt{G})(1+G)^{-\alpha/2}$ extends to a bounded operator on
$L_p(\R^2)$ for $\alpha>\n{1/p-1/2}$ and $1\leq p \leq \infty$.

%{\footnotesize We choose a small $\delta>0$ and $f\in \S$, supported in $\Omega$. Since
%$G_\epsilon$ is uniformly elliptic on $\Omega_1$, the operator
%\[(1+G_\epsilon)^{-1/4-\delta}(1-\Delta)^{1/4}\] is bounded on $L_p(\Omega_1)$, for $p\geq 1$.
%Hence, we get that $\dnorm{\exp(it\sqrt{G_\epsilon})f}_{L^{-1/4-\delta}_p} $ is bounded
%by $C \; \dnorm{f}_{L_p}$, with a constant $C$. Since we have finite speed of propagation
%we can deduce that
%$\exp(it\sqrt{G_\epsilon})(1+G_\epsilon)^{-1/4-\delta}\delta_{(x',u')}$ lies in $L_1$
%outside a compact set $K(x',u')$ and the $L_p$-norms are uniformly bounded. The
%distribution $(1+G_\epsilon)^{-\delta}\delta_0$ lies in $L_p$ for some $p>1$. Hence the
%$L_p^{-1/4-\delta}$-norm of
%$\exp(it\sqrt{G_\epsilon})(1+G_\epsilon)^{-\delta}\delta_{(x',u')}$ is bounded with
%$L_p^{-1/4-\delta}$-norm independent of $(x',u')$. Therefore, $\X_{K(x',u')} \
%\exp(it\sqrt{G_\epsilon})(1+G_\epsilon)^{-1/2-2\delta}\delta_{(x',u')}$ has bounded
%$L_1$-norm, uniformly in $(x',u')$. Since we can choose $\delta$ arbitrarily small, this
%shows (\ref{FIO0}).}\absatz
\begin{proof}
    Since
    \[\Big(-\d{x}^2-\Big(1+\frac{x^2}{x_1^2}+\frac{2x}{x_1}\Big)\d{u}^2\Big)f(\fa+x_1,\fa x_1)\vert_{(x-x_1,u/x_1)}
    =(-\d{x}^2-x^2\d{u}^2)f\vert_{(x,u)},\]
    we can express the kernel of $m(G)$ by the kernel of $m(G_{1/x_1})$ for every bounded function $m$. We have
    \begin{equation*}
        \begin{split}
        \int &K_{m(G_{1/x_1})}(x',u',x-x_1,u/x_1) \; f(x'+x_1,u'x_1) \; dx' \; du'\\
        &=\int K_{m(G)}(x',u',x,u) \; f(x',u') \; dx' \; du'
        \end{split}
    \end{equation*}
and thus
\begin{equation*}
    x_1^{-1} \; K_{m(G_{1/x_1})}(x'-x_1,u'/x_1,x-x_1,u/x_1)=K_{m(G)}(x',u',x,u).
\end{equation*}
By Proposition \ref{FIO} the operator
$\exp(t\sqrt{G_\epsilon})(1+G_\epsilon)^{-\alpha/2}$, with kernel
$M_{\epsilon,t}^\alpha$, is bounded from $L_1(\Omega_1)$ to $L_1$ for small $t<t_0$, with
$t_0$ a sufficiently small constant. Hence
\begin{equation*}
    \dnorm{M_{\epsilon,t}^\alpha(x',0,x,u}_{L^1(x,u)}\leq C_\alpha
\end{equation*}
for all $\n{x'}\leq 2$ and $C_\alpha$ independent of $\epsilon$. Since the operator
\[(1-\eta_N)(G_\epsilon)(1+G_\epsilon)^{\alpha/2}G_\epsilon^{-\alpha/2-\delta},\] for small $\delta$, is
bounded on $L_1$ we get that the kernel $\t{M}_{\epsilon,t}^{\alpha-\delta}$ of the
operator $h_t^{\alpha-\delta}(G_\epsilon)$ has bounded Schur norm. Hence
\begin{equation*}
    x_1^{-1} \; \int \n{\t{M}_{\epsilon,t}^{\alpha-\delta}(x'-x_1,0,x-x_1,u/x_1)} \; dx \; du \leq C_\alpha,
\end{equation*}
for all $\n{x'-x_1}\leq 1$. With $\epsilon=1/x_1$ this leads to
\begin{equation*}
    \int K_{h_t^{\alpha-\delta}(G)}(x',0,x,u) \; dx \; du \leq C_\alpha,
\end{equation*}
for all $\n{x'-x_1}\leq 1$. With $x'=x_1$ we end up with
\begin{equation*}
    \int \n{K_{h_t^{\alpha-\delta}(G)}(x',0,x,u)} \; dx \; du \leq C_\alpha,
\end{equation*}
uniformly in $x'$, provided $x'\geq 4$. By the homogeneity of $G$ with respect to the
dilation $(x,u)\mapsto (rx,r^2u)$, it follows that
\begin{equation*}
    \sup\limits_{\n{x'}\geq 4/t_0} \int K_{h^{\alpha-\delta}(G)}(x',0,x,u) \; dx \; du \leq
    C_\alpha.
\end{equation*}
By setting $C:=4/t_0$ the proposition has been proven.
\end{proof}

\subsubsection*{Fourier integral operators}

%\ffbox{If $A \in OPS^m$ then $A^{1/m}$ is in $OPS^1$. \cite{seeley}}

In this section we deal with operators of the form
\begin{equation}
    \label{FIOdef}
    Af(x)=\int a(x,\xi) \; e^{i\kphi(x,\xi)} \; \h{f}(\xi) \; d\xi,
\end{equation}
where the amplitude $a$ is a real valued function in a symbol class $S^m_{\rho,\delta}$,
$\rho>0$, $\delta<1$, and the phase function $\phi$ fulfills
\begin{list}{}
    \item{(a) $\kphi$ is smooth, real valued and homogeneous of degree $1$ in $\xi$.}
    \item{(b) the gradient $\grad_x\kphi$ is nowhere vanishing on the support of a, for
    all $\xi\not=0$.}
\end{list}
\absatz \noindent These operators are a special cases of \emph{Fourier integral
operators}. We refer to Sogge \cite{sogge} and Duistermaat \cite{duistermaat} for a
general definition.

The aspect of interest for us is that the operator defined in \eqref{FIOdef} is
essentially the solution operator to a strictly hyperbolic differential equation.

\subsubsection*{Sketch of the proof of Proposition \ref{FIO}}

%By using duality arguments it suffices to proof assertion \ref{FIO0} for $2\leq p \leq
%\infty$. \ffbox{$a \in S^m$ defined as
%\begin{equation}
%    \n{\d{\zeta}^{\alpha}\dt^{\beta}\d{z}^{\gamma} a}\leq C_{\alpha,\beta,\gamma} \;
%(1+\n{\zeta})^{m-\n{\alpha}}
%\end{equation}
%\begin{equation}
%    \n{\dt^{\alpha_1} \dx^{\alpha_2} \du^{\alpha_3} \d{\xi}^{\beta_1} \d{\eta}^{\beta_2}
%    a}\leq C_{\alpha,\beta} \; (1+\n{\zeta})^{m-\n{\beta}}
%\end{equation}}

We use the abbreviations $z:=(x,u)$, $\zeta:=(\xi,\eta)$. Let
$P_\epsilon:=\d{t}^2+G_\epsilon$. Since $P_{\epsilon}$ is strictly hyperbolic we can
factor its principal symbol,
\begin{equation*}
    p(x,u,\xi,\eta,\tau)=(\tau-\lambda_+(x,u,\xi,\eta))(\tau-\lambda_-(x,u,\xi,\eta))
\end{equation*}
with $\lambda_\pm=\pm (\xi^2 +(1+x\epsilon)^2\eta^2)^{1/2}$. The \emph{eikonal equation}
\begin{equation*}
    \label{eikonalgleichung}
    \left\{ \begin{array}{r@{\;=\;}l}
        \d{t}\kphi^\pm &\lambda_\iota(x,\grad_x\kphi^\pm)=\pm((\d{x}\kphi^\pm)^2+
        (1+x\epsilon)^2(\d{u}\kphi^\pm)^2)^{1/2},\\
        \kphi^\pm\vert_{t=0}& \skalar{z}{\zeta}\end{array} \right. ,
\end{equation*}
is a system of two first order nonlinear differential equation for $\kphi^+$ and
$\kphi^-$, and it can be solved at least for small $t$. By the initial condition
$\kphi^\pm(0,z,\zeta)=\skalar{z}{\zeta}$ the solutions $\kphi^+,\ \kphi^-$ will satisfy
the requirements (a) and (b) for small $t$, $\n{t}\leq \delta$, $\delta$ sufficiently
small and independent of $\epsilon$.

We suppose that an approximate solution $\t{v}$ to the Cauchy problem
\begin{equation}
    \label{cpreduction}
    (\d{t}^2+G_\epsilon)v=P_\epsilon v=0, \quad v\vert_{t=0}=f_0,\quad \d{t} v\vert_{t=0}=f_1,
\end{equation}
for $f_0,\ f_1 \in \S$, can be written as
\begin{equation*}
    \t{v}(t,x,u)=\sum_{\iota \in \{+,-\},\ k\in \{0,1\}}
    \int a^{\iota,k}(t,z,\zeta) \; e^{i\kphi^\iota(t,z,\zeta)} \; \h{f_k}(\zeta) \; d\zeta,
\end{equation*}
where $a^{\iota,k}$ is a symbol of order $-k$. For simplicity we only consider the case
$f_0 \in \S$ and $f_1=0$. The case $f_0=0$ and $f_1 \in \S$ can be obtained similarly.
Then, the general case follows since our wave equation is linear.

Put $f_1=0$. We suppose that $\t{v}$ is given as $\t{v}=\t{v}^+ +\t{v}^-$, with
\begin{equation}
    \label{udefinition}
    \t{v}^\iota(t,x,u)=
    \int a^{\iota}(t,z,\zeta) \; e^{i\kphi^\iota(t,z,\zeta)} \; \h{f_0}(\zeta) \; d\zeta,
\end{equation}
for $\iota \in \{+,-\}$ and $a^\iota$ is a symbol of order $0$ and given as a sum over
symbols $a_k^\iota$ in $S^{-k}$
\begin{equation}
    \label{summea}
    a^\iota(t,z,\zeta)=\sum\limits_{k\leq 0} a_k^\iota(t,z,\zeta).
\end{equation}
Now we compute the symbols $a_k^\iota$. Since $\kphi^\pm \vert_{t=0}=\skalar{z}{\zeta}$
and $u$ should be a solution with $u\vert_{t=0}=f_0$ and $\dt u\vert_{t=0}=0$ we have the
following equations for $a^+$ and $a^-$.
\begin{equation*}
    \begin{split}
        (a^+ + a^-)\vt&=1\\
        [i\dt\kphi^+a^++i\dt\kphi^-a^-+\dt(a^++a^-)]\vert_{t=0}&=0
    \end{split}
\end{equation*}
Since $\kphi^\pm$ fulfills the eikonal equation \eqref{eikonalgleichung} the second
equation can be written as
\begin{equation}
    \label{sigmaa+a-}
    [i\sigma(a^+ - a^-)+\dt(a^+ + a^-)]\vert_{t=0}=0,
\end{equation}
with $\sigma:=(\xi^2+(1+x\epsilon)^2\eta^2)^{1/2}$. Since $\sigma$ is of order $1$ and
$\dt(a^+ + a^-)$ is a symbol of order $0$, the highest order term in $a$ has to fulfill
\[\sigma(a_0^+ - a_0^-)\vert_{t=0}=0.\]
Put
\begin{equation}
    d_k:=a_k^+ +a_k^-,\quad b_k:=a_k^+-a_k^-
\end{equation}
In order to fulfill \eqref{sigmaa+a-} and with respect to the symbol orders of $\sigma,\
b_k$ and $d_k$ we choose $b_k$ and $d_k$ so that
\begin{equation}
    \label{dkbk}
    i\sigma b_k\vert_{t=0}=-(\dt d_k)\vert_{t=0}.
\end{equation}
Furthermore, to approximate a solution of the wave equation, i.e. to ensure that
$P_\epsilon \t{v}$ is small, $a_k^+$ and $a_k^-$ have to fulfill a transport equation. We
get these equation in the following way. By Applying our operator $P_\epsilon$ to
$\t{v}^\iota$, for $\iota \in \{+,-\}$, we get
\begin{equation*}
    P_\epsilon \t{v}^\iota=\int c^\iota(t,z,\zeta) \; e^{i\kphi^\iota(t,z,\zeta)}\; \h{f_0}(\zeta) d\zeta
\end{equation*}
with
\begin{equation}
    \label{FIO2}
    \begin{split}
    c^\iota(t,z,&\zeta)=e^{-i\kphi^\iota} \; P_\epsilon(e^{i\kphi^\iota} a^\iota)\\
    =&i\big(2(\d{t}\kphi^\iota)(\dt a^\iota)-2(\dx \kphi^\iota)(\dx a^\iota)-2(1+2\epsilon x+\epsilon^2
    x^2)(\du \kphi^\iota)(\du a^\iota)\big)\\
    &+i(P_\epsilon\kphi^\iota) a^\iota+(P_\epsilon a^\iota).
    \end{split}
\end{equation}
Because $\t{v}$ should be an approximate solution we have to show that $c^\pm$ has order
$-N$, with $N$ sufficiently large. Thus we set $c^\pm$ equal to zero and solve for the
symbols $a_k^\pm$. The leading term is
\[i\big(2(\d{t}\kphi^\iota)(\dt a_0^\iota)-2(\dx \kphi)(\dx a_0^\iota)-2(1+2\epsilon x+\epsilon^2
    x^2)(\du \kphi^\iota)(\du a_0^\iota)\big)\\
    +i(P_\epsilon\kphi^\iota) a_0^\iota\]
which is of order $1$. Define
\[V^\iota:=(\dt \kphi^\iota) \dt-(\dx \kphi^\iota) \dx-(1+2\epsilon x+\epsilon^2 x^2)(\du \kphi^\iota)\du.\]
We study now the \emph{transport equation}
\begin{equation*}
        (V^\iota a_0^\iota)(t,z,\zeta) + (P_\epsilon\kphi)a_0^\iota(t,z,\zeta) =0,
\end{equation*}
for $\iota \in \{+,-\}$ and with initial conditions
\begin{equation}
    (a_0^+-a_0^-)\vt=b_0\vert_{t=0}=0,\quad (a_0^+ + a_0^-)\vt=d_0\vert_{t=0}=1.
\end{equation}
Since $V^\iota$ is a real vector field, we can solve these equations, on the same
$t$-interval and get solutions $a_0^\iota$ in the symbol class $S^0$. Furthermore, we
have $a_0^+-a_0^-=0$ and $a_0^++a_0^-=1$. By Rewriting (\ref{FIO2}) and setting equal to
zero we get
\begin{equation*}
    \begin{split}
        V^\iota\Big(\sum_{k\leq -1} a_k^\iota\Big)+(P_\epsilon \kphi^\iota)\Big(\sum_{k\leq -1}
        a_k^\iota\Big)-i\Big(P_\epsilon\sum_{k \leq 0}a_k^\iota\Big)=0,
    \end{split}
\end{equation*}
for $\iota \in \{+,-\}$. This gives us new transport equations for $a_{-1}^+$ and
$a_{-1}^-$.
\begin{equation*}
        (Va_{-1}^\iota)(t,z,\zeta) + (P_\epsilon\kphi^\iota)a_{-1}^\iota(t,z,\zeta)-i(P_\epsilon a_0^\iota)=0.
\end{equation*}
Since we have already calculated $d_0$, we get for these transport equations with initial
conditions, chosen with respect to \eqref{dkbk},
\begin{equation*}
    i\sigma(a_{-1}^+-a_{-1}^-)\vt=i\sigma b_{-1}\vert_{t=0}=-\dt d_0\vt,\quad (a_{-1}^+ + a_{-1}^-)\vt=d_{-1}\vert_{t=0}=0
\end{equation*}
unique solutions $a_{-1}^+$ and $a_{-1}^-$.
Iteratively, we can solve the transport
equations
\begin{equation*}
        (V^\iota a_{k}^\iota)(t,z,\zeta) + (P_\epsilon\kphi^\iota)a_{k}^\iota(t,z,\zeta)-i(P_\epsilon a_{k+1}^\iota)=0.
\end{equation*}
with initial conditions
\begin{equation*}
    i\sigma(a_{k}^+ - a_{k}^-)\vt=i\sigma b_{k}\vert_{t=0}
    =-\dt d_{k+1}\vert_{t=0},\quad (a_{k}^+ + a_{k}^-)\vt=d_{k}\vert_{t=0}=0
\end{equation*}
for all $k\leq -1$ and $\iota \in \{+,-\}$.

We choose now a sufficiently large $N \in \N$. Put $\t{a}^\iota:=\sum_{-N \leq k\leq 0}
a_k^\iota$, for \mbox{$\iota \in \{+,-\}$}, and
\[\t{v}(t,z):=\sum_{\iota \in \{+,-\}} \int \t{a}^\iota(t,z,\zeta) \; e^{i\kphi^\iota(t,z,\zeta)} \;
\h{f_0}(\zeta) \; d\zeta.\] Then
% The $a_k^\iota$ can be "summed" by standard
%technics to a symbol $a^\iota\in S^0$ so that $a\sim \sum_{k\leq 0} a_k^\iota$. Then for
%$u:=u^+ + u^-$, with $u^+$ and $u^-$ defined as in \eqref{udefinition} we have
\begin{equation*}
    \begin{split}
    (P_\epsilon \t{v})(t,z)&=F(t,z),\\
    \t{v}\vert_{t=0}&=f_0,
    \end{split}
\end{equation*}
where \[F(t,z):=\sum_{\iota \in \{+,-\}} \int
    \t{c}^\iota(t,z,\zeta) \; e^{i\kphi^\iota(t,z,\zeta)} \; \h{f_0}(\zeta) \; d\zeta\] and
$\t{c}^\iota$ is a symbol of order $-N$, for $\iota \in \{+,-\}$. Furthermore, $\dt
\t{v}\vert_{t=0}$ is given by
\begin{equation*}
    \begin{split}
    \dt \t{v}(0,z)&=\sum_{k=-N}^0 \int [i\sigma b_k+(\dt d_k)](0,z,\zeta)
    \;  e^{i\skalar{z}{\zeta}} \; \h{f_0}(\zeta) \; d\zeta\\
    &=\int (\dt d_{-N})(0,z,\zeta)
    \;  e^{i\skalar{z}{\zeta}} \; \h{f_0}(\zeta) \; d\zeta.\\
    \end{split}
\end{equation*}
Put $g(z):=\int (\dt d_{-N})(0,z,\zeta) \;  e^{i\skalar{z}{\zeta}} \; \h{f_0}(\zeta) \;
d\zeta$.

Now, if a exact solution $v$ for the Cauchy problem \eqref{cpreduction} is given, the
function $w:=\t{v}-v$ fulfills the inhomogeneous wave equation
\begin{equation*}
    \left\{ \begin{array}{c@{\;=\;}l}
        (\dt^2+G_\epsilon)w & F, \\
        w\vt & 0,\\
        \dt w\vt & g.\end{array} \right.
\end{equation*}
Since our operator $G_\epsilon$ is just a smooth perturbation of the Laplacian, we have
finite speed of wave propagation. Since $f$ is compactly supported we can find an open
set $\Omega_2$ with $\ab{\Omega_2}$ compact and independent of $\epsilon$ such that
$\supp v \teil \Omega_2$.

%By general theory (see \cite{evans}), in fact by using the energy inequality, we obtain
%the estimate
%\begin{equation}
%    \dnorm{v(t,\fa)}_{W^{s}(\Omega)}\leq C_{s,t} \; \int_0^t \dnorm{F(\tau,\fa)}_{W^s(\Omega)} \; d\tau
%\end{equation}
%for $s \in \Z$ and $C_{s,t}$ independent of $\epsilon$. By using the symbol properties of
%$c(t,z,\zeta)$ we get
%\begin{equation*}
%    \dnorm{v(t,\fa)}_{W_s(\Omega)}\leq C_{s,t} \dnorm{f_j}_{L_\infty}.
%\end{equation*}
%By the Sobolev embedding theorem the $L^\infty$-norm of $v(t,\fa)$ on $\Omega$ is bounded
%by $\dnorm{v(t,\fa)}_{W_s(\Omega)}$ for $s>1/2$.

By general theory (see \cite{evans}), in fact by using the energy inequality, we obtain
the estimate
\begin{equation}
    \dnorm{w(t,\fa)}_{L_2(\Omega_2)}\leq C_{t}\Big(\dnorm{g}_{L_2(\Omega_2)} +\int_0^t \dnorm{F(\tau,\fa)}_{L_2(\Omega_2)} \;
    d\tau\Big),
\end{equation}
with $C_{t}$ independent of $\epsilon$. Since $\t{c}^\pm$ and $\dt d_{-N}$ are both
symbols of order $-N$, we get
\begin{equation*}
    \dnorm{w(t,\fa)}_{L_2(\Omega_2)}\leq C_{t} \dnorm{f_0}_{L_1}.
\end{equation*}
By the Cauchy-Schwarz inequality, the $L_1$-norm of $w(t,\fa)$ on $\Omega_2$ is bounded
by the $L_2$-norm of $w(t,\fa)$ on $\Omega_2$.

We are left with the estimation of Fourier integral operators of the form
\begin{equation*}
    f \mapsto Af(t,x,u)=
    \int a(t,z,\zeta) \; e^{i\kphi(t,z,\zeta)} \; \h{f}(\zeta) \; d\zeta,
\end{equation*}
where $a$ is a symbol of order $0$ and $\phi$ fulfills the requirements $(a)$ and $(b)$.
Furthermore, $a$ and $\kphi$ depend smoothly on $\epsilon$. By well known regularity
properties of Fourier integral operators, we obtain that the operator $A$ is bounded from
$W_p^\alpha$ to $L_p$ for $\alpha\geq \n{1/2-1/p}$ and $1<p<\infty$. We refer here to
\cite{seeger} and \cite{sogge}. This completes the sketch of the proof.

%Eigentlich brauche ich doch nur eine Abschätzung auf einem Kompaktum $K$, da das
%Huygenssche Prinzip gilt. Wegen endlicher Ausbreitungsgeschwindigkeit (Evans)

%\ffbox{ Für $s \in \N$. Wegen endlicher Ausbreitungsgeschwindigkeit und $f \in \Omega$,
%$\Omega$ kompakt, gilt
%\begin{equation}
%    \begin{split}
%    &\dnorm{F(\tau,\fa)}_{W_1^s}\lesssim\\
%    &=\dnorm{F(\tau,\fa)}_{W_2^s}\lesssim \dnorm{\int c(\tau,z,\zeta) \;
%    e^{i\kphi(\tau,z,\zeta)-iy\fa\zeta} \; d\zeta \; f(y)
%   \; dy}_{W^s}\\
%    &\lesssim \sum_{\n{\alpha}\leq s} \dnorm{\int [\d{z}^{\alpha} c(\tau,z,\zeta) \;
%    e^{i\kphi(\tau,z,\zeta)-iy\fa\zeta}] \; d\zeta \; f(y) \; dy}_{L_2(z)}\\
%    &\lesssim \sum_{\n{\alpha}\leq s} \sup_{y\in\Omega} \int \dnorm{\d{z}^{\alpha} c(\tau,z,\zeta) \;
%    e^{i\kphi(\tau,z,\zeta)-iy\fa\zeta}}_{L_2(z)} \; d\zeta \; \dnorm{f}_1
%    \end{split}
%\end{equation}}

% By Cauchy-Schwarz we get
%\begin{equation}
%    \begin{split}
%    \dnorm{F(\tau,\fa)}_{H^s}=\dnorm{\int c(\tau,z,\zeta) \; e^{i\kphi(\tau,z,\zeta)-iy\fa\zeta} \; d\zeta \; f(y)
%    \; dy}_{H^s}
%    \end{split}
%\end{equation}

%\[
%    \sum_{k=1,2} \int a^k (t,z,\zeta) \; e^{i\kphi^k(t,z,\zeta)} \; \h{f^k}(\zeta) \;
%d\zeta\] only by a smooth error.

\subsection{Reduction to an estimate for the local part of the kernel, part 2}

We now exchange the indicator function $\X_B$ by a smooth variant of it. Let $\t{\X}_B
\in \Cu(\R^4)$, with
\[\t{\X}_B(x',u',x,u)=1 \mbox{ for all } (x',u',x,u) \in \supp{\X_B}\]
and \[\t{\X}_B(x',u',x,u)=0 \mbox{ for all } (x',u',x,u) \mbox{ with } (x,u) \notin
B_4(x',u').\] $\t{\X}_B$ is a smooth function supported near the diagonal
$(x,u)=(x',u')$.

Put
    \[\Omega:=\menge{(x,u)\in \R^2}{\n{x}<2 \const_1}.\]\index{xbtilde}
\begin{satz}
    \label{reduclocalpart2}
  To prove the theorem, it suffices to show that for all $\alpha>1/2$ the operator
  \begin{equation}
    \label{reducprop2}
    f\mapsto \int \t{\X}_{B} K_{h^\alpha(G)}(x',u',\fa,\fa) \; f(x',u') \; dx' \; du'
  \end{equation}
  is bounded from $L_p(\Omega)$ to $L_p(\R^2)$ for every $1<p<\infty$.
\end{satz}

\begin{proof}
    Let $\alpha>1/2$.
    We assume now that the operator defined in \eqref{reducprop2} is bounded from
    $L_p(\Omega)$ to $L_p(\R^2)$ for every $1<p<\infty$.
    By Proposition \ref{reduclocalpart} we have to show that there exists a constant $C$ with
    \begin{equation}
        \label{reducprop5}
        \sup_{x'\leq \const_0,u'} \int \n{\X_B K_{h^{\alpha}(G)}(x',u',x,u)} \; dx \; du \leq C.
    \end{equation}
    Since $\alpha>1/2$ we can choose an $\epsilon>0$ such that
    $\alpha-\epsilon>1/2$. Put $\alpha':=\alpha-\epsilon$.
    %We know by Proposition \ref{reduktionx1} that
    %\begin{equation}
    %    \label{reducprop3}
    %    \sup_{\n{x'}\geq \const_0,u'} \int \n{K_{h^{\alpha'}(G)}(x',u',x,u)} \; dx \; du \lesssim
    %    1.
    %\end{equation}
    %Hence the condition of the proposition implies that the operator
    %\[f\mapsto \int \t{\X}_{B} K_{h^{\alpha'}(G)}(x',u',\fa,\fa) \; f(x',u') \; dx' \; du'\]
    %is bounded from $L_p(\R^2)$ to $L_p(\R^2)$ for every $1<p<\infty$.
    Since we known from Proposition \ref{reduclocalpart} that the kernel $(1-\t{\X}_{B})K_{h^{\alpha'}}$ has bounded Schur
    norm,
    we get that the operator $h^{\alpha'}(G)$ is bounded from $L_p(\Omega)$ to $L_p(\R^2)$ for $1<p<\infty$.

    In \cite{mueller} it was shown that $(1+L)^{-\epsilon}\delta_0$ is in $L_p$ for some $p>1$ and that
    $(1+L)^{-\epsilon}\delta_0$ is rapidly decreasing away from the origin. By
    the transfer principle, Proposition \ref{transferproposition}, we get that
    the function
    \[f_{x',u'}^1:(x,u)\mapsto K_{(1+G)^{-\epsilon}}(x',u',x,u)\X_\Omega(x,u),\]
    lies in $L_p$ and that
    the function
    \[f_{x',u'}^2:(x,u)\mapsto K_{(1+G)^{-\epsilon}}(x',u',x,u)(1-\X_\Omega)(x,u),\]
    lies in $L_2$.
    %\begin{equation*}
    %    \sup_{x',u'} \int \n{K_{(1+G)^{-\epsilon}}(x',u',x,u)}^p dx \; du
    %\end{equation*}
    %is bounded.
    Now, by applying the operator $h^{\alpha'}(G)$ to the functions
    $f_{x',u'}^1$ and $f_{x',u'}^{2}$ for $\n{x'}\leq C_0$, $C_0$ a constant, we get
    \begin{equation*}
        \dnorm{h^{\alpha'}(G)f^1_{x',u'}}_p\leq C, \quad \dnorm{h^{\alpha'}(G)f^2_{x',u'}}_2\leq C
    \end{equation*}
    with a constant $C$ only depending on $C_0$, $p$ and $\alpha$. Since
    \[K_{h^\alpha(G)}(x',u',x,u)=h^{\alpha'}(G)f_{x',u'}(x,u),\]
    %this implies that
    %\begin{equation*}
    %    \sup_{x',u'} \int \n{K_{h^{\alpha}(G)}(x',u',x,u)}^p dx \; du
    %\end{equation*}
    and since the function $(x,u)\mapsto \X_{B}(x',u',x,u)$ is in $L_1$, with norm bounded
    by $(1+\n{x'})$ for every $(x',u')$, there is for every $C_0>0$ a constant $C_1>0$ with
    \begin{equation*}
        \sup_{\n{x'}\leq C_0,u'} \int \n{\X_B K_{h^{\alpha}(G)}(x',u',x,u)} \; dx \; du \leq
        C_1.
    \end{equation*}
    This gives us the assertion \eqref{reducprop5} and proves the proposition.
\end{proof}
With slight modifications of the proof of Proposition \ref{reduclocalpart2} we can also
show that the Conjecture \ref{FIOvermutung} together with the assumption
\eqref{reducprop2} implies that for every $1\leq p\leq \infty$ the operator
$\exp(i\sqrt{G})(1+G)^{-\alpha/2}$ extends to a bounded operator on $L_p(\R^2)$, provided
$\alpha>\norm{1/p-1/2}$.
%
%
%
%
%
%
%%%%%%%%%%%%%%%%%%%%%%%%%%%%%%%%%%%%%%%%%%%%%%%%%%%%%%%%%%%%%%%%%%%%%%%%%%%%%%%%%%%%%%%%%%%%%
%%%%%%%%%%%%%%%%%%%%%%%%%%%%%%  Proof %%%%%%%%%%%%%%%%%%%%%%%%%%%%%%%%%%%%%%%%%%%%%%%%%%%%%%%%
%
%
\newpage
\Section{Preparations}
%
%
%
%
%
%%%%%%%%%%%%%%%%%%%%%%%%%%%%%%%%%%%%%%%%%%%%%%%%%%%%%%%%%%%%%%%%%%%%%%%%%%%%%%%%%%%%%%%%%%%%%
%%%%%%%%%%%%%%%%%%%%%%%%%%%%%%    Dyadische Zerlegung    %%%%%%%%%%%%%%%%%%%%%%%%%%%%%%%%%%%%
%
%
%
\subsection{Dyadic decomposition.}
\label{dyadicsection}
%
%
%\fbox{hier ist $\X_j(x)=\X(2^{-j}x)$}

We have seen in Section \ref{harmonicgrushin1} that the spectrum of the Gru\v{s}in
operator lies in the union of the rays
\begin{equation*}
    \begin{split}
    \ray_{n, \epsilon}&:=
    \Bigmenge{(\epsilon\lambda,\tau)}{\tau=(2n+1)\lambda,\; \lambda>0}, \quad \epsilon:=\pm 1,\ n \in \N_0\\
    \ray_\infty&:=\menge{(0,\tau)}{\tau\geq 0}.
    \end{split}
\end{equation*}
This is the joined spectrum of the operators $G$ and $iU$. Since $G$ is a positive
operator, the spectrum of $G$ is contained in $\bigcup_{n \in \N_0} \ray_{n,1} \cup
\ray_\infty$.

In Section \ref{carnot} we have studied spheres belonging to the optimal control metric
associated to $G$. This gave us a description of the singularities of the distribution
kernel of $\cos(\sqrt{G})$. The singularities lie in a rather complicated curve, that
contains many, for $x'=0$ infinitely many, edges.

In this section, we decompose the integral kernel of $h^\alpha(G)$ in a sum of integral
kernels $K_{k,j}^\epsilon$ such that these parts of the integral kernel coincides, in
some way, with the edges in the singular support of $\cos(\sqrt{G})$.

Roughly, for every $(k,j,\epsilon)$, we choose a rectangle in the joint spectrum of $iU$
and $G$ of length $2^{j}$ in the $n$-direction and length $2^{2k-j}$ in the $\lambda$
direction. $\epsilon(2n+1)$ corresponds to the spectrum of $-iGU^{-1}$ and $\lambda$ to
the spectrum of $iU$. The dyadic operators with integral kernels $K_{k,j}^\epsilon$ are
of the form $h^\alpha(G)\X_{2k-j}j(iU)\X_j(-iGU^{-1})$ where $\X$ is a cut-off function.
Since $G=(iU) (-iGU^{-1})$ we can compute the integral kernels of these dyadic parts by
the functional calculus of $iU$ and $-iGU^{-1}$.

The Fourier transform gives us a spectral decomposition of $iU$ and in Chapter
\ref{strichartzchapter} we derived the spectral decomposition of the operator $-iGU^{-1}$
as a sum over singular integral operators $\P_{n,\epsilon}$. We get an explicit formula
for the integral kernel $K_{k,j}^\epsilon(x',u',x,u)$ away from the diagonal
$u'=u$.\absatz

By Proposition \ref{spektralzerlegung1} the operator $h^\alpha(G)$ can be decomposed in
the following way
\begin{equation*}
    h^\alpha(G)f=\sum\limits_{\epsilon=\pm 1}\sum\limits_{n=0}^{\infty} \pint h^\alpha((2n+1)\lambda) \;
    [\P_{\lambda,n,\epsilon}f] \; d\lambda.
\end{equation*}
Let $\X_j,\ j \in \Z$ denote a dyadic decomposition of unity on $\R^+$. Define now
\begin{equation*}
    \begin{split}
    %h_{k,j}^{\epsilon}(\lambda,n)&:=h((2n+1)\lambda) \; \X_{2k-j}(\lambda) \;
    %\X_j(2n+1),\\
    H_{k,j}^\epsilon f&:=\sum\limits_{n=0}^{\infty} \pint h^\alpha((2n+1)\lambda) \; \X_{2k-j}(\lambda) \;
    \X_j(2n+1) \;
    [\P_{\lambda,n,\epsilon}f] \; d\lambda,
    \end{split}
\end{equation*}
for $j \in \N_0,\ k \in \Z,\ \epsilon\in\{-1,1\},\ \lambda \geq 0$.
%Then
%\begin{equation*}
%    h(G)=\sum\limits_{\ueber{\epsilon=\pm 1}{k\in \Z,\ j \in
%    \N_0}} H_{k,j}^\epsilon.
%\end{equation*}
Since
\begin{equation*}
  2^{k-1}\leq \sqrt{(2n+1)\lambda}\leq 2^{k+1}
\end{equation*}
on the support of $h_{k,j}^\epsilon$, we have $h_{k,j}^{\epsilon}=0$, unless $2^k\geq
N/4$. So, if we fix any $k_0\gg 1$, we may choose $N$ sufficiently large so that
\begin{equation}
  h^\alpha(G)=\sum\limits_{\ueber{\epsilon=\pm 1}{k\geq k_0,\ j\in \N_0}} H_{k,j}^{\epsilon}.
\end{equation}
Furthermore, we choose $k_0$ sufficiently large so that we can delete the factor
$(1-\eta)(\fa/N)$ in $h^\alpha$. In the following, we use the abbreviation $\l=2k-j$.
Define
\[\t{\X}(x):=\n{x}^{-\alpha/2} \X(x).\] Then
\begin{equation*}
    \t{\X}_{\l}(\lambda)\t{\X}_j(2n+1)=2^{k\alpha} \;
    \big((2n+1)\n{\lambda}\big)^{-\alpha/2} \; \X_\l(\lambda) \; \X_j(2n+1)
\end{equation*}
and
\begin{equation*}
    \begin{split}
    h^\alpha((2n+1)\lambda)  &\; \X_{\l}(\lambda) \; \X_j(2n+1)\\
    &=((2n+1)\lambda)^{-\alpha/2} \;
    e^{i\sqrt{(2n+1)\lambda}} \; \X_{\l}(\lambda) \; \X_j(2n+1)\\
    &=2^{-\alpha k} \; e^{i\sqrt{(2n+1)\lambda}} \; \t{\X}_{\l}(\lambda) \;
    \t{\X}_j(2n+1).
    \end{split}
\end{equation*}
Since $\t{\X}$ is of similar type as $\X$, we shall again write $\X$ in place of $\t{X}$.
Observe that
\[
    H_{k,j}^\epsilon f:=\sum_{\t{\epsilon}=\pm 1}\sum\limits_{n=0}^{\infty} \pint 2^{-\alpha k}
    e^{i\sqrt{(2n+1)\lambda}}
    \; \X_{\l}(\epsilon\t{\epsilon}\lambda) \;
    \X_j(2n+1) \;
    [\P_{\lambda,n,\t{\epsilon}}f] \; d\lambda,\]
and hence we get with Corollary \ref{spektralzerlegung2}
\begin{equation*}
  \begin{split}
  H_{k,j}^\epsilon f&=2^{-\alpha k} \sum_{\t{\epsilon}=\pm 1}
  \sum_n \X_j(\epsilon\t{\epsilon}(2n+1)) \; \gamma_n^{\epsilon}(iU) \; (\P_n^{\t{\epsilon}}
  f)\\
  &=2^{-\alpha k} \sum_n \X_j(2n+1) \; \gamma_n^{\epsilon}(iU) \; (\P_n^\epsilon f),
  \end{split}
\end{equation*}
with
\begin{equation*}
  \gamma_n^\epsilon(\lambda):=\t{\X}_{\l}(\epsilon\lambda) \; e^{i\sqrt{(2n+1)\norm{\lambda}}}.
\end{equation*}
We denote the integral kernel of the operator $H_{k,j}^{\epsilon}$ by
$K_{k,j}^{\epsilon}$. Away from the diagonal, $K_{k,j}^{\epsilon}$ is given by
\begin{equation*}
    \begin{split}
    K_{k,j}^\epsilon &(x',0,x,u)=
    \frac{2^{-\alpha k}}{2\pi}\\
    &\times\sum_{n=0}^{\infty} \X_j(2n+1) \int_{-\infty}^\infty P_n^\epsilon(x',0,x,u-s)
    \; e^{i\sqrt{(2n+1)\n{\lambda}}} \; \X_{\l}(\epsilon\lambda) \; e^{-i\lambda s} \; d\lambda \; ds.
    \end{split}
\end{equation*}
We put
\begin{equation}
  \Phi_{\l,n}^\epsilon(s):=\frac{1}{2\pi} \int_{-\infty}^{\infty} e^{i\sqrt{(2n+1)\n{\lambda}}} \;
  \X_{\l}(\epsilon\lambda) \; e^{-i\lambda s} \; d\lambda.
\end{equation}
Then $\Phi_{\l,n}^{-1}(s)=\Phi_{\l,n}^{1}(-s)$ and
\begin{equation*}
  K_{k,j}^\epsilon(x',0,x,u)=
  \frac{2^{-\alpha k}}{2\pi} \sum_{n=0}^{\infty} \X_j(2n+1) \int_{-\infty}^{\infty}
  P_n^\epsilon(x',0,x,u-s) \; \Phi_{\l,n}^\epsilon(s)
  \; ds.
\end{equation*}
Since $P_n^{-1}(x',0,x,u)=P_n^1(x',0,x,-u)$ and $\Phi_{\l,n}^{-1}(s)=\Phi_{\l,n}^{1}(-s)$
we get
\begin{equation*}
    K_{k,j}^{-1}(x',0,x,u)=K_{k,j}^1(x',0,x,u).
\end{equation*}
Hence we can restrict to the case $\epsilon=1$. Put
\[K_{k,j}:=K_{k,j}^1.\]
Thus, to prove the theorem it suffices to show the following proposition. Recall that in
the previous section we have defined
\[\Omega:=\menge{(x,u)\in \R^2}{\n{x}<2 \const_1}.\]
\begin{prop}
    \label{dyadicprop}
    If $\alpha>1/2$, then
    \begin{equation*}
        \sum\limits_{\ueber{\epsilon=\pm 1}{k\geq k_0,\ j\geq 0}}
        \dnorm{\t{\X}_B K_{k,j}}_{(L_p(\Omega),L_p)} < \infty.
    \end{equation*}
    for every $p$, $1<p<\infty$, where $\dnorm{K}_{(L_p(\Omega),L_p)}$ denotes the operator norm of
    the
    integral operator $f\mapsto \int K(x',u',\fa,\fa) \; f(x',u') \; dx' \; du'$ from
    $L_p(\Omega)$ to $L_p$.
\end{prop}
\subsection{Integral formulas for $K_{k,j}$}
In this section we completely follow the proof of Müller and Stein. More exactly, we
reproduce the Sections $1.2$ and $2$ of \cite{mueller}. The only difference here is that
we do not have translation invariant operators. Of course, we use our formulas for the
spectral projection operators $\P_n$ we derived in Section \ref{harmonicgrushin1} and
Section \ref{harmonicgrushin2}.\absatz

Recall the definition of $\Phi_{\l,n}^1$,
\begin{equation}
    \label{phidef}
  \Phi_{\l,n}^1(s):=\frac{1}{2\pi} \int_{-\infty}^{\infty} e^{i\sqrt{(2n+1)\n{\lambda}}} \;
  \X_{\l}(\lambda) \; e^{-i\lambda s} \; d\lambda.
\end{equation}
We define $\Phi_{\l,n}:=\Phi_{\l,n}^1$.

\begin{lemma}
    Let $f \in \Cnu(\R)$ be supported in $\intaa{1/2,2}$. For every $N\in \N$ there exist functions
    $f_0, \dots f_N \in \Cnu(\R)$ supported in $\intaa{1/4,4}$ and $E_N \in \Cu(\R^2)$, such that for
    $(a,b) \in \R^2$ with $\n{(a,b)}>1$
    \begin{equation*}
        \int_{-\infty}^{\infty} e^{i(ax-bx^2/2)} \; f(x) \; dx=e^{ia^2/(2b)} \; \sum_{\nu=0}^N
        b^{-1/2-\nu} \; f_{\nu}(a/b)+E_N(a,b),
    \end{equation*}
    where $E_N$ satisfies
    \[E_N^{(\alpha)}=O(\n{(a,b)}^{-N/2-1}), \mbox{ for every } \alpha \in \N^2.\]
\end{lemma}

\begin{proof}
    Müller and Stein \cite{mueller}, Lemma 1.4.
    The proof bases upon the method of stationary phase.
\end{proof}

We may apply the lemma to $\Phi_{\l,n}$, since $\sqrt{(2n+1)2^{\l}}\sim 2^k \gg 1$, and
obtain
\begin{equation}
    \label{diadic2-1}
    \begin{split}
    \Phi_{\l,n}(u)=&e^{i(2n+1)/(4u)} \; 2^\l \sum_{\nu=0}^N (2^{\l+1}u)^{-1/2-\nu} \; f_\nu
    \Big(\sqrt{\frac{2n+1}{2^\l}} \frac{1}{2u}\Big)\\
    &+2^\l E_N(\sqrt{(2n+1)2^\l},2^{\l+1}u),
    \end{split}
\end{equation}
with $f\nu$ and $E_N$ as in the lemma. Put $a_{n,\l}:=\sqrt{(2n+1)2^\l}$. Since
$a_{n,\l}\sim 2^k$ and since
\[
    (2^{\l+1}u)^{-1/2-\nu}=a_{n,\l}^{-1/2-\nu}a_{n,\l}^{1/2+\nu}(2^{\l+1}u)^{-1/2-\nu}=a_{n,\l}^{-1/2-\nu}\Big(\sqrt{\frac{M+2n}{2^\l}}\frac{1}{2u}\Big)^{1/2+\nu},
\]
the $\nu$-th term in (\ref{diadic2-1}) is given by
\[
    a_{n,\l}^{-1/2-\nu} \; \t{f}_\nu \Big(\sqrt{\frac{M+2n}{2^\l}}\frac{1}{2u}\Big),
\]
with $\t{f}_nu(x)=x^{1/2+\nu}f_{\nu}(x)$. Since $\t{f}_\nu$ is of the same type as
$f_\nu$ and we only have to sum over finitely many $\nu$ we may reduce to the case where
$\Phi_{\l,n}$ is either of the form
\begin{equation*}
  (a) \quad 2^\l \; a_{n,\l}^{-1/2-\nu} \; f\Big(\sqrt{\frac{2n+1}{2^\l}}\frac{1}{4u}\Big) \;
  e^{i(2n+1)/(4u)},
\end{equation*}
with $f \in \Cnu(\R)$ supported in $\intaa{1/8,2}$ and $\nu \in N_0$, or of the form
\begin{equation*}
    (b) \quad \Phi_{\l,n}(u):=2^{\l} \; E_N(a_{n,\l},2^{\l+1}u).
\end{equation*}
First we study the case (b).\absatz

Müller and Stein showed that the corresponding operators with convolution kernels
\[\t{K}_{k,j}:=2^{-\alpha k}\sum_n \X_j(1+2n)\int P_n^\H(x,y,u-s) \; \Phi_{2k-j,n}(s) \; ds,\]
where $P_n^\H$ is the kernel of the projection operator $\P_n^\H$ on the Heisenberg group
and $\Phi_{2k-j,n}$ is given by (b), are $L_p$ bounded for $1<p<\infty$. Furthermore,
they showed that one can sum up all $L_p$-operator norms for $\alpha>0$. Observe that
$\alpha>1/2$. Hence the operator with convolution kernel $\t{K}:=\sum_{k,j} K_{k,j}$ is
bounded on $L_p$ for $1<p<\infty$.\absatz

Let $\pi$ be the representation of $\H_1$ with $\pi(L)=G$. We mentioned in Section
\ref{transferencesection} that one can transfer estimates for operators on $\H_1$ to
operators on $\R^2$. By transfer methods the operator with integral kernel $K:=\sum_{k,j}
K_{k,j}$, where $K_{k,j}$ is given by
\begin{equation*}
  K_{k,j}(x',u',x,u)=
  \frac{2^{-\alpha k}}{2\pi} \sum_{n=0}^{\infty} \X_j(2n+1) \int_{-\infty}^{\infty}
  P_n(x',u',x,u-s) \; \Phi_{\l,n}(s)
  \; ds
\end{equation*}
and $\Phi_{\l,n}$ is given by (b), is bounded on $L_p$. Furthermore, also the operator
with truncated kernel $\t{\X}_B K$ is bounded on $L_p$. We do not want to go in the
details here. There is one difficulty. One can use transference, usually, only for
bounded measures and the convolution kernel $\t{K}$ is not bounded. We just want to
mention an argument that one has to use, to make the transference principle work.\absatz

{\footnotesize We take a dyadic decomposition of unity $\psi_r^\epsilon$ in the following
way. Put
\[\psi_r^\epsilon:=\sum_{\n{j}\leq r,\ \n{k}\leq r} \X_{2k-j}(\epsilon \lambda)
\X_j(2n+1),\] where $\X$ is chosen as in section \ref{dyadicsection}. We define
$\Psi_r^\epsilon:={\cal G}^{-1}(\psi_r^\epsilon)$, where $\cal{G}$ is the Gelfand
transform for the algebra of radial functions on $\H_1$. Furthermore, we put
$\Phi_r^\epsilon:=\pi(\Psi_r^\epsilon)$. Then the set
\[X:=\menge{\Phi_r^\epsilon(f)}{f \in \Cnu(\R^2),\ r\in \N_0,\ \epsilon=\pm 1}\] is dense
in $L_p(\R^2)$. Now we define the operator $A$ on $L_p(\R^2)$ by
\begin{equation}
    \label{partb}
    A(\Phi_r^\epsilon f):=\sum_{k,j} \pi(\conv{\t{K}_{k,j}}{\Psi_r^\epsilon})f.
\end{equation}
By transfer methods we can deduce that $\pi(\conv{\phi_r^\epsilon}{\t{K}_{k,j}})$ is
bounded on $L_p$, since $\conv{\phi_r^\epsilon}{\t{K}_{k,j}}$ is a bounded measure. The
convolution kernel of the original operator $H_{k,j}^1$ on the Heisenberg group is given
by ${\cal G}^{-1}(\phi_{k,j})$ where $\phi_{k,j}=(\lambda,n)\mapsto
\phi_{k,j}(\lambda,n)$ has compact support in $\lambda$ with $\lambda\sim 2^{2k-j}$
(compare \cite{mueller}, section 1.1). Thus we can localize the Fourier transform of
$\Phi_{\l,n}$, given by (b), to the same region. Hence for every $(r,\epsilon)$ there are
only finitely many $k,\ j$ we have to sum up in (\ref{partb}). Therefore, the operator
$A$ is bounded on $L_p$. But $A$ is equal to the operator with integral kernel
$K$.}\absatz

%\ffbox{Also bleibt zu studieren\\
%$K_{k,j}(x',x,t)= 2^{-\alpha_0 k} \sum_n \X_j(M+2n) \int P_n(x,x_1,u-s) \; 2^\l
%a_{n,\l}^{-1/2-\nu} \; f_\nu(\sqrt{(M+2n)/(2^\l)}\frac{1}{4s}) \; e^{i(M+2n)/(4s)} \;
%ds$}
\noindent The case that $\Phi_{\l,n}$ is given by (a) is left.

Since
\[ a_{n,\l}^{-1/2-\nu}\; \chi_j(2n+1)=2^{-k/2-\nu k} \; \t{\chi}_j(2n+1),\]
with $\t{\chi}$ of similar type than $\X_j$, we only have to prove the following
proposition instead of Proposition \ref{dyadicprop}. We have exchanged $\alpha$ by the
critical index $m$.

%\begin{prop}
%  Suppose that $K_{k,j}$ is given by
%  \begin{equation*}
%    K_{k,j}(x',x,t):= 2^{-m k} \sum_n \X_j(2n+1) \int P_n(x,x_1,(u-s)/4) \;
%    \Phi_{k,j,n}(s) \; ds
%  \end{equation*}
%  with
%  \begin{equation*}
%    \Phi_{k,j,n}(s):=2^{k/2+k-j}f\Big(\sqrt{\frac{M+2n}{2^{2k-j}}}\frac{1}{s}\Big)
%    \; e^{i(M+2n)/s},
% \end{equation*}
%  $f \in \Cnu(\R)$ supported in $\intaa{1/8,2}$, $k\geq 0$ sufficiently large. Then
%  \begin{equation*}
%    \sup\limits_{x'\geq 0} \; \sum_{j\geq 0} \dnorm{K_{k,j}(x',\fa,\fa)}_{L¹(B_{x'})}=O(k²).
%  \end{equation*}
%\end{prop}
%oder
\begin{prop}
    \label{dyadicprop2}
  Suppose that $K_{k,j}$ is given by
  \begin{equation*}
    K_{k,j}(x',0,x,u):= 2^{-m k} \sum_n \X_j(m+n) \int P_n(x',0,x,u/2-s/2) \;
    \Phi_{k,j,n}(s) \; ds
  \end{equation*}
  with
  \begin{equation*}
    \Phi_{k,j,n}(s):=2^{3k/2-j}f\Big(\sqrt{\frac{m+n}{2^{2k-j}}}\frac{1}{s}\Big)
    \; e^{i(m+n)/s},
  \end{equation*}
  $f \in \Cnu(\R)$ supported in $\intaa{1/8,2}$, $\chi \in \Cnu(\R)$ supported in $\intaa{1/2,2}$,
  $\chi_j=\chi(2^{-j}\fa)$ and $k\geq 0$ sufficiently large. Then for every $\epsilon>0$
  there exists a constant $C_\epsilon$ with
  \begin{equation*}
    \sup\limits_{\n{x'}\leq 2\const_1} \; \sum_{j\in \N_0} \; \dnorm{\t{\X}_{B}K_{k,j}(x',0,\fa,\fa)}_{L_1}
    \leq C_{\epsilon} \; 2^{\epsilon k}.
  \end{equation*}
\end{prop}
%\ffbox{Eigentlich wären die Träger bei $\intaa{1/4,4}$ für $\chi$ und
%$\intaa{1/8*2^{1/2},2^{3/2}}$ für f, das macht aber keinen Unterschied}

Before we prove Proposition \ref{dyadicprop2}, we do some more reductions. Let $\t{\chi}
\in \Cnu$ with $\t{\chi}(x)=1$ for $1/4\leq x \leq 16$ and supported in $\intaa{1/8,32}$.
Since $\chi_j(m+n)\Phi_{k,j,n}(s)=0$, unless $1/4\leq 2^{k-j} s \leq 16$
\[\chi_j(m+n)\Phi_{k,j,n}(s)=\chi_j(m+n)\t{\chi}(2^{k-j}s)\Phi_{k,j,n}(s).\]
%\ffbox{$2^{k-j}s\leq 1/4 \Rightarrow \sqrt{\frac{m+n}{2^{2k-j}}}\frac{1}{s}\geq
%\sqrt{\frac{2^j}{2^{2k-j+1}}}\frac{4}{2^{j-k}}=4/\sqrt{2}\geq 2$\\
%$2^{k-j}s\geq 16 \Rightarrow \sqrt{\frac{m+n}{2^{2k-j}}}\frac{1}{s}\leq
%\sqrt{\frac{2^{j+1}}{2^{2k-j}}}\frac{2^{k-j}}{16}=\sqrt{2}/16\leq 1/8$}
Moreover, writing
\[
    f\Big(\sqrt{\frac{m+n}{2^{2k-j}}}\frac{1}{s}\Big)=g\Big(\log\Big(\frac{m+n}{2^{2k-j}}s^{-2}\Big)\Big),\]
with $g$ smooth and supported in $\intaa{\log 1/16, \log 2}\teil \intaa{-\pi,\pi}$, we
see that
\begin{equation*}
    \Phi_{k,j,n}(s)=\sum_{\nu \in \Z} a_\nu \Big(\frac{m+n}{2^{2k-j}s^2}\Big)^{i\nu}\; 2^{3k/2-j}
    \; \t{\chi}(2^{k-j}s)\; e^{i(m+n)/s}=:\sum_{\nu \in \Z} a_\nu \Phi_{k,j,n,\nu}(s),
\end{equation*}
where
\begin{equation}
    \label{anu}
    a_\nu=O(\n{\nu}^{-N}),
\end{equation}
for every $N \in \N$. By the definition
\[
    K_{k,j,\nu}(x',0,x,u):= 2^{-m k} \sum_n \X_j(m+n) \int P_n(x',0,x,u/2-s/2) \;
    \Phi_{k,j,n,\nu}(s) \; ds
\]
we get
\begin{equation*}
    K_{k,j}=\sum_\nu a_\nu \; K_{k,j,\nu}.
\end{equation*}
By defining $\chi_{(\nu)}(s):=s^{-i2\nu} \t{\chi}(s)$, we still have $\chi_{(\nu)} \in
\Cnu$ and supported in $\intaa{1/8,32}$. Furthermore,
\begin{equation}
    \label{chinu}
    \dnorm{\chi_{(\nu)}^{(\alpha)}}_\infty=O((1+\n{\nu})^{\alpha}),
\end{equation}
for all $\alpha \in \N$ and
\begin{equation*}
    \begin{split}
    \Phi_{k,j,n,\nu}(s)&=2^{3k/2-j} \; \Big(\frac{m+n}{2^{2k-j}s^2}\Big)^{i\nu}
    \; \t{\chi}(2^{k-j}s)\; e^{i(m+n)/s}\\
    &=2^{3k/2-j} \; \chi_{(\nu)}(2^{k-j}s)\; 2^{-ij\nu} \; (m+n)^{i\nu} \; e^{i(m+n)/s}.
    \end{split}
\end{equation*}
%\ffbox{bis hier ok (Seite 312). Hier brauche ich jetzt die konkrete Form der $Q_n$'s.}
We define now \[\chi_{(\nu),j}(x):=\chi_{(\nu)}(2^{-j}x)=(2^{-j}x)^{i\nu} \chi(2^{-j}x)\]
and get $\chi_{(\nu),j}(m+n)=2^{-ij\nu}(m+n)^{i\nu}\chi_j(m+n)$. Thus,
\begin{equation*}
    \chi_j(m+n) \; \Phi_{k,j,n,\nu}(s)=2^{3k/2-j} \; \chi_{(\nu),j}(m+n) \; \chi_{(\nu)}(2^{k-j}s)\; e^{i(m+n)/s}.
\end{equation*}
By inserting the formulas for $\Phi_{k,j,n,\nu}$, we obtain
\begin{equation}
    \label{kkjnu}
    \begin{split}
    K_{k,j,\nu}&(x',0,x,u):= 2^{-m k} \sum_n \X_j(m+n)\\
    &\times \int P_n(x',0,x,(u-s)/2) \;
    \Phi_{k,j,n,\nu}(s) \; ds\\
    =&2^{k/2} \; 2^{(m-1)(j-k)}
    \sum_n \chi_{(\nu),j}(m+n) \\
    &\times \int \chi_{(\nu)}(2^{k-j}s) \;
    (2^{-mj}P_n)(x',0,x,(u-s)/2) \; e^{i(m+n)/s} \; ds.\\
    \end{split}
\end{equation}
Now we need our explicit formulas for $P_n$, which we derived in Chapter
\ref{strichartzchapter}. We have
\begin{equation}
    P_n(x',0,x,(u-s)/2)=C \; [Q_n-Q_{n-2}](x',0,x,(u-s)/2),
\end{equation}
where
\begin{equation*}
    Q_n(x',0,x,(u-s)/2)=(q_{n,m}^{+} \; e^{in\sigma} + \t{q}_{n,m}^{-} \; e^{-in\sigma})
    \frac{(x²+x_1²-i(u-s))^{n/2}}{(x²+x_1²+i(u-s))^{n/2+m+1}}.
\end{equation*}
$q_{n,m}^+$ is given by
\begin{equation*}
    \sum_{\l=0}^n
    \X(\l/n)\frac{\Gamma(\l+m+1)}{\Gamma(\l+1)}
    \frac{\Gamma(n-\l+m+1)}{\Gamma(n-\l+1)} \; e^{-i2\l\sigma},
\end{equation*}
where $\X\in \Cnu$ with $\X(x)=0$ for $x\geq 3/4$. Furthermore, $\t{q}_{n,m}^{-}$ is
similar to the function $q_{n,m}^-:=\ab{q_{n,m}^+}$. $\sigma$ is given by
\begin{equation}
    \label{defsigmaprep}
    e^{i\sigma}=\frac{2xx'+iw}{\sqrt{R^2+(u-s)^2}},
\end{equation}
with $w=\sqrt{(x^2-x'^2)^2+(u-s)^2}$ and $R=x^2+x'^2$. Define now
%\begin{equation}
%    2^{-mj}q_n=\sum_{\l=0}^n
%    \X(\l/n)\frac{\Gamma(\l+m+1)}{\Gamma(\l+1)}
%    \frac{\Gamma(n-\l+m+1)}{\Gamma(n-\l+1)} \; 2^{-mj} \; e^{-i2\l\sigma}.
%\end{equation}
\begin{equation*}
    \omega^+(x',x,u,s)=\arctan\Big(\frac{Rw-2xx'(u-s)}{R2xx'+(u-s)w}\Big)+\frac{1}{s}
\end{equation*}
and
\begin{equation*}
    \omega^-(x',x,u,s)=-\arctan\Big(\frac{Rw+2xx'(u-s)}{R2xx'-(u-s)w}\Big)+\frac{1}{s}.
\end{equation*}
$\arctan$ denotes the branch of $\tan^{-1}$ taking values in $\intaa{0,\pi}$, since
\mbox{$\n{Rw} \geq \n{2xx_1u}$} and thus the imaginary part is always positive (compare
\eqref{arctan1} and \eqref{arctan2}). Observe that
\begin{equation}
    \label{omega-=omega+}
    \begin{split}
    \omega^-(x',x,-u,-s)=&-\Big(\arctan\Big(\frac{Rw-2xx'(u-s)}{R2xx'+(u-s)w}\Big)+\frac{1}{s}\Big)\\
    =&-\omega^+(x',x,u,s)
    \end{split}
\end{equation}
holds. We define
\begin{equation*}
    Q_n^\pm(x',0,x,(u-s)/2):=q_n^{\pm} \; e^{\pm in\sigma} \frac{(x²+x_1²-i(u-s))^{n/2}}{(x²+x_1²+i(u-s))^{n/2+m+1}}
\end{equation*}
and $P_n^\pm:=Q_n^\pm+Q_{n-2}^\pm$. By these definitions we get
% For $z=R+iu$ we get
%\begin{equation*}
%    \begin{split}
%    \Big(&\frac{R-iu}{R+iu}\Big)^{1/2} \; e^{i\sigma}=\Big(\frac{\ab{z}}{z}\Big)^{1/2} \; e^{i\sigma}=
%    \Big(\frac{\ab{z}^2}{\n{z}^2}\Big)^{1/2} \; e^{i\sigma}
%    =\frac{\ab{z}}{\n{z}} \; e^{i\sigma}
%    =\frac{(R-iu)(2xx_1+iw)}{\n{z}^2}\\
%    &=\frac{R2xx_1+uw+i(Rw-2xx_1u)}{\n{z}^2}
%   =\frac{\n{z}^2}{\n{z}^2} \; \exp\Big(i\arctan\Big(\frac{Rw-2xx_1u}{R2xx_1+uw}\Big)\Big)\\
%    &=\exp\Big(i\arctan\Big(\frac{Rw-2xx_1u}{R2xx_1+uw}\Big)\Big)
%    \end{split}
%\end{equation*}
%where $\arctan$ denotes the branch of $\tan^{-1}$ taking values in $\intaa{0,\pi}$ since
%\mbox{$\n{Rw} \geq \n{2xx_1u}$} and thus the imaginary part is always positive.
%\begin{equation*}
%    \begin{split}
%    \Big(\frac{1}{R+iu}\Big)^{m+1} \; e^{i\sigma(m+1)}&=\Big(\frac{2xx_1+iw}{z}\Big)^{m+1} \n{z}^{-m-1}
%    =\Big(\frac{(R-iu)(2xx_1+iw)}{\ab{z}z}\Big)^{m+1} \n{z}^{-m-1}\\
%    &=\Big(\frac{R2xx_1+uw+i(Rw-2xx_1u)}{\n{z}^2}\Big)^{m+1} \n{z}^{-m-1}\\
%    &=\exp\Big(i(m+1)\arctan\Big(\frac{Rw-2xx_1u}{R2xx_1+uw}\Big)\Big)\n{z}^{-m-1}
%    \end{split}
%\end{equation*}
\begin{equation*}
    \begin{split}
    Q_n^\pm\; e^{i(m+n)/s}&=Q_n^\pm\; e^{i(m+n+1)/s} \; e^{-i/s}\\
    &=q_{n,m}^{\pm} \; e^{\pm in\sigma} \; \frac{(x²+x_1²-i(u-s))^{n/2}}{(x²+x_1²+i(u-s))^{n/2+m+1}} \;
    e^{i(m+n+1)/s} \; e^{-i/s}\\
    &=q_{n,m}^{\pm} \; e^{i(n+m+1)\omega^\pm} \; e^{-i/s} \; e^{\mp i(m+1)\sigma} \;
    \n{R+i(u-s)}^{-m-1}.
    \end{split}
\end{equation*}
By \eqref{defsigmaprep} and \eqref{omega-=omega+}, we find out that the term
$Q_n^-(x',0,x,(-u+s)) \; e^{i(m+n)/s}$ is similar to $\ab{Q_n^+(x',0,x,(u-s)) \;
e^{i(m+n)/s}}$. Since we are interested in $L_1$-norms of $K_{k,j,\nu}$ and by
\eqref{kkjnu}, we only have to consider $Q_n^+$ in our estimates to come. In fact, if we
define $K_{k,j,\nu}^+$ in such a way that it only involves $+$-terms and $K_{k,j,\nu}^-$
so that it only involves $-$-terms, we deduce from the $L_1$-boundedness of
$K_{k,j,\nu}^+$ the $L_1$-boundedness of $\ab{K_{k,j,\nu}^+}$. If we now exchange $u$ by
$-u$ and $s$ by $-s$ we get the $L_1$-boundedness of $K_{k,j,\nu}^-$.

%Since the formulas for $Q_n^+$ and $Q_n^-$ are very similar, we only consider $Q_n^+$ in
%the following.
Furthermore, define
\begin{equation}
    \zeta^\pm_{\nu,j,m}(x',x,u,s):=\sum_n \chi_{(\nu),j}(m+n) \; q_n^\pm \; 2^{-mj} \; e^{i
    (n+m+1)\omega^\pm}.
\end{equation}
Thus,
\begin{equation}
    \label{qnnu}
    \begin{split}
    2^{k/2} &\; 2^{(m-1)(j-k)}
    \sum_n \chi_{(\nu),j}(m+n)\\ &\times \int \chi_{(\nu)}(2^{k-j}s) \;
    (2^{-mj}Q_n^\pm)(x',0,x,(u-s)/2) \; e^{i(m+n)/s} \; ds\\
    =&2^{k/2} \; 2^{(m-1)(j-k)}\\
    &\times
    \int \frac{\zeta_{\nu,j,m}^\pm(x',x,u,s)}{(R^2+(u-s)^2)^{(m+1)/2}}
    \; e^{-i/s} \; e^{\mp i(m+1)\sigma} \; \chi_{(\nu)}(2^{k-j}s)\; ds.
    \end{split}
\end{equation}
%\begin{equation*}
%    \begin{split}
%    K_{k,j,\nu}^+ &(x,x_1,u)=2^{k/2} \; 2^{(m-1)(j-k)}\\
%    &\times \sum_n \chi_{(\nu),j}(m+n) \; \int \chi_{(\nu)}(2^{k-j}s) \;
%    (2^{-mj}P_n)(x,x_1,(u-s)/2) \; e^{i(m+n)/s} \; ds\\
%    =&2^{k/2} \; 2^{(m-1)(j-k)}
%    \int \frac{\zeta_{\nu,j}(x,x_1,u,s)}{(R^2+(u-s)^2)^{(m+1)/2}}
%    \; e^{-i/s} \; e^{-i(m+1)\sigma} \; \chi_{(\nu)}(2^{k-j}s)\; ds.
%    \end{split}
%\end{equation*}
%\ffbox{vergleiche hier S. 313 Formel 2.6'\\
%Nun Reduktion auf $\nu=0$, das liefert das Äquivalent zu 2.6}
%\begin{equation}
%    \begin{split}
%    K_{k,j}^+ &(x,x_1,u)=2^{k/2} \; 2^{(m-1)(j-k)}\\
%   &\times \int \frac{\zeta_{j}(x,x_1,u,s)}{(R^2+(u-s)^2)^{(m+1)/2}}
%   \; e^{-i/s} \; e^{-i(m+1)\sigma} \; \chi(2^{k-j}s)\; ds.
%    \end{split}
%\end{equation}
In order to simplify the notation, we only consider $\nu=0$ from now on. Our estimates
for $K_{k,j,\nu}$ are depending on estimates we get for $\zeta_{\nu,j,m}^\pm$ and
$\zeta_{\nu,j,m-1}^\pm$. For $\zeta_{\nu,j,m}^\pm$ resp. $\zeta_{\nu,j,m-1}^\pm$ we use
our estimates in Proposition \ref{propqn1} and Proposition \ref{propqn2}. These estimates
only depend on $4$ derivatives of $\X_\nu$. And hence, by \eqref{chinu}, these estimates
are bounded by $O((1+\n{\nu})^4)$. Since $a_\nu$ descends as fast, as we want (see
\eqref{anu}), it will be clear, at the end of our proof, that it suffices to show
estimates for $\nu=0$. Put
\[\zeta_{j,m}:=\zeta_{0,j,m}.\]
Hence we get, with this new definition and by \eqref{qnnu}, that
\begin{equation*}
    \begin{split}
    2^{k/2} &\; 2^{(m-1)(j-k)}\\
    \times& \sum_n \chi_{(0),j}(m+n) \; \int \chi_{(0)}(2^{k-j}s) \;
    (2^{-mj}Q_n^\pm)(x',0,x,(u-s)/2) \; e^{i(m+n)/s} \; ds\\
    =&2^{k/2} \; 2^{(m-1)(j-k)}
    \int \frac{\zeta_{j,m}^\pm(x',x,u,s)}{(R^2+(u-s)^2)^{(m+1)/2}}
    \; e^{-i/s} \; e^{\mp i(m+1)\sigma} \; \chi(2^{k-j}s)\; ds.
    \end{split}
\end{equation*}
For $Q_{n-2}^+$ we get
\begin{equation*}
    \begin{split}
    Q_{n-2}^+\; e^{i(m+n)/s}&=
    q_{n-2}^{+} \; e^{i((n-2)+m+1)\omega} \; e^{i/s} \; e^{-i(m-1)\sigma} \;
    \n{R+iu}^{-m-1}.
    \end{split}
\end{equation*}
Observe that
\[
    \X_j(m+n)=\X_j(m+n-2)+2^{-j}(2^j \X_j(m+n)-2^j \X_j(m+n-2)).
\]
Now let $\t{\X}_j(x):=2^j \X_j(x+2)-2^j\X_j(x)=2 \int_0^1 \X'((x+2-2t)/2^j) \; dt$. Since
this function has similar properties as $\X_j$, we obtain
\begin{equation*}
    \begin{split}
    2^{k/2+(m-1)(j-k)} \sum_n &\chi_{j}(m+n) \; \int \chi_{(0)}(2^{k-j}s) \;
    (2^{-mj}Q_{n-2}^\pm) \; e^{i(m+n)/s} \; ds\\
    =2^{k/2} \; 2^{(m-1)(j-k)} \; \Big(&\int \frac{\zeta_{j,m}^\pm(x',x,u,s)}{(R^2+(u-s)^2)^{(m+1)/2}}
    \; e^{i/s} \; e^{\mp i(m+1)\sigma} \; \chi(2^{k-j}s)\; ds\\
    +2^{-j} \; &\int \frac{\t{\zeta}_{j,m}^\pm(x',x,u,s)}{(R^2+(u-s)^2)^{(m+1)/2}}
    \; e^{i/s} \; e^{\mp i(m+1)\sigma} \; \chi(2^{k-j}s)\; ds\Big),
    \end{split}
\end{equation*}
with $\t{\zeta}_{j,m}$ of the same type as $\zeta_{j,m}$. From now on we denote
$K_{k,j,0}$ again by $K_{k,j}$. Thus,
\begin{equation}
    \label{Kkj,2^j-k<1}
    \begin{split}
    K_{k,j}&(x',0,x,u)=\sum_{\epsilon\in\{+,-\}} C \; 2^{k/2} \; 2^{(m-1)(j-k)}\\
    \times& \Big(\int \frac{\zeta_{j,m}^\epsilon (x',x,u,s)}{(R^2+(u-s)^2)^{(m+1)/2}}
    \; (e^{-i/s}-e^{i/s}) \; e^{-\epsilon i(m+1)\sigma} \; \chi(2^{k-j}s)\; ds\\
    &-2^{-j} \; \int \frac{\t{\zeta}^\epsilon_{j,m}(x',x,u,s)}{(R^2+(u-s)^2)^{(m+1)/2}}
    \; e^{i/s} \; e^{-\epsilon i(m+1)\sigma} \; \chi(2^{k-j}s)\; ds\Big),
    \end{split}
\end{equation}
%\ffbox{Nun berechnung für $R_n$}
For the region where $w(s)$ is small we use
(\ref{Pn=dtRn}) and obtain
\begin{equation*}
    R_n=(q_{n,m-1}^+ \; e^{in\sigma} + \t{q}_{n,m-1}^{-} \; e^{-in\sigma})
    \frac{(x²+x_1²-i(u-s))^{n/2}}{(x²+x_1²+i(u-s))^{n/2+m}}
\end{equation*}
with $q_{n,m-1}^+$ given by
\begin{equation*}
    \sum_{\l=0}^n
    \X(\l/n)\frac{\Gamma(\l+m)}{\Gamma(\l+1)}
    \frac{\Gamma(n-\l+m)}{\Gamma(n-\l+1)} \; e^{-i2\l\sigma}.
\end{equation*}
and $\t{q}_{n,m-1}^{-}$ similar to $q_{n,m-1}:=\ab{q_{n,m-1}^+}$. Put
\begin{equation*}
    \begin{split}
    R_n^+:=q_{n,m-1}^{+} \; e^{in\sigma} \; \frac{(x²+x_1²-iu)^{n/2}}{(x²+x_1²+iu)^{n/2+m}}
    \; e^{i(m+n)/s} \; e^{-i/s} \; e^{-i(m+n)/s}.
    \end{split}
\end{equation*}
Then
\begin{equation*}
    R_n^+ \; e^{i(m+n)/s}
    =q_{n,m-1}^{+} \; e^{i(n+m)\omega^+} \; e^{-im\sigma} \; \n{R+iu}^{-m}.
\end{equation*}
%\ffbox{For the spectral projection operators we get the following formula
%\begin{equation}
%  \begin{split}
%  P_{n,1}&=C_m \; [Q_n-Q_{n-2}]=\tfrac{1}{n!}\d{r}^n|_{r=0} \; \phi
%  =\tfrac{1}{n!}\d{r}^n|_{r=0} \; (\d{t}\Phi)\\
%  &=\tfrac{1}{n!} \d{t} \d{r}^n|_{r=0} \; \Phi
%  =C_m \; \d{t} R_n
%  \end{split}
%\end{equation}}
Thus we obtain a second formula for $K_{k,j}$. $K_{k,j}$ is given by
\begin{equation}
    \label{Kkj,2^k-j<1}
    \begin{split}
    K_{k,j}(x',0,x,u)&=\sum_{\epsilon \in \{+,-\}} C \; 2^{k/2} \; 2^{(m-3)(j-k)}\\
    &\times \int \frac{\zeta_{k,j,m-1}^\epsilon (x',x,u,s)}{(R^2+(u-s)^2)^{m/2}}
    \; e^{-\epsilon im\sigma} \; \X(2^{k-j}s)\; ds,
    \end{split}
\end{equation}
with
\[\zeta_{j,m-1}^\pm(x,x_1,u,s):=\sum_n \chi_{j}(m+n) \; q^\pm_{n,m-1} \; 2^{mj} \; e^{i (n+m)\omega^\pm}\]
and some constant $C$.\absatz

We now distinguish the cases when $j-k$ is small and when $k-j$ is small. We define the
constant $\const_2$ by
\begin{equation}
    \label{konstante2}\index{const2}
    \const_2:=\frac{1}{16^2} \frac{1}{(\const_0+\const_1)^2}.
\end{equation}
Furthermore, recall that we have to integrate $K_{k,j}$ over the set
\begin{equation}
    \label{otdef}
    \t{O}(x'):=\menge{(x,u)}{\n{x-x'}\leq 2 \const_0,\ \n{u}\leq 2\const_0(2+\n{x'})}
\end{equation}
and that $\n{x'}\leq 2 \const_1$.

\subsubsection*{$K_{k,j}$ for the case $2^{j-k}\leq 1/\const_2$}

Let $\rho \in \Cnu$ such that $\rho(x)=1$ for $\n{x}\leq 2^{12}$ and $\rho(x)=0$ for
$\n{x}\geq 2^{13}$. We split $K_{k,j}$ into
\begin{equation*}
    \begin{split}
    K&_{k,j}(x',0,x,u):=2^{-m k} \sum_n \X_j(m+n) \int P_n(x',0,x,(u-s)/2) \;
    \Phi_{k,j,n}(s) \; ds\\
    =&2^{-m k} \sum_n \X_j(m+n) \int P_n(x',0,x,(u-s)/2) \;
    \Phi_{k,j,n}(s) \; (1-\rho)(2^{4(k-j)}w(s)^2) \; ds\\
    &+2^{-m k} \sum_n \X_j(m+n) \int P_n(x',0,x,(u-s)/2) \;
    \Phi_{k,j,n}(s) \; \rho(2^{4(k-j)}w(s)^2)\; ds.
    \end{split}
\end{equation*}
%\ffbox{Achtung! Hier ist ein Unterschied zur Heisenbergsituation. Es wird betrachtet
%$2^{4(k-j)}w(s)$ statt $2^{4(k-j)}a(s)$ auf der Heisenberggruppe. Warum?} Using
%$P_{n}=C_m \; \d{t} R_n$, $\Phi_{k,j,n}'(s)=2^\l \t{\Phi}_{k,j,n}(s)$ and performing an
We use integration by parts in the second term. We find that $K_{k,j}$ is a sum of terms
of the following types \alpheqn
\begin{equation}
    \label{intparts1}
    \begin{split}
    2^{-m k} \sum_n \X_j(m+n) \int &P_n(x',0,,x,(u-s)/2) \;
    \Phi_{k,j,n}(s) \\ &\times(1-\rho)(2^{4(k-j)}w(s)^2)\; ds,
    \end{split}
\end{equation}
\begin{equation}
    \label{intparts2}
    \begin{split}
    2^{-m k+2k-j} \sum_n \X_j(m+n) \int &R_n(x',0,x,(u-s)/2) \;
    \t{\Phi}_{k,j,n}(s)\\ &\times\rho(2^{4(k-j)}w(s)^2)\; ds
    \end{split}
\end{equation}
and
\begin{equation}
    \label{intparts3}
    \begin{split}
    2^{-m k} \; 2^{4(k-j)} \; \sum_n \X_j(m+n) \int &R_n(x',0,x,(u-s)/2) \;
    \Phi_{k,j,n}(s) \; (u-s) \\ &\times\rho'(2^{4(k-j)}w(s)^2) \; ds.
    \end{split}
\end{equation}
\reseteqn We have seen before that it suffices to consider only $+$-terms, since we get
the $L_1$-estimates for the $-$-terms by exchanging $u$ with $-u$ and $s$ with $-s$. To
make our notation simpler, we put $\omega:=\omega^+$ and $\zeta_{j,m}:=\zeta_{j,m}^+$.
Observe that
\begin{equation*}
  \begin{split}
  \zeta_{j,m} \; e^{\pm i/s}&=\sum_{n=0}^\infty \X_j(m+n) \; q^+_{n,m}(\sigma) \;
  2^{-mj} \; e^{i(n+m+1\pm 1)\omega} \\ &\times \frac{R2xx'+(u-s)w \mp i(Rw-2xx'(u-s))}{R²+(t-s)²}\\
  &=\t{\zeta}_j \; \frac{R2xx'+(u-s)w \mp i(Rw-2xx'(u-s))}{R²+(t-s)²}
  \end{split}
\end{equation*}
with $\t{\zeta}_{j,m}^\pm:=\zeta_{j,m} \; e^{\pm i\omega}$ of the same type as $\zeta_j$.
Since $2^{-j}\leq 1$ we can write (\ref{intparts1}), only considering the $+$-terms, as a
finite sum of terms of the following types (compare \eqref{Kkj,2^j-k<1})
\begin{equation*}
    \begin{split}
    2^{k/2}& \; 2^{(m-1)(j-k)} \; \int
    \frac{\zeta_{j,m}(x',x,u,s)}{(R^2+(u-s)^2)^{(m+1)/2}}\\
    &\times \frac{R2xx'+(u-s)w + \epsilon i(Rw-2xx'(u-s))}{R²+(u-s)²} \; e^{-i(m+1)\sigma} \\
    &\times(1-\rho)(2^{4(k-j)}w(s)^2) \; \X(2^{k-j}s) \; ds
    \end{split}
\end{equation*}
with $\epsilon=\pm 1$.
%\begin{equation*}
%    \begin{split}
%    2^{k/2}& \; 2^{(m-1)(j-k)} \; \int
%    \frac{\zeta_{j,m}(x',x,u,s)}{(R^2+(u-s)^2)^{(m+1)/2}}\\
%   &\times \frac{R2xx_1+(u-s)w + i(Rw-2xx_1(u-s))}{R²+(t-s)²} \; e^{-i(m+1)\sigma} \\
%   &\times(1-\rho)(2^{4(k-j)}w(s)^2) \; \X(2^{k-j}s) \; ds
%   \end{split}
%\end{equation*}
We can restrict to the case $\epsilon=1$, since the second case is similar. For
(\ref{intparts2}), only considering the $+$-terms, we get (compare \eqref{Kkj,2^k-j<1})
\begin{equation*}
    \begin{split}
    2^{-m k+2k-j} \sum_n \X_j(m+n) \int &R_n^+(x',0,x,(u-s)/2) \;
    \t{\Phi}_{k,j,n}(s) \; \rho(2^{4(k-j)}w(s)^2)\; ds\\
    =2^{k/2} \; 2^{(m-3)(j-k)} \; \int &\frac{\zeta_{j,m-1}(x',x,u,s)}{(R^2+(u-s)^2)^{m/2}}
    \; e^{-im\sigma} \\ &\rho(2^{4(k-j)}w(s)^2) \; \X(2^{k-j}s) \; ds.
    \end{split}
\end{equation*}
%with
%\[\zeta_{j,m-1}(x,x_1,u,s):=\sum_n \chi_{j}(m+n) \; b_{n,m-1} \; 2^{mj} \; e^{i (n+m)\omega}.\]
%\ffbox{$-mk+2k-j+(3k/2-j)$(Def. von $\Phi_{k,j}$) $-mj$ (Def. von $\zeta$)\\
%$=k/2+(m-3)(j-k)$}
And for the last term (\ref{intparts3}), only considering the
$+$-terms, we get (compare \eqref{Kkj,2^k-j<1})
\begin{equation*}
    \begin{split}
    2^{-m k+4(k-j)} &\sum_n \X_j(m+n)\\
     &\times \int R_n^+(x',0,x,(u-s)/2) \;
    \t{\Phi}_{k,j,n}(s) \; (u-s) \; \rho'(2^{4(k-j)}w(s)^2)\; ds\\
    =&2^{k/2} \; 2^{(m-5)(j-k)-j} \; \int \frac{\zeta_{j,m-1}(x',x,u,s)}{(R^2+(u-s)^2)^{m/2}}
    \; (u-s) \; e^{-im\sigma} \\ &\rho'(2^{4(k-j)}w(s)^2) \;\X(2^{k-j}s)\; ds.
    \end{split}
\end{equation*}
%\ffbox{$-mk+4(k-j)+(3k/2-j)$(Def. von $\Phi_{k,j}$) $-mj$ (Def. von $\zeta$)\\
%$=k/2+(m-5)(j-k)-j$} Replacing now $x$ by $2^{(j-k)/2}x$, $x'$ by $2^{(j-k)/2}x_1$, $u$
%by $2^{j-k}u$ and $s$ by $2^{j-k}s$ we define \ffbox{das heißt nun $\t{x},\t{x}_1
%\lesssim 2^{(k-j)/2}$, $\t{x}_1=x_1\;2^{(k-j)/2}$ und $x_1 \lesssim 1$, da für
%$x_1\gtrsim 1$ das Ergebnis von Sogge benutzt werden kann.}

We now replace $x$ by $2^{(j-k)/2}x$, $u$ by $2^{j-k}u$ and $s$ by $2^{j-k}s$.
Furthermore, we set $x_1:=2^{(k-j)/2}x'$ and define
\[\zeta_{k,j,m}(x_1,x,u,s):=\sum_n \chi_{j}(m+n) \; q^+_{n,m} \; 2^{mj} \;
e^{i (n+m+1)\kphi(s)}\] and
\[\zeta_{k,j,m-1}(x_1,x,u,s):=\sum_n \chi_{j}(m+n) \; q^+_{n,m-1} \; 2^{mj} \; e^{i (n+m)\kphi(s)},\]
with
\begin{equation}
    \label{kphidef}
    \kphi(s):=\arctan\Big(\frac{Rw-2xx_1(u-s)}{R2xx_1+(u-s)w}\Big)+\frac{2^{k-j}}{s}.\index{kphidefinition}
\end{equation}
Hence we have to study the following terms.
%%%%%%%%%%%%%%%%%%%%%%%%%%%%%%%%%%      Definition von F_{k,j}     %%%%%%%%%%%%%%%%%%%%%%%%%%%%%%%%%%%%%%
\alpheqn
\begin{equation}
    \label{fkj}
    \begin{split}
    F_{k,j}(x_1,x,u)&:=2^{k/2} \; 2^{m(j-k)}
    \; \int \frac{\zeta_{k,j,m}(x_1,x,u,s)}{(R^2+(u-s)^2)^{(m+3)/2}} \;
    e^{-i(m+1)\sigma}\\
    &\times (R2xx_1+(u-s)w - i(Rw-2xx_1(u-s)))\\
    &\times (1-\rho)(2^{2(k-j)}w(s)^2) \; \X(s) \; ds
    \end{split}
\end{equation}
for \eqref{intparts1}). % \ffbox{von der Transformation kommt wegen $\t{x}=2^{k-j}x,\;
%d\t{x}/dx=2^{k-j}$ ein $2^{(j-k)/2}$ für $x$ und jeweils $2^{j-k}$ für $u$ und $s$. Wegen
%$(R^2+(u-s)^2)^{-(m+3)/2}$ erhält man $2^{7(k-j)/2}$ und wegen $(R2xx_1+(u-s)w -
%i(Rw-2xx_1(u-s)))$ nochmal $2^{2(j-k)}$ also insgesamt bekommt man $2^{j-k}$.}

%%%%%%%%%%%%%%%%%%%%%%%%%%%%%%%%%%      Definition von G_{k,j}     %%%%%%%%%%%%%%%%%%%%%%%%%%%%%%%%%%%%%%
\begin{equation}
    \label{gkj}
    \begin{split}
    G_{k,j}&(x_1,x,u):=2^{k/2} \; 2^{(m-1)(j-k)}\\
    &\times \int \frac{\zeta_{k,j,m-1}(x_1,x,u,s)}{(R^2+(u-s)^2)^{m/2}}
    \; e^{-im\sigma} \; \rho(2^{2(k-j)}w(s)^2) \; \X(s) \; ds
    \end{split}
\end{equation}
for \eqref{intparts2}. And at last
%\ffbox{von der Transformation kommt wegen $\t{x}=2^{k-j}x,\; d\t{x}/dx=2^{k-j}$ ein
%$2^{(j-k)/2}$ für $x$ und jeweils $2^{j-k}$ für $u$ und $s$. Wegen $(R^2+(u-s)^2)^{-m/2}$
%erhält man nochmal $2^{(k-j)/2}$ also insgesamt bekommt man $2^{2(j-k)}$.}

%%%%%%%%%%%%%%%%%%%%%%%%%%%%%%%%%%      Definition von H_{k,j}     %%%%%%%%%%%%%%%%%%%%%%%%%%%%%%%%%%%%%%
\begin{equation}
    \label{hkj}
    \begin{split}
    H_{k,j}&(x_1,x,u):=2^{k/2} \; 2^{(m-2)(j-k)-j}\\
    &\times \int \frac{\zeta_{k,j,m-1}(x_1,x,u,s)}{(R^2+(u-s)^2)^{m/2}}
    \; (u-s) \; e^{-im\sigma} \;  \rho'(2^{2(k-j)}w(s)^2) \;\X(s)\; ds
    \end{split}
\end{equation}
for \eqref{intparts3}.
%\ffbox{von der Transf. kommt $2^{(j-k)/2+2(j-k)}$ Wegen
%$(R^2+(u-s)^2)^{-m/2}$ erhält man $2^{(k-j)/2}$ und von $(u-s)$ nochmal $2^{j-k}$ also
%insgesamt bekommt man $2^{3(j-k)}$.} We use here the notation $x_1$ instead of $x'$,
%since $\n{x_1}=2^{(k-j)/2}\n{x'}\lesssim 2^{(k-j)/2}$ and hence depends on $k,j$ and is
%no longer bounded by uniform constant.
\reseteqn

\subsubsection*{$K_{k,j}$ for the case $2^{j-k}\geq \const_2$}

In this case we only use formula \eqref{Kkj,2^k-j<1}. Thus we have to study
\begin{equation}
    \label{gkjb}
    \begin{split}
    G_{k,j}(x,x_1,u):=&2^{k/2} \; 2^{(m-1)(j-k)}
    \; \int \frac{\zeta_{k,j,m-1}(x_1,x,u,s)}{(R^2+(u-s)^2)^{m/2}}
    \; e^{-im\sigma} \; \X(s) \; ds.
    \end{split}
\end{equation}

%
%
%
%%%%%%%%%%%%%%%%%%%%%%%%%    Berechnung der Ableitungen von \phi   %%%%%%%%%%%%%%%%%%%%%%%%%%
%
%
%
\subsection{Calculations for the phase function $\kphi$}

%\ffbox{\[    \arctan(x)'=1/(1+x²)\]}
In this section, we study the phase function $\kphi$ in $\zeta_{k,j,m}$ resp.
$\zeta_{k,j,m-1}$. This function was defined in \eqref{kphidef}. In the next chapter, we
estimate $F_{k,j},\ G_{k,j}$ and $H_{k,j}$ by partial integration, hence we have to study
the first and second derivative of $\kphi$.

We use the abbreviations
\begin{equation*}
\begin{split}
  R&:=x²+x_1²,\quad \index{Rs}
%$, $z:=R+iu$, $p:=4xx_1$,
  w:=\sqrt{R²-4(xx_1)²+(u-s)²},\\ \index{ws}
  a&:=\sqrt{R²+(u-s)²}. \index{as}
\end{split}
\end{equation*}
Let
\begin{equation*}
  \begin{split}
    \kphi_0(s)&:=\arctan\left(\frac{x²+x_1²}{u-s}\right)+\frac{2^{k-j}}{s},\\
    %\kphi_1(s)&:=\arctan\left(\frac{2xx_1}{\sqrt{x^4+x_1^4-2(xx_1)^2+(t-s)²}}\right),\\
    \kphi(s)&:=\arctan\left(\frac{Rw-2xx_1(u-s)}{2Rxx_1+(u-s)w}\right)
    +\frac{2^{k-j}}{s}.
  \end{split}
\end{equation*}
Some easy computations yield
\begin{equation*}
  \begin{split}
  \d{s}w(s)&=\frac{1}{w}(u-s)=\frac{u-s}{((x²-x_1²)+(u-s)²)^{1/2}},\\
  \d{s}a(s)&=\frac{1}{a}(u-s)=\frac{u-s}{((x²+x_1²)+(u-s)²)^{1/2}},\\
  \d{s}\kphi_0(s)&=\frac{R}{R²+(u-s)²}-\frac{2^{k-j}}{s²},\\
  %\d{s}\kphi_1(s)&=\frac{2xx_1(t-s)}{(R²+(t-s)²)w},\\
  \d{s}\kphi(s)&=\frac{Rw-2xx_1(u-s)}{(R²+(u-s)²)w}-\frac{2^{k-j}}{s²},\\
  \d{s}²\kphi_0(s)&=\frac{R}{(R²+(t-s)²)²}2(t-s)+2\frac{2^{k-j}}{s³},\\
  \d{s}²\kphi(s)&=\frac{Rw^{-1}(u-s)+2xx_1}{a^2w}\\
    &-\frac{Rw-2xx_1(u-s)}{a^4w²}\big((2(u-s)w+a^2w^{-1}(u-s)\big)+
    2\frac{2^{k-j}}{s^3}.
  \end{split}
\end{equation*}
\begin{bemerkung}
For $x_1=0$, the phase function is the same than in the Heisenberg case. We have
\begin{equation*}
  \begin{split}
  \kphi(s)&=\arctan\left(\frac{R}{t-s}\right)+\frac{2^{k-j}}{s}\\
  \d{s}\kphi(s)&=\frac{R}{R²+(t-s)²}-\frac{2^{k-j}}{s²}.
  \end{split}
\end{equation*}
\end{bemerkung}

\newpage
\begin{lemma}
  \label{ll}
  If $s \sim 1$ then
  \begin{equation*}
    \begin{split}
    \norm{\d{s} \kphi(s)}&\leq c2^{k-j}+2a^{-1}\\
    \norm{\d{s}²\kphi(s)}&\leq c2^{k-j}+6(aw)^{-1}\lesssim (2^{(k-j)/2}+(aw)^{-1/2})^2
    \end{split}
  \end{equation*}
  and hence
  \begin{equation*}
    \bignorm{2^{-j}\frac{\d{s}²\kphi}{(\d{s}\kphi)²}(s)}\lesssim 2^{-j}
    \; \left(\frac{2^{(k-j)/2}+(aw)^{-1/2}}{\d{s}\kphi}\right)^2.
  \end{equation*}
\end{lemma}
\begin{proof}
  Follows easily since $\norm{u-s}\leq w$ and $w \leq a$ and thus
    \[\norm{\d{s}²\kphi(s)}\leq 2a^{-1}w^{-1}+2a^{-3}w+2a^{-1}w^{-1}+c2^{k-j}\leq 6 (aw)^{-1}+c2^{k-j}.\]
\end{proof}

%
%
%
%%%%%%%%%%%%%%%%%%%%%%%%%%%%%%%%%%%%%%%%%%%%%%%%%%%%%%%%%%%%%%%%%%%%%%%%%%%%%%%%%%%%%%%%%%%%%
%%%%%%%%%%%%%%%%%%%%%%%%%%%%    Koordinatentransformation    %%%%%%%%%%%%%%%%%%%%%%%%%%%%%%%%
%
%
%
\subsection{The change of coordinates}

Let
\begin{equation*}
  \t{\kphi}(s):=\arctan\left(\frac{Rw-2xx_1(u-s)}{2Rxx_1+(u-s)w}\right).
\end{equation*}
Then
\begin{equation*}
  \begin{split}
  \d{s}\t{\kphi}(s)&=\frac{Rw-2xx_1(u-s)}{(R²+(u-s)²)w}.
  \end{split}
\end{equation*}
The following change of coordinates turns out to be useful
\begin{equation*}
  \begin{split}
  &X:=\frac{Rw-2xx_1(u-s)}{2Rxx_1+(u-s)w},\quad Y:=\frac{(R²+(u-s)²)w}{2Rxx_1+(u-s)w},\\
  &s:=s,\quad \psi(x,u,s):=(X,Y,s).
  \end{split}
\end{equation*}\index{X} \index{Y}
Observe that $\sgn X=\sgn Y$. Some easy computations show
\begin{equation}
    \label{coord1}
  \begin{split}
  &\frac{a²}{\norm{2Rxx_1+(u-s)w}}=\j{X},\quad w=\frac{\norm{Y}}{\j{X}},\\
  &\frac{Rw-2xx_1(u-s)}{(R²+(u-s)²)w}=\frac{X}{Y}=\d{s}\t{\kphi}(s).
  \end{split}
\end{equation}

\noindent One can compute the inverse of this transformation by solving a third order
equation in $x^2$ and a second order equation in $u$. But unfortunately we have no simple
solutions for these equations.

Nevertheless, we have a simple expression for the functional determinant in the new
coordinates.
\begin{equation}
    \label{funktionaldeterminante}
  \begin{split}
  \norm{\det(\psi')}^{-1}&=\frac{w\norm{2xx_1R+(u-s)w}³}
    {a^4\norm{2x^5-2xx_1^4+2x(u-s)²-2x_1(u-s)w}}\\
    &=\frac{\norm{Y}}{2\j{X}^3} \frac{1}{\sqrt{\norm{XY-x_1²X²}}}.
  \end{split}
\end{equation}
This calculation will be done in the following two lemmata.
\begin{lemma}
  The following equation holds
  \begin{equation*}
    \begin{split}
    \norm{\det(\psi')}^{-1}&=\frac{w\norm{2xx_1R+(u-s)w}³}{a^4}
    \frac{1}{\norm{2x^5-2xx_1^4+2x(u-s)²-2x_1(u-s)w}}\\
    &=\frac{\norm{Y}}{\j{X}^4} \frac{a²}{\norm{2x^5-2xx_1^4+2x(u-s)²-2x_1(u-s)w}}.
    \end{split}
  \end{equation*}
\end{lemma}

\begin{proof}
  The proof is elementary but complex.
  For these computations an algebra system like Maple is very useful.
\end{proof}

\begin{lemma}
  The following holds
  \begin{equation*}
    XY-x_1²X²=\frac{(2x^5-2xx_1^4+2x(u-s)²-2x_1(u-s)w)²}{4(2xx_1R+(u-s)w)²}.
  \end{equation*}
\end{lemma}

\begin{bemerkung}
  For $x_1=0$ we get $XY=\frac{R(R^2+(u-s)^2)}{(u-s)^2}$.
%  which, of course, is the same functional determinant as in \cite{mueller} except for an additional $R$ which
%  reflects the different dimension of the manifold.
\end{bemerkung}

\begin{proof}
  \begin{equation*}
    \begin{split}
    XY-x_1²X²&=\frac{(Rw-2xx_1(u-s))((u-s)²+R²)w}{(2xx_1R+(u-s)w)²}-x_1²
    \frac{(Rw-2xx_1(u-s))²}{(2xx_1R+(u-s)w)²}\\
    &=\frac{(2x^5-2xx_1^4+2x(u-s)²-2x_1(u-s)w)²}{4(2xx_1R+(u-s)w)²}
    \end{split}
  \end{equation*}
  As before, an algebra system is very useful for these calculations.
\end{proof}
\noindent Thus we get
\begin{equation*}
  \begin{split}
    \norm{\det(\psi')}^{-1}&=\frac{w\norm{2xx_1R+(u-s)w}³}{a^4}
    \frac{1}{\norm{2x^5-2xx_1^4+2x(u-s)²-2x_1(u-s)w}}\\
    &=\frac{w\norm{2xx_1R+(u-s)w}^2}{2a^4}
    \frac{1}{\sqrt{\norm{XY-x_1²X²}}}\\
    &=\frac{w}{2\j{X}²} \frac{1}{\sqrt{\norm{XY-x_1²X²}}}
    =\frac{\norm{Y}}{2\j{X}³} \frac{1}{\sqrt{\norm{XY-x_1²X²}}}
  \end{split}
\end{equation*}
and we have proven equation (\ref{funktionaldeterminante}).

\begin{bemerkung}
  For $x_1=0$ we use the same coordinates as in \cite{mueller}, i.e.
  \[
    X=\frac{R}{u-s},\quad Y=\frac{R²+(u-s)²}{u-s},\quad s=s.
  \]
  The functional determinant
  \[
    \norm{\det(\psi')}^{-1}=\frac{\norm{Y}}{2\j{X}^4}\frac{\j{X}}{\sqrt{\norm{XY}}},\]
    is, of course, also the same, except for the extra term
    $\j{X}/\sqrt{\norm{XY}}=R^{-1/2}$, which reflects the fact that our manifold $\R^2$ has one dimension less
    than the Heisenberg group $\H_1$.
\end{bemerkung}

In Section \ref{carnot} we mentioned that the case $x_1\not=0$ is much harder to
understand than the case $x_1=0$. Here is one reason for this fact. Our coordinate
transform we have to use is much more complicated for $x_1\not=0$ as for $x_1=0$.\absatz

As far as we know, there is no easier transform that allows us to estimate the $L_1$-norm
of the kernel $K_{k,j}$ properly.\absatz

To see that we really can use this transformation, we have to take a closer look at the
inverse of it. Though we can not hope to get an explicit and useful formula for
$\psi^{-1}$, we are able to show that $\psi$ can be restricted to a finite number of sets
$\Omega_n$ such that $\psi|_{\Omega_n}$ is invertible. Hence we are allowed to use
the transformation formula.\\

Observe that by \eqref{coord1} we get the following equations for $\t{u}=u-s$
\begin{equation}
    \label{coord2}
    \begin{split}
    \frac{\n{X}}{\j{X}}&=\frac{Rw-2xx_1\t{u}}{R^2+\t{u}^2}\\
    \t{u}^2&=\frac{Y^2}{\j{X}^2}-R^2+4x^2x_1^2
    \end{split}
\end{equation}
From the first equation we deduce
\begin{equation*}
    \t{u}=\pm\sqrt{\frac{R\n{Y}}{\n{X}}-R^2+x^2x_1^2\frac{\j{X}^2}{X^2}}-xx_1\frac{\j{X}}{\n{X}}.
\end{equation*}
This together with the second equation from \eqref{coord2} gives us for
$\epsilon\in\{-1,1\}$
\begin{equation*}
    \frac{Y^2}{\j{X}^2}-R^2+4x^2x_1^2=\Big(\epsilon\sqrt{\frac{R\n{Y}}{\n{X}}-R^2+x^2x_1^2\frac{\j{X}^2}{X^2}}-xx_1\frac{\j{X}}{\n{X}}\Big)^2,
\end{equation*}
which implies
\begin{equation*}
   \Big(\frac{Y^2}{\j{X}^2}+4x^2x_1^2-2x^2x_1^2\frac{\j{X}^2}{X^2}
   -R\frac{Y}{X}\Big)^2-4\Big(R\frac{Y}{X}-R^2+x^2x_1^2\frac{\j{X}^2}{X^2}\Big)x^2x_1^2\frac{\j{X}^2}{X^2}=0.
\end{equation*}
But this an equation of third order (due to $R^2x^2x_1^2=(x^2+x_1^2)^2x^2x_1^2$) in $x^2$
and hence can be solved.

Now, for a given $(x,\t{u})$, one can find a locally defined function $\psi_{-1}$ with
\[\psi_{-1}(X,Y)=\psi_{-1}\Big(\frac{Rw-2xx_1\t{u}}{2Rxx_1+\t{u}w},\frac{(R²+\t{u}²)w}{2Rxx_1+\t{u}w}\Big)=x^2\]
and hence $x=\sgn(x)\sqrt{\psi_{-1}(X,Y)}$. In addition,
\[Y^2/\j{X}^2-(\psi_{-1}(X,Y)+x_1^2)^2+4\psi_{-1}(X,Y)x_1^2= \t{u}^2\] and
$\t{u}=\sgn(u)\sqrt{Y^2/\j{X}^2-(\psi_{-1}(X,Y)+x_1^2)^2+4\psi_{-1}(X,Y)x_1^2}$.

Furthermore, there is a finite number $N$, a measurable set $\Omega_0$ of measure $0$ and
open sets $\Omega_1, \dots, \Omega_N$ such that $\bigcup_{n=0}^N \Omega_n=\R^2$ and
$\psi|_{\Omega_n}$ is a diffeomorphism for every $n\in \{1, \dots , N\}$. On all these
sets we may apply the transformation formula. We obtain the following proposition.

\begin{satz}
    There exists a constant $C$ such that for every measurable \mbox{$f:\R^3\to\C$} with
    $(X,Y,s)\mapsto f(X,Y,s) \norm{Y}\; \j{X}^{-3} \; \norm{XY-x_1²X²}^{-1/2} \in L_1(\R^3)$
    \begin{equation*}
        \begin{split}
        \int \n{f(X(x,&u,s),Y(x,u,s),s)}\; dx \; du \; ds\leq C \int
        \n{f(X,Y,s)}\; \norm{\det(\psi')}^{-1} \; dX \; dY \; ds\\
        &=C \int \n{f(X,Y,s)}\; \frac{\norm{Y}}{2\j{X}³} \frac{1}{\sqrt{\norm{XY-x_1²X²}}} \; dX \; dY \; ds
        \end{split}
    \end{equation*}
    holds.
\end{satz}

\newpage
\Section{The proof of Proposition \ref{dyadicprop2}}

Define
\begin{equation}
    \label{odef}
    O(x_1):=\menge{(x,u)}{\n{x-x_1}\leq 2^{(k-j)/2}2 \const_0,\ \n{u}\leq 2^{k-j}
    2\const_0(2+2\const_1)}.
\end{equation}
\index{o(x_1)}
 To prove Proposition
\ref{dyadicprop2}, we now show that the following two estimates hold (compare
\eqref{otdef}).\absatz

\noindent{\it (1) For every $\epsilon>0$ and for all $A \in \{F_{k,j},\ G_{k,j},\
H_{k,j}\}$, defined by \eqref{fkj}, \eqref{gkj} and \eqref{hkj}, there exists a constant
$C_\epsilon$ such that the estimate
\begin{equation}
    \label{proof1}
    \sum_{j\in \N_0,\ 2^{j-k}\leq 1/\const_2} \sup_{\n{x_1}\leq 2\const_1 2^{(k-j)/2}} \dnorm{A}_{L_1(O(x_1))}\leq C_\epsilon \; 2^{k\epsilon}
\end{equation}
holds.}\absatz\absatz

\noindent{\it (2) For every $\epsilon>0$ and for $G_{k,j}$, defined by \eqref{gkjb},
there exists a constant $C_\epsilon$ such that the estimate
\begin{equation}
    \label{proof2}
    \sum_{j\in \N_0,\ 2^{k-j}\leq \const_2} \sup_{\n{x_1}\leq 2\const_1 2^{(k-j)/2}} \dnorm{G_{k,j}}_{L_1(O(x_1))}\leq C_\epsilon \; 2^{k\epsilon}
\end{equation}
holds.}\absatz

Unfortunately, we cannot express $O(x_1)$ in the new coordinates. We only have to
integrate over such $X$ and $Y$ so that for given $s$ and $x_1$ the corresponding $x$ and
$u$, with $X=X(x_1,x,u,s)$ and $Y=Y(x_1,x,u,s)$, are in $O(x_1)$. For this we write
\[\int_{O(x_1)} \dots \; dX \; dY.\]
%
%
%
%%%%%%%%%%%%%%%%%%%%%%%%%%%%%%%%%%%%%%%%%%%%%%%%%%%%%%%%%%%%%%%%%%%%%%%%%%%%%%%%%%%%%%%%%%%%%
%%%%%%%%%    Integralberechnungen in neuen Koordinaten für spätere Abschätzungen    %%%%%%%%%
%
%
%
\subsection{Some integrations in the new coordinates}

Before we prove the assertions \eqref{proof1} and \eqref{proof2}, we show a few technical
lemmata concerning integration with respect to new coordinates $X$ and $Y$.
\begin{lemma}
  \label{lemmayint1}
  The following estimates hold true

  \noindent(a)
  \begin{equation*}
    \int\limits_{c_0 \leq \norm{Y}\leq c_1}
    \frac{\sqrt{\norm{Y}}}{\sqrt{\norm{Y-x_1²X}}} \; dY\lesssim c_1,
  \end{equation*}
  \noindent(b)
  \begin{equation*}
    \begin{split}
    \int\limits_{c_0 \leq \norm{Y}\leq c_1}
    \frac{1}{\sqrt{\norm{Y}}}\frac{1}{\sqrt{\norm{Y-x_1²X}}} \; dY
    \lesssim 1 + \log\Big(\frac{c_1}{c_0}\Big).
    \end{split}
  \end{equation*}
\end{lemma}

\begin{proof}
  This estimates follows easily by the transformation $Y\to Y/(x_1²\n{X})$. In fact
  \begin{equation*}
    \begin{split}
    &\int\limits_{c_0 \leq \norm{Y}\leq c_1}
    \frac{\sqrt{\norm{Y}}}{\sqrt{\norm{Y-x_1²X}}} \; dY
    =x_1²\n{X}\int\limits_{c_0/(x_1²\n{X}) \leq \norm{Y}\leq c_1/(x_1²\n{X})}
    \frac{\sqrt{\norm{Y}}}{\sqrt{\norm{Y-1}}} \; dY.
    \lesssim c_1.
    \end{split}
  \end{equation*}
  For $x_1^2 \n{X}\lesssim 1$ we get for $\n{Y}\leq 2$ the bound $x_1^2 \n{X} \; c\lesssim c_1$ and for
  $\n{Y}\geq 2$ we get $x_1^2 \n{X} \; c_1/(x_1^2\n{X})=c_1$.
  For $x_1^2 \n{X}\gtrsim 1$ we get for $\n{Y}\leq 2$ the bound
  $c_1 \; \int_{\n{Y}\leq 2} (\sqrt{\n{Y}}\sqrt{\n{Y-1}})^{-1} \; dY\lesssim c_1$.
  and for $\n{Y}\geq 2$ we get $x_1^2\n{X} \; c_1/(x_1^2\n{X})=c_1$. This gives us
  estimate (a). For (b) observe that
  \begin{equation*}
    \begin{split}
    \int\limits_{c_0 \leq \norm{Y}\leq c_1}
    &\frac{1}{\sqrt{\norm{Y}}}\frac{1}{\sqrt{\norm{Y-x_1²X}}} \; dY
    =\int\limits_{\frac{c_0}{x_1^2\n{X}} \leq \norm{Y}\leq \frac{c_1}{x_1²\n{X}}}
    \frac{1}{\sqrt{\n{Y}}}\frac{1}{\sqrt{\norm{Y-1}}} \; dY\\
    \lesssim &\int\limits_{0 \leq \norm{Y} \leq 1/2} \frac{1}{\sqrt{\n{Y}}} \; dY+
    \int\limits_{\frac{c_0}{x_1^2\n{X}} \leq \norm{Y}\leq \frac{c_1}{x_1²\n{X}},\ \n{Y}\geq 2} \frac{1}{\n{Y}} \;
    dY
    +\int\limits_{Y \sim 1} \frac{1}{\sqrt{\norm{Y-1}}} \; dY\\
    \lesssim &1+\log\Big(\frac{c_1}{c_0}\Big)
    \end{split}
  \end{equation*}
  holds.
\end{proof}

\begin{lemma}
    \label{lemmayint2}
  \begin{itemize}
    \item[(a)] For all $\epsilon \geq 0$
        \[\int\limits_{\n{Y-a_2}\leq c} \ed{\sqrt{\n{Y-a_1}}} \; dY
        \lesssim c^{1-\epsilon} \n{a_1-a_2}^{-1/2+\epsilon}\]
        holds with a constant not depending on $a_1$, $a_2$ and $c$.
    \item[(b)] For all $\epsilon > 0$
        \[\int\limits_{\n{Y-a_2}\geq c} \frac{c}{\n{Y-a_2}}
        \; \ed{\sqrt{\n{Y-a_1}}} \; dY
        \lesssim c^{1-\epsilon} \n{a_1-a_2}^{-1/2+\epsilon}\]
        holds with a constant not depending on $a_1$, $a_2$ and $c$.
  \end{itemize}
\end{lemma}
%\noindent By this lemma we get the following estimates.\absatz
%
%{\it \noindent For all $\epsilon \geq 0$
%\begin{equation}
%    \label{lemmayint2}
%  \int\limits_{\ueber{\n{Y-2^{j-k}s²X}}{\lesssim 2^{-j/2+j-k}\j{X}}} \ed{\sqrt{\n{Y-x_1^2X}}} \ dY
%  \leq C_{\epsilon}
%  \frac{(2^{-j/2+j-k}\j{X})^{1-\epsilon}}{\n{(x_1^2-2^{j-k}s²)X}^{1/2-\epsilon}}
%\end{equation}
%holds, with a constant $C_\epsilon$ not depending on $j$, $k$, $s$, $X$, and
%$x_1$.\absatz
%
%\noindent For all $\epsilon>0$
%\begin{equation}
%  \label{lemmayint3}
%  \int\limits_{\ueber{\n{Y-2^{j-k}s²X}}{\gtrsim 2^{-j/2+j-k}\j{X}}}
%  \frac{2^{-j/2+j-k}\j{X}}{\n{Y-2^{j-k}s²X}} \ed{\sqrt{\n{Y-x_1^2X}}} \ dY
%  \leq C_{\epsilon}
%  \frac{(2^{-j/2+j-k}\j{X})^{1-\epsilon}}{\n{(x_1^2-2^{j-k}s²)X}^{1/2-\epsilon}}
%\end{equation}
%holds, with a constant $C_\epsilon$ not depending on $j$, $k$, $s$, $X$, and $x_1$.}

\begin{proof}
  \begin{equation*}
    \begin{split}
    \int\limits_{\n{Y-a_2}\leq c} \ed{\sqrt{\n{Y-a_1}}} \; dY
    &=\int\limits_{\n{Y}\leq c \n{a_1-a_2}^{-1}} \sqrt{\n{a_1-a_2}} \; \ed{\sqrt{Y-1}} \; dY
    \end{split}
  \end{equation*}
  For $Y\ll 1$ the integral is bounded by $c^{1-\epsilon} \n{a_1-a_2}^{-1/2+\epsilon}$.

  \noindent For $Y\sim 1$ we have $1\lesssim c^{1-\epsilon} \n{a_1-a_2}^{-1+\epsilon}$.
  Since the integral converges we get $c^{1-\epsilon} \n{a_1-a_2}^{-1/2+\epsilon}$.

  \noindent For $Y\gg 1$ we have $\sqrt{\n{Y-1}}\geq \sqrt{\n{Y}}\geq \n{Y}^\epsilon$. The integration gives
  us the bound
  $c^{1-\epsilon} \; \n{a_1-a_2}^{-1+\epsilon} \; \sqrt{a_1-a_2}$. Thus (a) holds.
  \begin{equation*}
    \begin{split}
    \int\limits_{\n{Y-a_2}\geq c} \frac{c}{\n{Y-a_2}}
    \; \ed{\sqrt{\n{Y-a_1}}} \; dY
    &\lesssim \int \Big(\frac{c}{\n{Y}}\Big)^{1-\epsilon} \ed{\n{a_1-a_2}^{1/2-\epsilon}} \;
    \ed{\sqrt{Y-1}} \; dY
    \end{split}
  \end{equation*}
  For $Y \ll 1$  and $Y\sim 1$ we get by integration $c^{1-\epsilon}\n{a_1-a_2}^{-1/2+\epsilon}$.

  \noindent For $Y\gg 1$ we have
  $\sqrt{Y-1}\gtrsim \sqrt{Y}$ and the integral converges. We get $c^{1-\epsilon}\n{a_1-a_2}^{-1/2+\epsilon}$.
  Thus (b) holds.
  \end{proof}

\begin{lemma}
    \label{hkjlemma}
    Let $\epsilon>0$ and $C_1$ an arbitrary constant. We define
    \begin{equation*}
        \psi_{k,j}(X,s):=\frac{2^j}{(1+2^j\n{1-e^{i(\arctan(X)+2^{k-j}/s)}})^{1+\epsilon}}.
    \end{equation*}
    Then there exists a constant $C$ such that for all $j$ and $k$, with $2^{j-k}\leq
    C_1$,
    and all $X\in \R$ the estimate
    \[\bignorm{\int \psi_{k,j}(X,s) \; \X(s) \; ds}\leq C\]
    holds.
\end{lemma}

\begin{proof}
    Put $z:=\arctan(X) \in \intaa{0,\pi}$. Since $\supp \X \teil \intaa{1/8,32}$, we have
    \begin{equation*}
        \begin{split}
        \int \n{\psi_{k,j}(X,s) \; \X(s)} \; ds&\leq 2^{j-k} \int_{2^{k-j-4}}^{2^{k-j+3}}
        \frac{2^j}{(1+2^j\n{1-e^{i(z+s)}})^{1+\epsilon}}\; ds\\
        &\leq 2^{j-k} \; (1+2^{k-j}) \int \frac{2^j}{(1+2^j\n{s})^{1+\epsilon}}\;
        ds\lesssim 1
        \end{split}
    \end{equation*}
    (compare Lemma 3.1 in \cite{mueller}).
\end{proof}

\begin{lemma}
    \label{trivialestimate}
    Let $k,\ j \in \N_0$ with $2^{k-j} \leq \const_2$. Then
    \begin{equation*}
    \begin{split}
        I:=\int\limits_{O(x_1)} \psi_{k,j}(X,s) \;
        \frac{\norm{Y}^{1/2}}{\sqrt{\norm{XY-x_1²X²}}} \; \X(s) \; dX \; dY \; ds\lesssim 1
    \end{split}
    \end{equation*}
    holds.
\end{lemma}

\begin{proof}
    In this case we have $Y\sim 1,\ \j{X}\sim 1$ and $\n{X}\lesssim 2^{k-j}$.
    We consider the case $\n{Y}\leq 2x_1^2\n{X}$ first. Since $x_1^2\lesssim 2^{k-j}$, we get
    \begin{equation*}
        \begin{split}
        2^{(k-j)/2} &\int\limits_{Y \sim 1,\ \n{Y}\leq 2x_1^2\n{X}} \psi_{k,j}(X,s) \;
        \frac{1}{\sqrt{\n{X}}\sqrt{\norm{Y-x_1²X}}} \; \X(s) \; dX \; dY \; ds\\
        &\lesssim 2^{(k-j)/2} \int\limits_{\n{X} \leq 1} \psi_{k,j}(X,s) \;
        \frac{1}{\sqrt{\n{X}}} \; \X(s) \; dX \; ds.
        \end{split}
    \end{equation*}
    For $\n{X}\leq 2^{k-j}$, we do the $s$-integration first and obtain
    \begin{equation*}
        2^{(k-j)/2} \; 2^{j-k} \int_{\n{X}\leq 2^{k-j}} \frac{1}{\sqrt{\n{X}}} \;
        dX\lesssim 1.
    \end{equation*}
    For $\n{X}\geq 2^{k-j}$, we do the $X$-integration first and obtain
    \begin{equation*}
        \begin{split}
        2^{(k-j)/2} &\int\limits_{s\sim 1,\ \n{X} \geq 1} \psi_{k,j}(X,s) \;
        \frac{1}{\sqrt{\n{X}}} \; dX \; ds\\
        &\lesssim 2^{(k-j)/2} \; 2^{(j-k)/2} \; \int\limits_{s\sim 1,\ \n{X} \geq 1} \psi_{k,j}(X,s) \;
        dX \; ds
        \lesssim \int\limits_{s\sim 1} \; ds\lesssim 1.
        \end{split}
    \end{equation*}
    Now, the case $\n{Y}\geq 2x_1^2\n{X}$ is left. It is not difficult to see that for
    every $x'\leq 2\const_1$ and $s$ in the support of $\X$ there exists a set
    $\Omega(x',s)$ such that
    \begin{equation}
        Y(x_1,x,u,s)\in \Omega(x',s),
    \end{equation}
    for $x_1=2^{(k-j)/2}x'$ and all $x,u \in O(x_1)$. In addition, there exists a constant
    $C$ such that for all $x'$ and $s$ the estimate
    \[\n{\Omega(x',s)}\lesssim 2^{k-j}\]
    holds.
    To prove this we write
    \begin{equation}
        \frac{Y}{\j{X}^2}=\frac{(R^2+(u-s)^2)w}{a^4}(2Rxx_1+(u-s)w).
    \end{equation}
    Now, we have that $\j{X},\ (u-s)$ and $w$ are $1+O(2^{k-j})$. Furthermore, $u,R,xx_1$ are
    bounded by $2^{k-j}$ and $s$ is similar to $1$.
    Since $2^{k-j}\leq \const_2$ is small enough we get the result. We do not want to go
    into the details here. For $x'=0$ this is evident, since in this case we have
    \begin{equation}
        \frac{Y}{\j{X}^2}=u-s
    \end{equation}
    and $\n{u}\lesssim 2^{k-j}$.
    With this result we get
    \begin{equation}
        \begin{split}
        \int\limits_{O(x_1),\ \n{Y}\geq 2x_1^2\n{X}} &\psi_{k,j}(X,s) \;
        \frac{\norm{Y}^{1/2}}{\sqrt{\norm{XY-x_1²X²}}} \; \X(s) \; dX \; dY \; ds\\
        &\lesssim \int\limits_{O(x_1),\ \n{Y}\geq 2x_1^2\n{X}} \psi_{k,j}(X,s) \;
        \frac{1}{\sqrt{\norm{X}}} \; \X(s) \; dX \; dY \; ds\\
        &\lesssim 2^{k-j} \int\limits_{\n{X}\lesssim 1} \psi_{k,j}(X,s) \;
        \frac{1}{\sqrt{\norm{X}}} \; \X(s) \; dX \; ds\lesssim 2^{(k-j)/2}\lesssim 1.
        \end{split}
    \end{equation}
\end{proof}

\begin{lemma}
    \label{gkj2a}
    Let $\epsilon>0$ and $C$ an arbitrary constant. We define
    \begin{equation*}
        \psi_{k,j}(X,s):=\frac{2^j}{(1+2^j\n{1-e^{i(\arctan(X)+2^{k-j}/s)}})^{1+\epsilon}}
    \end{equation*}
    and $\Sigma:=\frac{2^{-j/2+j-k}\j{X}}{\n{Y-2^{j-k}s²X}}$.
  %Suppose $\rho_j$ is a periodic function
  %and the following estimate holds
  %\begin{equation}
  %  \rho_j(\omega)\leq \frac{2^j}{(1+2^j\n{\omega})^{(1+\epsilon)}}.
  %\end{equation}
  %Then
  %\begin{equation}
   % I:=2^{k/2+(k-j)/2}
   % \int\limits_{\n{Y-2^{j-k}s²X}\lesssim 2^{-j/2+j-k}\j{X}} \psi_{k,j}(X,s) \; \X(s) \;
   % \frac{\sqrt{Y}}{\sqrt{X}\sqrt{\n{Y-x_1²X}}} \; \frac{1}{\j{X}^{5/2}} \; dX \; dY \; ds
   % \lesssim 2^{k\epsilon}.
  %\end{equation}
    \begin{itemize}
    \item[(a)] For all $j$ and $k$ with $2^{j-k}\leq C$, the estimate
        \begin{equation*}
        \begin{split}
            I:=&2^{k/2+(k-j)/2}\\
            &\times \int\limits_{\Sigma \gtrsim 1} \psi_{k,j}(X,s) \X(s)
            \frac{\sqrt{Y}}{\sqrt{X}\sqrt{\n{Y-x_1²X}}} \frac{1}{\j{X}^{5/2}} dX dY
            ds \lesssim 2^{k\epsilon}.
        \end{split}
        \end{equation*}
        holds.
    \item[(b)] And similarly, for all $j$ and $k$ with $2^{j-k}\leq C$, the estimate
        \begin{equation*}
        \begin{split}
            I:=&2^{k/2+(k-j)/2}\\
            &\times \int\limits_{\ueber{\Sigma\lesssim 1,}{\n{Y} \lesssim 2^{j-k}\j{X}}} \psi_{k,j}(X,s) \X(s)
            \frac{(\Sigma+(2^{-j}\Sigma)^{1/2}) \sqrt{Y}}{\sqrt{X}\sqrt{\n{Y-x_1²X}}} \frac{1}{\j{X}^{5/2}} dX dY ds
            \lesssim 2^{k\epsilon}
        \end{split}
        \end{equation*}
        holds.
  \end{itemize}
\end{lemma}

\begin{proof}
  Define $\rho_j(x):=2^j \; (1+2^j\n{1-e^{ix}})^{-1-\epsilon}$.
  Then
  \begin{equation*}
    \begin{split}
    &\int \psi_{k,j}(X,s) \; \X(s) \; ds=
    \int \rho_j(\arctan(X)+2^{k-j}/s) \; \X(s) \; ds\\
    &\lesssim 2^{j-k} \; \int_{s \sim 2^{k-j}} \rho_j(\arctan(X)+s) \; ds
    \lesssim \int_{s \sim 1} \frac{2^j}{(1+2^j\n{\arctan(X)+s})^{1+\epsilon}}\; ds \lesssim 1
    \end{split}
  \end{equation*}
  since $\rho_j$ is a periodic function.

%  \ffbox{Furthermore we need an estimate for $\sum_{n \sim 2^{k-j}} \n{n-a}^{-1/2+\epsilon}$ for arbitrary $a$.
%  This sum is bounded by a $c \; 2^{(k-j)(1/2+\epsilon)}$ where $c$ is a constant independent of $a$.%
%%
%
%  1. Fall $a\geq C 2^{k-j}$ daraus folgt $\n{n-a}\geq \n{a} \geq 2^{k-j}$.
%  Somit $\sum_{n \sim 2^{k-j}} \n{n-a}^{-1/2} \lesssim \sum_{n \sim
%  2^{k-j}} 2^{(j-k)/2}\lesssim 2^{(k-j)/2}$
%
%  2. Fall $a\leq C 2^{k-j}$ daraus folgt $\sum_{n \sim 2^{k-j}} \n{n-a}^{-1/2} \lesssim
%  \n{n-a}^{1/2}\vert_{n\sim 2^{k-j}} \lesssim
%  2^{(k-j)/2}$}
  \noindent \emph{(a)}

  First we study the case $Y \geq 2 x_1^2 X$ or $Y\leq x_1^2 X/2$.
  Here we have the estimate
  \begin{equation*}
    I \lesssim 2^{k/2+(k-j)/2} \; 2^{-j/2+j-k}\int \psi_{k,j}(X,s) \; \X(s) \;
    \frac{1}{\sqrt{X}} \; \frac{1}{\j{X}^{3/2}} \; dX \; ds\lesssim 1
  \end{equation*}
  Now $x_1^2 \n{X}/2 \leq \n{Y}\leq 2 x_1^2 \n{X}$ and we suppose that
  $\n{Y} \geq 4 \; 2^{j-k}\n{X}$ or \mbox{$\n{Y}\leq 2^{j-k}\n{X}/4$}.
  We get
  \begin{equation*}
    \begin{split}
    I&\lesssim 2^{k/2+(k-j)/2}\\
    &\int\limits_{\n{Y}\lesssim 2^{-j/2+j-k}\j{X}} \psi_{k,j}(X,s) \; \X(s) \;
    \frac{\sqrt{Y}}{\sqrt{X}\sqrt{\n{Y-x_1²X}}} \; \frac{1}{\j{X}^{5/2}} \; dX \; dY \; ds
    \end{split}
  \end{equation*}
  Since we only integrate where $Y \sim x_1^2 X$ and since
  \[\n{x_1}\n{X}^{1/2}\lesssim \n{Y}^{1/2}\lesssim 2^{-j/4+(j-k)/2}\j{X}\]
  we get form the $Y$-integration
  \begin{equation*}
    \begin{split}
    I&\lesssim 2^{k/2+(k-j)/2} \; \int
    \psi_{k,j}(X,s) \; \X(s) \; 2^{-j/4+(j-k)/2} \; \j{X}^{1/2}
    \frac{\n{x_1}\sqrt{X}}{\sqrt{X}} \; \frac{1}{\j{X}^{5/2}} \; dX \; ds\\
    &\lesssim 2^{k/2-j/4} \; \int
    \psi_{k,j}(X,s) \; \X(s) \;
    \frac{2^{-j/4+(j-k)/2} \; \j{X}^{1/2}}{\sqrt{X}} \; \frac{1}{\j{X}^{2}} \; dX \; ds\lesssim 1.\\
    \end{split}
  \end{equation*}
  We are left with the case $Y \sim 2^{j-k}X$ and $Y \sim x_1²X$. First we estimate
  $\n{\sqrt{Y}}$ by $2^{(j-k)/2} \sqrt{\n{X}}$. We get
  \begin{equation*}
    I\lesssim 2^{k/2+(k-j)/2} \; 2^{(j-k)/2} \int_{\Sigma\gtrsim 1} \psi_{k,j}(X,s)\X(s) \;
    \frac{1}{\n{Y-x_1^2X}} \; \frac{1}{\j{X}^{5/2}} \; dX \; dY \; ds.
  \end{equation*}
  By Lemma \ref{lemmayint2} (a) we
  obtain
  \begin{equation*}
    I \lesssim 2^{(j-k)/2} \int \psi_{k,j}(X,s) \; \X(s) \;
    \frac{2^{(-j/2+j-k)(-\t{\epsilon})}}{\n{(2^{j-k}s²-x_1²)X}^{1/2-\t{\epsilon}}}
    \; \frac{1}{\j{X}^{3/2+\t{\epsilon}}} \; dX \; ds
  \end{equation*}
  for every $\t{\epsilon}\geq 0$.
  For $X \geq 1$ we get by the transformation $\mu:=\arctan X\in \intaa{0,\pi}$ and with
  $\t{\epsilon}=0$
  \begin{equation*}
    \begin{split}
      I&\lesssim 2^{(j-k)/2} \int\limits_0^{\pi} \int \rho_j(\mu+2^{k-j}/s) \; \X(s) \;
      \frac{1}{\n{2^{j-k}s²-x_1²}^{1/2}}
      \; d\mu \; ds\\
      &\lesssim 2^{(j-k)+(j-k)/2} \int\limits_0^{\pi} \int\limits_{s\sim 2^{k-j}}
      \rho_j(\mu+s) \frac{1}{\n{2^{k-j}s^{-2}-x_1^{2}}^{1/2}}
      \; d\mu \; ds\\
      &\lesssim 2^{(j-k)+(j-k)/2} \int\limits_0^{\pi} \int\limits_{s\sim 2^{k-j}}
      \rho_j(\mu+s) \frac{2^{k-j}}{\n{2^{k-j}-x_1^{2}s^2}^{1/2}}
      \; d\mu \; ds\\
    \end{split}
  \end{equation*}
  Since $x_1^2$ is comparable to $2^{j-k}$ we get
  \begin{equation*}
    \begin{split}
        &2^{(j-k)/2} \int\limits_0^{\pi} \int\limits_{s\sim 2^{k-j}}
        \rho_j(\mu+s) \frac{2^{(k-j)/2}}{\n{s^2-2^{k-j}x_1^{-2}}^{1/2}}
        \; d\mu \; ds\\
        &=2^{(j-k)/2} \int\limits_0^{\pi} \int\limits_{s\sim 2^{k-j}}
        \rho_j(\mu+s) \frac{2^{(k-j)/2}}{\n{s-2^{(k-j)/2}x_1^{-1}}^{1/2}\n{s+2^{(k-j)/2}x_1^{-1}}^{1/2}}
        \; d\mu \; ds\\
    \end{split}
  \end{equation*}
  We know that $s+\sgn(x_1)2^{(k-j)/2}x_1^{-1}\geq s\gtrsim 2^{k-j}$. We can assume that
  $\sgn(x_1)=1$. The case $\sgn(x_1)=-1$ is similar. We obtain
  \begin{equation*}
    \begin{split}
    I&\lesssim 2^{(j-k)/2} \int\limits_0^{\pi} \int\limits_{s\sim 2^{k-j}}
        \rho_j(\mu+s) \frac{1}{\n{s-2^{(k-j)/2}x_1^{-1}}^{1/2}}
        \; d\mu \; ds\\
        &\lesssim 2^{(j-k)/2} \;
      \sum\limits_{\ueber{n\sim 2^{k-j}}{n \sim 2^{k-j},\; n\in \N}} \int\limits_{s \in \intaa{-\pi,\pi}} \int \rho_j(\mu+n+s) \;
      \frac{1}{\n{s+n-2^{(k-j)/2}x_1^{-1}}^{1/2}} \; d\mu \; ds\\
      &=2^{(j-k)/2} \;
      \sum\limits_{\ueber{n\sim 2^{k-j}}{n \sim 2^{k-j},\;n\in \N}} \int\limits_{s \in \intaa{-\pi,2\pi}} \int \rho_j(n+s) \;
      \frac{1}{\n{s+n-\mu-2^{(k-j)/2}x_1^{-1}}^{1/2}} \; d\mu \; ds\\
      \end{split}
  \end{equation*}
  For $\n{n-2^{(k-j)/2}x_1^{-1}}\leq 4\pi$ we get by integration with respect to $\mu$
  \[
    2^{(j-k)/2} \;
    \sum_{\ueber{\n{n-2^{(k-j)/2}x_1^{-1}}\leq 4\pi}{n\in \N}} \; \int\limits_{s \in \intaa{-\pi,2\pi}} \rho_j(n+s) \;
    \n{s+n-2^{(k-j)/2}x_1^{-1}}^{1/2} \; ds.
  \]
  $\n{s+n-2^{(k-j)/2}x_1^{-1}}^{1/2}$ is bounded by a
  constant. The last integration gives us a one more constant. Hence we get
  $I\lesssim 2^{(j-k)/2}$.

  For $\n{n-2^{(k-j)/2}x_1^{-1}}\geq 4\pi$ we get
  $\n{s+n-\mu-2^{(k-j)/2}x_1^{-1}}\geq \n{n-2^{(k-j)/2}x_1^{-1}}/2$ and hence
  \begin{equation*}
    \begin{split}
    I&\lesssim 2^{(j-k)/2} \;
    \sum_{\ueber{\n{n-2^{(k-j)/2}x_1^{-1}}\geq 4\pi}{n \sim 2^{k-j},\;n \in \N}} \; \int\limits_{s \in \intaa{-\pi,2\pi}} \rho_j(n+s) \;
    \frac{1}{\n{n-2^{(k-j)/2}x_1^{-1}}^{1/2}} \; ds\\
    &\lesssim 2^{(j-k)/2} \;
    \sum_{\ueber{\n{n-2^{(k-j)/2}x_1^{-1}}\geq 4\pi}{n \sim 2^{k-j},\;n \in \N}} \;
    \frac{1}{\n{n-2^{(k-j)/2}x_1^{-1}}^{1/2}}.
    \end{split}
  \end{equation*}
  It is easy to observe that this last sum is bounded by $2^{(k-j)/2}$ and hence $I\lesssim 1$.

  For $X \leq 1$ and for $\sgn(x_1)=1$, which we assume, we get with similar arguments and $\arctan X \sim X$
  \begin{equation*}
    \begin{split}
      &I\lesssim 2^{j\epsilon/2+(j-k)/2}
      \int\limits_{\ueber{s \sim 2^{k-j}}{\n{X}\leq 1}}
      \rho_j(X+s) \frac{1}{\n{s-2^{(k-j)/2}x_1^{-1}}\n{X}^{1/2-\epsilon}}
      dX \; ds\\
      &\lesssim 2^{j\epsilon/2+(j-k)/2} \;
      \sum\limits_{\ueber{n\sim 2^{k-j}}{n\in\N}} \int\limits_{\ueber{s \in \intaa{-\pi,\pi}}{\n{X}\leq1}} \rho_j(X+n+s) \;
      \frac{\n{X}^{-1/2+\epsilon}}{\n{s+n-2^{(k-j)/2}x_1^{-1}}^{1/2-\epsilon}} dX \; ds\\
      &=2^{j\epsilon/2+(j-k)/2} \;
      \sum\limits_{\ueber{n\sim 2^{k-j}}{n \in \N}} \int\limits_{\ueber{s \in \intaa{-\pi,2\pi}}{\n{X}\leq1}} \rho_j(n+s) \;
      \frac{\n{X}^{-1/2+\epsilon}}{\n{s+n-X-2^{(k-j)/2}x_1^{-1}}^{1/2-\epsilon}} dX \; ds\\
    \end{split}
  \end{equation*}
  As before we get, for $\n{n-2^{(k-j)/2}x_1^{-1}}\leq 4\pi$,
  \begin{equation*}
    2^{j\epsilon/2} \; 2^{(j-k)/2} \;
    \sum_{\n{n-2^{(k-j)/2}x_1^{-1}}\leq 4\pi} \; \int\limits_{s \in \intaa{-\pi,2\pi}} \rho_j(n+s) \;
    ds\lesssim 2^{(j-k)/2} \; 2^{k\epsilon/2}
  \end{equation*}
  and, for $\n{n-2^{(k-j)/2}x_1^{-1}}\geq 4\pi$, the estimate
  \begin{equation*}
    \begin{split}
    2^{j\epsilon/2} \; 2^{(j-k)/2} \;
    &\sum\limits_{\ueber{\n{n-2^{(k-j)/2}x_1^{-1}}\geq 4\pi}{n\sim 2^{k-j},\;n\in\N}} \; \int\limits_{s \in \intaa{-\pi,2\pi}} \rho_j(n+s) \;
    \n{n-2^{(k-j)/2}x_1^{-1}}^{-1/2+\epsilon} \; ds\\
    &\lesssim 2^{k\epsilon}\\
    \end{split}
  \end{equation*}
  holds.\absatz

  \noindent\emph{(b)}

  The proof of \emph{(b)} is very similar to the proof of \emph{(a)}.

\noindent For $Y \geq 2 x_1^2 X$ or $Y\leq 1/2 x_1^2 X$ we use \[\Sigma^{1/2}\leq
\Sigma^{1/2-\epsilon/2}\leq
2^{j\epsilon/4+(k-j)\epsilon/2}2^{-j/2+j-k}\n{Y-2^{j-k}s^2X}^{-1+\epsilon/2}\] and get by
integrating with respect to $Y$.
  \begin{equation*}
    I \lesssim 2^{k\epsilon} \; \int \psi_{k,j}(X,s) \; \X(s) \;
    \frac{1}{\sqrt{X}} \; \frac{1}{\j{X}^{3/2-\epsilon/2}} \; dX \; ds\lesssim
    2^{k\epsilon}.
  \end{equation*}
  \noindent The case $Y \geq 4 \; 2^{j-k}X$ or $Y\leq 1/4 \; 2^{j-k}X$.
  For this part we have the estimate
  \begin{equation*}
    \begin{split}
    I&\lesssim
    \int\limits \psi_{k,j}(X,s) \; \X(s) \;
    \frac{1}{\sqrt{X}\sqrt{\n{Y-x_1²X}}} \; \frac{1}{\j{X}^{3/2}} \; dX \; dY \; ds\\
    \end{split}
  \end{equation*}
  Since we only integrate where $Y \sim x_1^2 X$ and since $\n{Y}\lesssim 2^{j-k}\j{X}$ we get form the $Y$-integration
  \begin{equation*}
    \begin{split}
    I&\lesssim 2^{(j-k)/2} \; \int
    \psi_{k,j}(X,s) \; \X(s) \; \frac{1}{\sqrt{X}} \; \frac{1}{\j{X}} \; dX \; ds\lesssim 1\\
    \end{split}
  \end{equation*}

  \noindent The case $Y \sim 2^{j-k}X$ and $Y \sim x_1²X$ is left. By Lemma \ref{lemmayint2} (b) we get

  \begin{equation*}
    I \lesssim 2^{(j-k)/2} \int \psi_{k,j}(X,s) \; \X(s) \;
    \frac{2^{(-j/2+j-k)(-\epsilon)}}{\n{(2^{j-k}s²-x_1²)X}^{1/2-\epsilon}}
    \; \frac{1}{\j{X}^{3/2+\epsilon}} \; dX \; ds
  \end{equation*}
  By similar estimates as in \emph{(a)} we get $I\lesssim 2^{k\epsilon}$.
\end{proof}

We get a similar result for $2^{k-j}\lesssim 1$.
\begin{lemma}
  \label{gkj2}
  Let $\epsilon>0$ and $C\leq 1$ a constant. We define
    \begin{equation*}
        \psi_{k,j}(X,s):=\frac{2^j}{(1+2^j\n{1-e^{i(\arctan(X)+2^{k-j}/s)}})^{1+\epsilon}}
    \end{equation*}
    and $\Sigma:=\frac{2^{-k/2}\j{X}}{\n{Y-2^{j-k}s²X}}$.
  %Suppose $\rho_j$ is a periodic function
  %and the following estimate holds
  %\begin{equation}
  %  \rho_j(\omega)\leq \frac{2^j}{(1+2^j\n{\omega})^{(1+\epsilon)}}.
  %\end{equation}
  \begin{itemize}
    \item[(a)] For all $j$ and $k$ with $2^{k-j}\leq C$, the estimate
        \begin{equation*}
            \begin{split}
        I:=&2^{k/2+(k-j)/2}\\
        &\times \int\limits_{\ueber{\Sigma\gtrsim 1,\ X\lesssim 1,}{Y\sim 1}} \psi_{k,j}(X,s) \X(s)
        \frac{\sqrt{Y}}{\sqrt{X}\sqrt{\n{Y-x_1²X}}} \frac{1}{\j{X}^{5/2}} dX  dY ds
        \lesssim 2^{j\epsilon}
        \end{split}
        \end{equation*}
        holds.
    \item[(b)] And similarly, for all $j$ and $k$ with $2^{k-j}\leq C$, the estimate
        \begin{equation*}
            \begin{split}
        I:=&2^{k/2+(k-j)/2}\\
        &\times \int\limits_{\ueber{\Sigma \; \lesssim 1,\ X\lesssim 1,}{Y\sim 1}}
        \psi_{k,j}(X,s)  \X(s)
        \frac{\Sigma \sqrt{Y}}{\sqrt{X}\sqrt{\n{Y-x_1²X}}} \frac{1}{\j{X}^{5/2}} dX  dY ds
        \lesssim 2^{j\epsilon}
        \end{split}
        \end{equation*}
        holds.
  \end{itemize}
\end{lemma}

\begin{proof}
  Define $\rho_j(x):=2^j \; (1+2^j\n{1-e^{ix}})^{-1-\epsilon}$.\absatz

  \noindent \emph{(a)}

  The case $Y \geq 2 x_1^2 X$ or $Y\leq 1/2 x_1^2 X$.
  For this part we have the estimate
  \begin{equation*}
    I \lesssim 2^{k/2+(k-j)/2} \; 2^{-k/2} \; \int \psi_{k,j}(X,s) \; \X(s) \;
    \frac{1}{\sqrt{X}} \; \frac{1}{\j{X}^{3/2}} \; dX \; ds
  \end{equation*}
  Now, for $\n{X}\geq 2^{k-j}$, we obtain
  $I \lesssim 2^{(k-j)/2} \; 2^{(j-k)/2} =1$,
  by integration over $X$.
  For $\n{X}\leq 2^{k-j}$, we first do the $s$-integration and obtain
  \[2^{(k-j)/2} \; 2^{j-k} \; \int_{\n{X}\leq 2^{k-j}} \frac{1}{\sqrt{X}} \; dX,\]
  but this is bounded by a constant.

  The case $Y \geq 4 \; 2^{j-k}X$ or $Y\leq 1/4 \; 2^{j-k}X$.
  %For this part we have the estimate
  %\begin{equation*}
  %  \begin{split}
  %  I&\lesssim 2^{k/2+(k-j)/2}
  %  \int\limits_{\n{Y}\lesssim 2^{-j/2+j-k}\j{X}} \psi_{k,j}(X,s) \; \X(s) \;
  %  \frac{\sqrt{Y}}{\sqrt{X}\sqrt{\n{Y-x_1²X}}} \; \frac{1}{\j{X}^{5/2}} \; dX \; dY \; ds\\
  %  \end{split}
  %\end{equation*}
  Since we only integrate where $Y \sim x_1^2 X$ we get form the $Y$-integration
  \begin{equation*}
    \begin{split}
    I&\lesssim 2^{k/2+(k-j)/2} \; \int
    \psi_{k,j}(X,s) \; \X(s) \; 2^{-k/4} \; \j{X}^{1/2}
    \frac{\n{x_1}\sqrt{X}}{\sqrt{X}} \; \frac{1}{\j{X}^{5/2}} \; dX \; ds\\
    &\lesssim 2^{k/2+(k-j)/2} \; \int
    \psi_{k,j}(X,s) \; \X(s) \; 2^{-k/4} \; \j{X}^{1/2}
    \frac{2^{-k/4} \; \j{X}^{1/2}}{\sqrt{X}} \; \frac{1}{\j{X}^{5/2}} \; dX \; ds\lesssim 1\\
    \end{split}
  \end{equation*}

  \noindent The case $Y \sim 2^{j-k}X$ and $Y \sim x_1²X$

  Since $X \lesssim 1$ and $Y\sim 1$ we get by Lemma \ref{lemmayint2} (a)
  \begin{equation*}
    I \lesssim 2^{j\epsilon/2+(k-j)\epsilon} \int \psi_{k,j}(X,s) \; \X(s) \;
    \frac{1}{\n{(2^{j-k}s²-x_1²)X}^{1/2-\epsilon}} \; X^{-1+\epsilon} \; dX \; ds
  \end{equation*}
  Now, with $\arctan X \sim X$
  and by the transformation $s:=2^{k-j}/s$, we obtain
  \begin{equation*}
    \begin{split}
      &2^{k\epsilon/2+(k-j)/2} \; 2^{j-k} \; \int\limits_{X\lesssim 1} \int\limits_{s \sim 2^{k-j}}
      \rho_j(X+s) \frac{1}{{\n{x_1^2-2^{k-j}/s^2}}^{1/2-\epsilon}} \; X^{-1+\epsilon} \; dX \;
      ds\\
      &\lesssim 2^{k\epsilon/2+(k-j)/2} \; 2^{(j-k)2\epsilon} \; \int\limits_{X\lesssim 1} \int\limits_{s \sim 2^{k-j}}
      \rho_j(X+s) \frac{x_1^{-1+2\epsilon}}{{\n{s^2-2^{k-j}/x_1^2}}^{1/2-\epsilon}} \; X^{-1+\epsilon} \; dX \; ds
     \end{split}
  \end{equation*}
  Since $x_1^2X \sim Y \sim 1$, the factor $x_1^{-1+2\epsilon}$ is bounded by $\n{X}^{1/2-\epsilon}$. We assume that $\sgn(x_1)=1$ and hence
  $\n{s+2^{(k-j)/2}x_1^{-1}}\gtrsim 2^{k-j}$. We obtain
  \begin{equation*}
    \begin{split}
    I\lesssim& 2^{k\epsilon/2} \; 2^{(j-k)\epsilon} \; \int\limits_{X\lesssim 1} \int\limits_{s \sim 2^{k-j}}
    \rho_j(X+s) \frac{1}{{\n{s-2^{(k-j)/2}/x_1}}^{1/2-\epsilon}} \; X^{-1/2} \; dX \;ds\\
    =& 2^{k\epsilon/2} \; 2^{(j-k)\epsilon}\\ &\times \sum_{n\sim 2^{k-j}} \;
    \int\limits_{X\lesssim 1} \int\limits_{s \sim 1}
    \rho_j(n+s) \frac{1}{{\n{s+n-X-2^{(k-j)/2}/x_1}}^{1/2-\epsilon}} \; X^{-1/2} \; dX \; ds
    \end{split}
  \end{equation*}

  Since $2^{k-j}\leq 1$ there is only one summand.
  For $\n{n-2^{(k-j)/2}x_1^{-1}}\leq 64$, we compute the $x$-integration and get
  \begin{equation*}
    2^{k\epsilon/2} \; 2^{(j-k)\epsilon}
    \sum_{\n{n-2^{(k-j)/2}x_1^{-1}}\leq 64} \; \int\limits_{s \sim 1} \rho_j(n+s) \;
    ds\lesssim 2^{j\epsilon}
  \end{equation*}
  \noindent For $\n{n-2^{(k-j)/2}x_1^{-1}}\geq 64$, we get
  \begin{equation*}
    \begin{split}
    2^{k\epsilon/2} \; &2^{(j-k)\epsilon} \;
    \sum_{\ueber{n \sim 2^{k-j},}{\n{n-2^{(k-j)/2}x_1^{-1}}\geq 64}}
    \int\limits_{s \sim 1} \rho_j(n+s) \;
    \n{n-2^{(k-j)/2}x_1^{-1}}^{-1/2+\epsilon} \; ds\lesssim 2^{j\epsilon}
    \end{split}
  \end{equation*}
  with a new $\epsilon$.

  Since the proof for \emph{(b)} is very similar, we omit it.
\end{proof}
%
%
%
%
%
%%%%%%%%%%%%%%%%%%%%%%%%%%%%%%%%%%%%%%%%%%%%%%%%%%%%%%%%%%%%%%%%%%%%%%%%%%%%%%%%%%%%%%

%%%%%%%%%%%%%%%%%%%%      Abschätzungen für 2^{j-k}\leq C      %%%%%%%%%%%%%%%%%%%%%%%
%H_{k,j}, F_{k,j}, G_{k,j}
%
%
\subsection{Estimation for $2^{j-k}\leq 1/\const_2$}
\subsubsection*{Estimation of $H_{k,j}$}

%\begin{quote}
%(Modifikation) Hier wird die Abschätzung $\sum_{\l=0}^{n/2} \l^{-1/2} e^{i\sigma \l}\leq
%\sigma^{-1/2}$ benutzt!! Außerdem wird hier die partielle Integration $P_n=\d{t}R_n$ benutzt mit
%dem Zusatz, das der Gewinn $2^{k-j}w$ sein sollte. Diese Abschätzungen sind unabhängig von $x_1$
%und von der Größe $2^{j-k}$,man benötigt lediglich $2^{-j}\leq 1$!
%\end{quote}
Recall the definition of $H_{k,j}$ given by \eqref{hkj}. According to this
\begin{equation*}
    \begin{split}
    H_{k,j}&(x_1,x,u)=2^{k/2} \; 2^{(m-2)(j-k)-j}\\
    &\times \int \frac{\zeta_{k,j,m-1}(x_1,x,u,s)}{(R^2+(u-s)^2)^{m/2}}
    \; (u-s) \; \frac{(2xx_1-iw)^{m}}{a^{m}} \;  \rho'(2^{2(k-j)}w(s)^2) \;\X(s)\; ds.
    \end{split}
\end{equation*}
Let
\begin{equation*}
    \kappa(s):=(u-s) \; \frac{(2xx_1-iw)^m}{a^{2m}} \; \rho'(2^{2(k-j)}w(s)^2).
\end{equation*}
Then $a^m w^{-m}\n{\kappa}\lesssim (2xx_1-iw)^m \; a^{-m} w^{-m+1}=w^{1-m}$.
\begin{bemerkung}
  On the Heisenberg group $\H_m$, $\kappa^{\H}$ is defined as
  \[\kappa^{\H}:=(u-s)(R²+(t-s)²)^{-m/2}\rho'(2^{2(k-j)}a²)\]
  and we have the estimate $\kappa^\H\lesssim a^{1-m}$.
\end{bemerkung}
\noindent We know, by Proposition \ref{propqn2}, that
\begin{equation*}
  \begin{split}
  \zeta_{k,j,m-1}&\lesssim 2^{-mj} \; 2^{j/2+j\epsilon} \; \frac{a^{m}}{w^{m}} \; \psi_{k,j}(X,s)
  = 2^{j\epsilon} \; \frac{a^{m}}{w^{m}} \; \psi_{k,j}(X,s)\\
  \end{split}
\end{equation*}
with
\[\psi_{k,j}(X,s):=\frac{2^j}{(1+2^j\n{1-e^{i\arctan(X)+2^{k-j}/s}})^{1+\epsilon}}.\]
Thus
\[\norm{H_{k,j}}
    \leq 2^{k/2} \; 2^{(m-2)(j-k)-j+\epsilon j}\; \int \psi_{k,j}
    \; w^{1-m} \; \rho'(2^{2(k-j)}w(s)^2) \;\X(s)\; ds.
    \]
%\begin{equation}
%  \begin{split}
%  H_{k,j}&:=2^{k/2} \; 2^{(m-2)(j-k)-j}\\
% &\times \int \frac{\zeta_j(x,x_1,t-s,2^{k-j}/s)}{w^m} \;
%  \frac{(2xx_1-iw)^{m+1}}{a^{m+1}} \; (t-s) \;
% \rho'(2^{2(k-j)}w(s)²) \; \X(s) \; ds\\
%  &\leq 2^{k/2} \; 2^{(m-2)(j-k)-j} \;
% \int \zeta_j(x,x_1,t-s,2^{k-j}/s) \;
% \frac{w}{w^{m}} \; \rho'(2^{2(k-j)}w(s)²) \; \X(s) \; ds
%  \end{split}
%\end{equation}
Observe that for $\rho'(s)\not=0$ we have
\begin{equation*}
  \frac{\n{Y}}{\j{X}}=w(s)\sim 2^{j-k}.
\end{equation*}
Lemma \ref{hkjlemma} implies that the $s$-integration just gives us a constant. In
combination with Lemma \ref{lemmayint1} (a) we find that
\[
  \begin{split}
  \dnorm{H_{k,j}&}_{L_1(O(x_1))}\lesssim 2^{k/2} \; 2^{(m-2)(j-k)-j+j\epsilon} \; 2^{j-k}\\
  &\times \int\limits_{\norm{Y} \sim 2^{j-k} \j{X}}
  \psi_{k,j}(X,s) \; \frac{\j{X}^{1/2}}{\norm{Y}^{1/2}}\;
  \frac{\norm{Y}}{2\j{X}^3} \frac{1}{\sqrt{\norm{XY-x_1²X²}}}\;ds\;dX\; dY\\
  \lesssim  & 2^{k/2} \; 2^{(m-2)(j-k)-j+\epsilon j} \; 2^{j-k} \;
  \int \frac{2^{j-k}\j{X}}{\j{X}^{2+1/2}\norm{X}^{1/2}}\; dX \lesssim 2^{-j/2+\epsilon j}
  \end{split}
\]
holds.
%\ffbox{ Hier wurde $w$ gegen $2^{j-k}$. Auf der Heisenberggrupe gilt für das Integral die
%gleiche Abschätzung $2^{k/2-mk+(m-1)j}= 2^{-j/2}$ für $m=1/2$.} \ffbox{überprüft am
%27.1.2006}
\subsubsection*{Definition of $\Sigma(s)$}
This was the easy part of the proof. For $F_{k,j}$ and $G_{k,j}$ we use partial summation
in the $s$-variable to get the desired estimates. For this, we use a "control quantity",
denoted by $\Sigma(s)$. If $\Sigma(s) \lesssim 1$ we "win" by partial integration a
factor of size $2^{-j/2}$. For $\Sigma(s) \gtrsim 1$ we use support estimates and also
"win" a factor of size $2^{-j/2}$. We define \index{Sigma}
\begin{equation}
  \label{Sigmadef}
  \Sigma(s):=2^{-j} \; \Big(\frac{2^{(k-j)/2}+w^{-1}}{\d{s}\kphi}\Big)^2
\end{equation}
Then
\begin{equation*}
    \begin{split}
    \Sigma(s)&=2^{-j} \; \Big(\frac{2^{(k-j)/2}+w^{-1}}
  {\frac{X}{Y}-\frac{2^{k-j}}{s²}}\Big)^2\\
  &\sim 2^{-j} \; \Big(\frac{(2^{(j-k)/2}+2^{j-k}w^{-1})Y}
  {Y-2^{j-k}s²X}\Big)^2
  =2^{-j} \; \Big(\frac{2^{(j-k)/2}\norm{Y}+2^{j-k}<X>}
  {Y-2^{j-k}s²X}\Big)^2.
  \end{split}
\end{equation*}
\index{Sigma1}
%
%
%
%%%%%%%%%%%%%%%%%%%%%%%%%%%%%%%%%%%%%%%%%%%%%%%%%%%%%%%%%%%%%%%%%%%%%%%%%%%%%%%%%%%%%%%%%%%%%%
%%%%%%%%%%%%%%%%%%%%%%%%%%%%%%%%        F_k,j           %%%%%%%%%%%%%%%%%%%%%%%%%%%%%%%%%%%%%%
%
%
%
\subsubsection*{Estimation of $F_{k,j}$}

%\ffbox{Hier wird die Abschätzung $\sum_{\l=0}^{n/2} \l^{-1/2} e^{i\sigma \l}\leq
%\sigma^{-1/2}$ benutzt!! Hier wird \underline{keine} partielle Integration $P_n=\d{t}R_n$
%benutzt!.}
Recall the definition of $F_{k,j}$. $F_{k,j}$ is given by
\begin{equation}
    \begin{split}
    F_{k,j}(x_1,x,u)=&2^{k/2} \; 2^{m(j-k)}
    \; \int \frac{\zeta_{k,j,m}(x,x_1,u,s)}{(R^2+(u-s)^2)^{(m+3)/2}} \;
    e^{-i(m+1)\sigma}\\
    &\times (R2xx_1+(u-s)w - i(Rw-2xx_1(u-s)))\\
    &\times (1-\rho)(2^{2(k-j)}w(s)^2) \; \X(s) \; ds.
    \end{split}
\end{equation}

%Let
%\begin{equation*}
%  \begin{split}
%  \t{F}_{k,j}&:=2^{k/2} \; 2^{(m-1)(j-k)}\\
% &\times \int
%  \frac{\zeta_j(x,x_1,t-s,1/s)}{w^{m+1}} \;
%  \X(2^{k-j}s) \; (1-\rho)(2^{4(k-j)}w(s)²) \; e^{-i(m+1)\sigma} \; e^{i/s} \; ds
%  \end{split}
%\end{equation*}
%with
%\begin{equation*}
%  \zeta_j(x,x_1,t-s,1/s):=2^{-mj} \; \sum_{n=0}^\infty \X_j(m+n) \; b_n(\sigma) \;
%  e^{i(n+m+1)(\arctan(X)+1/s)}\; \frac{w^{m+1}}{a^{m+1}}
%\end{equation*}
%Beachte das gilt:
%\begin{equation}
%  \begin{split}
% e^{i/s}\zeta_j&=2^{-mj} \; \sum_{n=0}^\infty \X_j(m+n) \; b_n(\sigma) \;
% e^{i(n+m+2)(\arctan(X)+1/s)}\; \frac{w^{m+1}}{a^{m+1}}\\
% &\times \frac{(t-s)w-R2xx_1-i(Rw+2xx_1(t-s))}{R²+(t-s)²}\\
% &=\t{\zeta}_j \; \frac{(t-s)w-R2xx_1-i(Rw+2xx_1(t-s))}{R²+(t-s)²}
% \end{split}
%end{equation}
%Mit einer neuen Funktion $\t{\zeta}_j$. Da sich diese Funktion ebenso wie $\zeta$
%verhält, setzen wir $\zeta_j:=\t{\zeta}_j$. Now we replace $x$ by $2^{(j-k)/2}x$, $t$ by
%$2^{j-k}t$ and obtain

%Neue Formel mit neuer Abschätzung!!!
%\begin{equation}
%  \begin{split}
% F_{k,j}&:=2^{k/2} \; 2^{m(j-k)} \; \int
% \frac{\zeta_j(x,x_1,t-s,2^{k-j}/s)}{w^{m+1}} \;
%  \frac{(2xx_1-iw)^{m+1}}{a^{m+1}}\\
%  &\times \frac{(t-s)w-R2xx_1-i(Rw+2xx_1(t-s))}{a²} \;
%    (1-\rho)(2^{2(k-j)}w(s)²) \; \X(s) \; ds\\
% \end{split}
%\end{equation}
\noindent Let
\begin{equation*}
    \begin{split}
    \kappa(s):=&\frac{(2xx_1-iw)^{m+1}}{a^{m+1}} \; \frac{R2xx_1+(u-s)w -
    i(Rw-2xx_1(u-s))}{a^{m+3}}\\&\times(1-\rho)(2^{2(k-j)}w(s)²).
    \end{split}
\end{equation*}
Then we have
\begin{equation}
    \label{Fkappasigma}
    a^{m+1} \; w^{-m-1} \;
    \norm{\kappa}\lesssim\norm{2xx_1-iw}^{m+1}a^{-m-1}w^{-m-1}=w^{-m-1}.
\end{equation}
% \ffbox{$(a/w)^{m+1}$
%ist $\norm{\sigma}^{-3/2}$, welches aus der Abschätzung für $\zeta_j$ kommt.}
\begin{bemerkung}
  On the Heisenberg group $\H_m$, $\kappa^{\H}$ is defined by
  \[\kappa^{\H}:=(u-s-iR)(R²+(u-s)²)^{-(m+2)/2}(1-\rho)(2^{2(k-j)}a²)\]
  and we have the estimate $\kappa^\H\lesssim a^{-m-1}$. Once more we see that the role of
  $a$ on the Heisenberg group coincides with the role of $w$ in our proof.
\end{bemerkung}
\noindent Since
\[\norm{R2xx_1+(u-s)w - i(Rw-2xx_1(u-s))}^2=R^2a^2+(u-s)^2a^2=a^4\]
we get
\begin{equation}
    \n{\d{s} \kappa}\lesssim \n{\kappa} \; a^{-1}.
\end{equation}
 \noindent For $s$ in the support of $(1-\rho)$ and by \eqref{odef} %(\ref{reduction1})
 we have
\begin{equation*}
    c_0 \; 2^{j-k} \leq w(s) \leq c_1 \; 2^{k-j},
\end{equation*}
for suitable constants $c_0$ and $c_1$ and $c_0\geq 64$. Thus
\begin{equation}
   c_0 \; 2^{j-k} \; <X> \leq \norm{Y} \leq c_1 \; 2^{k-j}<X>.
\end{equation}
Hence $\norm{Y-s² 2^{j-k}X}\geq \norm{Y}$ for all $s \in \supp \X$ and
\begin{equation*}
  \begin{split}
  \Sigma(s)&=2^{-j} \Big(\frac{2^{(j-k)/2}\norm{Y}+2^{j-k}<X>}
  {Y-2^{j-k}s²X}\Big)^2
  \lesssim 2^{-j} \Big(\frac{2^{(j-k)/2}\norm{Y}+2^{j-k}<X>}{\norm{Y}}\Big)^2\\
  &\lesssim 2^{-j} (2^{(j-k)/2}+1)^2\sim 2^{-j}.
  \end{split}
\end{equation*}
Observe that
\begin{equation*}
  2^{-j}\bignorm{\frac{\kappa'}{\kphi'\kappa}}\lesssim 2^{-j}\frac{w^{-1}}{\norm{\kphi'}}
  \lesssim (2^{-j}\Sigma)^{1/2}\leq 2^{-j}(2^{(j-k)/2}+1)^2\sim2^{-j}(2^{j-k}+1)
\end{equation*}
and
\begin{equation*}
  \begin{split}
  2^{-j}\frac{\norm{\X'}}{\norm{\kphi'\X}}&
  \lesssim \frac{1}{\norm{\kphi}}\lesssim 2^{-j}\frac{\norm{(a^{-1}+2^{k-j})}}{\norm{\d{s} \kphi}^2}
  \lesssim \Sigma(s).
  \end{split}
\end{equation*}
hold true. By Lemma \ref{ll} we have
\[2^{-j}\Bignorm{\frac{\d{s}^2\kphi}{\kphi'(s)^2}}\lesssim \Sigma(s)\]
and hence
\[2^{-j}\d{s}\Big(\frac{\kappa\X}{\d{s}\kphi}\Big)\lesssim (\Sigma(s)+(2^{-j}\Sigma(s))^{1/2})\n{\kappa\X}(s).
\]
Furthermore,
\[\Bignorm{\frac{2^{-j}w^{-1}}{\d{s}\kphi}}\lesssim (2^{-j}\Sigma(s))^{1/2}.\]
By the definition of $\kappa$ we have
\[F_{k,j}=2^{k/2} \; 2^{m(j-k)} \; \int
  \zeta_{k,j,m}(x_1,x,u,s) \;
  \kappa(s) \; \X(s) \; ds \]
and
\begin{equation*}
  \begin{split}
  \zeta_{k,j,m}(x_1,x,u,s)=&\sum_{n=0}^\infty \X_j(m+n) \; q_{n,m}(\sigma) \; 2^{-mj} \\
  &\times \frac{1}{i(n+m+1)\d{s}\kphi} \; [\d{s} e^{i(n+m+1)\kphi(s)}]
  \end{split}
\end{equation*}
with $\kphi(s):=\arctan(X)+2^{k-j}/s$. Thus
\begin{equation*}
  \begin{split}
  F_{k,j}&=2^{k/2} \; 2^{m(j-k)} \; \int \t{\zeta}_{j,m}(x_1,x,u-s,2^{k-j}/s) \;
  2^{-j} \d{s} \Big(\frac{\kappa \X}{\d{s}\kphi}\Big) \; ds\\
  &+2^{k/2} \; 2^{m(j-k)} \; \int \t{\t{\zeta}}_{j,m}(x_1,x,u-s,2^{k-j}/s) \;
  \frac{2^{-j} \; w^{-1}}{\d{s}\kphi} \; \kappa \; \X \; ds
  \end{split}
\end{equation*}
with
\begin{equation*}
  \begin{split}
  \t{\zeta}_{k,j,m}&:=\sum_{n=0}^\infty \X_j(m+n) \; q_{n,m}(\sigma)\;
  \frac{2^{(1-m)j}}{i(n+m+1)} \; e^{i(n+m+1)\kphi}\\
  \t{\t{\zeta}}_{k,j,m}&:=\sum_{n=0}^\infty \X_j(m+n) \;
  [\d{s}q_{n,m}(\sigma)] \; \frac{2^{(1-m)j} \; w}{i(n+m+1)} \; e^{i(n+m+1)\kphi}.
  \end{split}
\end{equation*}
We know, by Proposition \ref{propqn1}, that the following two inequalities
\begin{equation}
  \begin{split}
  \t{\zeta}_{k,j,m}&\lesssim 2^{-mj} \; 2^{j/2+j\epsilon} \; \frac{a^{m+1}}{w^{m+1}} \; \frac{2^j}{(1+2^j\n{\kphi})^{1+\epsilon}}
  \sim 2^{j\epsilon} \; \frac{a^{m+1}}{w^{m+1}} \; \psi_{k,j}(X,s),\\
  \t{\t{\zeta}}_{k,j,m}&\lesssim 2^{-mj} \; 2^{j/2+j\epsilon} \; \frac{a^{m+2}}{w^{m+2}}\;
  a^{-1} \; w \; \frac{2^j}{(1+2^j\n{\kphi})^{1+\epsilon}}
  \sim 2^{j\epsilon} \; \frac{a^{m+1}}{w^{m+1}} \; \psi_{k,j}(X,s)
  \end{split}
\end{equation}
hold, with
\[\psi_{k,j}(X,s):=\frac{2^j}{(1+2^j\n{1-e^{i\arctan(X)+2^{k-j}/s}})^{1+\epsilon}}.\]
%\ffbox{$\d{s}\sigma\leq a^{-1}$!}
%\ffbox{In der Heisenbergsituation tritt nur $\t{\zeta}_j$ auf, da hier $b_n(\sigma)$ eine
%Konstante mit $b_n \sim n^m$ ist.}
Let
\[J:=\intaa{c_0 \; 2^{j-k} \j{X},c_1 \; 2^{k-j}\j{X}}.\]
%\[\t{\psi}_{k,j}=\psi_{k,j} \frac{w^{m+1}}{a^{m+1}}.\]
%\[\t{\t{\psi}}_{k,j}=\n{\t{\t{\zeta}}_j}\frac{w^{m+1}}{a^{m+1}}\]
By Lemma \ref{hkjlemma} the $s$-integration $\int \psi_{k,j}(X,s) \; \X(s) \; ds$
%\[\int \rho_j(X,s) \; ds\lesssim 2^{j\epsilon}\]
yields a constant. In combination %(\ref{Fkappasigma}) with lemma \ref{lemmayint} and
%lemma \ref{lemmasint1} we thus find
with Lemma \ref{lemmayint1} (b) we thus find that
\begin{equation*}
      \begin{split}
      \Bigdnorm{&2^{k/2} \; 2^{m(j-k)} \; \int \t{\zeta}_{k,j,m}(x_1,x,u,s) \;
      2^{-j} \d{s} \Big(\frac{\kappa \X}{\d{s}\kphi}\Big) \; ds}_{L_1(O(x_1))}\\
      &\leq 2^{k/2} \; 2^{m(j-k)} \; \Bigdnorm{\int \n{\t{\zeta}_{k,j,m}} \;
      (\Sigma(s)+(2^{-j}\Sigma(s))^{1/2}) \; \n{\kappa \X} \; ds}_{L_1(O(x_1))}\\
      &\lesssim 2^{-j/2+\epsilon j} \int\limits_{\norm{Y}\in J} \psi_{k,j}(X,s) \;
      \frac{\j{X}^{3/2}}{\norm{Y}^{3/2}}
      \frac{\norm{Y}}{2\j{X}^3} \frac{1}{\sqrt{\n{XY-x_1²X²}}} \;
      \; dX \; dY \; ds\\
      %&\lesssim 2^{-j/2} \int\limits_{\norm{Y}\in J} \rho_j(X,s) \;
      %\frac{\j{X}^{3/2}}{\norm{Y}^{3/2}}
      %\frac{\norm{Y}}{2\j{X}^3} \frac{1}{\sqrt{\n{XY-x_1²X²}}} \;
      %\; dX \; dY \; ds\\
      &\lesssim 2^{-j/2+\epsilon j} \; \int\limits_{\norm{Y}\in J}
      \frac{1}{\norm{Y}^{1/2}\j{X}^{3/2}} \;
      \frac{1}{\sqrt{\norm{XY-x_1²X²}}} \; dX \; dY\\
      &\lesssim 2^{-j/2+\epsilon j} \; \int \big(1+\log\big(\frac{c_1}{c_0} \; 2^{2(k-j)}\big)\big) \;
      \frac{1}{\n{X}^{1/2} \j{X}^{3/2}} \; dX \lesssim 2^{-j/2+\epsilon j} \; k.
  \end{split}
\end{equation*}
holds. For the term involving $\t{\t{\zeta}}_{k,j,m}$ we get the same result. This
implies
\[
    \sum_{j-k\lesssim 1} \dnorm{F_{k,j}}_{L_1(0(x_1))}\lesssim k^2.\]
%\begin{remark}
%    On the Heisenberg group Müller and Stein derived for $F_{k,j}^\H$, the corresponding
%    term, the inequality
%   \begin{equation}
%        \sum_{j-k\lesssim 1} \dnorm{F_{k,j}}_{L_1(O^\H(0))} \lesssim k^2 \; 2^{-k/2}.
%    \end{equation}
%\end{remark}

%
%
%
%%%%%%%%%%%%%%%%%%%%%%%%%%%%%%%%%%%%%%%%%%%%%%%%%%%%%%%%%%%%%%%%%%%%%%%%%%%%%%%%%%%%%%%%%%%%%%
%%%%%%%%%%%%%%%%%%%%%%%%%%%%%%%%        G_k,j           %%%%%%%%%%%%%%%%%%%%%%%%%%%%%%%%%%%%%%
%
%
%
\subsubsection*{Estimation of $G_{k,j}$}

The proof for $G_{k,j}$ is similar to the proof for $F_{k,j}$, except for that we have
here a set of points for which the partial integration does not yields $2^{-j}$. We use
estimates for the measure of this set, to get also an additional $2^{-j/2}$.
%\ffbox{Hier wird die Abschätzung $\sum_{\l=0}^{n/2} \l^{-1/2} e^{i\sigma \l}\leq
%\sigma^{-1/2}$ benutzt!! Außerdem wird hier die partielle Integration $P_n=\d{t}R_n$
%benutzt mit dem Zusatz, das der Gewinn $2^{k-j}w$ sein sollte. (Abschätzungen für
%$2^{j-k}\lesssim 1$)}
$G_{k,j}$ is defined by
\begin{equation*}
    \begin{split}
    G_{k,j}&(x_1,x,u)=2^{k/2} \; 2^{(m-1)(j-k)}\\
    &\times \int \frac{\zeta_{j,m-1}(x_1,x,u,s)}{(R^2+(u-s)^2)^{m/2}}
    \; \frac{(2xx_1-iw)^{m}}{a^{m}} \; \rho(2^{2(k-j)}w(s)^2) \; \X(s) \; ds.
    \end{split}
\end{equation*}
with
\[\zeta_{k,j,m-1}(x_1,x,u,s)=\sum_n \X_j(m+n) \; q_{n,m-1} \; 2^{mj} \; e^{i(n+m)\kphi}.\]
Let
\[\kappa(s):=\frac{(2xx_1-iw)^{m}}{a^{2m}} \; \rho(2^{2(k-j)}w(s)²).\]
Then $a^m \; w^{-m} \; \norm{\kappa}\lesssim \norm{2xx_1-iw}^{m} \; a^{-m}
w^{-m}=w^{-m}$.
\begin{bemerkung}
  On the Heisenberg group $\H_m$, $\kappa^{\H}$ is defined as
  \[\kappa^{\H}:=(R²+(t-s)²)^{-m/2}\rho(2^{2(k-j)}a²)\]
  and we have the estimate $\kappa^\H\lesssim a^{-m}$.
\end{bemerkung}
\noindent For $s$ in the support of $\rho$ we have
\begin{equation*}
  w(s) \lesssim 2^{j-k}
\end{equation*}
%(wie sieht C genau aus?)
and thus
\begin{equation}
  \norm{Y}\lesssim 2^{j-k} <X>.
\end{equation}
Let
\begin{equation*}
  \begin{split}
  \Sigma(s)&:=2^{-j} \; \Big(\frac{2^{(k-j)/2}+w^{-1}}{\d{s}\kphi}\Big)^2
  =2^{-j} \; \Big(\frac{2^{(k-j)/2}+w^{-1}}
  {\frac{X}{Y}-\frac{2^{k-j}}{s²}}\Big)^2\\
  &\sim 2^{-j} \; \Big(\frac{(2^{(j-k)/2}+2^{j-k}w^{-1})Y}
  {Y-2^{j-k}s²X}\Big)^2
  =2^{-j} \; \Big(\frac{2^{(j-k)/2}\norm{Y}+2^{j-k}\j{X}}
  {Y-2^{j-k}s²X}\Big)^2.
  \end{split}
\end{equation*}
Since $\norm{Y}\lesssim 2^{j-k} \; <X>$ we see that
\begin{equation}
  \Sigma(s)\sim 2^{-j} \; \Big(\frac{2^{j-k}\j{X}}
  {Y-2^{j-k}s²X}\Big)^2.
\end{equation}
holds. This implies that
\begin{equation}
  \norm{Y-2^{j-k}s²X}\lesssim 2^{-j/2+j-k}\j{X},
\end{equation}
if $\Sigma(s)\geq 1$. Now, fix a cut-off function $\rho_1$ supported in $\norm{x}\leq 2$
with $\rho_1(x)=1$ for $\norm{x}\leq 1$. We write
\begin{equation*}
    \begin{split}
    \norm{G_{k,j}}&(x,x_1,u)=2^{k/2} \; 2^{(m-1)(j-k)}\\
    &\times \Bignorm{\int \zeta_{k,j,m-1}(x,x_1,u,s)\; \kappa(s)
    \big(\rho_1(\Sigma(s))+(1-\rho_1)(\Sigma(s))\big) \; \X(s) \; ds}\\
    \lesssim &G_{k,j}^1+G_{k,j}^2+G_{k,j}^3+G_{k,j}^4
    \end{split}
\end{equation*}
with
\begin{equation*}
    \begin{split}
    G_{k,j}^1&:=2^{k/2+(m-1)(j-k)} \; \int\limits_{1\leq \Sigma\leq 2}
    \Bignorm{\t{\zeta}_{k,j,m-1}(x_1,x,u,s)\; \Big(\frac{\kappa\X}{\kphi'}\Big)(s) \; 2^{-j} \; \Sigma'(s)} \; ds\\
    G_{k,j}^2&:=2^{k/2+(m-1)(j-k)} \; \int\limits_{\Sigma\leq 2}
    \Bignorm{\t{\zeta}_{k,j,m-1}(x_1,x,u,s)\;
    \; 2^{-j}\d{s}\Big(\frac{\kappa\X}{\kphi'}\Big)(s)} \; ds\\
    G_{k,j}^3&:=2^{k/2+(m-1)(j-k)} \; \int\limits_{\Sigma\geq 1}
    \Bignorm{\zeta_{k,j,m-1}(x_1,x,u,s) \; (\kappa\X)(s)} \; ds\\
    G_{k,j}^4&:=2^{k/2+(m-1)(j-k)} \; \int\limits_{\Sigma\leq 2}
    \Bignorm{\t{\t{\zeta}}_{k,j,m-1}(x_1,x,u,s) \; 2^{-j} \; \Big(\frac{\kappa\X}{w\kphi'}\Big)(s)} \; ds
  \end{split}
\end{equation*}
and
\begin{equation*}
    \begin{split}
    \t{\zeta}_{k,j,m-1}(x,x_1,u,s)&=\sum_n \X_j(m+n) \; q_{n,m-1} \; \frac{2^{(1+m)j}}{i(n+m)} \;
    e^{i(n+m)\kphi}\\
    \t{\t{\zeta}}_{k,j,m-1}(x,x_1,u,s)&=\sum_n \X_j(m+n) \; (\d{s}q_{n,m-1}) \; \frac{2^{(1+m)j}}{i(n+m)} \;
    e^{i(n+m)\kphi}
    \end{split}
\end{equation*}
Notice that
\begin{equation}
  \begin{split}
  \norm{\Sigma'}(s)&\lesssim 2^{-j/2}\Sigma^{1/2}\Big(\Bignorm{\frac{w^{-2}}{\kphi'}}
  +2^{j/2}\Sigma^{1/2}\Bignorm{\frac{\kphi''}{\kphi'}}\Big)\\
  &\lesssim (2^{j/2}\Sigma^{3/2}+2^{j}\Sigma^{2})\norm{\kphi'}\lesssim 2^j\norm{\kphi'},
  \end{split}
\end{equation}
if $\Sigma(s) \lesssim 1$, and hence $\n{G_{k,j}^1}\lesssim \n{G_{k,j}^3}$.
\subsubsection*{Estimates for $G_{k,j}^1$ and $G_{k,j}^3$.}
We get by Proposition \ref{propqn1}
\begin{equation}
  \begin{split}
  \t{\zeta}_{k,j,m-1}&\lesssim 2^{-mj} \; 2^{j/2+j\epsilon} \; \frac{a^{m}}{w^{m}} \; \frac{2^j}{(1+2^j\n{\kphi})^{1+\epsilon}}
  \sim 2^{j\epsilon} \; \frac{a^{m}}{w^{m}} \; \psi_{k,j}(X,s),\\
  \end{split}
\end{equation}
with
\[\psi_{k,j}:=\frac{2^j}{(1+2^j\n{1-e^{i\kphi}})^{1+\epsilon}}.\]
By Lemma \ref{gkj2a}(a) we deduce
\begin{equation*}
  \begin{split}
  \Bigdnorm{2^{k/2} \; &2^{(m-1)(j-k)} \; \int\limits_{\Sigma(s)\geq 1}
  \t{\zeta}_{j,m-1}(x_1,x,u,s)\; \kappa(s) \; \X(s) \; ds}\\
  \lesssim &2^{k/2+\epsilon j} \; 2^{(m-1)(j-k)}\\ &\times\int\limits_{\Sigma(s)\geq 1}
  \psi_{k,j}(X,s) \; \frac{<X>^{1/2}}{\norm{Y}^{1/2}}
  \frac{\norm{Y}}{2<X>^3} \frac{1}{\sqrt{\norm{XY-x_1²X²}}}  \;
  dX \; dY \; ds\\
  =&2^{k/2-1+\epsilon j} \; 2^{(m-1)(j-k)}\\ &\times\int\limits_{\Sigma(s)\geq 1}
  \psi_{k,j}(X,s) \; \frac{\norm{Y}^{1/2}}{<X>^{5/2}} \frac{1}{\sqrt{\norm{XY-x_1²X²}}}  \;
  dX \; dY \; ds\\
  \lesssim &2^{2\epsilon k}.
  \end{split}
\end{equation*}

%\ffbox{
%  For the Heisenberg group erhält man an dieser Stelle
%  \[2^{(m-1/2)(j-k)} \; \int \frac{\norm{X}^{m-1}}{\j{X}^{m+1}} \; dX\]
%  Für $m=1/2$ also bis auf logarithmische Terme das identische Ergebnis.}

\subsubsection*{Estimates for $G_{k,j}^2$ and $G_{k,j}^4$.}

We know by Proposition \ref{propqn1}
\begin{equation}
  \begin{split}
  \t{\zeta}_{k,j,m-1}&\lesssim 2^{mj} \; 2^{-j/2+j\epsilon} \; \frac{a^{m}}{w^{m}} \; \frac{2^j}{(1+2^j\n{\kphi})^{1+\epsilon}}
  \sim 2^{j\epsilon}\frac{a^{m}}{w^{m}} \; \psi_{k,j}(X,s)\\
  \t{\t{\zeta}}_{k,j,m-1}&\lesssim 2^{mj} \; 2^{-j/2+j\epsilon} \; \frac{a^{m+1}}{w^{m+1}}\;
    a^{-1} \; w \; \frac{2^j}{(1+2^j\n{\kphi})^{1+\epsilon}}
    \sim 2^{j\epsilon} \; \frac{a^{m}}{w^{m}} \; \psi_{k,j}(X,s),
    \end{split}
\end{equation}
with
\[\psi_{k,j}:=\frac{2^j}{(1+2^j(1-e^{i\omega}))^{1+\epsilon}}.\]
Easy calculations show
\begin{equation}
    2^{-j} \; \Big(\frac{\kappa\X}{w\kphi'}\Big)(s)\lesssim (2^{-j}\Sigma(s))^{1/2} \;
    (\kappa\X)(s)
\end{equation}
and
\[2^{-j}\d{s}\Big(\frac{\kappa\X}{\kphi'}\Big)(s)\leq
(\Sigma+(2^{-j}\Sigma)^{1/2})(s)(\kappa\X)(s).\] In this case we gain by the integration
by parts. We have the following estimate
\begin{equation*}
  2^{-j/2+j-k}<X> \lesssim \norm{Y-2^{j-k}s²X}\lesssim 2^{j-k}<X>
\end{equation*}
and hence
\begin{equation*}
  2^{k-j}<X>^{-1}\lesssim \Sigma(s)\leq 1.
\end{equation*}
This implies, by Lemma \ref{gkj2a}(b),
\begin{equation*}
    \begin{split}
    \max&\{\dnorm{G_{k,j}^2}_{L_1(O(x_1))},\dnorm{G_{k,j}^4}_{L_1(O(x_1))}\}
    \lesssim 2^{k/2+\epsilon j} \; 2^{(m-1)(j-k)}\\
    \int\limits_{\Sigma\leq 2}&
    \psi_{k,j}(X,s) \; (\Sigma+(2^{-j}\Sigma)^{1/2}) \; \X(s)\;
    \frac{\n{Y}^{1/2}}{\j{X}^{5/2}\n{X}^{1/2}} \;
    \frac{1}{\sqrt{\norm{Y-x_1²X}}} \; dX \; dY \; ds\\
    &\leq 2^{2\epsilon k}.
    \end{split}
\end{equation*}
%
%
%
%
%
%
%%%%%%%%%%%%%%%%%%%%%%%%%%%%%%%%%%%%%%%%%%%%%%%%%%%%%%%%%%%%%%%%%%%%%%%%%%%%%%%%%%%
%%%%%%%%%%%%%%%%%%%%      Abschätzungen für 2^{k-j}<<1      %%%%%%%%%%%%%%%%%%%%%%%
%%%%%%%%%%%%%%%%%%%%%%%%%%%%%%%%%%%%%%%%%%%%%%%%%%%%%%%%%%%%%%%%%%%%%%%%%%%%%%%%%%%
%G_{k,j}
%
%
%
%
%
\subsection{Estimation for $2^{k-j}\leq \const_2$}

\subsubsection*{Definition of $\Sigma(s)$}

We define the control quantity $\Sigma$ in this case by
\begin{equation*}
  \begin{split}
  \Sigma(s)&:=2^{-j} \; \Big(\frac{2^{(k-j)/2}}{\d{s}\kphi}\Big)^2
  =2^{-j} \; \Big(\frac{2^{(k-j)/2}}
  {\frac{X}{Y}-\frac{2^{k-j}}{s²}}\Big)^2
  \sim 2^{-j} \; \Big(\frac{2^{(j-k)/2}Y}
  {Y-2^{j-k}s²X}\Big)^2.
  \end{split}
\end{equation*}
\index{Sigma2}

\subsubsection*{Estimation of $G_{k,j}$}

%\begin{quote}
%$1/2\leq s \leq 2$. $2^{-M}$ muss so klein gewählt werden, dass $2^{-M/2}\leq 1/4$,
%$2^{-M}\t{x_1}\leq 1/2$ für alle $\t{x_1}$
%Hier soll $\t{x}_1 \leq C$ gelten mit einer großen Konstanten $C\geq 1$. Weiterhin ist
%hier $2^{k-j}\leq c$ mit einer sehr kleinen Konstanten $c$, so dass
%$2^{k-j}(1+\t{x}_1)\ll 1$ gilt. Hier wird die Abschätzung $\sum_{\l=0}^{n/2} \l^{-1/2}
%e^{i\sigma \l}\leq \sigma^{-1/2}$ benutzt!! Außerdem wird hier die partielle Integration
%$P_n=\d{t}R_n$ benutzt mit dem Zusatz, das der Gewinn $2^{k-j}w$ sein sollte.
%\end{quote}
Recall that in this case $G_{k,j}$ is given by
\begin{equation}
  \begin{split}
   G_{k,j}&:=2^{k/2} \; 2^{(m-1)(j-k)} \; \int
   \frac{\zeta_{k,j,m-1}(x_1,x,u,s)}{(R^2+(u-s)^2)^{m/2}} \;
   \frac{(2xx_1-iw)^{m}}{a^{m}} \; \X(s) \; ds.
  \end{split}
\end{equation}
Let
\[\kappa(s):=\frac{(2xx_1-iw)^{m}}{a^{2m}}.\]
Then $a^{m} \; w^{-m} \; \norm{\kappa}=\norm{2xx_1-iw}^{m} \; a^{-m}\; w^{-m}=w^{-m}$.
\begin{bemerkung}
  On the Heisenberg group $\H_m$, $\kappa^{\H}$ is defined as
  \[\kappa^{\H}:=(R²+(u-s)²)^{-m/2}\]
  and we have the estimate $\kappa^\H\lesssim a^{-m}$.
\end{bemerkung}
\noindent
The support of $\X$ is contained in $\intaa{1/8,32}$. By \eqref{konstante2} and
\eqref{odef}
%(\ref{reduction0})
we get for $s$ in the support of $\X$ the following estimates
\begin{equation*}
  \label{Fsupprho}
  \begin{split}
  \norm{x²-x_1²}&=\n{(x+x_1)(x-x_1)}\leq (2\const_0 2^{(k-j)/2} +4\const_1
  2^{(k-j)/2})2^{(k-j)/2+1}\const_0\\
  &\leq 2^{k-j+3}(\const_0+\const_1)^2\leq 8 \const_2 (\const_0+\const_1)^2\leq 1/32,\\
  R&=x^2+x_1^2\leq 16 \const_2(\const_0+\const_1)^2 \leq 1/16,\\
  u&\leq 2^{k-j+3} (\const_0+\const_1)^2\leq 8\const_2(\const_0+\const_1)^2 \leq 1/32,\\
  (u-s)&\sim 1, \quad w(s)=\sqrt{(x^2-x_1^2)^2+(u-s)^2} \sim 1.
  \end{split}
\end{equation*}
We get
\begin{equation*}
    \begin{split}
    \norm{X}&=\Bignorm{\frac{Rw+2xx_1(u-s)}{(u-s)w-2Rxx_1}}\lesssim \frac{Rw}{w}\lesssim 2^{k-j}<1\\
    \norm{Y}&=\Bignorm{\frac{(R²+(u-s)²)w}{(u-s)w-2Rxx_1}}\sim 1,\\
    \end{split}
\end{equation*}
which implies
\begin{equation}
  \norm{Y}\sim \j{X} \sim 1.
\end{equation}
Furthermore, $\n{\d{s}²\kphi}\lesssim 2^{k-j}$ and $\n{\d{s} \kphi} \lesssim 2^{k-j}$.
Hence easy calculations show
\begin{equation}
    2^{-j} \; \Big(\frac{\kappa\X}{w\kphi'}\Big)(s)\lesssim \frac{2^{k-2j}}{\n{\kphi'}^2}(\kappa\X)(s)=\Sigma(s) \;
    (\kappa\X)(s)
\end{equation}
and
\[2^{-j}\d{s}\Big(\frac{\kappa\X}{\kphi'}\Big)(s)\lesssim \frac{2^{k-2j}}{\n{\kphi'}^2}(\kappa\X)(s)
=\Sigma(s)\;(\kappa\X)(s).\]
%\begin{equation*}
%    2^{-j}\Big(\Bignorm{\frac{\kappa'}{\kphi'\t{\kappa}}}+\Bignorm{\frac{\X'}{\kphi'\t{\X}}}+
%    \Bignorm{\frac{\d{s}^2\kphi}{\kphi'^2}}\Big)
%    \lesssim \frac{2^{k-2j}}{\n{\kphi'}}=\Sigma(s).
%\end{equation*}
Since $\norm{Y} \sim 1$ we see that
\begin{equation}
  \Sigma(s)\sim \frac{2^{-k}}
  {\n{Y-2^{j-k}s²X}^2}
\end{equation}
holds. This implies
\begin{equation}
  \norm{Y-2^{j-k}s²X}\lesssim 2^{-k/2},
\end{equation}
if $\Sigma(s)\geq 1$. Fix a cut-off function $\rho_1$ supported in $\norm{x}\leq 2$ with
$\rho_1(x)=1$ for $\norm{x}\leq 1$. We can write $G_{k,j}$ as
\begin{equation*}
  \begin{split}
  \norm{G_{k,j}}=&2^{k/2} \; 2^{(m-1)(j-k)}\\
  &\Bignorm{\int
  \zeta_{k,j,m-1}(x_1,x,u,s)\; \kappa(s) \;
  \rho_1(\Sigma(s)+(1-\rho_1(\Sigma(s)) \; \X(s) \; ds}\\
  \lesssim &G_{k,j}^1+G_{k,j}^2+G_{k,j}^3+G_{k,j}^4
  \end{split}
\end{equation*}
with
\begin{equation*}
  \begin{split}
    G_{k,j}^1&:=2^{k/2+(m-1)(j-k)-j} \; \int\limits_{1\leq \Sigma\leq 2}
    \Bignorm{\t{\zeta}_{k,j,m-1}(x_1,x,u,s)\; \Big(\frac{\kappa\X}{\kphi'}\Big)(s) \; (\d{s}\Sigma)(s)}ds\\
    G_{k,j}^2&:=2^{k/2+(m-1)(j-k)} \; \int\limits_{\Sigma\leq 2}
    \Bignorm{\t{\zeta}_{k,j,m-1}(x_1,x,u,s) \; (\Sigma+(2^{-j}\Sigma)^{1/2})(s) \;
    (\kappa\X)(s)} \; ds\\
    G_{k,j}^3&:=2^{k/2+(m-1)(j-k)} \; \int\limits_{\Sigma\geq 1}
    \norm{\zeta_{k,j,m-1}(x_1,x,u,s) \; (\kappa\X)(s)} \; ds\\
    G_{k,j}^4&:=2^{k/2+(m-1)(j-k)} \; \int\limits_{\Sigma\leq 2}
    \Bignorm{\t{\t{\zeta}}_{k,j,m-1}(x_1,x,u,s) \; (2^{-j}\Sigma)^{1/2}(s) \; (\kappa\X)(s)} \;
    ds,
  \end{split}
\end{equation*}
and $\t{\zeta}_{k,j,m-1}$ and $\t{\t{\zeta}}_{k,j,m-1}$ defined as before. We have
$\n{G_{k,j}^1}\lesssim \n{G_{k,j}^3}$. Notice that
\begin{equation}
  \begin{split}
  \norm{\d{s}\Sigma}(s)&\lesssim \n{2^{-j}\frac{2^{k-j}}{(\d{s}\kphi)^3} \d{s}^2\kphi} \lesssim 2^j\norm{\kphi'},
  \end{split}
\end{equation}
if $\Sigma(s) \lesssim 1$.
%\begin{equation}
%  \begin{split}
%  \norm{\d{s}\Sigma}(s)&\lesssim 2^{-j/2}\Sigma^{1/2}\Big(\Bignorm{\frac{w^{-2}}{\kphi'}}
%  +2^{j/2}\Sigma^{1/2}\Bignorm{\frac{\kphi''}{\kphi'}}\Big)\\
%  &\lesssim (2^{j/2}\Sigma^{3/2}+2^{j}\Sigma^{2})\norm{\kphi'}\lesssim 2^j\norm{\kphi'},
% \quad \quad \mbox{if } \Sigma \lesssim 1.
%  \end{split}
%\end{equation}
And hence $\n{G_{k,j}^1}\lesssim \n{G_{k,j}^3}$.

\subsubsection*{Estimates for $G_{k,j}^1$ and $G_{k,j}^3$.}

In this case we have by Lemma \ref{gkj2}
\begin{equation*}
  \begin{split}
  \dnorm{G_{k,j}^3}_{L_1(O(x_1))}=
  &\Bigdnorm{2^{k/2} \; 2^{(m-1)(j-k)} \; \int\limits_{\ueber{O(x_1)}{\Sigma(s)\geq 1}}
  \zeta_{k,j,m}(x_1,x,u,s)\; \kappa(s) \; \X(s) \; ds}_{L_1(O(x_1))}\\
  \lesssim& 2^{k/2+\epsilon j} \; 2^{(m-1)(j-k)}\\
  &\times \int\limits_{\ueber{O(x_1)}{\Sigma(s)\geq 1}}
  \psi_{k,j}(X,s) \;
  \frac{\norm{Y}^{1/2}}{2<X>^{5/2}} \frac{1}{\sqrt{\norm{XY-x_1²X²}}}  \;
  dX \; dY \; ds\\
  \lesssim &2^{2\epsilon j}.
  \end{split}
\end{equation*}
In addition with the trivial estimate in Lemma \ref{trivialestimate},
\begin{equation*}
  \begin{split}
  \dnorm{G_{k,j}^3}_{L_1(O(x_1))}\lesssim &2^{k/2+\epsilon j} \; 2^{(m-1)(j-k)}\\
  &\times \int\limits_{O(x_1)} \psi_{k,j}(X,s) \;
  \frac{\norm{Y}^{1/2}}{2<X>^{5/2}} \frac{1}{\sqrt{\norm{XY-x_1²X²}}} \; dX \; dY \; ds\\
  %\lesssim &2^{k-j/2+\epsilon j} \; \int\limits_{O(x_1)}
  %\psi_{k,j}(X,s) \; \frac{1}{\sqrt{\n{X}}} \; \frac{1}{\sqrt{\norm{Y-x_1²X}}} \; dX \; dY \; ds\\
  \lesssim &2^{k-j/2+\epsilon j},
  \end{split}
\end{equation*}
we get
\begin{equation*}
  \dnorm{G_{k,j}^1}_{L_1(O(x_1))}+\dnorm{G_{k,j}^3}_{L_1(O(x_1))}\lesssim \min\{2^{2j\epsilon},
  2^{k-j/2+\epsilon j}\}.
\end{equation*}
Hence, with $2^M \sim 1/\const_2$,
\begin{equation*}
  \begin{split}
  \sum_{j>k+M} \dnorm{G_{k,j}^1}_{L_1(O(x_1))}&+\dnorm{G_{k,j}^3}_{L_1(O(x_1))}\\
  &\lesssim \sum_{j>k+M,\ j \leq 2k}
  2^{2\epsilon j}+
  \sum_{j>k+M,\ j>2k} 2^{k-j/2+\epsilon j}\\&\lesssim k \; 2^{2\epsilon k} + \sum_{j=0}^{\infty} 2^{-j/2+\epsilon j}
  \lesssim k \; 2^{k\epsilon}.
  \end{split}
\end{equation*}
%\begin{bemerkung}
%  For the Heisenberg group erhält man an dieser Stelle
%  \[2^{(m-1/2)(j-k)} \; \int \frac{\norm{X}^{m-1}}{<X>^{m+1}} \; dX\]
%  Für $m=1/2$ also bis auf logarithmische Terme das identische Ergebnis.
%\end{bemerkung}

\subsubsection*{Estimates for $G_{k,j}^2$ and $G_{k,j}^4$.}

In this case we gain by the integration by parts.
%We have the following estimate
%\begin{equation*}
%  2^{-k/2}(1+2^{(j-k)/2}) \lesssim \norm{Y-2^{j-k}s²X}\lesssim 1+2^{j-k}.
%\end{equation*}
%This implies (zumindest für $x_1=0$)
%\begin{equation*}
%  \begin{split}
%  \int\limits_{\Sigma\leq 1,\ Y \sim 1}
% &\norm{\Sigma}^{1/2} \;
% \frac{\norm{Y}^{1/2}}{\sqrt{\norm{Y-x_1²X}}} \; dY\\
% &\lesssim \int\limits_{2^{-k/2}(1+2^{(j-k)/2}) \lesssim \norm{Y-2^{j-k}X}\lesssim 1+2^{j-k}}
% \frac{2^{-k/2}(1+2^{(j-k)/2})}{\norm{Y-2^{j-k}X}} \; dY\\
% &\lesssim 2^{-k/2}(1+2^{(j-k)/2}) \; \max\{j,k\}
% \end{split}
%\end{equation*}
We deduce by Lemma \ref{gkj2}
\begin{equation*}
  \begin{split}
  \dnorm{G_{k,j}^2}_{L_1(O(x_1))}+&\dnorm{G_{k,j}^4}_{L_1(O(x_1))}\\
  \lesssim &2^{k/2+\epsilon j+(m-1)(j-k)}\\
  &\times \int\limits_{\ueber{O(x_1)}{\Sigma\leq 2}}
  \psi_{k,j}(X,s) \; \Sigma(s) \;
  \frac{\norm{Y}^{1/2}}{\sqrt{\norm{XY-x_1²X^2}}} \; dX \; dY \; ds\\
  \lesssim &2^{2\epsilon j}.
  \end{split}
\end{equation*}
%Für $j<C \; k$ ist alles ok. Diese Abschätzung ist hier leider zu schlecht, da über alle $j>k$
%summiert wird!!!
%Aber wir haben die triviale Abschätzung (\ref{lemmayint1})
Together with the trivial estimate, Lemma \ref{trivialestimate},
\begin{equation*}
  2^{k/2+\epsilon j+ (m-1)(j-k)} \; \int\limits_{\ueber{O(x_1)}{\Sigma\leq 1,\ Y \sim 1}}
  \norm{\Sigma}^{1/2} \;  \frac{\norm{Y}^{1/2}}{\sqrt{\norm{XY-x_1²X^2}}} \; dY\lesssim 2^{k-j/2+\epsilon j}
\end{equation*}
we get, with $2^M \sim 1/\const_2$,
%Das heißt wir haben auch (nun $j>C \; k$)
%\begin{equation}
%  \begin{split}
%  \dnorm{G_{k,j}^2}&\lesssim 2^{k/2+(m-1)(j-k)} \int\limits_{\Sigma\leq 2}
%  \psi_{k,j}(X,s) \; \Sigma(s) \;
%  \frac{\norm{Y}^{1/2}}{\sqrt{\norm{Y-x_1²X}}} \; dX \; dY \; ds\\
% &\lesssim \min\{2^{(k-j)/2}2^{(j-k)/2} \; j,2^{k-j/2}\}
% =\min\{j,2^{k-j/2}\}
% \end{split}
%\end{equation}
%Damit
\begin{equation*}
    \begin{split}
  \sum_{j>k+M} \dnorm{G_{k,j}^2}_{L_1(O(x_1))}+\dnorm{G_{k,j}^4}_{L_1(O(x_1))}&=\sum_{j>k+M,\ j<2k}
  2^{2\epsilon j}+\sum_{j>k+M,j>2k}
  2^{k-j/2+\epsilon j}\\
  &\lesssim k \; 2^{2\epsilon k}.
    \end{split}
\end{equation*}
%Eigentlich brauche ich, da $2^{k-j}(1+\t{x}_1)\leq 2^{-M}$ nur über $j$ summieren mit
%$j>k+M+\log{1+\t{x}_1}$. Scheint doch irgendwie zu gehen.
This completes the proof of Proposition \ref{dyadicprop2} and hence completes the proof
of Theorem \ref{theorem2}.

\newpage
\thispagestyle{myheadings} \markboth{Index of notation}{Index of notation} %\pagenumbering{Roman}

\addcontentsline{toc}{section}{Index of notation}

\noindent{\Large \bf Index of
notation}\\
\absatz

%\noindent{\textbf{Spaces of functions and distributions}}

\noindent\begin{tabular}{ll}
    \hspace{1.3cm} &\\
    $\C$ & set of complex numbers\\
    $C_{\infty}$ & set of continuous, in $\infty$ vanishing functions\\
    $C^{\infty}$ & set of smooth functions\\
    $C^{\infty}_0$ & set of compactly supported smooth function\\
    $D'$ & set of distributions\\
    $\H_1$ & Heisenberg group of dimension $3$, page 11\\
    $\H_m$ & Heisenberg group of dimension $2m+1$, page 11\\
    $L_1$ & set of integrable functions\\
    $L_p$ & set of measurable functions with $\n{f}^p \in L_1$\\
    $L_p^\alpha$ & set of measurable function $f$ with $\dnorm{f}_{L_p^\alpha}<\infty$\\
    $\N$ & set of natural numbers\\
    $\N_0$ & $\N \cup \{0\}$\\
    $\R$ & set of real numbers\\
    $\R^+$ & set of real numbers $x$ with $x\geq 0$\\
    $\S$ & set of Schwartz functions\\
    $W^s_p$ & $L_p$-Sobolev space of order $s$\\
    $\Z$ & set of integers\\
%\end{tabular}
%\noindent\begin{tabular}{ll}
    \hspace{1.3cm} &\\
    \hspace{1.3cm} &\\
    $\dnorm{\fa}_{L_p^\alpha}$ & for a differential operator $A$, $\dnorm{f}_{L_p^\alpha}=\dnorm{(1+A)^{\alpha/2}f}_p$, page 1\\
    $\dnorm{\fa}_{Schur}$ & Schur norm, page 10\\
    $a$ & $((x^2+x'^2)^2+u^2)^{1/2}$ resp. $((x^2+x'^2)^2+(u-s)^2)^{1/2}$, pages 34, 74\\
    $\alpha$ & $>1/2$, page 50\\
    $\alpha(d,p)$ & $(d-1)\n{1/p-1/2}$, page 2\\
    $B_{A}$ & ball with respect to the metric $d_{A}$, page 18\\
    $\const_0,\ \const_1,\ \const_2$ & general constants, pages 20, 26, 71\\
    $d_{A}$ & optimal control metric of the differential operator $A$, page 17\\
    $\delta_r$ & $(x,u)\mapsto(rx,r^2u)$, automorphic dilation, page 11\\
    $G$ & Gru\v{s}in operator, page 11\\
    $h^\alpha$ & page 50\\
    $L$ & sub-Laplacian on $\H_m$, page 12\\
    $m$ & $1/2$, page 26\\
    $m^\alpha(G)$  & $\exp(i\sqrt{G}) (1+G)^{-\alpha/2}$, page 26\\
    $O(x_1)$ & page 79\\
    %$\kphi(s)$ & page 71\\
    $\P_{n,\epsilon}$ & spectral projection operator belonging to the ray $\ray_{n,\epsilon}$,
    page 31\\
    $R$ & $x^2+x'^2$, pages 34, 74\\
    $\ray_{n,\epsilon}$ & ray in the joint spectrum of $G$ and $iU$, page 30
\end{tabular}
\noindent\begin{tabular}{ll}
    $\sigma$ & page 34\\
    $\Sigma$ & pages 91, 97\\
    $S_C,\ S_{\const_1}$ & pages 5, 25\\
    $U$ & $\d{u}$, page 30\\
    $w$ & $((x^2-x'^2)^2+u^2)^{1/2}$ resp. $((x^2-x'^2)^2+(u-s)^2)^{1/2}$, pages 34, 74\\
    $X,\ Y$ & new coordinates, page 75\\
    $\X_A$ & indicator function of the set $A$\\
    $\t{\X}_B$ & "smooth indicator function" associated to the ball $B$, page 59\\
\end{tabular}

%\begin{theindex}

%\end{theindex}


\newpage
\begin{thebibliography}{444}

\addcontentsline{toc}{section}{References}

\bibitem{andrews} G.~E.~Andrews, R.~Askey, R.~Roy. \emph{Special Functions},
Encyclopedia of mathematics and its applications {\bf 71}, Cambridge university press,
Cambridge, 1999.
%
\bibitem{coifman} R.~R.~Coifman, G.~Weiss. \emph{Transference methods in analysis},
Conference Board of the Mathematical Sciences, Regional conference series in mathematics
{\bf 31}, American Mathematical Society, Providence, R. I., 1977.
%
\bibitem{cowling} M.~Cowling, A.~Sikora.
\emph{A spectral multiplier theorem for a sublaplacian on $SU(2)$}, Mathematische
Zeitschrift {\bf 238} (2001), 1--36.
%
\bibitem{duistermaat} J.~J.~Duistermaat. \emph{Fourier Integral Operators}, Courant
Institute Lecture Notes, New York, 1969.
%
\bibitem{evans} L.~C.~Evans. \emph{Partial Differential Equations}, Graduate Studies in
Mathematics {\bf 19}, American Mathematical Society, Providence, R. I., 1991.
%
\bibitem{folland} G.~Folland. \emph{Harmonic Analysis in Phase Space}, Annals of Mathematical
Studies {\bf 122}, Princeton University Press, Princeton, N. J., 1989.
%
\bibitem{FoS} G.~Folland, E.~M.~Stein. \emph{Hardy Spaces on Homogeneous Groups}, Princeton
University Press, Princeton, N. J., 1982.
%
\bibitem{greiner} P.~C.~Greiner, D.~Holcman, Y.~Kannai.
\emph{Wave kernels related to second-order operators}, Duke mathematical journal {\bf
114} (2002), No. 2, 329--386.
%
\bibitem{grushin} V.~V.~Gru\v{s}in. \emph{On a class of hypoelliptic operators},
Mathematics of the USSR-Sbornik {\bf 12} (1970), No. 3, 458--476.
%
\bibitem{hebisch} W.~Hebisch. \emph{Multiplier theorem on generalized Heisenberg groups},
Colloquium Mathematicum {\bf 65} (1993), 231--239
%
\bibitem{hoermander} L.~Hörmander. \emph{Hypoelliptic second-order differential equations},
Acta Mathematica {\bf 119} (1967), 147--171.
%
\bibitem{hulanicki} A.~Hulanicki. \emph{A functional calculus for Rockland operators on
nilpotent Lie groups}, Studia Mathematica {\bf 78} (1984), 253-266.
%
%\bibitem{kempe} M.~Kempe, \emph{Restriktionssätze für getwistete Sub-Laplace-Operatoren
%und Anwendungen auf Rieszmittel}, Dissertation, Mathematisches Seminar der Universität
%Kiel, 2002.
%
\bibitem{melrose} R.~Melrose. \emph{Propagation for the wave group of a positive
subelliptic second order differential operator}, Hyperbolic equations and related topics,
Academic Press, Boston, MA, 1986.
%
\bibitem{miyachi} A.~Miyachi.
\emph{On some estimates für the wave equation in $L^p$ and $H^p$}, Journal of the faculty
of science, University of Tokyo {\bf 27} (1980), 331--354.
%
\bibitem{muellerrestriktion} D.~Müller. \emph{A restriction theorem for the Heisenberg
group}, Annals of Mathematics {\bf 131} (1990), 567-587.
%
\bibitem{muellerberlin} D.~Müller. \emph{Functional Calculus on Lie Groups and Wave
Propagation}, Proceedings of the International Congress of Mathematics, Berlin 1998, Vol.
2, Documenta Mathematica, Journal der Deutschen Mathematiker-Vereinigung (1998), 679--689
%
%\bibitem{muellerskript} D.~Müller. \emph{Die Heisenberggruppe}, notes of a lecture in Kiel, 1997
%
\bibitem{muellerricci1} D.~Müller, F.~Ricci, E.~M.~Stein.
\emph{Marcinkiewicz multipliers and multiparameter structure on Heisenberg (-type) groups
I}, Inventiones Mathematicae {\bf 119} (1995), 199--233.
%
\bibitem{muellerricci2} D.~Müller, F.~Ricci, E.~M.~Stein.
\emph{Marcinkiewicz multipliers and multiparameter structure on Heisenberg (-type) groups
II}, Mathematische Zeitschrift {\bf 221} (1996), 267--291.
%
\bibitem{mueller} D.~Müller, E.~M.~Stein.
\emph{$L^p$-estimates for the wave equation on the Heisenberg group}, Revista Matematica
Iberoamericana {\bf 15} (1999), No. 2, 297--334.
%
\bibitem{ms} D.~Müller, E.~M.~Stein.
\emph{On spectral multipliers for Heisenberg and related groups}, Journal de Mathématique
pures et aplliquées / 9.Sér. {\bf 73} (1994), No. 4, 413--440.
%
\bibitem{peral} J.~Peral. \emph{$L^p$ estimates for the wave equation},
Journal of functional analysis {\bf 36} (1980), 114--145.
%
\bibitem{ricci} F.~Ricci. \emph{A contraction of $SU(2)$ to the Heisenberg group},
Monatshefte für Mathematik  {\bf 101} (1986), No. 3, 211--225.
%
\bibitem{seeger} A.~Seeger, C.~D.~Sogge, E.~M.~Stein.
\emph{Regularity properties of Fourier integral operators}, Annals of Mathematics {\bf
134} (1991), 231--251.
%
%\bibitem{seeley} R.~Seeley, \emph{Complex powers of an elliptic operator},
%Proc. Symp. Pure Math {\bf 10} (1967), 288--307.
%
\bibitem{sogge} C.~D.~Sogge.
\emph{Fourier integrals in classical analysis}, Cambridge tracts in mathematics {\bf
105}, Cambridge University Press, Cambridge, 1993.
%
\bibitem{sjoestrand} S.~Sjöstrand.
\emph{On the Riesz means of the solutions of the Schrödinger equation}, Annali della
Scuola Normale Superiore di Pisa {\bf 24} (1970), 331--348.
%
\bibitem{steininterpolation} E.~M.~Stein. \emph{Interpolation of linear operators},
Transactions of the american mathematical Society {\bf 87} (1958), 159--172.
%
\bibitem{steinharmonic} E.~M.~Stein. \emph{Harmonic Analysis: Real-Variable Methods,
Orthogonality, and Oscillatory Integrals}, Princeton University Press, Princeton, N. J.,
1993.
%
\bibitem{strichartz} R.~S.~Strichartz.
\emph{$L^p$ Harmonic Analysis and Radon Transforms on the Heisenberg Group}, Journal of
functional analysis {\bf 96} (1991), 350--406.
%
\bibitem{strichartzg} R.~S.~Strichartz. \emph{Sub-Riemannian geometry},
Journal of differential Geometry {\bf 24} (1986), 221--263.
%
%\bibitem{taylor} M.~E.~Taylor, \emph{Noncommutative microlocal analysis, Part I},
%Memoirs of the American Mathematical Society {\bf 52}, No. 313, 1984
%
\bibitem{taylor} M.~E.~Taylor. \emph{Noncommutative Harmonic Anaylsis},
Mathematical Surveys and Monographs {\bf 22}, American Mathematical Society, Providence,
\mbox{R. I.}, 1986.
%
\bibitem{taylorpseudo} M.~E.~Taylor. \emph{Pseudodifferential Operators},
Princeton University Press, Princeton, N. J., 1981.

\end{thebibliography}
\end{document}